%% file: from_spinors_to_horospheres_a_geometric_tour.tex
\documentclass{article}

\usepackage{amsfonts}
\usepackage{amsmath}
\usepackage{amssymb}
\usepackage{amsthm}

\usepackage{authblk}

\usepackage[nottoc]{tocbibind}

\usepackage[margin=3cm]{geometry} 
\DeclareFontFamily{OT1}{pzc}{}
\DeclareFontShape{OT1}{pzc}{m}{it}{<-> s * [1.10] pzcmi7t}{}
\DeclareMathAlphabet{\mathpzc}{OT1}{pzc}{m}{it}

\usepackage{booktabs} 

\usepackage[pagebackref, pdftex]{hyperref}

\renewcommand*{\backref}[1]{}
\renewcommand*{\backrefalt}[4]{
  \ifcase #1
  [No citations.]
  \or [#2]
  \else [#2]
  \fi }

\usepackage{graphicx} 

\usepackage{tikz}
\usetikzlibrary{calc, arrows, decorations.markings, decorations.pathmorphing, positioning, decorations.pathreplacing}
\usepackage{capt-of}

\AtBeginDocument{%
   \def\MR#1{}
}

\newcommand{\To}{\longrightarrow}

\newcommand{\C}{\mathbb{C}}

\newcommand{\CP}{\mathbb{CP}}
\newcommand{\D}{\mathcal{D}}
\newcommand{\Disc}{\mathbb{D}}

\newcommand{\f}{\mathbf{f}}
\newcommand{\F}{\mathbf{F}}
\newcommand{\g}{\mathbf{g}}
\newcommand{\G}{\mathbf{G}}
\newcommand{\h}{\mathbf{h}}
\renewcommand{\H}{\mathbf{H}}
\newcommand{\horo}{\mathpzc{h}}
\newcommand{\horos}{\mathfrak{H}}
\newcommand{\HH}{\mathcal{H}}
\newcommand{\hyp}{\mathbb{H}}
\renewcommand{\i}{\mathbf{i}}
\newcommand{\I}{\mathbf{I}}
\renewcommand{\j}{\mathbf{j}}
\newcommand{\J}{\mathbf{J}}
\renewcommand{\k}{\mathbf{k}}
\newcommand{\K}{\mathbf{K}}

\newcommand{\M}{\mathcal{M}}

\newcommand{\p}{\mathbf{p}}

\newcommand{\Q}{\mathbb{Q}}

\newcommand{\R}{\mathbb{R}}

\newcommand{\RP}{\mathbb{RP}}

\renewcommand{\S}{\mathcal{S}}

\newcommand{\U}{\mathbb{U}}
\newcommand{\V}{\mathcal{V}}

\newcommand{\Z}{\mathbb{Z}}
\newcommand{\ZZ}{\mathcal{Z}}

\DeclareMathOperator{\Fr}{Fr}

\DeclareMathOperator{\Gr}{Gr}

\DeclareMathOperator{\Hopf}{Hopf}

\let\Im\relax
\DeclareMathOperator{\Im}{Im}
\let\Re\relax
\DeclareMathOperator{\Re}{Re}

\DeclareMathOperator{\Isom}{Isom}

\DeclareMathOperator{\Span}{Span}
\DeclareMathOperator{\Spin}{Spin}
\DeclareMathOperator{\Stereo}{Stereo}

\DeclareMathOperator{\Trace}{Trace}

\numberwithin{equation}{section}

\newtheorem{theorem}[equation]{Theorem}
\newtheorem{thm}{Theorem}

\newtheorem{lem}[equation]{Lemma}

\newtheorem{prop}[equation]{Proposition}

\newtheorem{defn}[equation]{Definition}
\theoremstyle{definition}
\newtheorem{eg}[equation]{Example}

\newcommand{\refsec}[1]{Section~\ref{Sec:#1}}
\newcommand{\refdef}[1]{Definition~\ref{Def:#1}}
\newcommand{\refeg}[1]{Example~\ref{Eg:#1}}
\newcommand{\reffig}[1]{Figure~\ref{Fig:#1}}

\newcommand{\refeqn}[1]{\eqref{Eqn:#1}}
\newcommand{\reflem}[1]{Lemma~\ref{Lem:#1}}
\newcommand{\refprop}[1]{Proposition~\ref{Prop:#1}}
\newcommand{\refthm}[1]{Theorem~\ref{Thm:#1}}

\begin{document}

\title{From Spinors to Horospheres: A Geometric Tour} 

\author{Daniel V. Mathews}
\affil{School of Mathematics, Monash University \\
School of Physical and Mathematical Sciences, Nanyang Technological University \\ 
\texttt{dan.v.mathews@gmail.com}}

\author{Varsha}
\affil{Department of Mathematics, 
University College London \\ 
\texttt{varsha.varsha.24@ucl.ac.uk}}

\maketitle

\begin{abstract}
This article is an exposition and elaboration of recent work of the first author on spinors and horospheres. It presents the main results in detail, and includes numerous subsidiary observations and calculations. It is intended to be accessible to graduate and advanced undergraduate students with some background in hyperbolic geometry.

The main result is the spinor--horosphere correspondence, which is a smooth, $SL(2,\C)$-equivariant bijection between two-component complex spin vectors and spin-decorated horospheres in  three-dimensional hyperbolic space. The correspondence includes constructions of Penrose--Rindler and Penner, which respectively associate null flags in Minkowski spacetime to spinors, and associate horospheres to points on the future light cone. The construction is presented step by step, proceeding from spin vectors, through spaces of Hermitian matrices and Minkowski space, to various models of 3-dimensional hyperbolic geometry. Under this correspondence, we show that the natural inner product on spinors corresponds to a 3-dimensional, complex version of lambda lengths, describing a distance between horospheres and their decorations.

We also discuss various applications of these results. An ideal hyperbolic tetrahedron with spin-decorations at its vertices obeys a Ptolemy equation, generalising the Ptolemy equation obeyed by 2-dimensional ideal quadrilaterals. More generally we discuss how  real spinors describe 2-dimensional hyperbolic geometry. We also discuss the relationships between spinors, horospheres, and various sets of matrices.
\end{abstract}

\tableofcontents

\section{Introduction}

\subsection{Overview}

At least since Descartes, mathematics has sought ways to describe geometry using algebra --- usually, though perhaps not always, in the hope that complicated geometric problems can be reduced to simpler algebraic calculations. 

In this paper we discuss a way to describe certain objects in 3-dimensional \emph{hyperbolic} geometry, called \emph{horospheres}, using pairs of complex numbers. Our use of pairs of complex numbers builds on that of Roger Penrose and Wolfgang Rindler in their book \cite{Penrose_Rindler84}, where they were considered as \emph{spinors}. Our results build on their work, so we follow their terminology.

Spinors arise in various contexts in physics. At least since Einstein, physics has sought ways to describe physical objects geometrically. From this perspective, this paper discusses how to describe spinors in terms of the geometry of horospheres. 

Horospheres are standard objects in hyperbolic geometry. Though we define them below, we do assume some background in hyperbolic geometry. However, this paper is designed to be broadly accessible, and we hope that, for readers with a little knowledge of hyperbolic geometry, reading this paper may strengthen that knowledge, and inspire them to learn more.

The goal of this paper is to explain in detail the following theorem of the first author in \cite{Mathews_Spinors_horospheres}, and some of its ramifications. The theorem says that pairs of complex numbers correspond to horospheres with some decorations on them, which we will define in due course.
\begin{thm} 
\label{Thm:spinors_to_horospheres}
There exists an explicit, smooth, bijective, $SL(2,\C)$-equivariant correspondence between nonzero spinors, and horospheres in hyperbolic 3-space $\hyp^3$ with spin decorations.
\end{thm}

So, given a pair of complex numbers $(\xi, \eta)$, what is the corresponding horosphere, and what is the decoration? We give an explicit answer in \refthm{explicit_spinor_horosphere_decoration}.

Having a bijective correspondence between two mathematical objects is good, but it is even better when that correspondence preserves various structures on each side. A particularly nice aspect the  correspondence in \refthm{spinors_to_horospheres} is that it can tell us the \emph{distance} between horospheres, and more, from some elementary operations on complex numbers. \refthm{main_thm} tells us how to do this.

A bijective correspondence between two mathematical objects is also nice when structures on one side can illuminate structures on the other. We will see various instances of this throughout the paper. One example is that, when we have four pairs of complex numbers, they obey certain equations called \emph{Pl\"{u}cker relations}. These correspond to equations relating distances between horospheres which we call \emph{Ptolemy equations}, as they have the same form as Ptolemy's theorem from classical Euclidean geometry \cite{Ptolemy_Almagest}.

The full proof of \refthm{spinors_to_horospheres} takes us on a tour through various interesting mathematical constructions. Along the way we will see, for instance, Pauli matrices from quantum mechanics, Minkowski space from relativity theory, the Hopf fibration, stereographic projection, and the hyperboloid, conformal disc, and upper half space models of hyperbolic space. It is quite a journey and in this paper we take the time to explain each step along the way, making various observations as we proceed. In this sense, this paper is a fuller exposition of \cite{Mathews_Spinors_horospheres}, with some further details, pictures, and calculations. The proof brings together several existing constructions in relativity theory and hyperbolic geometry, including the null flag construction of Penrose--Rindler in \cite{Penrose_Rindler84} and the relation of the light cone to horocycles given by Penner in \cite{Penner87}. 

It is perhaps worth noting that part of the motivation for Penrose--Rindler's work \cite{Penrose_Rindler84} was that, using their constructions, complex numbers describe structures from both quantum mechanics, and relativity theory. Such phenomena arise here where, as we will see, for instance, the Pauli matrices of quantum mechanics arise in a relativistic context, and the group $SL(2,\C)$ plays several roles, simultaneously describing linear transformations of spinors, conformal transformations of the celestial sphere (regarded as $\CP^1$), and isometries of Minkowski space (i.e. Lorentz transformations). The potential for these mathematical ideas to describe physics has been taken up in the program of \emph{twistor theory} (see e.g. \cite{Huggett_Tod94, Penrose21}). In that context, the results of this paper give a further, very concrete and explicit, geometric interpretation of spinors, that may be of relevance elsewhere. However, the constructions we consider here are prior to the notion of twistors; they only concern spinors. As far as relativity theory is concerned, it is the special theory, not the general theory.

Whatever the case, the spinor--horosphere correspondence of \refthm{spinors_to_horospheres} has already found several applications within geometry and topology, from generalising Descartes' circle theorem \cite{me_Zymaris}, to finding hyperbolic structures \cite{Mathews_Purcell_Ptolemy}, and inter-cusp distances in knot complements \cite{Howie_Mathews_et_al}.

\subsection{Horospheres and their decorations}
\label{Sec:intro_horospheres_decorations}

So, what is a horosphere? 
\begin{defn} \
\label{Def:intro_horosphere}
\begin{enumerate}
\item
A \emph{horoball} is the limit of increasing hyperbolic balls tangent to a given plane in $\hyp^3$ at a given point on a given side, as their radius tends to infinity. 
\item
A \emph{horosphere} is the boundary of a horoball.
\end{enumerate}
\end{defn}
See \reffig{horospheres_defn} for a picture of this construction. It may not be particularly informative at first instance, but horospheres appear distinctively in the various standard models of hyperbolic 3-space $\hyp^3$. In this paper we consider the hyperboloid model, which we denote $\hyp$; the conformal ball model, which we denote $\Disc$; and the upper half space model, which we denote $\U$. These are discussed in texts on hyperbolic geometry such as \cite{Anderson05, CFKP97, Iversen92, Ramsay_Richtmyer95, Ratcliffe19, Thurston97}.

\begin{center}
\begin{tabular}{cc}
 \begin{tikzpicture}[scale=0.8]
    \draw[green] (0,0) ellipse (2cm and 0.4cm);
    \fill[white] (-2,0)--(2,0)--(2,0.5)--(-2,0.5);
    \shade[ball color = green!40, opacity = 0.2] (0,0) circle (2cm);
    \draw[green] (0,0) circle (2cm);
    \draw[dashed,green] (0,0) ellipse (2cm and 0.4cm);
    \shade[ball color = red!40, opacity = 0.1] (0,1) circle (1cm);
    \draw (0,1) circle (1cm);
    \fill (0,0) circle (0.055cm);
    \shade[ball color = red!40, opacity = 0.1] (0,0.75) circle (0.75cm);
    \draw (0,0.75) circle (0.75cm);
    \shade[ball color = red!40, opacity = 0.1] (0,0.5) circle (0.5cm);
    \draw (0,0.5) circle (0.5cm);
    \shade[ball color = red!40, opacity = 0.1] (0,0.25) circle (0.25cm);
    \draw (0,0.25) circle (0.25cm);
    \fill (0,2) circle (0.055cm);
    \node[black] at (0,-1.5) {$\Disc$};
    \node at (-0.75,1.4){$\horo$};
\end{tikzpicture} 
&
\begin{tikzpicture}[scale=0.8]
    \draw[green] (-2,-0.5)--(2,-0.5)--(3,0.5)--(-1,0.5)--(-2,-0.5);
    \draw (-1,-0.5)--(0,0.5)--(0,3.5)--(-1,2.5)--(-1,-0.5);
    \fill[white] (0.5,1) circle (1cm);
    \shade[ball color = red!40, opacity = 0.1] (0.5,1) circle (1cm);
    \draw (0.5,1) circle (1cm);
    \shade[ball color = red!40, opacity = 0.1] (0.25,1) circle (0.75cm);
    \draw (0.25,1) circle (0.75cm);
    \shade[ball color = red!40, opacity = 0.1] (0,1) circle (0.5cm);
    \draw (0,1) circle (0.5cm);
    \shade[ball color = red!40, opacity = 0.1] (-0.25,1) circle (0.25cm);
    \draw (-0.25,1) circle (0.25cm);
	\fill[black] (0.5,0) circle (0.07cm);
    \fill[black] (-0.5,1) circle (0.07cm);
		\node[black] at (3,1.5) {$\U$};
		\node[black] at (1.8,-0.2) {$\C$};
        \node at (0.4,2){$\horo$};
\end{tikzpicture}\\
(a) & (b)
\end{tabular}
    
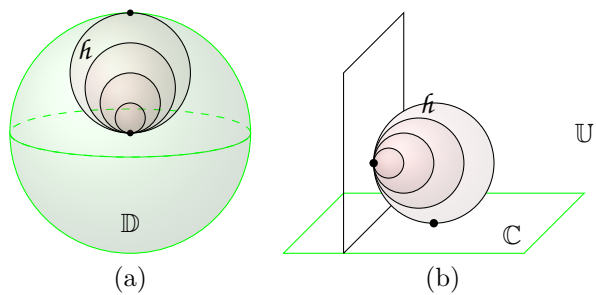
\captionof{figure}{Horosphere definition in the (a) disc model and (b) upper half space model.}
    \label{Fig:horospheres_defn}
    \end{center}

In the hyperboloid model $\hyp$, a horosphere $\horo$ appears as the intersection of the hyperboloid with an affine 3-plane whose normal lies in the light cone. Roughly speaking, such planes are ``on a 45 degree angle"; in the context of conic sections, they are the planes which intersect the cone in parabolic sections. In the conformal ball model $\Disc$, a horosphere appears as a sphere tangent to the sphere at infinity. This point at infinity is called the \emph{centre} of the horosphere. In the upper half space model $\U$, with the boundary at infinity regarded as $\C \cup \{\infty\}$ in the usual way, a horosphere appears either as a horizontal plane, if its centre is $\infty$, and otherwise a sphere tangent to $\C$ at its centre. See \reffig{horospheres}.

\begin{center}
\begin{tabular}{ccc}
\begin{tikzpicture}[scale=0.8]
  \draw (-0.2,3.7) .. controls (-1,0.25) .. (1.8,4.27);
  \fill[white] (-4,3.7)--(0,0)--(4,3.7)--(-4,3.7);
  \fill[white] (4,4)--(0,0)--(-0.75,0.75)--(1.9,4.3)--(4,4.3);
  \draw[blue] (-4,4)--(0,0)--(4,4);
  \draw[dashed, thick] plot[variable=\t,samples=1000,domain=-75.5:75.5] ({tan(\t)},{sec(\t)});
  \fill[white] (2,3)--(2.2,2.3)--(1.33,2);
  \draw[blue] (0,4) ellipse (4cm and 0.4cm);
  \draw[dotted, thick] (-0.2,3.7) .. controls (-1,0.25) .. (1.8,4.27);
  \draw (0,4) ellipse (3.85cm and 0.3cm);
  \node[blue] at (-3.5,3){$L^+$};
  \draw[dashed] (0,4) ellipse (4cm and 0.4cm);
  \draw[dashed] (0,4) ellipse (3.85cm and 0.3cm);
  \draw[dashed] (-4,4)--(0,0)--(4,4);
  \node at (-0.75,2.5){$\mathpzc{h}$};
  \node at (-2.25,3){$\hyp$};
\end{tikzpicture} & \begin{tikzpicture}[scale=0.8]
    \draw[green] (0,0) ellipse (2cm and 0.4cm);
    \fill[white] (-2,0)--(2,0)--(2,0.5)--(-2,0.5);
    \shade[ball color = green!40, opacity = 0.2] (0,0) circle (2cm);
    \draw[green] (0,0) circle (2cm);
    \draw[dashed,green] (0,0) ellipse (2cm and 0.4cm);
    \shade[ball color = red!40, opacity = 0.1] (-0.8,0.1) circle (1cm);
    \draw (-0.8,0.1) circle (1cm);
    \fill (-1.7,0.1) circle (0.055cm);
    \shade[ball color = red!40, opacity = 0.1] (1.1,-0.2) circle (0.8cm);
    \draw (1.1,-0.2) circle (0.8cm);
    \fill (1.5,-0.2) circle (0.055cm);
    \node[black] at (0,-1.5) {$\Disc$};
    \node at (-0.75,1.4){$\horo_1$};
    \node[black] at (1.1, 0.9) {$\horo_2$};
\end{tikzpicture} & \begin{tikzpicture}[scale=0.8]
    \draw[green] (-2,-0.5)--(2,-0.5)--(3,0.5)--(-1,0.5)--(-2,-0.5);
    \fill[white] (-0.1,0.5) circle (0.5cm);
    \shade[ball color = red!40, opacity = 0.1] (-0.1,0.5) circle (0.5cm);
    \draw (-0.1,0.5) circle (0.5cm);
    \draw (-2,1.5)--(2,1.5)--(3,2.5)--(-1,2.5)--(-2,1.5);
		\node[black] at (3,1.5) {$\U$};
		\node[black] at (1.8,-0.2) {$\C$};
        \node at (0.4,2){$\horo_1$};
        \node[black] at (0.7, 0.8) {$\horo_2$};

\end{tikzpicture}\\
(a) & (b) & (c)
    \end{tabular}
    
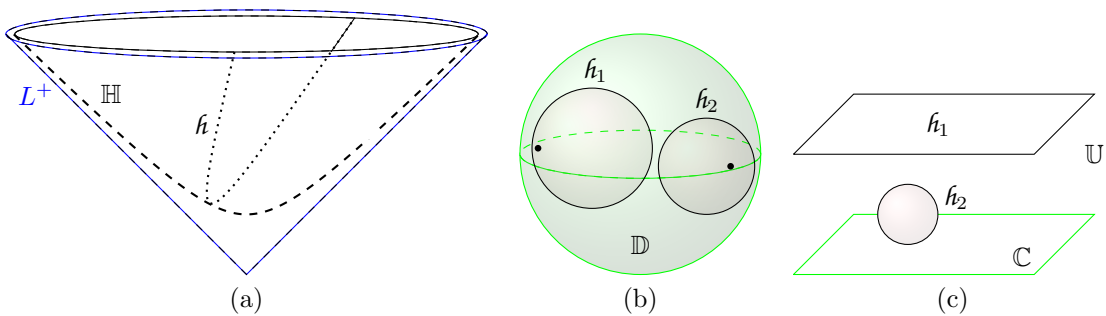
\captionof{figure}{Horospheres $\horo, \horo_1, \horo_2$ in the (a) hyperboloid model (drawn schematically, one dimension down), (b) conformal ball model and (c) upper half space model.}
    \label{Fig:horospheres}
\end{center}

As it turns out, a horosphere is isometric to the Euclidean plane. Even though hyperbolic 3-space $\hyp^3$ is negatively curved, horospheres are flat surfaces living inside $\hyp^3$. Perhaps this is most easily seen for those horospheres which appear as horizontal planes in the upper half space model $\U$. Using the standard description of $\U$ as
\begin{equation}
\label{Eqn:upper_half_space}
\U = \left\{ (x,y,z) \in \R^3 \, \mid \, z > 0 \right\}
\quad \text{with Riemannian metric} \quad
ds^2 = \frac{dx^2 + dy^2 + dz^2}{z^2},
\end{equation}
fixing $z$ to be a constant $z_0$ shows that the hyperbolic metric on the horosphere $z=z_0$ is a constant multiple of the Euclidean metric on the $xy$-plane.

The \emph{decorations} we consider on horospheres take advantage of their Euclidean geometry. If we place a tangent vector at a point on a horosphere $\horo$, we may transport it around $\horo$ by parallel translation, to obtain a \emph{parallel tangent vector field} on $\horo$. Note this cannot be done on surfaces with nonzero curvature: parallel transport of a vector around a loop will in general not result in the same vector. By the Gauss--Bonnet theorem, the vector will be rotated by an angle equal to the curvature inside the loop.

In a horosphere decoration, we are only interested in the direction of the vector, not its length. So a decoration is a \emph{parallel oriented line field}. (Alternatively, we could consider it as a parallel unit vector field.) Some decorated horospheres in the disc model and upper half space models are shown in \reffig{decorated_horospheres}.

\begin{center}
\begin{tabular}{ccc}
\begin{tikzpicture}[scale=0.8]
    \draw[green] (0,0) ellipse (2cm and 0.4cm);
    \fill[white] (-2,0)--(2,0)--(2,0.5)--(-2,0.5);
    \shade[ball color = green!40, opacity = 0.2] (0,0) circle (2cm);
    \draw[green] (0,0) circle (2cm);
    \draw[dashed,green] (0,0) ellipse (2cm and 0.4cm);
    \shade[ball color = red!40, opacity = 0.1] (-0.8,0.1) circle (1cm);
    \draw (-0.8,0.1) circle (1cm);
    \fill (-1.7,0.1) circle (0.055cm);
    \draw[->, red] (-1.7,0.1) to[out=90,in=180] (-0.7,1);
    \draw[->, red] (-1.7,0.1) to[out=60,in=180] (-0.2,0.7);
    \draw[->, red] (-1.7,0.1) to[out=30,in=150] (-0.1,0.2);
    \draw[->, red] (-1.7,0.1) to[out=0,in=135] (-0.1,-0.2);
    \draw[->, red] (-1.7,0.1) to[out=-15,in=110] (-0.4,-0.6);
    \draw[->, red] (-1.7,0.1) to[out=-30,in=90] (-0.8,-0.8);
    \draw[->, red] (-1.7,0.1) to[out=-45,in=90] (-1.3,-0.7);
\end{tikzpicture} & \begin{tikzpicture}[scale=0.8]
    \draw[green] (-2,-0.5)--(2,-0.5)--(3,0.5)--(-1,0.5)--(-2,-0.5);
    \fill[white] (-0.1,0.5) circle (0.5cm);
    \shade[ball color = red!40, opacity = 0.1] (-0.1,0.5) circle (0.5cm);
    \draw (-0.1,0.5) circle (0.5cm);
		\fill[red] (-0.1,0) circle (0.07cm);
    \draw[->, red] (-0.1,0) to[out=135,in=0] (-0.4,0.2);
    \draw[->, red] (-0.1,0) to[out=120,in=0] (-0.5,0.4);
    \draw[->, red] (-0.1,0) to[out=90,in=-45] (-0.4,0.7);
    \draw[->, red] (-0.1,0) to[out=60,in=-60] (-0.2,0.9);
    \draw[->, red] (-0.1,0) to[out=45,in=-45] (0.1,0.8);
    \draw[->, red] (-0.1,0) to[out=30,in=-90] (0.3,0.4);
    \draw (-2,1.5)--(2,1.5)--(3,2.5)--(-1,2.5)--(-2,1.5);
    \begin{scope}[xshift=0.5cm]
		\draw[red,->] (-1.1,1.7)--(-1.4,2);
    \draw[red,->] (-0.4,1.7)--(-1,2.4);
    \draw[red,->] (0.2,1.7)--(-0.4,2.4);
    \draw[red,->] (0.8,1.7)--(0.2,2.4);
    \draw[red,->] (1.2,2)--(0.8,2.4);
		\end{scope}
		\node[black] at (3,1.5) {$\U$};
		\node[black] at (1.8,-0.2) {$\C$};
\end{tikzpicture}\\
(a) & (b)
    \end{tabular}
    
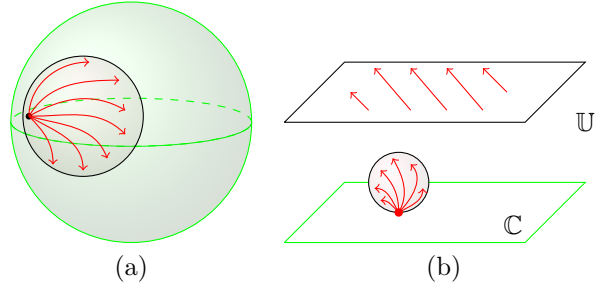
\captionof{figure}{Decorated horospheres in the (a) conformal  ball and (b) upper half space models.}
    \label{Fig:decorated_horospheres}
\end{center}

A decoration on a horosphere can be rotated through any angle. If we rotate it through an angle of $2\pi$, it returns to the same decoration. It turns out that it is possible to define a \emph{spin decoration}, which \emph{does not} return to the same decoration after rotating through $2\pi$, but \emph{does} return to the same decoration after rotation through $4\pi$. A rigorous definition is given in \refdef{spin_decoration}. It requires some technical details relating to the geometry of \emph{spin}, the same geometry that allows an electron to return to its initial state after rotating through $4\pi$, but not $2\pi$.

If we do not worry about spin, then \refthm{spinors_to_horospheres} also gives a smooth, bijective, $SL(2,\C)$-equivariant correspondence between nonzero spinors \emph{up to sign}, and decorated horospheres. The $SL(2,\C)$ action then factors through $PSL(2,\C)$. We prove this in \refprop{main_thm_up_to_sign}.

It is most convenient to describe a decorated horosphere explicitly in the upper half space model $\U$. It is common to think of the horizontal, $xy$-plane in $\U$ as the complex plane, and introduce a complex coordinate $z = x+yi$. The boundary at infinity of hyperbolic space can then be regarded as $\partial \U = \C \cup \{\infty\}$. Thus, $\U$ can alternately be described as
\[
\U = \{ (z,h) \in \C \times \R \, \mid \, h > 0 \} = \C \times \R^+.
\]
A horosphere $\horo$ in $\U$ thus has its centre in $\C \cup \{\infty\}$. If $\horo$ has centre $\infty$ then it appears as a  horizontal plane in $\U$ at some height, and because it is parallel to $\C$, directions along $\horo$ may be specified by complex numbers. If $\horo$ has centre at $z \neq \infty$, then it appears as a Euclidean sphere in $\U$, with some diameter; and at its highest point, or \emph{north pole}, its tangent space is again parallel to $\C$, so directions along $\horo$ may be specified by complex numbers. (Two complex numbers which are positive multiples of each other specify the same direction.) Because a decoration  is a \emph{parallel} oriented line field on $\horo$, if suffices to describe a decoration on $\horo$ at one point, and the north pole will suffice. Further details are given in \refsec{U_horospheres_decorations}. 

\begin{thm}
\label{Thm:explicit_spinor_horosphere_decoration}
Under the correspondence of \refthm{spinors_to_horospheres}, a  nonzero spinor $(\xi, \eta) \in \C^2$ corresponds to a horosphere $\horo$ in $\U$, centred at $\xi/\eta$, with a spin-decoration.
\begin{enumerate}
\item
If $\eta \neq 0$, then $\horo$ appears in $\U$ as a sphere with  Euclidean diameter $|\eta|^{-2}$, and its decoration is specified at the north pole by $i \eta^{-2}$.
\item 
If $\eta = 0$ then $\horo$ appears in $\U$ as a plane at height $|\xi|^2$, and its decoration is specified by $i \xi^2$.
\end{enumerate}
\end{thm}
This theorem makes \refthm{spinors_to_horospheres} explicit, and in particular locates precisely the horosphere corresponding to a spinor. See \reffig{upper_half_space_decorated_horosphere}. However, it only describes decorations, rather than spin decorations. Indeed, in \refthm{explicit_spinor_horosphere_decoration}, the spinors $\pm (\xi, \eta)$ both yield the same decorated horosphere. When spin is fully taken into account, the two spinors $(\xi,\eta)$ and $-(\xi,\eta)$ correspond to spin-decorations on the same horosphere which differ by a $2\pi$ rotation.

\begin{center}
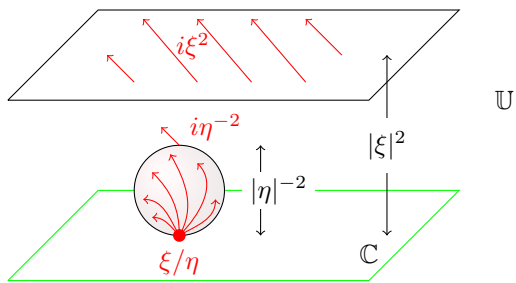

\begin{tikzpicture}[scale=1.2]
    \draw[green] (-2,-0.5)--(2,-0.5)--(3,0.5)--(-1,0.5)--(-2,-0.5);
    \fill[white] (-0.1,0.5) circle (0.5cm);
    \shade[ball color = red!40, opacity = 0.1] (-0.1,0.5) circle (0.5cm);
    \draw (-0.1,0.5) circle (0.5cm);
		\fill[red] (-0.1,0) circle (0.07cm);
    \draw[->, red] (-0.1,0) to[out=135,in=0] (-0.4,0.2);
    \draw[->, red] (-0.1,0) to[out=120,in=0] (-0.5,0.4);
    \draw[->, red] (-0.1,0) to[out=90,in=-45] (-0.4,0.7);
    \draw[->, red] (-0.1,0) to[out=60,in=-60] (-0.2,0.9);
    \draw[->, red] (-0.1,0) to[out=45,in=-45] (0.1,0.8);
    \draw[->, red] (-0.1,0) to[out=30,in=-90] (0.3,0.4);
		\draw[red, ->] (-0.1,1)--(-0.3,1.2);
		\node[red] at (0.3,1.2) {$i \eta^{-2}$};
		\node[red] at (-0.1,-0.3) {$\xi/\eta$};
		\draw[<->] (0.8,0)--(0.8,1);
		\fill[white] (0.6,0.3)--(1.4,0.3)--(1.4,0.7)--(0.6,0.7)--cycle;
		\node[black] at (1,0.5) {$|\eta|^{-2}$};
    \draw (-2,1.5)--(2,1.5)--(3,2.5)--(-1,2.5)--(-2,1.5);
    \begin{scope}[xshift=0.5cm]
		\draw[red,->] (-1.1,1.7)--(-1.4,2);
    \draw[red,->] (-0.4,1.7)--(-1,2.4);
    \draw[red,->] (0.2,1.7)--(-0.4,2.4);
    \draw[red,->] (0.8,1.7)--(0.2,2.4);
    \draw[red,->] (1.2,2)--(0.8,2.4);
		\node[red] at (-0.45,2.1) {$i \xi^2$};
		\end{scope}
		\draw[<->] (2.2,0)--(2.2,2);
		\fill[white] (1.8,0.7)--(2.6,0.7)--(2.6,1.3)--(1.8,1.3)--cycle;
		\node[black] at (2.2,1) {$|\xi|^2$};
		\node[black] at (3.5,1.5) {$\U$};
		\node[black] at (2,-0.2) {$\C$};
\end{tikzpicture}
    \captionof{figure}{Decorated horospheres in the upper half space model corresponding to spinors $\kappa = (\xi, \eta)$.}
    \label{Fig:upper_half_space_decorated_horosphere}
\end{center}

\subsection{Spinor inner product and distances between horospheres}

How can we describe the distance between two horospheres --- or even better, between two spin-decorated horospheres?

Consider two horospheres $\horo_1, \horo_2$, with centres $p_1, p_2$. Then the geodesic $\gamma$ from $p_1$ to $p_2$ intersects both horospheres orthogonally. Let the intersection points of $\gamma$ with $\horo_1, \horo_2$ be $q_1, q_2$ respectively. Assuming $\horo_1, \horo_2$ are disjoint, the shortest path from $\horo_1$ and $\horo_2$ is given by $\gamma$ from $q_1$ to $q_2$. Denote this shortest distance between the horospheres by $\rho$. 

If $\horo_1, \horo_2$ have decorations, then we can say more --- there is also an \emph{angle} between them. Precisely, the decoration on $\horo_1$ describes a direction at $q_1$, and if we parallel translate this direction along $\gamma$ to $q_2$, then there is some angle $\theta$, such that rotating the direction at $q_2$ by $\theta$ around $\gamma$ aligns the two decorations.

The angle $\theta$ between the two decorations is well defined modulo $2\pi$. If we consider \emph{spin} decorations, then the angle is well defined modulo $4\pi$. Rigorous definitions are given in \refsec{complex_lambda_lengths}. See \reffig{3}.

\begin{figure}[h]
\def\svgwidth{0.5\columnwidth}
\begin{center}
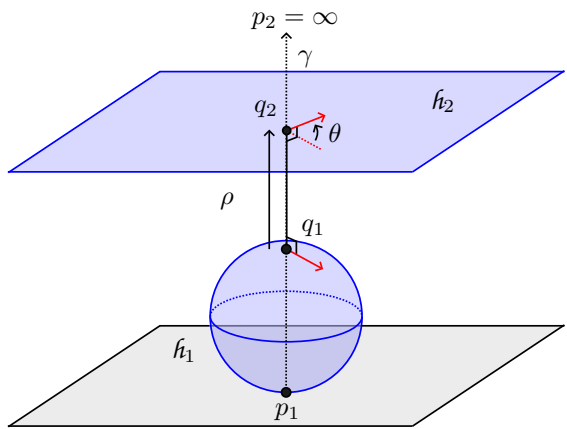 
\caption{Complex translation distance between decorated horospheres.}
\label{Fig:3}
\end{center}
\end{figure}

In this way, we can define a \emph{complex distance} $d$ between spin-decorated horospheres, given by
\[
d = \rho + i \theta.
\]
Our next theorem shows us that we can find the complex distance between two spin-decorated horospheres, from an elementary operation on the corresponding spinors.
\begin{thm}
\label{Thm:main_thm_2}
\label{Thm:main_thm}
Given two spinors $\kappa_1, \kappa_2$, with corresponding spin-decorated horospheres $\mathpzc{h}_1, \mathpzc{h}_2$, 
\[
\{\kappa_1, \kappa_2\} = \exp\left(\frac{d}{2}\right),
\]
where $\{ \cdot, \cdot \}$ is the inner product of spinors, and $d$ is the complex distance between $\mathpzc{h}_1$ and $\mathpzc{h}_2$.
\end{thm}

Thus, the complex distance --- including both the distance between horospheres, and angle between decorations --- can be calculated simply from the inner product of spinors. But what is this inner product? As it turns out, it just amounts to arranging the two complex numbers of $\kappa_1$, and the two complex numbers of $\kappa_2$, as the columns of a matrix, and taking the determinant.
\begin{defn}
\label{Def:bilinear_form_defn}
The \emph{spinor inner product} $\{ \cdot, \cdot \} \colon \C^2 \times \C^2 \To \C$ is defined for $\kappa_1 = (\xi_1,\eta_1)$ and $\kappa_2 = (\xi_2, \eta_2)$ by
\[
\left\{ \kappa_1 , \kappa_2 \right\} 
= \det (\kappa_1, \kappa_2)
= \det \begin{pmatrix} \xi_1 & \xi_2 \\  \eta_1 & \eta_2 \end{pmatrix}
= \xi_1 \eta_2 - \xi_2 \eta_1.
\]
\end{defn}
Equivalently, $\{ \cdot, \cdot \}$ can be regarded as the standard complex symplectic form on $\C^2$. If $\C^2$ has coordinates $(z_1, z_2)$, then the inner product above is (up to conventions about constants) just $dz_1 \wedge dz_2$.

We call the quantity $\exp(d/2)$ the \emph{complex lambda length} between spin-decorated horospheres, denoted $\lambda$.
\[
\lambda = \exp \left( \frac{d}{2} \right).
\]
It generalises the notion of \emph{lambda length}, defined by Penner in \cite{Penner87} as a real quantity in the 2-dimensional context. In two dimensions, one can define a distance between horocycles, but there is no angle involved. Our $\lambda$ here is a generalised, 3-dimensional, complex version of the lambda lengths from \cite{Penner87}. It is worth pointing out that the case when our spinors have \emph{real} coordinates essentially reduces to 2-dimensional geometry, though with some technicalities; and when the spinors are \emph{integers}, we can recover Ford circles: we discuss this in \refsec{real_spinors_H2}.

Note that as $\theta$ is well defined modulo $4\pi$, $d$ is well defined modulo $4\pi i$, so $d/2$ is well defined modulo $2\pi i$, and hence $\lambda = \exp (d/2)$ is well defined. However, if we drop spin and only consider decorations, then $\theta$ is only well defined modulo $2\pi$, so $d$ is only well defined modulo $2\pi i$, and $\lambda$ is then only well defined up to sign. The spinors $\kappa_1, \kappa_2$ are then also only well defined up to sign, so \refthm{main_thm_2} still holds, but with a sign ambiguity.

Although we have assumed the two horospheres $\horo_1, \horo_2$ are disjoint, in fact \refthm{main_thm} applies to any two spin-decorated horospheres. When horospheres overlap, the distance $\rho$ is well defined and negative; when they have the same centre, $\rho \rightarrow -\infty$ and $\lambda = 0$. We discuss this in \refsec{complex_lambda_lengths}. 

Taken together, \refthm{explicit_spinor_horosphere_decoration} and \refthm{main_thm} provide a powerful method for computations involving horospheres. Given a spinor, we can say precisely where the corresponding horosphere is, and what its decoration looks like. Conversely, given decorated horospheres, it is not difficult to find corresponding spinors. And given two spin-decorated horospheres, we can find the complex distance, or lambda length, between them, simply by taking a determinant.

{\flushleft \textbf{Example.} }
Consider the spinor $\kappa_1 = (1,0)$. By  \refthm{explicit_spinor_horosphere_decoration} it corresponds to the horosphere $\horo_1$ in $\U$, centred at $\infty$ --- hence a horizontal plane --- at height $1$, with decoration specified by $i$.

Similarly, $\kappa_2 = (0,1)$ corresponds to the horosphere $\horo_2$ in $\U$, centred at $0$, with Euclidean diameter $1$, and decoration specified at the north pole by $i$.

These two horospheres are tangent at $(0,0,1) \in \U$, and their decorations agree there. It turns out that their spin decorations agree too, so their complex distance is given by $d = \rho + i \theta$ where $\rho = 0$ and $\theta = 0$, i.e. $d=1$. Hence their lambda length is $\lambda = \exp(d/2) = 1$. We verify \refthm{main_thm} by checking that $\{\kappa_1, \kappa_2\} = 1$ also, given by taking the determinant of the identity matrix.

Multiplying $\kappa_1$ by $re^{i \theta}$ with $r>0$ and $\theta$ real moves the plane $\horo_1$ to height $r^2$ in $\U$, i.e. upwards by $2 \log r$, and rotates its decoration by $2\theta$. The complex distance between $\horo_1, \horo_2$ becomes $d = 2 \log r + 2 \theta i$, and we then find $\lambda = \exp(d/2) = r e^{i \theta}$, which again agrees with $\{\kappa_1, \kappa_2\}$. The situation is as in \reffig{3}.

\subsection{Equivariance}
\label{Sec:intro_equivariance}

\refthm{spinors_to_horospheres} includes a statement that the spinor--horosphere correspondence is $SL(2,\C)$-equivariant. This means that there are actions of $SL(2,\C)$ on the space $\C^2$ of spinors, and on the space of spin-decorated horospheres, and that the correspondence respects those actions.

The action of $SL(2,\C)$ on $\C^2$ is not complicated: it is just matrix-vector multiplication! It is easily computable.

The action of $SL(2,\C)$ on spin-decorated horospheres, on the other hand, is a little more subtle. The orientation-preserving isometry group of $\hyp^3$ is well known to be $PSL(2,\C)$, and this isomorphism can be made quite explicit in the upper half space model, where elements of $PSL(2,\C)$ describe M\"{o}bius transformations. Thus, $PSL(2,\C)$ acts on $\hyp^3$ by isometries, and hence also on horospheres and decorated horospheres.

However, spin decorations on horospheres live in a more complicated space. The group $SL(2,\C)$ is the double and universal cover of $PSL(2,\C)$, and can be regarded as the group of orientation-preserving isometries of $\hyp^3$ which also preserve spin structures. It is then possible to define an action of $SL(2,\C)$ on spin-decorated horospheres, and we do this precisely in \refsec{lifts_of_maps_spaces}.

The equivariance of \refthm{spinors_to_horospheres} thus means that applying an $SL(2,\C)$ linear transformation to a spinor corresponds to applying the corresponding isometry to a spin-decorated horosphere. This can be useful.

\subsection{Ptolemy equation and matrices}
\label{Sec:Ptolemy_matrices}

First appearing in Ptolemy's 2nd century \emph{Almagest} \cite{Ptolemy_Almagest}
is \emph{Ptolemy's theorem}, that in a cyclic quadrilateral $ABCD$ in the Euclidean plane one has
\[
AC \cdot BD = AB \cdot CD + AD \cdot BC.
\]

\begin{center}
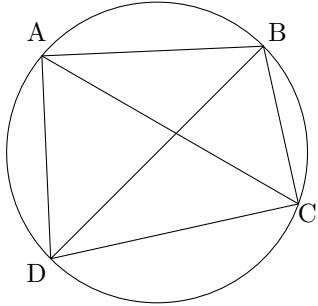

\begin{tikzpicture}
    \draw (0,0) circle (2cm);
    \draw (1.414,1.414)--(-1.532,1.285)--(-1.414,-1.414)--(1.879,-0.684)--(1.414,1.414)--(-1.414,-1.414);
    \draw (-1.532,1.285)--(1.879,-0.684);
    \node at (-1.6,1.6){A};
    \node at (1.6,1.6){B};
    \node at (2.0,-0.8){C};
    \node at (-1.6,-1.6){D};
\end{tikzpicture}\\
\captionof{figure}{Ptolemy's theorem.}
\label{Fig:Ptolemys_thm}
\end{center}

See \reffig{Ptolemys_thm}. Similar \emph{Ptolemy equations} arise in various mathematical contexts, such as representations of 3-manifold groups, e.g. \cite{GGZ15, Zickert16}, and more generally in \emph{cluster algebras}, see e.g. \cite{Fomin_Shapiro_Thurston08, Fomin_Thurston18, Williams14}. As part of their spinor algebra, Penrose--Rindler in \cite{Penrose_Rindler84} discuss an antisymmetric quantity $\varepsilon_{AB}$ describing the inner product $\{ \cdot , \cdot \}$. In particular, it obeys a Ptolemy-like equation (e.g. \cite[eq. 2.5.21]{Penrose_Rindler84}
\[
\varepsilon_{AC} \varepsilon_{BD} 
= \varepsilon_{AB} \varepsilon_{CD} 
+ \varepsilon_{AD} \varepsilon_{BC}.
\]
In our context, we obtain a Ptolemy equation as follows.
\begin{thm}
\label{Thm:main_thm_Ptolemy}
For any ideal tetrahedron in $\hyp^3$, with spin-decorated horospheres $\mathpzc{h}_i$ ($i=0,1,2,3$) about its vertices, and $\lambda_{ij}$ the lambda length between $\mathpzc{h}_i$ and $\mathpzc{h}_j$,
\begin{equation}
\label{Eqn:ptolemy}
\lambda_{02} \lambda_{13}
= \lambda_{01} \lambda_{23} + \lambda_{12} \lambda_{03}.
\end{equation}
\end{thm}
See \reffig{4}. Penner in \cite{Penner87} gave a similar equation for real lambda lengths in an ideal quadrilateral in the hyperbolic plane. \refthm{main_thm_Ptolemy} extends this result into 3 dimensions, using complex lambda lengths.

\begin{center}
    \begin{tikzpicture}[scale=2,>=stealth',pos=.8,photon/.style={decorate,decoration={snake,post length=1mm}}]
    \draw (-1,0)--(1.5,0.5);
    \fill[white] (0.75,0.35) circle (0.1 cm);
    \draw (0,1.5)--(-1,0)--(1,0)--(0,1.5)--(1.5,0.5)--(1,0);
    \draw[blue] (-0.83,0.1) circle (0.2);
    \draw[blue] (0.85,0.12) circle (0.2);
    \draw[blue] (0,1.3) circle (0.2);
    \draw[blue] (1.3,0.5) circle (0.2);
    \shade[ball color = blue!40, opacity = 0.1] (-0.83,0.1) circle (0.2cm);
    \shade[ball color = blue!40, opacity = 0.1] (0.85,0.12) circle (0.2cm);
    \shade[ball color = blue!40, opacity = 0.1] (0,1.3) circle (0.2cm);
    \shade[ball color = blue!40, opacity = 0.1] (1.3,0.5) circle (0.2cm);
    \draw[red,->] (-1,0) to[out=90,in=225] (-0.9,0.25);
    \draw[red,->] (-1,0) to[out=60,in=180] (-0.75,0.2);
    \draw[red,->] (-1,0) to[out=45,in=150] (-0.7,0.08);
    \draw[red,->] (-1,0) to[out=30,in=135] (-0.75,-0.05);
    \draw[red,->] (1,0) to[out=90,in=-45] (0.9,0.25);
    \draw[red,->] (1,0) to[out=130,in=0] (0.75,0.2);
    \draw[red,->] (1,0) to[out=135,in=60] (0.7,0.08);
    \draw[red,->] (1,0) to[out=150,in=45] (0.75,-0.05);
    \draw[red,->] (1.5,0.5) to[out=120,in=0] (1.2,0.6);
    \draw[red,->] (1.5,0.5) to[out=150,in=15] (1.15,0.5);
    \draw[red,->] (1.5,0.5) to[out=180,in=60] (1.2,0.35);
    \draw[red,->] (1.5,0.5) to[out=200,in=60] (1.3,0.34);
    \draw[red,->] (0,1.5) to[out=210,in=90] (-0.15,1.3);
    \draw[red,->] (0,1.5) to[out=225,in=90] (-0.1,1.2);
    \draw[red,->] (0,1.5) to[out=260,in=120] (0,1.15);
    \draw[red,->] (0,1.5) to[out=290,in=120] (0.1,1.2);
    \node at (-1,-0.25){1};
    \node at (1,-0.25){2};
    \node at (1.7,0.5){3};
    \node at (0,1.7){0};
    \draw [black!50!green, ultra thick, ->] (-0.5,-0.1) to [out=0, in=180] (0.5,0.1);
		\draw [black!50!green] (0,-0.2) node {$\lambda_{12}$};
		\draw [black!50!green, ultra thick, ->] (-0.4,1.1) to [out=240, in=60] (-0.6,0.4);
		\draw [black!50!green] (-0.7,0.75) node {$\lambda_{01}$};
		\draw [black!50!green, ultra thick, ->] (0.22,1) to [out=-60, in=120] (0.78,0.5);
		\draw [black!50!green] (0.4,0.65) node {$\lambda_{02}$};
		\draw [black!50!green, ultra thick, ->] (1.15,0.05) to [out=45, in=250] (1.18,0.27);
		\draw [black!50!green] (1.365,0.16) node {$\lambda_{23}$};
		\draw [black!50!green, ultra thick, ->] (0.35,1.17) to [out=-33, in=147] (1.15,0.85);
		\draw [black!50!green] (0.85,1.11) node {$\lambda_{03}$};
\end{tikzpicture}

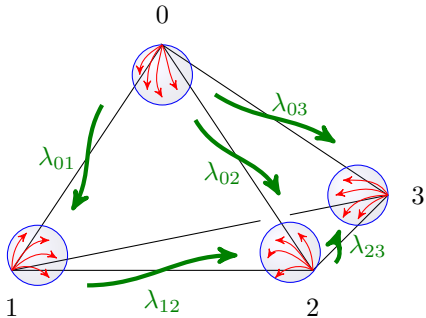
\captionof{figure}{Decorated horospheres and complex lambda lengths along the edges of an ideal tetrahedron.}
    \label{Fig:4}
\end{center}

It is perhaps more standard in 3-dimensional geometry and topology to describe hyperbolic ideal tetrahedra using \emph{shape parameters}, which are also \emph{cross-ratios} of the four ideal vertices. Shape parameters were used famously by Thurston to develop gluing and completeness equations for hyperbolic 3-manifolds \cite{Thurston_notes}. As we discuss in \refsec{shape_parameters}, from the lambda lengths of an ideal tetrahedron, one can recover the shape parameters.

The spinor--horosphere correspondence allows us to consider horospheres and their decorations via spinors, which are vectors in $\C^2$. So if we have \emph{several} spin-decorated horospheres, we then have \emph{several} vectors in $\C^2$, which can be arranged as the columns of a \emph{matrix}. We can then approach problems involving multiple horospheres, or ideal \emph{polygons} or \emph{polyhedra} by using the algebra of matrices. In a sense, \refthm{main_thm_Ptolemy} is the first result in this regard.

An ideal polyhedron in $\hyp^3$ has some number $d$ of ideal vertices. Decorating each ideal vertex with a spin-decorated horosphere, we obtain a bijective correspondence between suitably decorated ideal polyhedra, and $2 \times d$ complex matrices satisfying certain conditions. Moreover, if we want to consider such polyhedra up to \emph{isometry}, we can take a quotient by the $SL(2,\C)$ action. Taking a quotient of a space of $2 \times d$ matrices by a left action of $2 \times 2$ matrices is well known to produce \emph{Grassmannians}. So the spinor--horosphere correspondence allows us to relate spaces of polyhedra to Grassmannian-like objects built from matrices. We explore these ideas in \refsec{polygons_polyhedra_matrices}; they are also developed in \cite{Mathews_Spinors_horospheres}.
Similarly, we can relate \emph{ideal polygons} in $\hyp^2$ with $d$ ideal vertices to $2 \times d$  \emph{real} matrices. Lambda lengths are then real, and their sign can then be related to cyclic ordering around the circle at infinity; we discuss this in \refsec{spin_coherent_positivity}.

\subsection{The journey ahead: overview of proofs and constructions}

As we have mentioned, proving our main theorems involves a journey through several areas of mathematics. Let us now give an overview of where this journey will take us. Essentially, the proof of \refthm{spinors_to_horospheres} consists of carefully tracking spinors through various constructions. In \cite{Mathews_Spinors_horospheres} several steps are elided, and various spaces are implicitly identified. Here here we treat them  separately.

The journey proceeds in two stages, in \refsec{spin_vectors_to_decorated_horospheres} and \refsec{spin}. The first stage, in \refsec{spin_vectors_to_decorated_horospheres}, goes from spinors to decorated horospheres, but does not incorporate spin. The second stage, in \refsec{spin}, upgrades the spaces and maps of the first stage, to incorporate spin.

Once these two stages are complete, in \refsec{applications} we consider some applications.

\subsubsection{Pre-spin stage}

The first, or ``pre-spin" stage, in  \refsec{spin_vectors_to_decorated_horospheres}, has five steps. (In \cite{Mathews_Spinors_horospheres} they are elided to two.)

The first step goes from \emph{spinors} to \emph{Hermitian matrices}, and it is implicit when Penrose--Rindler form the expression
\[
\kappa^A \; \overline{\kappa}^{A'}.
\]
This corresponds to taking a spinor $\kappa = (\xi, \eta)$,  regarding it as a column vector, and multiplying it by its conjugate transpose $\kappa^*$. The result is a $2 \times 2$ Hermitian matrix.
\[
\kappa \kappa^* = \begin{pmatrix} \xi \\ \eta \end{pmatrix} \begin{pmatrix} \overline{\xi} & \overline{\eta} \end{pmatrix}.
\]

The second step goes from \emph{Hermitian matrices} to \emph{Minkowski space} $\R^{1,3}$, which has coordinates $(T,X,Y,Z)$ and metric $g = dT^2 - dX^2 - dY^2 - dZ^2$. The key fact is that $2 \times 2$ Hermitian matrices are precisely those which can be written in the form
\begin{equation}
\label{Eqn:spinvec_to_Hermitian}
\frac{1}{2} 
\begin{pmatrix} T+Z & X+iY \\ X-iY & T-Z \end{pmatrix}
= \frac{1}{2} \left( T \sigma_T + X \sigma_X + Y \sigma_Y + Z \sigma_Z \right)
\end{equation}
and hence such matrices can be \emph{identified} with points in $\R^{1,3}$. Here we observe the appearance of the \emph{Pauli matrices} of quantum mechanics,
\[
\sigma_T = \begin{pmatrix} 1 & 0 \\ 0 & 1 \end{pmatrix}, \quad
\sigma_X = \begin{pmatrix} 0 & 1 \\ 1 & 0 \end{pmatrix}, \quad
\sigma_Y = \begin{pmatrix} 0 & i \\ -i & 0 \end{pmatrix}, \quad
\sigma_Z = \begin{pmatrix} 1 & 0 \\ 0 & -1 \end{pmatrix}.
\]

Putting these two steps together, from a nonzero spinor we obtain a $2 \times 2$ Hermitian matrix, and then a point of $\R^{1,3}$. This construction arguably goes back much further than Penrose--Rindler, to the first uses of spinors in quantum theory. In any case, it turns out that the resulting point in Minkowski space always lies on the \emph{positive} or \emph{future light cone} $L^+$, which is given by
\[
T^2 - X^2 - Y^2 - Z^2 = 0 \quad \text{and} \quad T>0.
\]
Thus, to a spinor, our first two steps associate a point in $L^+$. This association, however, is not bijective, indeed far from it. After all, $\C^2$ is 4-dimensional, but $L^+$ is 3-dimensional. Thus Penrose--Rindler consider not just points on the light cone, but \emph{flags}. Roughly speaking, a flag consists of a \emph{point} on $L^+$, the \emph{ray} through that point, and a \emph{2-plane} containing the ray. The possible 2-planes provide an extra dimension of flexibility, and eventually provides the direction of a spin-decoration. So as it turns out, we must associate to a spinor not just a point on the light cone, but a \emph{flag}. Roughly, a flag consists of a point on the light cone (0-dimensional), the ray through it (1-dimensional), and a tangent plane (2-dimensional). See \reffig{flag}. We think of the ray as the flagpole, and the 2-plane as a flag unfurled from it!

\begin{center}
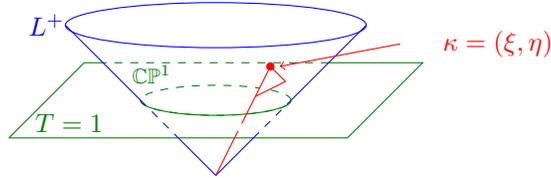

  \begin{tikzpicture}
  \draw[blue] (3.75,1.5) ellipse (2cm and 0.3cm);
  \draw[green!50!black] (3.75,0.5) ellipse (1cm and 0.2cm);
  \fill[white] (2.75,0.5)--(4.75,0.5)--(4.75,0.72)--(2.75,0.72);
  \draw[dashed, green!50!black] (3.75,0.5) ellipse (1cm and 0.2cm);
  \draw[green!50!black] (1,0)--(5.5,0)--(6.5,1)--(5.25,1);
  \draw[green!50!black] (2.25,1)--(2,1)--(1,0);
  \draw[dashed,green!50!black] (5.25,1)--(2.25,1);
  \draw[dashed,blue] (2.75,0.5)--(3.25,0);
  \draw[blue] (2.75,0.5)--(1.75,1.5);
  \draw[dashed, blue] (4.25,0)--(4.75,0.5);
  \draw[blue] (4.75,0.5)--(5.75,1.5);
  \draw[blue] (3.25,0)--(3.75,-0.5)--(4.25,0.0);
  \draw[red] (3.75,-0.5)--(4,0);
  \draw[dashed,red] (4,0)--(4.1875,0.375);
  \fill[white] (4.475,0.95)--(4.675,0.75)--(4.275,0.55);
  \draw[red] (4.1375,0.275)--(4.475,0.95)--(4.675,0.75)--(4.275,0.55);
  \node[blue] at (1.5,1.5){$L^+$};
  \fill[red] (4.475,0.95) circle (0.055cm);
  \node[red] at (7.5,1.25){$\kappa=(\xi,\eta)$};
  \draw[->,red](6.2,1.25)--(4.6,0.95);
  \node[green!50!black] at (1.8,0.2){$T=1$};
  \node[green!50!black] at (2.9,0.85){\footnotesize$\mathbb{CP}^1$};
\end{tikzpicture} 
    \captionof{figure}{A flag in Minkowski space (drawn a dimension down).}
    \label{Fig:flag}
\end{center}

However, if we are to proceed carefully and step by step, then flags in Minkowski space must come from spinors via an intermediate step in Hermitian matrices. As it turns out, we must consider flags in the space of Hermitian matrices. So the first two steps of our construction produce maps
\[
\{ \text{Spinors} \} \stackrel{\f}{\To} \{ \text{Hermitian matrices} \} \stackrel{\g}{\To} \{ \text{Future light cone in $\R^{1,3}$} \}
\]
which are then upgraded to maps
\[
\{ \text{Spinors} \} \stackrel{\F}{\To} \{ \text{Flags in Hermitian matrices} \} \stackrel{\G}{\To} \{ \text{Flags in $\R^{1,3}$} \}.
\]
These steps are carried out in \refsec{spin_vectors_to_Hermitian} to \refsec{flags}, making various observations along the way. (The composition $\g \circ \f$ is essentially the Hopf fibration under stereographic projection!) Roughly, \refsec{spin_vectors_to_Hermitian} considers the map $\f$, \refsec{hermitian_to_minkowski} considers the map $\g$, and \refsec{flags} considers flags and upgrades the maps to $\F$ and $\G$. As it turns out, each step has a ``lower case" version, which considers simpler structures, and an ``upper case" version, which includes some sort of tangent structure such as a flag or decoration. (In \cite{Mathews_Spinors_horospheres}, these two steps are elided into one, with $\f$ and $\g$ becoming $\phi_1$, and $\F, \G$ becoming $\Phi_1$.) These ideas are all in \cite{Penrose_Rindler84}; we give them a slightly different,  detailed and explicit treatment.

The third step, covered in \refsec{Minkowski_to_hyperboloid}, goes from the \emph{light cone} to \emph{horospheres in the hyperboloid model $\hyp$} of hyperbolic space, and from \emph{flags} to \emph{decorated horospheres in $\hyp$}. This step builds on a construction of Penner \cite{Penner87}, one dimension down. Given a point $p \in L^+$, we consider the 3-plane in $\R^{1,3}$ consisting of $x$ satisfying the linear equation
\begin{equation}
\label{Eqn:horosphere_eqn}
\langle p,x \rangle = 1
\end{equation}
in the Minkowski inner product. This is exactly the type of plane that intersects the hyperboloid $\hyp$ in a horosphere, and indeed it yields a map
\[
\{ \text{Future light cone in $\R^{1,3}$} \} \stackrel{\h}{\To} \{ \text{Horospheres in $\hyp$} \}.
\]
See \reffig{flag_horosphere}. It turns out that, if we also have a \emph{flag} based at the point $w$, then that flag intersects the horosphere in a way that precisely gives a decoration, and so this map can be upgraded to a map
\[
\{ \text{Flags in $\R^{1,3}$} \} \stackrel{\H}{\To} \{ \text{Decorated horospheres in $\hyp$} \}.
\]

\begin{center}
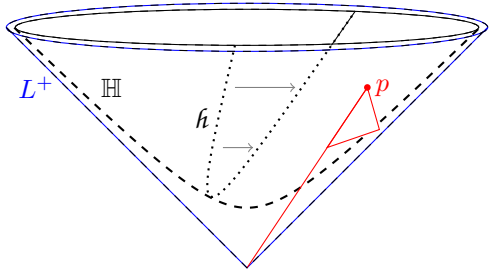

\begin{tikzpicture}[scale=0.8]
  \draw (-0.2,3.7) .. controls (-1,0.25) .. (1.8,4.27);
  \fill[white] (-4,3.7)--(0,0)--(4,3.7)--(-4,3.7);
  \fill[white] (4,4)--(0,0)--(-0.75,0.75)--(1.9,4.3)--(4,4.3);
  \draw[blue] (-4,4)--(0,0)--(4,4);
  \draw[dashed, thick] plot[variable=\t,samples=1000,domain=-75.5:75.5] ({tan(\t)},{sec(\t)});
  \fill[white] (2,3)--(2.2,2.3)--(1.33,2);
  \draw[blue] (0,4) ellipse (4cm and 0.4cm);
  \draw[dotted, thick] (-0.2,3.7) .. controls (-1,0.25) .. (1.8,4.27);
  \draw (0,4) ellipse (3.85cm and 0.3cm);
  \draw[red] (0,0)--(2,3);
  \fill[red] (2,3) circle (0.055cm);
  \node[blue] at (-3.5,3){$L^+$};
  \node[red] at (2.25,3){$p$};
  \draw[red] (2,3)--(2.2,2.3)--(1.33,2)--(2,3);
  \draw[dashed] (0,4) ellipse (4cm and 0.4cm);
  \draw[dashed] (0,4) ellipse (3.85cm and 0.3cm);
  \draw[dashed] (-4,4)--(0,0)--(4,4);
  \node at (-0.75,2.5){$\mathpzc{h}$};
  \node at (-2.25,3){$\hyp$};
  \draw[gray, ->] (-0.2,3)--(0.8,3);
  \draw[gray, ->] (-0.4,2)--(0.1,2);
\end{tikzpicture} 
    \captionof{figure}{Decorated horosphere in $\hyp$ arising from a flag (drawn a dimension down).}
    \label{Fig:flag_horosphere}
\end{center}

The fourth and fifth steps, covered in \refsec{hyperboloid_to_disc} and \refsec{Disc_to_U} respectively, are standard isometries between models of $\hyp^3$. As it turns out, for us the most straightforward route from the hyperboloid model $\hyp$ to the upper half space model $\U$ is via the conformal disc model $\Disc$. Our maps  transfer various structures between models,
\[
\{ \text{Horospheres in $\hyp$} \}
\stackrel{\i}{\To}
\{ \text{Horospheres in $\Disc$} \}
\stackrel{\j}{\To}
\{ \text{Horospheres in $\U$} \},
\]
the latter involving stereographic projection. The upper-case  versions handle decorations,
\[
\{ \text{Decorated horospheres in $\hyp$} \}
\stackrel{\I}{\To}
\{ \text{Decorated horospheres in $\Disc$} \}
\stackrel{\J}{\To}
\{ \text{Decorated Horospheres in $\U$} \}.
\]
(In \cite{Mathews_Spinors_horospheres}, all models of $\hyp^3$ are identified, so $\h, \i, \j$ are elided into $\phi_2$ and $\H, \I, \J$  into $\Phi_2$.)

Having completed these five steps, in \refsec{putting_maps_together} we put them together. We have a sequence of maps which start from a spinor, proceed to obtain a flag at a point on $L^+$, and then eventually finish up at a horosphere with a decoration. In \refprop{JIHGF_general_spin_vector} we prove \refthm{explicit_spinor_horosphere_decoration} for decorated horospheres.

Much of this story already appears in \cite{Penrose_Rindler84}, if we forget horospheres. The point $p$ on $L^+$ obtained from the spinor $\kappa = (\xi, \eta)$ yields a point on the celestial sphere $\S^+$, which is also the boundary at infinity of hyperbolic space $\partial \hyp^3$. Regarding this sphere as $\CP^1$ via stereographic projection, the point $p$ is at $\xi/\eta$; it is  the centre of the corresponding horosphere. The flag and/or decoration yields a tangent direction to $\CP^1$ at $\xi/\eta$, as discussed in \cite[ch. 1]{Penrose_Rindler84}. See \reffig{1}.

\begin{center}
\begin{tabular}{cc}
    \begin{tikzpicture}
  \draw[blue] (3.75,1.5) ellipse (2cm and 0.3cm);
  \draw[green] (3.75,0.5) ellipse (1cm and 0.2cm);
  \fill[white] (2.75,0.5)--(4.75,0.5)--(4.75,0.72)--(2.75,0.72);
  \draw[dashed, green!50!black] (3.75,0.5) ellipse (1cm and 0.2cm);
  \draw[green!50!black] (1,0)--(5.5,0)--(6.5,1)--(5.25,1);
  \draw[green!50!black] (2.25,1)--(2,1)--(1,0);
  \draw[dashed,green!50!black] (5.25,1)--(2.25,1);
  \draw[dashed,blue] (2.75,0.5)--(3.25,0);
  \draw[blue] (2.75,0.5)--(1.75,1.5);
  \draw[dashed, blue] (4.25,0)--(4.75,0.5);
  \draw[blue] (4.75,0.5)--(5.75,1.5);
  \draw[blue] (3.25,0)--(3.75,-0.5)--(4.25,0.0);
  \draw[red] (3.75,-0.5)--(4,0);
  \draw[dashed,red] (4,0)--(4.1875,0.375);
  \fill[white] (4.475,0.95)--(4.675,0.75)--(4.275,0.55);
  \draw[red] (4.1375,0.275)--(4.475,0.95)--(4.675,0.75)--(4.275,0.55);
  \node[blue] at (1.5,1.5){$L^+$};
  \fill[red] (4.475,0.95) circle (0.055cm);
  \node[red] at (7.5,1.25){$\kappa=(\xi,\eta)$};
  \draw[->,red](6.2,1.25)--(4.6,0.95);
  \node[green!50!black] at (1.8,0.2){$T=1$};
  \node[green!50!black] at (2.9,0.85){\footnotesize$\mathbb{CP}^1$};
\end{tikzpicture} & \begin{tikzpicture}
       \draw[green!50!black] (0,-0.25) ellipse (1.45cm and 0.25cm);
       \fill[white] (-1.45,-0.25)--(1.45,-0.25)--(1.45,0.05)--(-1.45,0.05);
       \draw[dashed,green!50!black] (0,-0.25) ellipse (1.45cm and 0.25cm);
       \shade[ball color = green!40, opacity = 0.1] (0,0) circle (1.5cm);
       \draw[green] (0,0) circle (1.5cm);
       \draw[dashed,green] (0,1.5)--(1,0.375);
       \draw[green!50!black] (1,0.375)--(2,-0.75);
       \fill (1,0.375) circle (0.055cm);
       \draw[->,red] (1,0.375)--(1.3,0.6);
       \draw[->,red] (2,-0.75)--(2.4,-0.7);
       \draw (-3,-0.9)--(3,-0.9)--(4,0.1)--(1.48,0.1);
       \draw[dashed] (1.48,0.1) -- (-1.48,0.1);
       \draw (-1.48,0.1)--(-2,0.1)--(-3,-0.9);
       \node[green!50!black] at (-1.4,1.2){$\mathbb{CP}^1$};
       \fill (2,-0.75) circle (0.055cm);
       \draw[<-,red] (0.9,0.375)--(-3,0.3);
       \node[red] at (2,-1.2){$\frac{\xi}{\eta}$};
       \node[red] at (2.4,-0.4){$\frac{i}{\eta^2}$};
    \end{tikzpicture}\\
    (a) & (b) 
    \end{tabular}
    
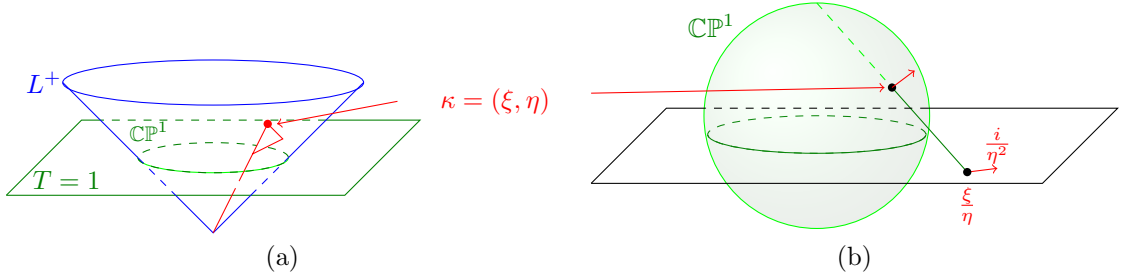
\captionof{figure}{Spinor $\kappa$ with (a) corresponding null flag, and (b) projection to $\CP^1$.}
    \label{Fig:1}
\end{center}

\subsubsection{Spin cycle}

In the second stage of our constructions, having completed the five steps of maps $\f,\g,\h,\i,\j$ and their upgrades to flags and decorations $\F,\G,\H,\I,\J$, we do need to go through the five steps in detail again. In \refsec{spin} we just upcycle them to include spin!

First there are the technicalities: we must define spin-decorated horospheres and various related notions. We do this in \refsec{spin-decorated_horospheres}.

Once this is done, in \refsec{topology_of_spaces_and_maps} we consider the topology of the maps $\F,\G,\H,\I,\J$ and spaces involved. Upcycling our maps to spin versions is essentially just lifting to universal covers, and we obtain
\begin{align*}
\{ \text{Spinors} \} &\stackrel{\widetilde{\F}}{\To} \{ \text{Spin flags in Hermitian matrices} \} \stackrel{\widetilde{\G}}{\To} \{ \text{Spin flags in $\R^{1,3}$} \} \\
& \stackrel{\widetilde{\H}}{\To} \{ \text{Spin-decorated horospheres in $\hyp$} \} 
\stackrel{\widetilde{\I}}{\To}
\{ \text{Spin-decorated horospheres in $\Disc$} \} \\
&\stackrel{\widetilde{\J}}{\To}
\{ \text{Spin-decorated Horospheres in $\U$} \}.
\end{align*}
We can then prove \refthm{spinors_to_horospheres} and \refthm{explicit_spinor_horosphere_decoration}.

It remains to prove \refthm{main_thm}. In \refsec{complex_lambda_lengths} we properly define lambda lengths, and in \refsec{proof_main_thm} we prove the theorem. 

\subsubsection{Post-spin cycle}

Having completed the spin cycle, we then examine a few applications in \refsec{applications}. \refsec{3d_hyp_geom} considers three-dimensional hyperbolic geometry, including the Ptolemy equation of \refthm{main_thm_Ptolemy}. \refsec{real_spinors_H2} considers what happens when spinors are real; we obtain some 2-dimensional hyperbolic geometry, and relations to positivity, triangulated polygons, and Ford circles and Farey fractions. \refsec{polygons_polyhedra_matrices} considers generalising to ideal hyperbolic polygons and polyhedra, and matrices built out of spinors.

\subsection{Notation}
\label{Sec:notation}

In the careful calculations and step-by-step approach of this paper, there is  unavoidably much notation. We have tried to be consistent throughout and avoid duplication of notation. We have followed some notation of Penrose--Rindler \cite{Penrose_Rindler84}, some that is standard in Minkowski geometry, and some that is standard in hyperbolic geometry; some however is probably not standard.

Throughout, complex numbers are denoted by lower case Greek letters, matrices are denoted by upper case Latin letters, and real numbers usually by lower case Latin letters. (These letters however can also denote other things.) The set of $m\times n$ matrices with entries from a set $\mathbb{F}$, is denoted $\mathcal{M}_{m\times n}(\mathbb{F})$. A ring, field or vector space $\mathbb{F}$ without its zero element is denoted $\mathbb{F}_\times$. In particular, the space of nonzero spinors $\C^2 \setminus \{(0,0)\}$ is abbreviated to $\C^2_\times$. Hyperbolic 3-space (independent of model) is denoted $\hyp^3$ and we use $\hyp, \Disc, \U$ to refer to various models.

An overline $\overline{x}$ is common to denote both complex conjugates, and elements of quotient spaces. We use both in close proximity, so to avoid potential confusion, we denote the latter by underlines. That is, $\overline{\alpha}$ is the complex conjugate of $\alpha$, and $\underline{S}$ is an element of a quotient space.

In Appendix \ref{Sec:Notation} there is a table of notation for the reader's convenience.

Unfortunately for our notation, the letter H is ubiquitous in this subject. Already in this introduction we have seen hyperbolic, hyperboloid, horospheres, Hermitian, height, $\hyp$, $\horo$, $h$, $\h$, $\H$ and $\widetilde{\H}$. There will also be $\HH$, $\mathfrak{H}$, and $\h_\partial$. We can only apologise.

\subsection{Acknowledgments}

The first author is supported by Australian Research Council grant DP210103136.

\section{From spinors to null flags to decorated horospheres}
\label{Sec:spin_vectors_to_decorated_horospheres}

In this section we establish the necessary constructions for the main theorems (without spin). We start with a definition following the terminology of \cite{Penrose_Rindler84} as we need it.
\begin{defn}
A \emph{spin vector}, or \emph{two-component spinor}, or just \emph{spinor}, is a pair of complex numbers. 
\end{defn}

\subsection{From spin vectors to Hermitian matrices}
\label{Sec:spin_vectors_to_Hermitian}

The first step in our journey goes from spin vectors to Hermitian matrices via the map $\f$. In \refsec{Hermitian_matrices_and_properties} we introduce various families of Hermitian matrices; they may seem obscure but we will see in \refsec{hermitian_to_minkowski} that they correspond to standard objects in Minkowski space. In \refsec{map_f} we define and discuss the map $\f$. In \refsec{SL2C_and_f} we discuss $SL(2,\C)$ actions and show $\f$ is $SL(2,\C)$-equivariant. Finally in \refsec{derivatives_of_f} we consider some derivatives of $\f$, motivating the need for flags.

\subsubsection{Hermitian matrices and their properties}
\label{Sec:Hermitian_matrices_and_properties}

\begin{defn} \
\begin{enumerate}
\item
The set of Hermitian matrices in $\mathcal{M}_{2\times2}(\C)$ is denoted $\HH$.
\item
$\HH_0=\{S\in\HH \, \mid \, \det S=0\}$ is the set of elements of $\HH$ with determinant zero.
\item
$\HH_0^{0+}=\{S\in\HH_0 \, \mid \, \Trace S \geq 0 \}$ is the set of elements of $\HH_0$ with non-negative trace.
\item
$\HH_0^+=\{S\in\HH_0 \, \mid \, \Trace(S)> 0 \}$ is the set of elements of $\HH_0$ with positive trace.
\end{enumerate}
\end{defn}

Observe that $\HH$ is a 4-dimensional real vector space with respect to, for instance,  the Pauli basis
\[
\sigma_T = \begin{pmatrix} 1 & 0 \\ 0 & 1 \end{pmatrix}, \quad
\sigma_X = \begin{pmatrix} 0 & 1 \\ 1 & 0 \end{pmatrix}, \quad
\sigma_Y = \begin{pmatrix} 0 & i \\ -i & 0 \end{pmatrix}, \quad
\sigma_Z = \begin{pmatrix} 1 & 0 \\ 0 & -1 \end{pmatrix}.
\]
Note however that none of  $\HH_0$, $\HH_0^{0+}$ or $\HH_0^+$ is closed under addition, hence none is a a vector space.  However, $\R$ acts on $\HH_0$ by multiplication: a real multiple of an element of $\HH_0$ again lies in $\HH_0$. Similarly,  the non-negative reals $\R^{0+}$ act on $\HH_0^{0+}$ by multiplication, and the positive reals $\R^+$ act on $\HH_0^+$ by multiplication.

We observe some basic facts about Hermitian matrices of determinant zero.
\begin{lem}
\label{Lem:H0_trace_diagonal}
For $S \in \HH_0$:
\begin{enumerate}
\item
The diagonal elements are both $\geq 0$, or both $\leq 0$.
\item
$S\in\HH_0^{0+}$ iff both diagonal entries are non-negative.
\item
$S\in\HH_0^{+}$ iff at least one diagonal entry is positive.
\item
$\HH_0^+ \subset \HH_0^{0+}$, with $\HH_0^{0+} \setminus \HH_0^+=\{0\}$.
\end{enumerate}
\end{lem}

\begin{proof}
Letting $S = \begin{pmatrix} a & b+ci \\ b-ci & d\end{pmatrix}$ where $a,b,c,d\in\R$, we observe that $\det S = ad - b^2 - c^2=0$. 
\begin{enumerate}
\item Since $ad = b^2 + c^2 \geq 0$, either $a,d \geq 0$ or $a,d \leq 0$. 
\item From (i), $\Trace S = a+d \geq0$ iff $a,d\geq 0$.
\item From (i) $\Trace S = a+d >0$ iff at least one of $a,d$ is positive. 
\item It is immediate from the definition that $\HH_0^+ \subseteq \HH_0^{0+}$. If $S \in \HH_0^{0+} \setminus \HH_0^+$ then $\det S=0=\Trace S$, so from (ii) $a=d=0$, thus $b^2+c^2 = 0$, so $b=c=0$, i.e., $S=0$.
\end{enumerate}
\end{proof}

Thus $\HH_0^{0+}$ can be defined as all $S\in\HH_0$ with both diagonal entries non-negative. Similarly $\HH_0^+$ can be defined as all $S\in\HH_0$ with one diagonal entry positive.

\subsubsection{The map from spin vectors to Hermitian matrices}
\label{Sec:map_f}

\begin{defn}
\label{Def:f}
The map $\f$ from spin vectors to Hermitian matrices is given by
\[
\f \colon \C^2 \To \HH, \quad
\f (\kappa) = \kappa \, \kappa^*.
\]
\end{defn}
Here we view $\kappa$ as a column vector, regarding $\C^2$ as $\M_{2 \times 1}(\C)$.

\begin{lem}
\label{Lem:f_surjectivity}
The map $\f$ is smooth and has the following properties:
\begin{enumerate}
\item
$\f(\C^2)=\HH_0^{0+}$.
\item
$\f(\kappa)=0$ iff $\kappa = 0$.
\item
The map $\f$ restricts surjectively to a map $\C^2_\times \To \HH_0^+$ (which we also  denote $\f$).
\end{enumerate}
\end{lem}

\begin{proof}
For general $\kappa = (\xi, \eta)$ we describe $\f$ explicitly; it is manifestly smooth.
\begin{equation}
\label{Eqn:f_formula}
\f(\xi, \eta) = 
\begin{pmatrix} \xi \\ \eta \end{pmatrix}
\begin{pmatrix} \overline{\xi} & \overline{\eta} 
\end{pmatrix}
=
\begin{pmatrix}
\xi  \overline{\xi} & \xi \overline{\eta} \\
\eta \overline{\xi} & \eta \overline{\eta}.
\end{pmatrix}
= \begin{pmatrix} 
|\xi|^2 & \xi \overline{\eta} \\
\eta \overline{\xi} & |\eta|^2
\end{pmatrix}
\end{equation}
\begin{enumerate}
\item
Observe $\f(\kappa)$ has determinant zero and trace $|\xi|^2 + |\eta|^2 \geq 0$. Thus the image of $\f$ lies in $\HH_0^{0+}$.

To see that the image is $\HH_0^{0+}$, take $S = \begin{pmatrix} a & re^{i\theta} \\ re^{-i\theta} & b \end{pmatrix} \in \HH_0^{0+}$, where  $r \geq 0$ and $a,b,\theta\in\R$. Then $ab=r^2$, and by  \reflem{H0_trace_diagonal}(ii) we have $a,b \geq 0$. Letting $\sqrt{\cdot}$ denote the non-negative square root of a non-negative real number, we may take, for example, $(\xi, \eta) = \left( \sqrt{k} e^{i\theta}, \sqrt{l} \right)$ or $\left( \sqrt{k}, \sqrt{l} e^{-i\theta} \right)$, and then $\f(\xi, \eta) = S$. 

\item Clearly $\f(0) = 0$. If $\f(\kappa) = 0$ then the diagonal elements of $\f(\kappa)$ are $|\xi|^2 = |\eta|^2 = 0$, so $\kappa=0$.

\item If $\kappa \neq 0$ then at least one of the diagonal entries of $\f(\kappa)$ is positive, so by \reflem{H0_trace_diagonal}(iii), $\f(\kappa) \in \HH_0^+$. For surjectivity, take $S \in \HH_0^+$, which by \reflem{H0_trace_diagonal}(iv) is equivalent to $S \in \HH_0^{0+}$ and $S \neq 0$. By (i) there exists  $\kappa \in \C^2$ such that $\f(\kappa) = S$. By (ii), $\kappa \neq 0$, i.e. $\kappa \in \C^2_\times$.
\end{enumerate}
\end{proof}

The map $\f$ is not injective; the next lemma describes precisely the failure of injectivity.
\begin{lem}
\label{Lem:when_f_equal}
$\f(\kappa) = \f(\kappa')$ iff $\kappa = e^{i\theta} \kappa'$ for some $\theta\in\R$.
\end{lem}

\begin{proof}
If $\kappa = e^{i \theta} \kappa'$ then we have $\f(\kappa) = \kappa \kappa^* = \left( \kappa' e^{i\theta} \right) \left( e^{-i\theta} \kappa'^* \right) = \kappa' \kappa'^* = \f(\kappa')$. 

For the converse, suppose $\f(\kappa) = \f(\kappa')$.
If $\f(\kappa) = \f(\kappa')=0$ then by \reflem{f_surjectivity}(ii) we have $\kappa = \kappa' = 0$ so the result holds trivially. Thus  we assume $\f(\kappa) = \f(\kappa')\neq0$, and hence, again using \reflem{f_surjectivity}(ii), $\kappa, \kappa' \neq (0,0)$. Let $\kappa = (\xi, \eta)$ and $\kappa' = (\xi', \eta')$. 

Considering \refeqn{f_formula} and equating diagonal entries gives $|\xi| = |\xi'|$ and $|\eta| = |\eta'|$.
We then
have $\xi = e^{i \theta} \xi'$ and $\eta = e^{i \phi} \eta'$ for some $\theta,\phi\in\R$. Thus
\[
\f(\kappa) = \begin{pmatrix} \xi \overline{\xi} & \xi \overline{\eta} \\ \eta \overline{\xi} & \eta \overline{\eta} \end{pmatrix}
= \begin{pmatrix} \xi' \overline{\xi'} & e^{i(\theta - \phi)} \xi' \overline{\eta'} \\ e^{i(\phi - \theta)} \eta' \overline{\xi'} & \eta' \overline{\eta'} \end{pmatrix}
\quad \text{while} \quad
\f(\kappa') = 
\begin{pmatrix} \xi' \overline{\xi'} & \xi' \overline{\eta'} \\ \eta' \overline{\xi'} & \eta' \overline{\eta'}
\end{pmatrix},
\]
therefore $\theta = \phi$ (mod $2\pi)$, and we have $(\xi,\eta) = e^{i\theta}(\xi',\eta')$ as desired.
\end{proof}

{\flushleft \textbf{Remark: $\f$ is the cone on the Hopf fibration.} } The \emph{Hopf fibration} is a fibration of $S^3$ as an $S^1$ bundle over $S^2$. We will discuss it in more detail in \refsec{f_compose_g} and \refsec{Hopf}, but we can see it already. 

The restriction of $\f$ to $S^3 = \{(\xi,\eta) \in \C^2 \, \mid \, |\xi|^2 + |\eta|^2 =1\}$, since it is smooth and identifies precisely those pairs $(\xi, \eta), (\xi', \eta')$ such that $(\xi, \eta) = e^{i\theta}(\xi', \eta')$, must topologically be the Hopf fibration $S^3 \To S^2$. Similarly, the restriction of $\f$ to $\C_\times^2 \cong S^3 \times \R$ is topologically the product of the Hopf fibration with the identity map on $\R$, $S^3 \times \R \To S^2 \times \R$. Extending to the full domain $\C^2$ then cones off both these spaces with the addition of a single extra point, extending $S^3 \times \R$ to $\C^2$ (the cone on $S^3$) and extending $S^2 \times \R$ to the cone on $S^2$. In other words, $\f$ is the cone on the Hopf fibration. The topology of $\HH$ and various subspaces will become clearer in \refsec{hermitian_to_minkowski} when we consider Minkowski space; see \reflem{Hermitian_topology} and surrounding discussion.


\subsubsection{$SL(2,\C)$ actions and equivariance}
\label{Sec:SL2C_and_f}

We now define $SL(2,\C)$ actions on $\C^2$ and $\HH$. We denote a general element of $SL(2,\C)$ by $A$ and a general element of $\HH$ by $S$. We denote both actions by a dot where necessary.
We already mentioned the action on $\C^2$ in the introductory \refsec{intro_equivariance}.
\begin{defn}
\label{Def:SL2C_action_on_C2}
$SL(2,\C)$ acts from the left on $\C^2$ by usual matrix-vector multiplication, $A\cdot\kappa = A \kappa$.
\end{defn}

\begin{lem}
\label{Lem:SL2C_by_symplectomorphisms}
For any $\kappa_1, \kappa_2 \in \C^2$ and $A \in SL(2,\C)$, we have
\[
\{A \cdot \kappa_1, A \cdot \kappa_2 \} = \{ \kappa_1, \kappa_2 \}.
\]
\end{lem}
In other words, the action of $SL(2,\C)$ on $\C^2$ is by symplectomorphisms, preserving the complex symplectic form $\{ \cdot, \cdot \}$.
\begin{proof}
Let $M\in\mathcal{M}_{2\times2}(\C)$ have columns $\kappa_1, \kappa_2$. Then by definition $\{ \kappa_1, \kappa_2 \} = \det M$. Further, $AM\in\mathcal{M}_{2 \times 2}(\C)$ has columns $A \kappa_1$ and $A \kappa_2$, so that $\{ A \kappa_1, A \kappa_2 \} = \det (AM)$. Since $A \in SL(2,\C)$ we have $\det A = 1$ so $\det(AM) = \det M$.
\end{proof}

\begin{defn}
\label{Def:SL2C_actions_on_C2_H}
\label{Def:standard_SL2C_actions}
$SL(2,\C)$ acts from the left on 
$\HH$ by $A\cdot S = ASA^*$.
\end{defn}
To see that we indeed have an action on $\HH$ note that $(ASA^*)^* = ASA^*$ and, for $A,A' \in SL(2,\C)$, we have 
\begin{equation}
\label{Eqn:group_action_on_Hermitian}
(AA')\cdot S = AA'S(AA')^* = AA'SA'^*A^* = A(A'SA'^*)A^* = A \cdot (A' \cdot S).
\end{equation}
Note also that, for $S,S' \in \HH$ and $a, a' \in \R$ we have
\begin{equation}
\label{Eqn:linear_action_on_Hermitian}
A \cdot \left( a S + a S' \right) 
= A \left( a S + a' S' \right) A^*
= a  ASA^* + a' AS'A^*.
= a A \cdot S + a' A \cdot S'
\end{equation}
so $SL(2,\C)$ acts by real linear maps on $\HH$.

Observe that
\begin{equation}
\label{Eqn:basic_equivariance}
\f (A\cdot\kappa) = (A\cdot\kappa)(A\cdot\kappa)^* = A \, \kappa \, \kappa^* \, A^* = A \f(\kappa) A^* = A\cdot \f(\kappa).
\end{equation}

\begin{lem}
\label{Lem:SL2C_preerves_Hs}
The action of $SL(2,\C)$ on $\HH$ restricts to actions on $\HH_0$, $\HH_0^{0+}$ and $\HH_0^+$.
\end{lem}

\begin{proof}
If $\det S = 0$ then $\det(A\cdot S) = \det(ASA^*) = \det(A) \det(S) \det(A^*) = 0$, so $\HH_0$ is preserved.

If $S \in \HH_0^{0+}$ then by \reflem{f_surjectivity}(i), $S = \f(\kappa)$ for some $\kappa$; by \refeqn{basic_equivariance} then $A \cdot S = A\cdot \f(\kappa) = \f(A\cdot\kappa)$, which by \reflem{f_surjectivity}(i) again lies in $\HH_0^{0+}$. Thus $\HH_0^{0+}$ is preserved.

If $S \in \HH_0^+$ then the same argument applies,  using \reflem{f_surjectivity}(iii) instead of (i). If $S \in \HH_0^+$ then $S = \f(\kappa)$ for some $\kappa \neq 0$. Since $A \in SL(2,\C)$, $\kappa \neq 0$ implies $A\cdot\kappa \neq 0$. Thus $A \cdot S = A \cdot \f(\kappa) = \f(A\cdot\kappa) \in \HH_0^+$ as desired.
\end{proof}

\begin{lem} \
\label{Lem:restricted_actions_on_H}
\begin{enumerate}
\item
The actions of $SL(2,\C)$ on $\C^2$ and $\HH_0^{0+}$ are equivariant with respect to $\f$.
\item
The actions of $SL(2,\C)$ on $\C^2_\times$ and $\HH_0^+$ are equivariant with respect to $\f$.
\end{enumerate}
\end{lem}

\begin{proof}
The equivariance is precisely expressed by \refeqn{basic_equivariance}.
\end{proof}

\begin{lem}
\label{Lem:SL2C_on_C2_transitive}
The action of $SL(2,\C)$ on $\C^2_\times$ is transitive. That is, for any $\kappa, \kappa' \in \C^2_\times$ there exists $A \in SL(2,\C)$ such that $A \cdot \kappa = \kappa'$.
\end{lem}
(Note the $A$ here is not unique.)

\begin{proof}
For an example of a matrix in $SL(2,\C)$ taking $(1,0)$ to $\kappa = (\xi, \eta) \in \C^2_\times$, consider
\[
A_\kappa
= \begin{pmatrix}
\xi & 0 \\
\eta & \xi^{-1}
\end{pmatrix}
\quad \text{or} \quad
\begin{pmatrix}
\xi & - \eta^{-1} \\
\eta & 0
\end{pmatrix}.
\]
As $\kappa \in \C^2_\times$, at least one of $\xi, \eta$ is nonzero, hence at least one of these matrices is well defined. Then the matrix $A_{\kappa'} A_\kappa^{-1}$ takes $\kappa$ to $\kappa'$.
\end{proof}

\subsubsection{Derivatives of $\f$}
\label{Sec:derivatives_of_f}

So far, we have associated to a spinor $\kappa\in\C^2$ a Hermitian matrix $\f(\kappa)$. We now proceed to associate to it some tangent information.

Consider the derivative of $\f$, as a \emph{real} smooth function, by regarding both $\C^2$ and $\HH$ as $\R^4$. The derivative of $\f$ at a point $\kappa = (\xi, \eta) = (a+bi,c+di) \in \C^2$ (corresponding to $(a,b,c,d) \in \R^4$) in the direction $\nu \in T_\kappa \C^2 \cong \C^2$ is given by
\[
D_\kappa \f (\nu) = \left. \frac{d}{ds} \f(\kappa+\nu s) \right|_{s=0}
\]
where $s$ is a real variable.
Regarding $\kappa,\nu\in\mathcal{M}_{2\times 1}(\C)$, we have 
\[
\f(\kappa+ \nu s) = (\kappa + \nu s)(\kappa+\nu s)^* = \kappa \kappa^* + \left( \kappa \nu^* + \nu  \kappa^* \right) s + \nu  \nu^* s^2
\]
so that
\begin{equation}
\label{Eqn:derivative_formula}
D_\kappa \f(\nu) = \kappa \nu^* + \nu\kappa^*.
\end{equation}
Since $\f$ has image in $\HH_0^{0+}\subset\HH$, and since the  tangent space to a real vector space is the space itself, this derivative lies in $\HH$, which is readily seen via the expression $\kappa \nu^* + \nu \kappa^*$.  However, while tangent vectors to $\HH_0^{0+}$ can be regarded as Hermitian matrices, these matrices do not generally lie in $\HH_0^{0+}$, and similar remarks apply to $\HH_0$ and $\HH_0^+$. Indeed, it is straightforward to check that in general $\kappa \nu^* + \nu \kappa^*$ does not lie in $\HH_0$.

Derivatives of $\f$ will be useful in the sequel and we note derivatives in some directions here.
\begin{lem}
\label{Lem:derivatives_of_f_in_easy_directions}
For any $\kappa \in C^2_\times$ we have
\[
D_\kappa \f(\kappa) = 2 \f(\kappa)
\quad \text{and} \quad
D_\kappa \f (i \kappa) = 0.
\]
\end{lem}
The first of these says that as $\kappa$ increases along a (real) ray from the origin, $\f(\kappa)$ also increases along a (real) ray from the origin. The second is equivalent to the fact from \reflem{when_f_equal} that $\f$ is constant along the circle fibres $e^{i\theta} \kappa$ over $\theta \in \R$, and $i\kappa$ is the fibre direction.
\begin{proof}
Using equation \refeqn{derivative_formula} we obtain
\begin{align*}
D_\kappa \f (\kappa) &= 2 \kappa \kappa^* = 2 \f(\kappa) \\
\D_\kappa \f (i \kappa) &= \kappa (i \kappa)^* + i \kappa \kappa^*
= \kappa \kappa^* (-i) + i \kappa \kappa^* = 0.
\end{align*}
\end{proof}

We observe that the action of $SL(2,\C)$ on $\C^2$ extends to tangent vectors $\nu$ in a standard way. If $\nu$ is tangent to $\C^2$ ($\cong \R^4$) at a point $\kappa$, and $A$ lies in  $SL(2,\C)$ (or indeed in $GL(4,\R)$), then $A\nu$ is a tangent vector to $\C^2$ at $A \kappa$. This is just the standard fact that the derivative of a linear map on a vector space is itself. Precisely, differentiating \refeqn{basic_equivariance}, we obtain
\begin{equation}
\label{Eqn:equivariance_of_derivative_of_f}
D_{A \kappa} \f ( A \nu) = A\cdot D_\kappa \f(\nu),
\end{equation}
so that the resulting action of $SL(2,\C)$ on tangent vectors is also equivariant. (Equation \refeqn{equivariance_of_derivative_of_f} also follows immediately from \refeqn{derivative_formula} and \refdef{SL2C_actions_on_C2_H}.)

Thus, to a spinor $\kappa$ and a ``tangent spinor" $\nu$ we associate a Hermitian matrix $\f(\kappa)$ and a tangent $D_\kappa \f(\nu)$. However, we want to obtain information from $\kappa$ only; and we do not want to lose any information in passing from $\kappa$ to $\f(\kappa)$ together with tangent data. We are thus interested in $\nu$ being a \emph{function} of $\kappa$. Letting
\[
\nu = \ZZ(\kappa)
\quad \text{for some real smooth function} \quad
\ZZ \colon \R^4 \To \R^4,
\]
we might then try to associate to a spinor $\kappa$ the Hermitian matrix $\f(\kappa)$ and its tangent $D_\kappa \f ( \ZZ(\kappa)) = \kappa \ZZ(\kappa)^* + \ZZ(\kappa) \kappa^*$. However, $\kappa$ is a four (real) dimensional object, and $\f$ has image in the three-dimensional space $\HH_0^{0+}$, so we can only reasonably expect one extra coordinate's worth of information from tangent data. Moreover, it will be difficult to obtain equivariance under $SL(2,\C)$. On the one hand, applying $A \in SL(2,\C)$ to $D_\kappa \f( \ZZ(\kappa) )$, we would associate to $A\kappa$ the tangent direction
\[
A \cdot D_\kappa \f(\ZZ(\kappa)) = A \left( \kappa \ZZ(\kappa)^* + \ZZ(\kappa) \kappa^* \right) A^*
\]
at $\f(A\kappa)$; but on the other hand, we would associate to $A \kappa$ the tangent direction
\[
D_{A \kappa} \f( \ZZ(A\kappa) ) = A \kappa \ZZ(A\kappa)^* + \ZZ(A\kappa) (A \kappa)^*.
\]

Penrose and Rindler describe a neat solution, providing the extra coordinate's worth of information equivariantly via a certain \emph{flag} based on $\f(\kappa)$. Such flags, however, are more easily seen in Minkowski space, and so we first introduce the map to Minkowski space.

\subsection{From Hermitian matrices to the positive light cone in Minkowski space}
\label{Sec:hermitian_to_minkowski}

Our second step is from Hermitian matrices to Minkowski space via the map $\g$ which, as mentioned in the introduction, may be described by Pauli matrices. The isomorphism $\g$ allows us to regard Hermitian matrices and Minkowski space as the same thing: for us, Hermitian matrices essentially \emph{are} points in Minkowski space. 

In \refsec{Minkowski_space_and_g} we discuss various notions in Minkowski space and the map $\g$. In \refsec{f_compose_g} we consider the composition $\g \circ \f$. In \refsec{Hopf} we discuss how $\g \circ \f$ is related to stereographic projection and the Hopf fibration. Finally, in \refsec{inner_products_spinors-Minkowski} we discuss a relationship between the inner products on spinors and Minkowski space. 

\subsubsection{Minkowski space and the map $\g$}
\label{Sec:Minkowski_space_and_g}

We start with definitions. Write points in Minkowski space as $p = (T,X,Y,Z)$, $p' = (T',X',Y',Z')$.
\begin{defn} \
\label{Def:light_cones}
\begin{enumerate}
\item
Minkowski space $\R^{1,3}$ is the 4-dimensional vector space $\R^4$, with inner product 
\[
\langle p,p' \rangle = TT' - XX' - YY' - ZZ',
\]
and the $(3+1)$-dimensional Lorentzian manifold structure on $\R^4$ with metric $ds^2 = dT^2 - dX^2 - dY^2 - dZ^2$.
\item
The \emph{light cone} $L \subset \R^{1,3}$ is
$L=\{(T,X,Y,Z) \in \R^{1,3} \, \mid \, T^2 - X^2 - Y^2 - Z^2 = 0\}$.
\item
The \emph{non-negative light cone} $L^{0+} \subset \R^{1,3}$ is $L^{0+}=\{(T,X,Y,Z) \in L \, \mid \, T \geq 0\}$.
\item
The \emph{positive light cone} $L^+ \subset \R^{1,3}$ is $L^+=\{(T,X,Y,Z) \in L \, \mid \, T>0\}$.
\end{enumerate}
\end{defn}

Clearly $L^+ \subset L^{0+} \subset L \subset \R^{1,3}$. As usual, we refer to vectors/points $p$ as \emph{timelike}, \emph{lightlike/null}, or \emph{spacelike}  accordingly as $T^2 - X^2 - Y^2 - Z^2$ is positive, zero, or negative.

\begin{defn}
\label{Def:celestial_sphere}
The \emph{(future) celestial sphere} $\S^+$ is either
\begin{enumerate}
\item
the projectivisation of $L^+$, or
\item
the intersection of the future light cone $L^+$ with the plane $T=1$ in $\R^{1,3}$.
\end{enumerate}
\end{defn}
In other words, the celestial sphere is the set of rays of $L^+$; projectivising identifies points along rays from the origin. Alternatively, we may take a subset of $L^+$ containing a single point from each ray; a standard subset given by intersecting with the 3-plane $T=1$. The two versions of $\S^+$ are related by the diffeomorphism sending each ray of $L^+$ to its point at $T=1$. We will need both versions; whenever we mention $\S^+$ we will specify which version we mean.

Since the equations $T=1$ and $T^2 - X^2 - Y^2 - Z^2 = 0$ imply $X^2 + Y^2 + Z^2 = 1$, we see $\S^+$ is diffeomorphic to $S^2$.

The isomorphism between $\HH$ and $\R^{1,3}$ is already given by  \refeqn{spinvec_to_Hermitian}. Any Hermitian matrix can be uniquely written as
\[
\begin{pmatrix}
a & b+ci \\ b-ci & d
\end{pmatrix}
\quad \text{or} \quad
\frac{1}{2} \begin{pmatrix} T+Z & X+Yi \\ X-Yi & T-Z \end{pmatrix} \]
where $a,b,c,d$ or $T,X,Y,Z$ are real, and we map to Minkowski space accordingly.
\begin{defn}
\label{Def:g_H_to_R31}
The map $\g$ from Hermitian matrices to Minkowski space is given by
\[
\g \colon \HH \To \R^{1,3}, \quad
\g \begin{pmatrix} a & b+ci \\ b-ci & d \end{pmatrix} = \left( a+d, 2b, 2c, a-d \right).
\]
\end{defn}
Since 
\[
\g^{-1} (T,X,Y,Z) = \frac{1}{2} \begin{pmatrix} T+Z & X+iY \\ X-iY & T-Z \end{pmatrix},
\]
it is  clear that $\g$ is a linear isomorphism of vector spaces, and diffeomorphism of smooth manifolds.

Under $\g$, determinant and trace become familiar expressions in Minkowski space. Our conventions perhaps produce some slightly unorthodox constants.
\begin{lem}
\label{Lem:det_trace_formulas}
Suppose $S \in \HH$ and $\g(S) = (T,X,Y,Z)$.
\begin{enumerate}
\item $4 \det S = T^2 - X^2 - Y^2 - Z^2$.
\item $\Trace S = T$.
\end{enumerate}
\end{lem}
\begin{proof}
Immediate calculation.
\end{proof}

\begin{lem}
\label{Lem:det0_lightcone_correspondence}
The isomorphism $\g \colon \HH \To \R^{1,3}$ restricts to bijections
\[
\text{(i) } \HH_0 \To L, \quad 
\text{(ii) } \HH_0^{0+} \To L^{0+}, \quad
\text{(iii) } \HH_0^+ \To L^+.
\]
\end{lem}

\begin{proof}
For (i), \reflem{det_trace_formulas}(i) shows that $\det S = 0$ iff $T^2 - X^2 - Y^2 - Z^2 = 0$. So $S \in \HH_0$ iff $\g(S) \in L$.

Suppose now that $S \in \HH_0$ and $\g(S) \in L$. By \reflem{det_trace_formulas}(ii), $\Trace S \geq 0$ iff $T \geq 0$, proving (ii). Similarly, $\Trace S > 0$ iff $T > 0$, proving  (iii).
\end{proof}

The positive light cone $L^+$ is diffeomorphic to $S^2 \times \R$;  the slice at constant $T$ is an $S^2$ with equation $X^2 + Y^2 + Z^2 = T^2$. The non-negative light cone is obtained by adding a singular point at the origin, and is the topological cone on $S^2$. The light cone $L$ is a double cone formed by joining two copies of the non-negative cone at the singular point; or alternatively by taking $S^2 \times \R$ and collapsing $S^2 \times \{0\}$ to a point. So we immediately have the following.
\begin{lem}
\label{Lem:Hermitian_topology}
$\HH_0^+ \cong L^+$ is diffeomorphic to $S^2 \times \R$, $\HH_0^{0+} \cong L^{0+}$ is a cone on $S^2$, and $\HH_0 \cong L$ is a double cone on $S^2$.
\qed
\end{lem}

The action of $SL(2,\C)$ on $\HH$ naturally gives an action on $\R^{1,3}$, defining it to be equivariant under the linear diffeomorphism $\g$. This is a standard action.
\begin{defn}
\label{Def:SL2C_on_R31}
$SL(2,\C)$ acts on $\R^{1,3}$ by
\[
A\cdot p = \g \left( A\cdot (\g^{-1} (p)) \right) \quad \text{for $A \in SL(2,\C)$ and $p \in \R^{1,3}$.}
\]
\end{defn}
Thus by definition $A\cdot \g(p) = \g (A\cdot p)$ and explicitly, for $p = (T,X,Y,Z)$,
\begin{equation}
\label{Eqn:SL2C_action_on_R31}
A\cdot (T,X,Y,Z)
=
\g \left( A\cdot \frac{1}{2} \begin{pmatrix} T+Z & X+iY \\ X-iY & T-Z \end{pmatrix} \right)
= \frac{1}{2} \, \g \left( A \begin{pmatrix} T+Z & X+iY \\ X-iY & T-Z \end{pmatrix}  A^* \right)
\end{equation}

\begin{lem}
\label{Lem:SL2C_action_on_light_cones}
For any $A \in SL(2,\C)$, the action of $A$ on $\R^{1,3}$ is a linear map $T_A \colon \R^{1,3} \To \R^{1,3}$ which preserves $L$, $L^{0+}$ and $L^+$.
\end{lem}

\begin{proof}
We have already seen in \refeqn{linear_action_on_Hermitian} that, for given $A \in SL(2,\C)$ the action of $A$ on $\HH$ is a linear map $\HH \To \HH$; since $\g$ and $\g^{-1}$ are  linear, $T_A$ is also a linear map $\R^{1,3} \To \R^{1,3}$.

By \reflem{SL2C_preerves_Hs}, the action of $A$ on $\HH$ preserves $\HH_0$, $\HH_0^{0+}$ and $\HH_0^+$; thus, applying the linear diffeomorphism $\g$ and \reflem{det0_lightcone_correspondence}, the action of $A$ on $\R^{1,3}$ preserves $L, L^{0+}$ and $L^+$.
\end{proof}

The linear maps on $\R^{1,3}$ preserving $L^+$ are precisely those in $O(1,3)^+$, i.e. those which preserve the Lorentzian inner product and are orthochronous (preserve the direction of time). 

The linear maps $T_A$ in fact lie in $SO(1,3)^+$, i.e. are also orientation-preserving. We can observe this directly by noting that the generators of $SL(2,\C)$
\[
\begin{pmatrix} re^{i\theta} & 0 \\ 0 & \frac{1}{r} e^{-i\theta} \end{pmatrix},
\quad
\begin{pmatrix} 1 & a+bi \\ 0 & 1 \end{pmatrix}, 
\quad
\begin{pmatrix} 1 & 0 \\ a+bi & 1 \end{pmatrix}
\]
(where $a,b,r,\theta\in\R$) map to $T_A$ given respectively by
\[
\begin{pmatrix}
\frac{r^2+r^{-2}}{2} & 0 & 0 & \frac{r^2-r^{-2}}{2} \\
0 & \cos 2\theta & -\sin 2\theta & 0 \\
0 & \sin 2\theta & \cos 2\theta & 0 \\
\frac{r^2-r^{-2}}{2} & 0 & 0 & \frac{r^2+r^{-2}}{2}
\end{pmatrix}, 
\quad
\begin{pmatrix}
1+\frac{a^2+b^2}{2} & a & b & -\frac{a^2+b^2}{2} \\
a & 1 & 0 & -a \\
b & 0 & 1 & -b \\
\frac{a^2+b^2}{2} & a & b & 1-\frac{a^2+b^2}{2}
\end{pmatrix},
\quad
\begin{pmatrix}
1+\frac{a^2+b^2}{2} & -a & -b & \frac{a^2+b^2}{2} \\
a & 1 & 0 & a \\
-b & 0 & 1 & -b \\
-\frac{a^2+b^2}{2} & a & b & 1-\frac{a^2+b^2}{2}
\end{pmatrix}
\]
which all have determinant $1$.

\subsubsection{Putting $\f$ and $\g$ together}
\label{Sec:f_compose_g}

We now compose $\f$ and $\g$,
\[
\C^2 \stackrel{\f}{\To} \HH \stackrel{\g}{\To} \R^{1,3}.
\]
This composition sends a spinor $\kappa$ to the point $(T,X,Y,Z) \in \R^{1,3}$ such that
\begin{equation}
\label{Eqn:Pauli_Hermitian}
\kappa \, \kappa^* = \frac{1}{2} \left( T \sigma_T + X \sigma_X + Y \sigma_Y + Z \sigma_Z \right).
\end{equation}
We consider some properties of this composition, and perform some calculations.
\begin{lem}
\label{Lem:gof_properties}
The map $\g \circ \f \colon \C^2 \To \R^{1,3}$ is smooth and has the following properties.
\begin{enumerate}
\item
$\g \circ \f (\kappa) = 0$ precisely when $\kappa = 0$.
\item
The image of $\g \circ \f$ is $L^{0+}$.
\item
$\g \circ \f$ restricts to a surjective map $\C_\times^2 \To L^+$.
\item
$\g \circ \f(\kappa) = \g \circ \f(\kappa')$ iff $\kappa = e^{i\theta} \kappa'$ for some real $\theta$.
\item
The actions of $SL(2,\C)$ on $\C^2$ and $\R^{1,3}$ are equivariant with respect to $\g \circ \f$. These actions restrict to actions on $\C_\times^2$ and $L, L^+, L^{0+}$ which are also appropriately equivariant.
\end{enumerate}
\end{lem}

\begin{proof}
Immediate from \reflem{f_surjectivity}, \reflem{when_f_equal}, \reflem{restricted_actions_on_H} and \reflem{det0_lightcone_correspondence}.
\end{proof}

We can calculate $\g \circ \f$ explicitly, and prove some of its properties. For the rest of this subsection, let $\kappa = (\xi, \eta) = (a+bi,c+di) \in \C^2$, where $a,b,c,d \in \R$.
\begin{lem}
\label{Lem:spin_vector_to_TXYZ}
Let $\g \circ \f(\kappa) = (T,X,Y,Z)$. Then
\begin{align*}
T &= |\xi|^2 + |\eta|^2 = a^2 + b^2 + c^2 + d^2 \\
X &= 2 \Re \left( \xi \overline{\eta} \right) = 2 \, |\eta|^2 \, \Re (\xi/\eta) = 2(ac+bd) \\
Y &= 2 \Im \left( \xi \overline{\eta} \right) = 2 \, |\eta|^2 \, \Im (\xi/\eta) = 2(bc-ad) \\
Z &= |\xi|^2 - |\eta|^2 = a^2+b^2-c^2-d^2.
\end{align*}
\end{lem}

\begin{proof}
From \refeqn{f_formula} we have 
\begin{equation}
\label{Eqn:f_kappa_in_real_coords}
\f(\kappa) = 
\begin{pmatrix} 
\xi \overline{\xi} & \xi \overline{\eta} \\ 
\eta \overline{\xi} & \eta \overline{\eta}.
\end{pmatrix}
=
\begin{pmatrix}
a^2 + b^2 & (ac+bd)+(bc-ad)i \\
(ac+bd)-(bc-ad)i & c^2 + d^2
\end{pmatrix}
\end{equation}
Applying the definition of $\g$ from \refdef{g_H_to_R31} and the fact $\overline{\eta} = \eta^{-1} \, |\eta|^2$ then gives the claim.
\end{proof}

We already noted in \refsec{map_f} that $\f$ is the cone on the Hopf fibration. In Minkowski space, the picture is perhaps a little more intuitive, and we can add some explicit details.

\begin{lem} 
\label{Lem:C2_to_R31_Hopf_fibrations}
Let $S^3_r = \{ \kappa \in \C^2 \, \mid \, |\xi|^2 + |\eta|^2 = r^2 \}$ be the 3-sphere of radius $r>0$ in $\C^2 \cong \R^4$, and let $S^3 = S^3_1$.
\begin{enumerate}
\item
The restriction of $\g \circ \f$ to each $S^3_r$ yields a surjective map from $S^3_r$ onto the 2-sphere $L^+ \cap \{ T=r^2 \} = r^2 \S^+ \cong S^2$ which is the Hopf fibration. In particular, the restriction to $S^3$ yields a Hopf fibration onto the celestial sphere $S^3 \To \S^+ \cong S^2$.
\item
The map $\g \circ \f \colon \C^2 \To L^{0+}$ is the cone on the Hopf fibration.
\end{enumerate}
\end{lem}
In (i) we regard $\S^+$ as $L^+ \cap \{T=1\}$, i.e. \refdef{celestial_sphere}(ii).

\begin{proof}
In \refsec{map_f} we saw that, since $\f(\kappa) = \f(\kappa')$ iff $\kappa = e^{i \theta} \kappa'$, $\f$ is a smooth map on each $S^3_r$ collapsing each fibre of the Hopf fibration to a point, so is the Hopf fibration. As $\g$ is a diffeomorphism, the same is true for $\g \circ \f$.

By \reflem{spin_vector_to_TXYZ}, $\g \circ \f (\xi, \eta)$ has $T$-coordinate $|\xi|^2 + |\eta|^2 = r^2$, and by \reflem{gof_properties}(iii), $\g \circ \f (\C^2_\times) = L^{+}$. So the image of $S^3_r$ under $\g \circ \f$ is the intersection of $L^{+}$ with $T=r^2$, as claimed. 

Thus, the family of $3$-spheres $S^3_r$ foliating $\C^2_\times$ are mapped under $\g \circ \f$ by Hopf fibrations to the family of $2$-spheres $L^+ \cap \{T=1\}$ foliating $L^+$. See \reffig{cone_on_Hopf}. Hence we can regard the restriction of $\g \circ \f$ to $\C_\times^2$ as the product of the Hopf fibration with the identity map, $\C^2_\times \cong S^3 \times \R \To S^2 \times \R \cong L^+$.

\begin{center}
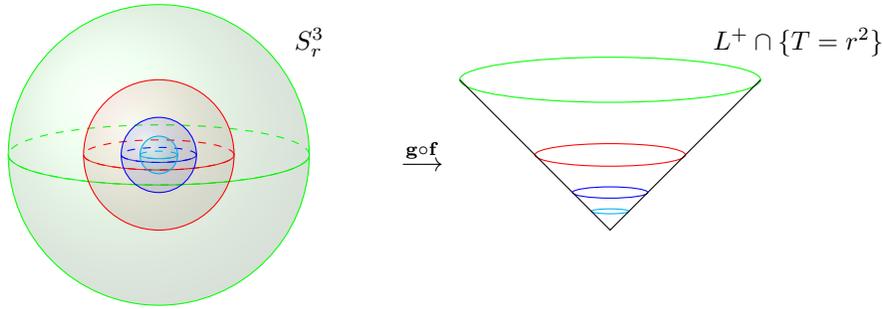

\begin{tikzpicture}
    \draw[green] (0,0) ellipse (2cm and 0.4cm);
    \fill[white] (-2,0)--(2,0)--(2,0.5)--(-2,0.5);
    \draw[red] (0,0) ellipse (1cm and 0.2cm);
    \fill[white] (-1,0)--(1,0)--(1,0.5)--(-1,0.5);
    \draw[blue] (0,0) ellipse (0.5cm and 0.1cm);
    \fill[white] (-0.5,0)--(0.5,0)--(0.5,0.5)--(-0.5,0.5);
    \draw[cyan] (0,0) ellipse (0.25cm and 0.05cm);
    \fill[white] (-0.25,0)--(0.25,0)--(0.25,0.5)--(-0.25,0.5);
    \shade[ball color = green!40, opacity = 0.2] (0,0) circle (2cm);
    \draw[green] (0,0) circle (2cm);
    \draw[dashed,green] (0,0) ellipse (2cm and 0.4cm);
    \shade[ball color = red!80, opacity = 0.1] (0,0) circle (1cm);
    \draw[red] (0,0) circle (1cm);
    \draw[dashed,red] (0,0) ellipse (1cm and 0.2cm);
    \shade[ball color = blue!160, opacity = 0.1] (0,0) circle (0.5cm);
    \draw[blue] (0,0) circle (0.5cm);
    \draw[dashed,blue] (0,0) ellipse (0.5cm and 0.1cm);
    \shade[ball color = cyan!320, opacity = 0.1] (0,0) circle (0.25cm);
    \draw[dashed,cyan] (0,0) ellipse (0.25cm and 0.05cm);
    \draw[cyan] (0,0) circle (0.25cm);
    \node[black] at (2,1.5) {$S_r^3$};
    \draw[green] (6,1) ellipse (2cm and 0.3cm);
  \draw[red] (6,0) ellipse (1cm and 0.15cm);
  \draw[blue] (6,-0.5) ellipse (0.5cm and 0.075cm);
  \draw[cyan] (6,-0.75) ellipse (0.25cm and 0.0325cm);
  \draw (4,1)--(6,-1)--(8,1);
  \node at (3.5,0){$\stackrel{\g\circ\f}{\To}$};
  \node at (8.5,1.5){$L^+\cap \{T=r^2$\}};
\end{tikzpicture}
\captionof{figure}{The map $\g \circ \f$ as the cone on the Hopf fibration (drawn one dimension down).}
\label{Fig:cone_on_Hopf}
\end{center}

Adding the $0$ into $\C^2$ and $L^+$, since $\g \circ \f (0)= 0$, $\g \circ \f$ is the cone on the Hopf fibration.
\end{proof}

The following computation will be useful when we consider lines and planes containing $\g \circ \f (\kappa)$.
\begin{lem}
\label{Lem:gof_celestial_sphere}
For any $\kappa \in \C_\times^2$, the line $\R (\g \circ \f (\kappa))$ intersects $\S^+$ in the unique point
\[
\left( 1, \frac{2(ac+bd)}{a^2+b^2+c^2+d^2}, \frac{2(bc-ad)}{a^2+b^2+c^2+d^2}, 
\frac{a^2+b^2-c^2-d^2}{a^2+b^2+c^2+d^2} \right).
\]
\end{lem}
Here we regard $\S^+$ as $L^+ \cap \{T=1\}$, i.e \refdef{celestial_sphere}(ii).

\begin{proof}
This follows immediately from \reflem{spin_vector_to_TXYZ}, scaling $\g \circ \f(\kappa)$ to have $T$-coordinate $1$.
\end{proof}

\subsubsection{The Hopf fibration and stereographic projection}
\label{Sec:Hopf}

We have seen the Hopf fibration in $\g \circ \f$; we can also describe this directly and explicitly.
Perhaps the most standard definition of the Hopf fibration is as follows.
\begin{defn}
The \emph{Hopf fibration} is the map 
\[
\text{Hopf} \colon S^3 \To S^2 \cong \CP^1, \quad
(\xi, \eta) \mapsto \frac{\xi}{\eta}.
\]
\end{defn}
Here we regard $S^3$ as $\{(\xi, \eta) \; \mid \; |\xi|^2 + |\eta|^2 = 1 \} \subset \C^2$, and $\CP^1 = \C \cup \{\infty\} $ as $S^2$.
We can translate from the Riemann sphere to the unit 2-sphere in $\R^3$ by stereographic projection; again, perhaps the most standard definition is as follows. It is the map obtained from projecting the $xy$-plane in $\R^3$, viewed as $\C$, to the unit sphere, as in \reffig{1}. It extends to a map from $\CP^1 = \C \cup \{\infty\}$.
\begin{defn}
\label{Def:stereographic_projection}
\emph{Stereographic projection} is the map
\[
\text{Stereo} \colon \CP^1  \To S^2, \quad
a+bi \mapsto \left( \frac{2a}{1+a^2+b^2}, \frac{2b}{1+a^2+b^2}, \frac{-1+a^2+b^2}{1+a^2+b^2} \right), \quad
\infty \mapsto (0,0,1).
\]
\end{defn}
If we compute the Hopf fibration from the standard $S^3 \subset \CP^1$, to the standard Euclidean $S^2 \subset \R^3$ using stereographic projection, we obtain expressions we have seen before!
\begin{lem}
\label{Lem:gof_Hopf}
Let $\pi_{XYZ} \colon \R^{1,3} \To \R^3$ be the projection onto the $XYZ$ 3-plane in Minkowski space. Then the composition $\Stereo \circ \Hopf \colon S^3 \To S^2$ is given by
\[
\Stereo \circ \Hopf  = \pi_{XYZ} \circ \g \circ \f|_{S^3}.
\]
\end{lem}
Here the projection $\pi_{XYZ}$ simply maps $(X,Y,Z,T) \mapsto (X,Y,Z)$.
In other words, the $X,Y,Z$ coordinates of $\g \circ \f$ are precisely the Hopf fibration computed with stereographic projection.

\begin{proof}
Let $(\xi, \eta) = (a+bi, c+di) \in S^3$ where $a,b,c,d \in \R$.
We compute
\[
\Hopf (\xi,\eta) = \frac{a+bi}{c+di} = \frac{ac+bd}{c^2+d^2} + i \frac{bc-ad}{c^2+d^2}
\]
and then applying $\Stereo$ yields
\[
\left( \frac{ 2 \left( \frac{ac+bd}{c^2+d^2} \right) }{1 + \left( \frac{ac+bd}{c^2+d^2} \right)^2 + \left( \frac{bc-ad}{c^2+d^2} \right)^2 }, \;
\frac{ 2 \left( \frac{bc-ad}{c^2+d^2} \right) }{1 + \left( \frac{ac+bd}{c^2+d^2} \right)^2 + \left( \frac{bc-ad}{c^2+d^2} \right)^2 }, \;
\frac{ -1 + \left( \frac{ac+bd}{c^2+d^2} \right)^2 + \left( \frac{bc-ad}{c^2+d^2} \right)^2 }{ 1 + \left( \frac{ac+bd}{c^2+d^2} \right)^2 + \left( \frac{bc-ad}{c^2+d^2} \right)^2 }
\right)
\]
which, fortunately enough, simplifies to
\[
\frac{1}{a^2+b^2+c^2+d^2} 
\left( 2(ac+bd), \;
2 (bc-ad), \;
a^2+b^2 - c^2 - d^2 \right).
\]
Since $a^2+b^2+c^2+d^2 = |\xi|^2 + |\eta|^2 = 1$, comparison with \reflem{spin_vector_to_TXYZ} gives the desired result.
\end{proof}

\subsubsection{Inner products on spinors and Minkowski space}
\label{Sec:inner_products_spinors-Minkowski}

Two spinors $\kappa, \kappa' \in \C^2$ have an inner product $\{\kappa, \kappa'\}$; we also now have the two points in the light cone $\g \circ \f (\kappa), \, \g \circ \f (\kappa')$, on which we can consider the Lorentzian inner product $\langle \g \circ \f(\kappa), \, \g \circ \f(\kappa') \rangle$. If one of $\kappa,\kappa'$ is a real multiple of the other, then $\{\kappa, \kappa'\} = 0$, and equally, $\g \circ \f(\kappa)$ and $\g \circ \f(\kappa')$ are proportional lightlike vectors, so $\langle \g \circ \f(\kappa), \g \circ \f (\kappa') \rangle = 0$. In fact, we have the following. Compare \cite[lem. 4.5]{Penner12}.
\begin{prop}
\label{Prop:complex_Minkowski_inner_products}
For $\kappa, \kappa' \in \C^2_\times$,
\[
2 \left| \left\{ \kappa, \kappa' \right\} \right|^2
= \langle  \g \circ \f (\kappa), \, \g \circ \f(\kappa') \rangle.
\]
\end{prop}

Let $\kappa = (\xi, \eta)$, $\kappa' = (\xi', \eta')$, and $\xi = a+bi,\ \eta = c+di,\ \xi' = a'+b'i,\ \eta' = c'+d'i$ where $a,b,c,d,a',b',c',d'$ are all real. It is convenient for the proof to think of $\kappa, \kappa'$ as real vectors $(a,b,c,d)$, $(a',b',c',d')$, and consider the $2 \times 4$ matrix
\[
M = \begin{pmatrix} a & b & c & d \\ a' & b' & c' & d' \end{pmatrix}
\]
with those vectors as its rows. We denote by $M_{ij}$ the submatrix of $M$ formed from its $i$ and $j$ columns. Thus, for instance,
\[
M_{34} = \begin{pmatrix} c & d \\ c' & d' \end{pmatrix},
\quad
\det M_{13} = ac' - ca',
\quad \text{etc.}
\]
It is then true that
\begin{equation}
\label{Eqn:Plucker_24}
\det M_{13} \det M_{24} = \det M_{12} \det M_{34} + \det M_{14} \det M_{23}.
\end{equation}
This can be checked directly; it is a Pl\"{u}cker relation, which arises in the theory of Grassmannians (see e.g. \cite[ch. 1.5]{Griffiths_Harris94}). We will use it later in \refsec{3d_hyp_geom} to prove our Ptolemy equation. The strategy of the proof of \refprop{complex_Minkowski_inner_products} is to write all quantities in terms of the $M_{ij}$.
\begin{lem}
\label{Lem:complex_inner_product_subdeterminants}
With $\kappa,\kappa'$ as above,
\[
\left\{\kappa,\kappa'\right\}
= \left( \det M_{13} - \det M_{24} \right) + \left( \det M_{14} + \det M_{23} \right) i.
\]
\end{lem}
This lemma is really a general fact about $2 \times 2$ complex matrices $N$: if we make its entries into $1 \times 2$ real matrices, and obtain a $2 \times 4$ real matrix $M$, then $\det N$ is given by the right hand side above.

\begin{proof}
\begin{align*}
\det \begin{pmatrix} a+bi & a'+b'i \\ c+di & c'+d'i \end{pmatrix}
&= (a+bi)(c'+d' i)-(a'+b'i)(c+di) \\
&= \left( ac' - ca' + db'-bd' \right) + \left( ad'-da' + bc'-cb' \right)i,
\end{align*}
which is the desired combination of determinants.
\end{proof}

\begin{lem}
\label{Lem:Minkowski_inner_product_subdeterminants}
With $\kappa,\kappa'$ as above,
\[
\frac{1}{2} \langle \g \circ \f (\kappa), \, \g \circ \f (\kappa') \rangle
=
\det M_{13}^2 + \det M_{14}^2 + \det M_{23}^2 + \det M_{24}^2 - 2 \det M_{12} \det M_{34}.
\]
\end{lem}

\begin{proof}
Using \reflem{spin_vector_to_TXYZ} we have
\begin{align*}
\g \circ \f(\kappa)
&= \left( a^2 + b^2 + c^2 + d^2, \, 2(ac+bd), \, 2(bc-ad), \, a^2 + b^2 - c^2 - d^2 \right) \\
\g \circ \f(\kappa')
&= \left( a'^2 + b'^2 + c'^2 + d'^2, \, 2(a'c'+b'd'), \, 2(b'c'-a'd'), \, a'^2 + b'^2 - c'^2 - d'^2 \right)
\end{align*}
so applying $\langle \cdot, \cdot \rangle$ yields $\langle \g \circ \f (\kappa), \, \g \circ \f (\kappa') \rangle$ as
\begin{align*}
\left( a^2 + b^2 + c^2 + d^2 \right)
\left( a'^2 + b'^2 + c'^2 + d'^2 \right)
& - 4 (ac+bd)(a'c'+b'd')
- 4 (bc-ad)(b'c'-a'd') \\
&- \left(a^2 + b^2 - c^2 - d^2 \right) 
\left( a'^2 + b'^2 - c'^2 - d'^2 \right)
\end{align*}
This simplifies to
\[
2(ac'-ca')^2 + 2(ad'-da')^2 + 2(bc'-cb')^2 + 2(bd'-db')^2 
- 4(ab'-ba')(cd'-dc')
\]
giving the desired equality.
\end{proof}

\begin{proof}[Proof of \refprop{complex_Minkowski_inner_products}]
By \reflem{complex_inner_product_subdeterminants} and \reflem{Minkowski_inner_product_subdeterminants}, it remains to show that the following equation holds:
\[
\left( \det M_{13} - \det M_{24} \right)^2 + \left( \det M_{14} + \det M_{23} \right)^2
=
\det M_{13}^2 + \det M_{14}^2 + \det M_{23}^2 + \det M_{24}^2 - 2 \det M_{12} \det M_{34}.
\]
Upon expanding and simplifying, this reduces to the Pl\"{u}cker equation \refeqn{Plucker_24}.
\end{proof}

\subsection{Flags}
\label{Sec:flags}

We now pick up the idea, left off in \refsec{derivatives_of_f}, of defining a flag using the map $\f$ and its derivative in a certain direction $\ZZ(\kappa)$ at each point $\kappa \in \C^2_\times$. 

\begin{defn}
A \emph{flag} in a vector space $V$ is an ascending sequence of subspaces
\[
V_1 \subset \cdots \subset V_k.
\]
Letting $d_i = \dim V_i$, the $k$-tuple $(d_1, \ldots, d_k)$ is called the \emph{signature} of the flag.
\end{defn}

We will use the map $\f$ to span a 1-dimensional subspace of $\HH$, and then use its derivative as described by $\ZZ$ to span a 2-plane. Thus, the flag involved will be
\[
\R \f(\kappa) \subset \R \f(\kappa) \oplus \R D_\kappa \f(\ZZ(\kappa)),
\]
and this assignment of flags to spin vectors turns out to be equivariant under the action of $SL(2,\C)$.

Such flags are flags in $\HH$, but as seen in \refsec{hermitian_to_minkowski}, there is a linear isomorphism $\g$ between $\HH$ and $\R^{1,3}$ preserving all relevant structure, so these flags can also be considered in $\R^{1,3}$, after applying $\g$ appropriately.

The flags we consider all have signature $(1,2)$, but not every such flag arises by this construction. There are certain geometric constraints on the subspaces, relating to the \emph{light cone} $L$ of \emph{null vectors} in $\R^{1,3}$, or the space of singular Hermitian matrices $\HH_0$. Moreover, in order to obtain our desired bijections, we need further structure in our flags of a distinguished point, and orientations. Hence we call the flag structures we need \emph{pointed oriented null flags}.

To most readers, we suspect geometric constraints are more easily understood in terms of the light cone in Minkowski space, than in terms of singular Hermitian matrices. On the other hand, the map $\f$ maps directly into Hermitian matrices, while the map $\g$ then applies a further linear transformation, so the algebra of flags is simpler in terms of Hermitian matrices. 

Thus, we discuss flags both in $\HH$ and $\R^{1,3}$, but prefer $\HH$ for simpler algebra, and $\R^{1,3}$ for geometric intuition.  We will define flags in $\HH$ and $\R^{1,3}$ simultaneously. 

In \refsec{Z} and we introduce the map $\ZZ$, needed for defining the flag direction. In \refsec{PNF} we introduce \emph{pointed null flags}, with ``null" having its usual meaning in $\R^{1,3}$, and then in \refsec{PONF} we introduce \emph{pointed oriented null flags}, the precise type of flag structure we need, which also have some orientation in their structure. In \refsec{describing_flags} we develop notation for describing flags. Then in \refsec{map_F} we can define the map $\F$ from spin vectors to flags. In \refsec{SL2c_action_on_flags_HH} we discuss the $SL(2,\C)$ action on flags, and in \refsec{equivariance_of_F} prove equivariance of the action. This discussion of the $SL(2,\C)$ action is in terms of Hermitian matrices $\HH$, so in \refsec{flags_Minkowski_space} we translate these results into Minkowski space. In \refsec{calculating_flags_Minkowski} we explicitly calculate details of flags in Minkowski space corresponding to spin vectors, and in \refsec{rotating_flags} we consider rotating them. This allows us to show in \refsec{F_surjectivity} that the maps $\F$ and $\G \circ \F$ are surjective, more precisely 2--1 maps.

\subsubsection{The map $\ZZ$}
\label{Sec:Z}

\begin{defn}
\label{Def:Z_C2_to_C2_and_J}
Define $\ZZ \colon \C^2 \To \C^2$ by
\[
\ZZ \begin{pmatrix}\alpha\\ \beta\end{pmatrix}  = \begin{pmatrix}
    
 \overline{\beta} \, i\\ \, -\overline{\alpha} \, i \end{pmatrix}
\quad \text{i.e.} \quad
\ZZ (\kappa) = J \, \overline{\kappa}
\quad \text{where} \quad
J = \begin{pmatrix} 0 & i \\ -i & 0 \end{pmatrix}.
\]
\end{defn}
With this definition of $\ZZ$, using \refeqn{derivative_formula}, we obtain
\begin{equation}
\label{Eqn:derivative_flag_dirn}
D_\kappa f(\ZZ(\kappa)) = \kappa \ZZ(\kappa)^* + \ZZ(\kappa) \kappa^*
= \kappa \kappa^T J + J \overline{\kappa} \kappa^*.
\end{equation}
The following observations are significant in the sequel and help to motivate the definition of $\ZZ$.
\begin{lem}
\label{Lem:bilinear_Z_negative_imaginary}
\label{Lem:Z_forms_basis}
For any $\kappa \in \C^2_\times$,
\begin{enumerate}
\item $\{\kappa, \ZZ(\kappa)\}$ is negative imaginary;
\item
$\kappa$ and $\ZZ(\kappa)$ form a basis for $\C^2$ as a complex vector space.
\end{enumerate}
\end{lem}

\begin{proof}
Let $\kappa=(\xi,\eta) \in \C^2_\times$, then from \refdef{bilinear_form_defn},
\[
\{\kappa,\ZZ(\kappa)\}=
\det 
\begin{pmatrix} 
\xi & \overline{\eta} \, i \\
\eta & - \overline{\xi} \, i
\end{pmatrix}
=
\xi(-\overline{\xi}i)-\eta(\overline{\eta}i)
=- \left( |\xi|^2+|\eta|^2 \right) i,
\]
which is negative imaginary. Being nonzero, the matrix columns are linearly independent over $\C$.
\end{proof}

For another, possibly motivating, perspective on $\ZZ$, identify $(\xi,\eta)=(a+bi,c+di)$ with the quaternion $q=a+b\pmb{i}+c\pmb{j}+d\pmb{k}$, where 
$1, \pmb{i}, \pmb{j}, \pmb{k}$
are the elementary quaternions. Then, as a map on quaternions, $\ZZ$ is given by
\[
\ZZ(q)=-\pmb{k} q=-\pmb{k}(a+b\pmb{i}+c\pmb{j}+d\pmb{k})=(d+c\pmb{i}-b\pmb{j}-a\pmb{k})\leftrightarrow(d+ci,-b-ai).
\]
Thus, in the Euclidean metric on $\C^2 \cong \R^4$, $\ZZ (q)$ is orthogonal to $q$. 

On the unit $S^3$ centred at the origin in the quaternions, the tangent space to $S^3$ at $\kappa$ has basis $\pmb{i} \kappa, \pmb{j} \kappa, \pmb{k} \kappa$. The $\pmb{i}\kappa$ direction is the direction of the fibre of the Hopf fibration, and $\f$ is constant in that direction. This perhaps motivates why we take the $\pmb{k} \kappa$ direction. (The choice of $-$ rather than $+$, and $\pmb{k}$ rather than $\pmb{j}$, is somewhat arbitrary.)

\subsubsection{Pointed null flags}
\label{Sec:PNF}

All the flags we consider will be of signature $(1,2)$ in $\HH \cong \R^{1,3}$. By \reflem{det0_lightcone_correspondence}, the subset $\HH_0^+ \subset \HH$ corresponds under $\g$ to the positive light cone $L^+ \subset \R^{1,3}$. Vectors on $L^+$ are null, hence the name.

\begin{defn}
\label{Def:null_flag_in_Minkowski}
A \emph{null flag} in $\R^{1,3}$ (resp. $\HH$) is a flag of signature $(1,2)$ in $\R^{1,3}$ (resp. $\HH$)
\[
V_1 \subset V_2
\]
where 
\begin{enumerate}
\item $V_1$ is spanned by some $p \in L^+$ (resp. $S \in \HH_0^+$).
\item $V_2$ is spanned by the same $p$ (resp. $S$), together with some $v \in T_p L^+$ (resp. $U \in T_S \HH_0^+$).
\end{enumerate}
\end{defn}
Thus in a null flag $V_1 \subset V_2$ in $\R^{1,3}$, the first space $V_1$ is a line in the light cone, and the second space $V_2$ is a 2-plane tangent to the light cone. Although $p$ in the above definition is null (indeed, has future-pointing lightlike position vector), the tangent vector $v$ to $L^+$ at $p$ is not null. See \reffig{flag}.

The definitions of null flags in $\HH$ and $\R^{1,3}$ correspond under the isomorphism $\g$: $V_1 \subset V_2$ is a null flag in $\HH$ iff $\g(V_1) \subset \g(V_2)$ is a null flag in $\R^{1,3}$. Thus $\g$ provides a bijection between null flags in $\HH$ and null flags in $\R^{1,3}$. 

From a spinor $\kappa$, we already have a point $\f(\kappa) \in \HH_0^+$ or $\g \circ \f(\kappa) \in L^+$, so our flags come with a distinguished basepoint, as in the following definition.
\begin{defn}
\label{Def:pointed_null_flag}
A \emph{pointed null flag} in $\R^{1,3}$ (resp. $\HH$) is a point $p \in L^+$ (resp. $S \in \HH_0^+$) together with a null flag $\R p \subset V$ (resp. $\R S \subset V$).

We denote the set of pointed null flags in $\R^{1,3}$ (resp. $\HH$) by $\mathcal{F_P}(\R^{1,3})$ (resp. $\mathcal{F_P}(\HH)$ ).
\end{defn}
When the distinction between $\HH$ and $\R^{1,3}$ is unimportant we simply write $\mathcal{F_P}$.

We denote a pointed null flag as above in
\begin{itemize}
    \item  $\R^{1,3}$ by $(p,V)$ or $[[p,v]]$, where $v \in T_p L^+$ and $V$ is spanned by $p$ and $v$;
    \item  $\HH$ by $(S, V)$ or $[[S,U]]$, where $U \in T_S \HH_0^+$ and $V$ is spanned by $S$ and $U$.
\end{itemize}

All the notions in $\HH$ and $\R^{1,3}$ in the definition of pointed null flags correspond under the isomorphism $\g$:  $(S,V)\in\mathcal{F_P}(\HH)$ iff $(\g(S), \g(V))\in\mathcal{F_P}(\R^{1,3})$. So $\g$ yields a bijection $\mathcal{F_P}(\HH) \To \mathcal{F_P}(\R^{3,1})$, given by $(S,V) \mapsto (\g(S),\g(V))$ or $[[S,U]] \mapsto [[\g(S), \g(U)]]$. 

The notation $(p,V)$ is unique: if $(p,V) = (p',V')$ then $p=p'$ and $V=V'$. However the same is not true for the notation $[[p,v]]$: a given pointed null flag may be described by different pairs $p,v$. The following lemma clarifies when two descriptions are equal.
\begin{lem}
\label{Lem:characterise_equal_PNFs}
Suppose $p,p' \in L^+$ and $v,v' \in \R^{1,3}$.
The following are equivalent:
\begin{enumerate}
\item $[[p,v]]$ and $[[p',v']]$ describe the same pointed null flag.
\item $p=p'$, and $v,v'$ both lie in $T_p L^+$, and the real spans of $(p,v)$ and $(p',v')$ are 2-dimensional and equal.
\item $p=p'$, and $v,v'$ both lie in $T_p L^+$, and $v,v'$ are not real multiples of $p$, and there exist real numbers $a,b,c$, not all zero, such that $ap+bv+cv'=0$.
\end{enumerate}
\end{lem}
A similar statement applies for pointed null flags in $\HH$, if we replace $p,p' \in L^+$ with $S,S' \in \HH_0^+$, $v,v' \in \R^{1,3}$ with $U,U' \in \HH$, and $T_p L^+$ with $T_S \HH_0^+$.

\begin{proof}
That (i) is equivalent to (ii) is immediate from the definition: the points $p,p'$ must be equal, and the planes spanned by $(p,v)$ and $(p',v')$ must be tangent to $L^+$ (resp. $\HH_0^+$) and equal.  That (ii) is equivalent to (iii) is elementary linear algebra: $(p,v)$ and $(p,v')$ span equal 2-dimensional planes iff $(p,v)$ and $(p,v')$ are linearly independent but $(p,v,v')$ is linearly dependent.
\end{proof}

\subsubsection{Pointed oriented null flags}
\label{Sec:PONF}

In general, an \emph{oriented flag} is a flag
\[
\{0\} = V_0 \subset V_1  \subset \cdots \subset V_k
\]
where each quotient $V_i/V_{i-1}$, for $i=1, \ldots, k$, is endowed with an orientation. Equivalently, these orientations amount to orienting $V_1$, and then orienting each quotient $V_2/V_1, V_3/V_2, \ldots, V_k/V_{k-1}$.

We regard an \emph{orientation} of a vector space $V$, in standard fashion, as an equivalence class of ordered bases of $V$, where two ordered bases are equivalent when they are related by a linear map with positive determinant.

A pointed null flag $(p,V)\in\mathcal{F_P}$ already naturally contains some orientation data: the 1-dimensional space $\R p$ can be oriented in the direction of $p$. Thus it remains to orient the quotient $V/\R p$, as per the following definition.

\begin{defn}
\label{Def:pointed_oriented_null_flag}
A \emph{pointed oriented null flag} in $\R^{1,3}$  is the data $(p, V, o)$ where:
\begin{enumerate}
\item $(p,V)\in\mathcal{F_P}(\R^{1,3})$, with $\R p$ is oriented in the direction of $p$;
\item $o$ is an orientation of $V/\R p$.
\end{enumerate}
The set of pointed oriented  null flags in $\R^{1,3}$ is denoted $\mathcal{F_P^O}(\R^{1,3})$.
\end{defn}

Similarly, a pointed oriented null flag in $\HH$ consists of $(S, V, o)$, where $(S,V) \in \mathcal{F_P}(\HH)$, $\R S$ is oriented in the direction of $S$, and $o$ is an orientation of $V/\R S$.  Since $(S,V)$ is a pointed null flag, $S \in \HH_0^+$, and $V$ is a 2-dimensional subspace containing $S$ and tangent to $\HH_0^+$.
The set of pointed oriented null flags in $\HH$ is denoted $\mathcal{F_P^O}(\HH)$.

When the distinction between $\HH$ and $\R^{1,3}$ is unimportant we simply write $\mathcal{F_P^O}$.

Pointed oriented null flags are the structure we need to describe spinors. Henceforth we will simply refer to them as \emph{flags}.

The space $\mathcal{F_P^O}(\R^{1,3})$ of pointed null flags is 4-dimensional. To see this, note that $p$ lies in the 3-dimensional positive light cone $L^+$. The tangent space $T_p L^+$ is 3-dimensional and contains $\R p$ as a subspace. The set of relatively oriented 2-planes $V$ in the 3-dimensional vector space $T_p L^+$ containing $\R p$ is 1-dimensional; there is an $S^1$ worth of such 2-planes, rotating around $\R p$. In fact, we will see later in \refsec{topology_of_spaces} that $\mathcal{F_P^O}$ naturally has the topology of $\textnormal{UT}S^2 \times \R$, the product of the unit tangent bundle of $S^2$ with $\R$.

Just as for pointed null flags, there is a bijection $\mathcal{F_P^O}(\HH) \To \mathcal{F_P^O}(\R^{1,3})$, as we now show. Let $(S,V,o) \in \mathcal{F_P^O}(\HH)$, consisting of subspaces $\R S \subset V$. Just as for pointed null flags, we can directly apply $\g$ to $S \in \HH_0^+$ and $V \subset \HH$ to obtain $\g(S)$, and $\g(V)$. We can also apply $\g$ to the orientation $o$ as follows. The orientation $o$ is represented by an equivalence class of ordered bases of $V/\R S$.
(As $V/\R S$ is 1-dimensional, such an ordered basis consists of just one element.)
The isomorphism $\g \colon \HH \To \R^{1,3}$ restricts to isomorphisms $V \To \g(V)$ and $\R S \To \R \g(S)$, and hence provides an isomorphism of quotient spaces $\underline{\g} \colon V / \R S \To \g(V) / \R \g(S)$.
Taking $\underline{B}$ to be an ordered basis of $V/\R S$ representing $o$, then we define $\g(o)$ to the the orientation represented by $\g(\underline{B})$. 
\begin{defn}
\label{Def:G}
The map $\G$ from (pointed oriented null) flags in $\HH$, to (pointed oriented null) flags in $\R^{1,3}$, is given by
\[
\G \colon \mathcal{F_P^O}(\HH) \To \mathcal{F_P^O}(\R^{1,3}), \quad
\G(S,V,o) = (\g(S),\g(V),\g(o)).
\]
\end{defn}

\begin{lem}
\label{Lem:G_bijection}
$\G$ is well defined and a bijection.
\end{lem}
In other words, $(S,V,o)\in\mathcal{F_P^O}(\HH)$ iff $(\g(S),\g(V),\g(o))\in\mathcal{F_P^O}(\R^{1,3})$

\begin{proof}
The isomorphism $\g$ maps $S \in \HH_0^+$ to a point $\g(S) \in L^+$ (\reflem{det0_lightcone_correspondence}). The 2-plane $V$ is spanned by $S$ and an element of $T_S \HH_0^+$, so $\g(V)$ is a 2-plane spanned by $\g(S)$ and an element of $T_{\g(S)} L^+$. Thus $\R \g(S) \subset \g(V)$ is a null flag in $\R^{1,3}$ and in fact $(\g(S), \g(V)) \in \mathcal{F_P} (\R^{1,3})$.

Considering orientations, since $\g(S) \in L^+$, the 1-dimensional space $\R \g(S)$ is oriented towards the future, in the direction of $\g(S)$. To see that $\g(o)$ is well defined, let $\underline{B}, \underline{B'}$ be two ordered bases of $V/\R S$ representing $o$ (in fact each basis consists of one vector); we show that $\g(\underline{B}), \g(\underline{B'})$ represent the same orientation of $\g(V)/\R \g(S)$. Since $\underline{B}, \underline{B'}$ represent $o$ and consist of single vectors, then $\underline{B'} = m \underline{B}$ where $m$ is positive real, so $\g(\underline{B'}) = M \g (\underline{B})$. As $m > 0$ then $\g(\underline{B'})$ and $\g(\underline{B})$ represent the same orientation $\g(V)/\R \g(S)$. So $\g(o)$ is well defined, and indeed $\G$ is well defined.

The same arguments applied to the isomorphism $\g^{-1}$ show that $\G^{-1}$ is a well defined inverse to $\G$, so $\G$ is a bijection.
\end{proof}

\subsubsection{Describing flags}
\label{Sec:describing_flags}

Above we introduced notation $[[p,v]]$ for pointed null flags. We now extend this notation to (pointed oriented null) flags.

\begin{defn}
\label{Def:pv_notation_PONF}
Let $p \in L^+$ and $v \in T_p L^+$, such that $p,v$ are linearly independent. Then $[[p,v]]$ denotes $(p,V,o)\in\mathcal{F_P^O}(\R^{1,3})$, where $V$ is the span of $p$ and $v$, and $o$ is the orientation on $V/\R p$ represented by $v + \R p$.
\end{defn}
The definition works similarly in $\mathcal{F_P^O}(\HH)$: for $S \in \HH_0^+$ and $U \in T_S \HH_0^+$, such that $S,U$ are linearly independent, $[[S,U]]$ denotes $(S,V,o)\in\mathcal{F_P^O}(\HH)$ where $V$ is the span of $S$ and $U$, and $o$ is the orientation on $V/\R S$ given by $U + \R S$.

Intuitively, the orientations can be understood as follows. The 2-plane $V$ is spanned by $p$ and $v$; $p$ gives an orientation on the line $\R p$, which is towards the future in $\R^{1,3}$ since $p \in L^+$. Choosing an orientation on $V/\R p$ amounts to choosing one of the two sides of the line $\R p$ on the plane $V$; we choose the side to which $v$ points.

We have seen that flags in $\HH$ and $\R^{1,3}$ are related by the bijection $\G$, which has a simple description in this notation.
\begin{lem}
\label{Lem:G_in_pv_notation}
For $[[S,U]] \in \mathcal{F_P^O}(\HH)$, we have $\G [[S,U]] = [[\g(S), \g(U)]]$.
\end{lem}

\begin{proof}
Let $V$ be the 2-plane spanned by $S,U$ and $o$ the orientation on $V/\R S$ given by $U$, so $[[S,U]] = (S,V,o)$. Applying $\G$ to this flag, by \refdef{G}, yields $(\g(S),\g(V),\g(o))$. Now $\g(V)$ is the span of $\g(S)$ and $\g(U)$, and $\g(o)$ is the orientation on $\g(V)/\R \g(S)$ induced by $\g(U)$, so $(\g(S),\g(V),\g(o)) = [[\g(S),\g(U)]]$.
\end{proof}

Just as for pointed null flags, a given $(p,V,o)\in\mathcal{F_P^O}(\R^{1,3})$ can be described by many different $[[p,v]]$, and the following lemma, refining \reflem{characterise_equal_PNFs}, describes when they are equal.
\begin{lem}
\label{Lem:characterise_equal_PONFs}
Suppose $p,p' \in L^+$ and $v,v' \in \R^{1,3}$. The following are equivalent.  
\begin{enumerate}
\item $[[p,v]]$ and $[[p',v']]$ describe the same (pointed oriented null) flag.
\item $p=p'$, and $v,v'$ both lie in $T_p L^+$, and the sets
\[
\R p + \R^+ v = \left\{ ap+bv \mid a,b \in \R, b > 0 \right\}, \quad
\R p' + \R^+ v' = \left\{ ap'+b v' \mid a,b \in \R, b > 0 \right\}
\]
are equal 2-dimensional half-planes.
\item $p=p'$, and $v,v'$ both lie in $T_p L^+$, and $v,v'$  are not real multiples of $p$, and  there exist real numbers $a,b,c$ such that $ap+bv+cv'=0$, where $b,c$ are nonzero and have opposite sign.
\end{enumerate} 
\end{lem}
As usual, a similar statement applies to flags in $\HH$, replacing $\R^{1,3}$ with $\HH$, $p,p' \in L^+$ with $S,S' \in \HH_0^+$, $v,v' \in \R^{1,3}$ with $U,U' \in \HH$, and $T_p L^+$ with $T_S \HH_0^+$.

Note that when $v,v'$ are not real multiples of $p$, then an equation $ap+bv+cv'=0$ with $a,b,c$ not all zero must have $b$ and $c$ nonzero, and so can be rewritten as $v' = dv+ep$ or $v = d'v'+e'p$, expressing $v'$ in terms of the basis $\{v,p\}$, or $v$ in terms of the basis $\{v',p\}$ respectively.  Having $b$ and $c$ of opposite sign is then equivalent to $d$ and $d'$ being positive, since $d = -b/c$ and $d'=-c/b$. In other words, $v$ is a positive multiple of $v'$, modulo multiples of $p$; and equivalently, $v'$ is a positive multiple of $v$ modulo multiples of $p$.

\begin{proof}
First we show the equivalence of (i) and (ii). By \reflem{characterise_equal_PNFs}, $[[p,v]]$ and $[[p',v']]$ describe the same pointed null flag if and only if $p=p'$, $v,v'$ both lie in $T_p L^+$, and the real spans of $(p,v)$ and $(p',v')$ are 2-dimensional and equal; let this span be $V$. It remains to show that the orientations on $V/\R p$ given by $v+\R p$ and $v'+\R p$ are equal if and only if $\R p + \R^+ v = \R p + \R^+ v'$.

Now $V$ is divided into two half planes by the line $\R p$. They are respectively given by
\[
\R p + \R^+ v = \left\{ ap+bv \mid a,b \in \R, b > 0 \right\} \quad \text{and} \quad
\R p - \R^+ v = \left\{ ap-bv \mid a,b \in \R, b > 0 \right\}.
\]
These two half-planes map down to the 1-dimensional quotient space $V/\R p$ to give the two components of the complement of the origin: the first half-plane yields the positive real span of $v+\R p$; the second yields the negative real span of $v+\R p$. The first defines the co-orientation given by $v+\R p$. For $(p,v')$ we have a similar description of two half-planes $\R p + \R^+ v'$ and $\R p - \R^+ v'$, and we see that the half-plane $\R p + \R^+ v'$ yields the positive real span of $v'+ \R p$ in $V/\R p$, corresponding to the orientation given by $v' + \R p$. Thus, the two orientations are equal if and only if the two claimed sets are equal.

Now we show that (ii) is equivalent to (iii). We note that if the two sets in (ii) are equal, then $v' = ap+bv$ for some real $a,b$ with $b$ positive. Then $ap+bv-v'=0$ provides the equation required for (iii). Conversely, if $ap+bv+cv'=0$ with $b,c$ of opposite sign, then we may write $v'=dv+ep$ where $d$ is positive. Thus $v' \in \R p + \R^+ v$, so the half-plane $\R p + \R^+ v$ must coincide with the half-plane $\R p + \R^+ v'$.
\end{proof}

\subsubsection{The map from spin vectors to flags}
\label{Sec:map_F}

We now upgrade the map $\f$ to $\F$. Whereas $\f$ associates to a spinor $\kappa$ a matrix in $\HH_0^{0+}$, the map $\F$ associates to $\kappa$ a flag in $\HH$. The point in the pointed flag is just $\f(\kappa)$. As discussed at the beginning of \refsec{flags}, the 2-plane incorporates tangent data, using the derivative of $\f$ in a direction specified by the map $\ZZ$. We will see that the resulting construction is equivariant.

\begin{defn}
\label{Def:spinors_to_PNF}
The map $\F$ from nonzero spin vectors to (pointed oriented null) flags is given by
\[
\F \colon \C_\times^2 \To  \mathcal{F_P^O}(\HH), \quad
\F(\kappa) = [[ \f(\kappa), \; D_\kappa \f(\ZZ(\kappa)) ]].
\]
\end{defn} 
Using \refeqn{derivative_flag_dirn} we thus have, for $\kappa \in \C^2_\times$,
\begin{equation}
\label{Eqn:F_explicitly}
\F(\kappa) = [[ \f(\kappa), \; \kappa \kappa^T J + J \, \overline{\kappa} \kappa^* ]].
\end{equation}
Although $\F$ as stated could equally well map to less elaborate structures, for instance dropping the ``pointed or ``oriented" details, we need the full data of a pointed oriented null flag for our construction.

The domain of $\F$ is $\C_\times^2$ rather than $\C^2$, since $\f(0)=0$, which does not span a 1-dimensional subspace in $\HH$; moreover there is no well defined tangent space to $\HH_0^+$ or $\HH_0^{0+}$ there. For $\kappa \neq 0$ we have $0 \neq \f(\kappa) \in \HH_0^+$, so we obtain a well defined 1-dimensional subspace for our null flag. Although it is clear $D_\kappa \f(\ZZ(\kappa)) \in T_{\f(\kappa)} \HH_0^+$, it is perhaps not so clear that, with $\f(\kappa)$, it spans a 2-dimensional vector space. We verify this, and in fact prove something stronger, in \reflem{flag_well_defined} below.

We saw in \reflem{G_bijection}, that the linear isomorphism $\g \colon \HH \To \R^{1,3}$ induces a bijection $\G$ on flags; this immediately allows us to transport the flags on $\HH$, constructed by $\F$, over to Minkowski space. 

Before proving \reflem{flag_well_defined} to verify that $\F$ is well defined, we first prove a general observation in linear algebra about factorisation of spin vectors. Statements equivalent to this first lemma appear in Penrose and Rindler \cite{Penrose_Rindler84}, and probably elsewhere. 
Recall (\refsec{notation}) that $\M_{m \times n}(\mathbb{F})$ denotes $m \times n$ matrices with entries in $\mathbb{F}$, and $\M_{m \times n}(\mathbb{F})_\times$ denotes such matrices which are nonzero.
\begin{lem}
\label{Lem:spinor_factorisation}
Suppose $M,M'\in\mathcal{M}_{2\times 1}(\C)_\times$, and $N,N'\in\mathcal{M}_{1\times 2}(\C)_\times$. If $MN = M'N'$ then there exists $\mu\in\C_\times$ such that $M = \mu M'$ and $N = \mu^{-1} N'$.
\end{lem}

\begin{proof}
Let
\[
M = \begin{pmatrix} \alpha \\ \beta \end{pmatrix}, \quad
M' = \begin{pmatrix} \alpha' \\ \beta' \end{pmatrix}, \quad
N= \begin{pmatrix} \gamma & \delta \end{pmatrix}, \quad
N' = \begin{pmatrix} \gamma' & \delta' \end{pmatrix}.
\quad \text{Also let} \quad
v = \begin{pmatrix} -\delta \\ \gamma \end{pmatrix}
\]
so that $Nv=0$. Then $M'N'v = MNv=0$, which can be written out as
\[
M'N' v = 
M' \begin{pmatrix} \gamma' & \delta' \end{pmatrix}
\begin{pmatrix} -\delta \\ \gamma \end{pmatrix}
= M' (-\gamma' \delta + \delta' \gamma)
=
\begin{pmatrix} 0 \\ 0 \end{pmatrix}.
\]
Since $M'$ is nonzero, we have $-\gamma' \delta + \delta' \gamma = 0$, so that $N$ and $N'$ are (complex) proportional. A similar argument shows that $M$ and $M'$ are (complex) proportional. Since $MN=M'N'$, these proportions are inverses. Thus $M = \mu M'$ and $N = \mu^{-1} N'$ for some complex $\mu$.
\end{proof}

\begin{lem}
\label{Lem:flag_well_defined}
For any $\kappa \neq 0$, the three Hermitian matrices
\[
\f(\kappa), \quad
D_\kappa \f(\ZZ(\kappa)), \quad
D_\kappa \f (i \ZZ(\kappa))
\]
are linearly independent over $\R$.
\end{lem}
It follows that $D_\kappa \f(\ZZ(\kappa))$ is not a real multiple of $\f(\kappa)$, and hence $\F$ is well defined.

\begin{proof}
Applying \refeqn{derivative_flag_dirn}, we must show that for all $\kappa \neq 0$, the Hermitian matrices
\[
\kappa \kappa^*, \quad
\kappa \kappa^T J + J \overline{\kappa} \kappa^*, \quad
-i \left( \kappa \kappa^T J - J \overline{\kappa} \kappa^* \right)
\]
are linearly independent over $\R$. Suppose to the contrary that they are not: then we have
\[
a \kappa \kappa^* + b \left( \kappa \kappa^T J + J \overline{\kappa} \kappa^* \right) - ci \left(\kappa \kappa^T J - J \overline{\kappa} \kappa^* \right) = 0,
\]
for some real $a,b,c$, not all zero. We may rewrite this as 
\[
\kappa \left( a \kappa^* + b \kappa^T J - c i \kappa^T J \right) = \left( b J \overline{\kappa} + c i J \overline{\kappa} \right) \left( - \kappa^* \right).
\]
Let $\beta = b + ci$. Note $\beta = 0$ implies $a \kappa \kappa^* = 0$, a contradiction since $\kappa \in \C^2_\times$ and $a,b,c$ are not all zero; so $\beta \neq 0$. The equation can be written as
\[
\kappa \left( a \kappa^* + \overline{\beta} \kappa^T J \right) = \left( J \overline{\kappa} \right) \left( - \beta \kappa^* \right),
\]
where both sides are a product of a $2 \times 1$ and $1 \times 2$ complex matrix. On the right hand side, both factors are nonzero, hence the same must be true on the left hand side.

Applying \reflem{spinor_factorisation} we have $\kappa = \mu J \overline{\kappa}$ for some $\mu\neq0\in\C$. Letting $\kappa = (\xi, \eta)$ we thus have
\[
\begin{pmatrix} \xi \\ \eta \end{pmatrix}
= \mu
\begin{pmatrix} 0 & i \\ -i & 0 \end{pmatrix}
\begin{pmatrix} \overline{\xi} \\ \overline{\eta} \end{pmatrix}
= \mu \begin{pmatrix} \overline{\eta} \, i \\ - \overline{\xi} \, i \end{pmatrix},
\]
so that $\xi = \mu \overline{\eta} i$ and $\eta = -\mu \overline{\xi} i$, hence  $\overline{\eta} = \overline{\mu} \xi i$. But putting these together yields
\[
\xi 
= \mu \overline{\eta} i
= \mu (\overline{\mu} \xi i) i
= -|\mu|^2 \xi.
\]
Thus $\xi = 0$, which implies $\eta = 0$, contradicting $\kappa \neq 0$.
\end{proof}

After \reflem{flag_well_defined}, we can give quite a precise description of the derivative of $\f$. At a point $\kappa$, the derivative $D_\kappa \f$ is a real linear map between tangent spaces $T_\kappa \C^2 \To T_{\f(\kappa)} \HH$. As both $\C^2$ and $\HH$ are real vector spaces, we may identify these tangent spaces with $\C^2$ and $\HH$ respectively.
\begin{lem}
\label{Lem:structure_of_derivative_of_f}
For any $\kappa \in \C^2_\times$, the derivative
$D_\kappa \f$, considered as a real linear map $\C^2 \To \HH$, has the following properties.
\begin{enumerate}
\item The kernel of $D_\kappa \f$ is 1-dimensional, spanned by $i \kappa$.
\item $\kappa, \ZZ(\kappa), i \ZZ(\kappa) \in \C^2$ are linearly independent over $\R$, and their 3-dimensional span maps isomorphically onto the image of $D_\kappa \f$.
\end{enumerate}
\end{lem}
We will see later in \reflem{orthonormal_basis_from_spinor} some nice properties of the three vectors in (ii) and their images.

\begin{proof}
By \reflem{Z_forms_basis}, $\{ \kappa, \ZZ(\kappa)\}$ is a complex basis for $\C^2$, hence $\{ \kappa, i \kappa, \ZZ(\kappa), i \ZZ(\kappa) \}$ is a real basis for $\C^2$. We consider the effect of $D_\kappa \f$ on this basis.

We saw in \reflem{derivatives_of_f_in_easy_directions} that $i \kappa \in \ker D_\kappa \f$, so the kernel of $D_\kappa \f$ has dimension $\geq 1$ and the image of $D_\kappa \f$ has dimension $\leq 3$.

Since $D_\kappa \f (\kappa) = 2 \f(\kappa)$ (\reflem{derivatives_of_f_in_easy_directions}), \reflem{flag_well_defined} tells us that the images of $\kappa, \ZZ(\kappa), i \ZZ(\kappa)$ under $D_\kappa \f$ are linearly independent. So the image of $D_\kappa \f$ has dimension exactly $3$, spanned by the image of these 3 vectors, and the kernel has dimension has exactly $1$, spanned by $i \kappa$.
\end{proof}

Combining \refdef{spinors_to_PNF}, equation \refeqn{F_explicitly} and \reflem{G_in_pv_notation}, we immediately obtain the following description of $\G \circ \F \colon \C_\times^2 \To \mathcal{F_P^O}(\R^{1,3})$. This shows how to associate a flag in Minkowski space to a spin vector.
\begin{lem}
\label{Lem:GoF_in_pv_form}
\[
\G \circ \F (\kappa) = [[ \g \circ \f (\kappa), \g \left( D_\kappa \f (\ZZ(\kappa)) \right) ]]
= [[ \g \left( \kappa \kappa^* \right) , \g \left( \kappa \kappa^T J + J \overline{\kappa} \kappa^* \right) ]].
\]
\qed
\end{lem}

\subsubsection{$SL(2,\C)$ action on flags in $\HH$}
\label{Sec:SL2c_action_on_flags_HH}

We now explain how $SL(2,\C)$ acts on flags in $\HH$. In \refsec{equivariance_of_F} we consider equivariance of $\F$ with respect to this action.

We have considered flags both in $\HH$ and $\R^{1,3}$, but the isomorphism $\G$ shows that it is equivalent to consider either space of flags. Although $\R^{1,3}$ is perhaps easier to understand geometrically, it is more straightforward algebraically to consider the action on flags in $\HH$, and so we will consider $\HH$ first. From \refsec{flags_Minkowski_space} onwards we will consider $\R^{1,3}$.

To define the action of $SL(2,\C)$ on the space of flags $\mathcal{F_P^O}(\HH)$, we need to consider its actions on subspaces of $\HH$, their quotient spaces, and their orientations. We start with subspaces, extending the action on $\HH$ from \refdef{standard_SL2C_actions}.
\begin{defn}
\label{Def:matrix_on_Hermitian_subspace}
Let $V$ be a real vector subspace of $\HH$, and $A \in SL(2,\C$). Then the action of $A$ on $V$ is given by
\[
A\cdot V = \left\{ A\cdot S \mid S \in V \right\}
= \left\{ ASA^* \mid S \in V \right\} = AVA^*.
\]
\end{defn}
The same calculation as for $\HH$ \refeqn{group_action_on_Hermitian} shows that, for $A,A' \in SL(2,\C)$, we have $(AA') \cdot V = A \cdot (A' \cdot V)$, so we indeed have an action of $SL(2,\C)$ on the set of subspaces of $\HH$.

In fact, as we now see, this action is by linear isomorphisms. 
\begin{lem}
Let $V$ be a real $k$-dimensional subspace of $\HH$ and $A \in SL(2,\C)$.
\label{Lem:SL2C_action_preserves_dimension}
\begin{enumerate}
\item
The map $V \To A \cdot V$ defined by $S \mapsto A \cdot S$ for $S \in V$ is a linear isomorphism. In particular, $A\cdot V$ is also a $k$-dimensional subspace of $\HH$.
\item 
\refdef{matrix_on_Hermitian_subspace}
defines an action of $SL(2,\C)$ on the set of real $k$-dimensional subspaces of $\HH$.
\end{enumerate}
\end{lem}
The set of $k$-dimensional subspaces of $\HH$ forms the \emph{Grassmannian} $\Gr(k,\HH)$, so the above lemma says that $SL(2,\C)$ acts on $\Gr(k,\HH)$ by linear isomorphisms. 

\begin{proof}
The map $V \To A \cdot V$ is given by the action of $A$ on individual elements $S$ of $\HH$, i.e. $S \mapsto A \cdot S = A S A^*$. This is a real linear map, as shown explicitly in \refeqn{linear_action_on_Hermitian}. It is also invertible, with inverse given by the action of $A^{-1}$. Thus $V$ and $A \cdot V$ must have the same dimension.
\end{proof}

Next we consider the action of $SL(2,\C)$ on quotients of subspaces of $\HH$, and their bases.

For the rest of this subsection, $V \subset W$ are real subspaces of $\HH$, and $A \in SL(2,\C)$.
\begin{lem} \
\label{Lem:SL2C_action_subspaces_facts}
\begin{enumerate}
\item
$A \cdot V \subset A \cdot W$, so the quotient $(A \cdot W) / (A \cdot V)$ is well defined.
\item
Let $\underline{S} = S + V \in W/V$, i.e. $S \in W$ represents $\underline{S}$.
Then $A \underline{S} A^*$ is a well-defined element of $(A\cdot W)/(A\cdot V)$, represented by $A\cdot S = A S A^* \in A\cdot W$. 
\item
The map $W/V \To (A \cdot W) / (A \cdot V)$ defined by $\underline{S} \mapsto A \underline{S} A^*$ is a linear isomorphism. 
\item
\label{Lem:action_on_ordered_bases}
If  $\underline{S}_1, \ldots, \underline{S}_k$ is a basis of of $W/V$, then $A \underline{S}_1 A^*, \ldots, A \underline{S}_k A^*$ is a basis of $(A\cdot W)/(A\cdot V)$.
\end{enumerate}
\end{lem}

In (ii) above, we think of $A \underline{S} A^*$ as the action of $A$ on $\underline{S} \in W/V$, and define $A \cdot \underline{S} = A \underline{S} A^* \in (A \cdot W)/(A \cdot V)$. If $A,A' \in SL(2,\C)$ then for $\underline{S}$ an element of $W/V$, we have a similar calculation as \refeqn{group_action_on_Hermitian} 
\begin{equation}
\label{Eqn:group_action_on_quotient}
(AA') \cdot \underline{S} = (AA') \underline{S} (AA')^* = A A' \underline{S} A'^* A^* = A \cdot (A' \underline{S} A'^*) = A \cdot (A' \cdot \underline{S}),
\end{equation}
showing that we have a group action of $SL(2,\C)$ on quotients of subspaces of $\HH$.

\begin{proof} \
\begin{enumerate}
\item
An element of $A \cdot V$ can be written as $A \cdot S$ for some $S \in V$; as $V \subset W$ then $S \in W$, so $A \cdot S \in A \cdot W$. Thus $A \cdot V \subset A \cdot W$.
\item
If $S' \in [S]$ is another representative of $\underline{S}$, then $S-S'  \in V$, so  $A\cdot S - A\cdot S' = A\cdot (S - S') \in A\cdot V$.
\item
The same calculation as in \refeqn{linear_action_on_Hermitian}
shows that $\underline{S} \mapsto A \underline{S} A^*$ is linear in $\underline{S}$. And as in \reflem{SL2C_action_preserves_dimension}, this linear map is invertible, with inverse given by the action of $A^{-1}$.
\item
Immediate from the previous part, since a linear isomorphism sends a basis to a basis.
\end{enumerate}
\end{proof}

In (iv) above, we think of the basis $A \underline{S}_i A^*$ as the action of $A$ on the basis $\underline{S}_i$. Writing  $\underline{B} = (\underline{S}_1, \ldots, \underline{S}_k)$ for the ordered basis, we define $A \cdot \underline{B} = (A \cdot \underline{S}_1, \ldots, A \cdot \underline{S}_k)$.

For $A,A' \in SL(2,\C)$ and $\underline{B}$ an ordered basis, we then have $(AA') \cdot \underline{B} = A \cdot (A' \cdot \underline{B})$, by a similar calculation as \refeqn{group_action_on_quotient}.
Thus, we have a group action of $SL(2,\C)$ on ordered bases of quotients of subspaces of $\HH$.

Next, consider \emph{two} ordered bases $\underline{B} = (\underline{S}_1, \ldots, \underline{S}_k)$ and $\underline{B}' = (\underline{S}'_1, \ldots, \underline{S}'_k)$, and their orientations. By \reflem{SL2C_action_subspaces_facts}(iv) then $A \cdot \underline{B}$ and $A \cdot \underline{B}'$ are ordered bases of $(A \cdot W)/(A \cdot V)$.
\begin{lem}
\label{Lem:change_of_basis_matrix_after_action}
\label{Lem:action_on_coorientation}
Let $\underline{B}, \underline{B}'$ be two ordered bases of $W/V$ as above.
\begin{enumerate}
\item 
Let $M$ be the linear map of $W/V$ taking the ordered basis $\underline{B}$ to $\underline{B}'$, and $N$ the linear map of $(A \cdot W)/(A \cdot V)$ taking the ordered basis $A \cdot \underline{B}$ to $A \cdot \underline{B}'$. Then $\det M= \det N$.
\item 
If $\underline{B}$ and $\underline{B}'$ are ordered bases of $W/V$ representing the same orientation, then $A\cdot \underline{B}$ and $A\cdot \underline{B}'$ represent the same orientation of $(A\cdot W)/(A\cdot V)$.
\end{enumerate}
\end{lem}

\begin{proof}
By \reflem{SL2C_action_subspaces_facts}(iii), the map $T_A \colon W/V \To (A \cdot W)/(A \cdot V)$ given by $\underline{S} \mapsto A \cdot \underline{S}$ is a linear isomorphism, and by definition it sends the ordered basis $\underline{B}$ to $A \cdot \underline{B}$ and $\underline{B}'$ to $A \cdot \underline{B}'$. Thus $T_A M = N T_A$, and the matrix of $M$ with respect to $\underline{B}$ (or $\underline{B}'$) is equal to the matrix of $N$ with respect to $A \cdot \underline{B}$ (or $A \cdot \underline{B}'$). Thus $\det M = \det N$.

If $\underline{B}, \underline{B}'$ represent the same orientation, then $\det M > 0$, so $\det N = \det M > 0$. Thus $A \cdot \underline{B}$ and $A \cdot \underline{B}'$ represent the same orientation.
\end{proof}

Recall from \refdef{pointed_oriented_null_flag} that the orientations in flags are orientations on quotients of subspaces. For an orientation $o$ on $W/V$ then we can define $A \cdot o$ to be the orientation on $(A \cdot W)/(A \cdot V)$ represented by $A \cdot \underline{B}$, where $\underline{B}$ is any ordered basis of $W/V$ representing $o$. By the above lemma, $A \cdot o$ is well defined.

For $A,A' \in SL(2,\C)$, we observe that $(AA')\cdot o = A\cdot (A' \cdot o)$. Indeed, taking a basis $\underline{B}$ representing $o$, we saw that $(AA') \cdot \underline{B} = A \cdot (A' \cdot \underline{B})$, which are bases representing the orientations $(AA') \cdot o$ and $A \cdot (A' \cdot o)$ respectively. Thus we have a group action of $SL(2,\C)$ on orientations of quotients of subspaces of $\HH$.

We can now define an action of $SL(2,\C)$ on flags in $\HH$.
\begin{defn}
\label{Def:matrix_on_PONF}
Consider $(S,V,o)\in\mathcal{F_P^O}(\HH)$ and let $A \in SL(2,\C)$. Define $A$ to act on $(S,V,o)$ by
\[
A\cdot (S,V,o) = (A\cdot S, A\cdot V, A\cdot o).
\]
\end{defn}

\begin{lem}
\label{Lem:SL2C_act_on_PONF_H}
\refdef{matrix_on_PONF} defines an action of $SL(2,\C)$ on $\mathcal{F_P^O}(\HH)$.
\end{lem}

\begin{proof}
First we check that $(A\cdot S, A\cdot V, A \cdot o)$ is indeed a pointed oriented null flag. We know that $SL(2,\C)$ acts on $\HH_0^+$ (\reflem{SL2C_preerves_Hs}), so $A \cdot S \in \HH_0^+$. As the $SL(2,\C)$ action preserves 2-dimensional subspaces (\reflem{SL2C_action_preserves_dimension}), $A \cdot V$ is 2-dimensional. We also observe that $\R S \subset V$ implies $\R(A\cdot S) = \R(ASA^*) = A(\R S)A^* \subset AVA^* = A \cdot V$. 

As $(S,V) \in \mathcal{F_P}(\HH)$, by definition there exists $v \in T_S \HH_0^+$ such that $S$ and $v$ span $V$. Since the action of $A$ on subspaces is by linear isomorphisms (\reflem{SL2C_action_preserves_dimension}), then $A\cdot S$ and $A\cdot v$ span $A\cdot V$, and moreover, since $\HH_0^+$ lies in the vector space $\HH$, on which the action of $A$ is linear, we have $A\cdot v \in T_{A\cdot S} \HH_0^+$. Thus $\R(A\cdot S) \subset A\cdot V$ is a null flag and $(A\cdot S,A\cdot V) \in \mathcal{F_P}(\HH)$.

By \reflem{action_on_coorientation} and subsequent remarks, $A\cdot o$ is an orientation on $(A \cdot V) / (A\cdot \R S)$. Thus $(A \cdot S, A \cdot V, A \cdot o)$ is a pointed oriented null flag.

The actions of $SL(2,\C)$ on $\HH$, subspaces of $\HH$, and orientations are all group actions, by \refdef{SL2C_actions_on_C2_H}, \refdef{matrix_on_Hermitian_subspace}, and \reflem{action_on_coorientation} (and subsequent comments) respectively. So for $A,A' \in SL(2,\C)$ we have $(AA')\cdot (S,V,o) = A\cdot (A' \cdot (S, V, o))$, yielding the desired group action.
\end{proof}

The action of $SL(2,\C)$ on $\mathcal{F_P^O}(\HH)$ is described naturally in the notation $[[S,U]]$ of \refdef{pv_notation_PONF}. 
\begin{lem}
\label{Lem:action_on_pv_notation}
\label{Lem:action_on_pv_notation_PONF}
Let $[[S,U]] \in \mathcal{F_P^O}(\HH)$, and $A \in SL(2,\C)$, then
\[
A\cdot [[S,U]] = [[A\cdot S, A\cdot U]] = [[ASA^*, AUA^*]].
\]
\end{lem}

\begin{proof}
Letting $V$ be the real span of $S$ and $U$, and $o$ the orientation induced by $U$ on $V/\R S$, we have $[[S,U]] = (S, V, o)$. In particular, $\underline{U} = U + \R S \in V / \R S$ is an (ordered!) basis of the 1-dimensional quotient space $V / \R S$, and $o$ is the orientation given by $\underline{U}$.

By \refdef{matrix_on_PONF}, $A \cdot (S,V,o) = (A \cdot S, A \cdot V, A \cdot o)$. As $S,U$ is a basis of $V$, and $A$ acts by linear isomorphisms (\reflem{SL2C_action_preserves_dimension}), then $A \cdot S, A \cdot U$ is basis of $A \cdot V$. Moreover, the action of $A$ induces an isomorphism of quotient spaces $V / \R S \To (A \cdot V) / (A \cdot \R S)$ sending $\underline{U}$ to $A \cdot \underline{U}$ (\reflem{SL2C_action_subspaces_facts}), and $A \cdot o$ is the orientation given by $A \cdot \underline{U}$.

In other words, $A \cdot o$ is the orientation induced by $A \cdot U$ on $(A \cdot V)/(A \cdot \R S)$. Thus $(A \cdot S, A \cdot V, A \cdot o) = [[A \cdot S, A \cdot U]]$.
\end{proof}

\subsubsection{Equivariance of actions on spin vectors and flags in  $\HH$}
\label{Sec:equivariance_of_F}

In this section prove equivariance of $\F$ , as follows.
\begin{prop}
\label{Prop:SL2C_spinors_PNF_H_equivariant}
The actions of $SL(2,\C)$ on $\C_\times^2$ and $\mathcal{F_P^O}(\HH)$ are equivariant with respect to $\F$. In other words, for $\kappa \in \C_\times^2$ and $A \in SL(2,\C)$,
\[
A\cdot \F(\kappa) = \F(A\cdot\kappa).
\]
\end{prop}
The proof of \refprop{SL2C_spinors_PNF_H_equivariant} is essentially the first time we actually use $A \in SL(2,\C)$: the actions of $SL(2,\C)$ in \refdef{standard_SL2C_actions}, \reflem{restricted_actions_on_H}, and \refdef{matrix_on_Hermitian_subspace}--\reflem{action_on_pv_notation}
all work for $A \in GL(2,\C)$.

We will give two proofs of \refprop{SL2C_spinors_PNF_H_equivariant}, one conceptual, and one explicit. The first, conceptual proof is based on the following lemma.
\begin{lem}
\label{Lem:conceptual} 
For two spinors $\kappa,\nu\in\C^2_\times$, the following are equivalent:
\begin{enumerate}
\item $\{\kappa,\nu\}$ is negative imaginary,
\item $\nu=\alpha\kappa+b\ZZ(\kappa)$, where $\alpha\in\C,b\in\R^+$,
\item $[[\f(\kappa),D_\kappa \f(\nu)]]=\F(\kappa)$.
\end{enumerate}
\end{lem}
To motivate this lemma, note that all three equivalent conditions say, in various senses, that ``$\nu$ is like $\ZZ(\kappa)$". \reflem{bilinear_Z_negative_imaginary}
tells us that $\{ \kappa, \ZZ(\kappa) \}$ is negative imaginary, so (i) says that $\{\kappa, \nu\}$ is like $\{\kappa_, \ZZ(\kappa)\}$. Condition (ii) says that $\nu$ is, up to multiples of $\kappa$, a positive multiple of $\ZZ(\kappa)$. And \refeqn{F_explicitly} tells us that $\F(\kappa) =
[[\f(\kappa),D_\kappa \f(\ZZ(\kappa))]]$, so (iii) says that using the directional derivative of $\f$ in the direction $\nu$ yields the same flag as $\F$, which uses the direction $\ZZ(\kappa)$.

\begin{proof}
We first show (i) and (ii) are equivalent. Since $\{\cdot, \cdot\}$ is complex bilinear, if (ii) holds then
\[
\{\kappa, \nu\} 
= \alpha \{ \kappa, \kappa \} + b \{ \kappa, \ZZ(\kappa) \}
= b \{ \kappa, \ZZ(\kappa) \}
\]
which is negative imaginary by \reflem{bilinear_Z_negative_imaginary}, so (i) holds. For the converse, if $\{\kappa, \nu\}$ is negative imaginary then $\{\kappa, b\ZZ(\kappa)\} = \{\kappa, \nu\}$ for some positive $b$. As $\{\cdot,\cdot\}$ is a complex symplectic form on a complex 2-dimensional vector space, any two vectors yielding the same value for $\{\kappa,\cdot\}$ differ by a complex multiple of $\kappa$, so (ii) holds.

Next we show (ii) and (iii) are equivalent. For convenience, let $S = \f(\kappa)$, $U = D_\kappa \f(\nu)$ and $U' = D_\kappa \f(\ZZ(\kappa))$. Suppose (ii) holds, so that $\nu = \alpha \kappa + b \ZZ(\kappa)$, and we show that 
\[
[[\f(\kappa),D_\kappa \f(\nu)]]=[[\f(\kappa), D_\kappa \f(\ZZ(\kappa))]], \quad \text{i.e.} \quad [[S,U]] = [[S,U']].
\]

Let $\alpha = c + di$, where $c,d \in \R$. Then by the (real) linearity of the derivative of $\f$, and using the calculations of derivatives in the $\kappa$ direction (proportional to $\f(\kappa)$  and $i \kappa$ directions (the fibre direction) from \reflem{derivatives_of_f_in_easy_directions},
we have
\begin{align*}
U &= D_\kappa \f(\nu)
= D_\kappa \f ( c \kappa + d i \kappa + b \ZZ(\kappa) ) \\
&= c D_\kappa \f(\kappa) + d D_\kappa \f (i \kappa) + b D_\kappa \f (\ZZ(\kappa)) \\
&= 2 c \f(\kappa) + b D_\kappa \f(\ZZ(\kappa))
= 2 c S + b U'.
\end{align*}
We now apply \reflem{characterise_equal_PONFs}. Since $\F(\kappa) = [[S,U']]$ is a bona fide flag, $U'$ is not a real multiple of $S$. Since $U = 2cS + bU'$, we see that $U$ is not a real multiple of $S$ either. The equation $-2c S + U - bU' = 0$ above is a linear dependency between $S,U,U'$ with coefficients of opposite sign on $U$ and $U'$. Thus the flags are equal. Alternatively, one can observe that $\R S + \R^+ U = \R S + \R^+ U'$.

For the converse, suppose $[[S,U]] = [[S,U']]$. By \reflem{characterise_equal_PONFs}, we have a linear dependency and rearranging it, we have $U = a S + b U'$ where $a,b$ are real and $b>0$. Thus 
\[
D_\kappa \f(\nu) = a \f(\kappa) + b D_\kappa \f(\ZZ(\kappa)).
\]
Since $D_\kappa \f(\kappa) = 2 \f(\kappa)$ (\reflem{derivatives_of_f_in_easy_directions}), using the real linearity of $D_\kappa \f$, we have
\[
D_\kappa \f \left( \nu - \frac{a}{2} \kappa - b \ZZ(\kappa) \right) = 0.
\]
By \reflem{structure_of_derivative_of_f}, $D_\kappa \f$ has kernel spanned by $i \kappa$. Thus we have 
$\nu - \frac{a}{2} \kappa - b \ZZ(\kappa)  = c i \kappa$
for some real $c$. Letting $\alpha = a/2 + ci$, we have $\nu = \alpha \kappa + b \ZZ(\kappa)$, as required for (ii).
\end{proof}

\begin{proof}[Proof 1 of \refprop{SL2C_spinors_PNF_H_equivariant}]
We have $\F(\kappa)=[[\f(\kappa), D_\kappa \f(\ZZ(\kappa)]]$ so
\[
A\cdot \F(\kappa) = [[A \cdot \f(\kappa), A\cdot D_\kappa \f(\ZZ(\kappa))]] = [[\f(A\kappa), D_{A\kappa} \f(A(\ZZ(\kappa)))]],
\]
applying \reflem{action_on_pv_notation},  equivariance of $\f$ (\reflem{restricted_actions_on_H}) and its derivative \refeqn{equivariance_of_derivative_of_f}. Now as $A \in SL(2,\C)$, by \reflem{SL2C_by_symplectomorphisms} it acts on $\C^2$ by symplectomorphisms, so $\{A\kappa,A(\ZZ(\kappa))\} = \{\kappa,\ZZ(\kappa)\}$. But $\{\kappa, \ZZ(\kappa)\}$ is negative imaginary (\reflem{bilinear_Z_negative_imaginary}), so by \reflem{conceptual} then $[[ \f(A\kappa), D_{A\kappa} \f(A(\ZZ(\kappa)))]] = \F(A\kappa)$.
\end{proof}

The second, explicit proof of \refprop{SL2C_spinors_PNF_H_equivariant} is based on the following, perhaps surprising, identity.

\begin{prop}
\label{Prop:crazy_identity}
For any spin vector $\kappa \in \C^2$ and $A \in SL(2,\C)$,
\begin{align*}
\left[ A \kappa \kappa^T J A^* + A J \overline{\kappa} \kappa^* A^* \right] 
\left( \kappa^* A^* A \kappa \right)
=
\left[ A \kappa \kappa^T A^T J + J \overline{A} \overline{\kappa} \kappa^* A^* \right] \left( \kappa^* \kappa \right) ,
+
\left[ A \kappa \kappa^* A^* \right]
\left( \kappa^T J A^* A \kappa + \kappa^* A^* A J \overline{\kappa} \right).
\end{align*}
\end{prop}

\begin{proof}
Let $A = \begin{pmatrix} \alpha & \beta \\ \gamma & \delta \end{pmatrix}$ and $\kappa = \begin{pmatrix} \xi \\ \eta \end{pmatrix}$, and expand and simplify, using $\alpha \delta - \beta \gamma = 1$.
\end{proof}

\begin{proof}[Proof 2 of \refprop{SL2C_spinors_PNF_H_equivariant}]
From \refdef{spinors_to_PNF}  we have $\F(\kappa) = [[ \f(\kappa), D_\kappa \f(\ZZ(\kappa)) ]]$, and by \reflem{action_on_pv_notation_PONF} we have
\[
A\cdot \F(\kappa) = [[A\cdot \f(\kappa), A\cdot D_\kappa \f(\ZZ(\kappa)) ]].
\]
On the other hand, $A$ acts on $\kappa$ simply by matrix-vector multiplication, and we have
\begin{align*}
\F(A\cdot\kappa) &= \F(A\kappa) = [[ \f(A\kappa), D_{A\kappa} \f(\ZZ(A \kappa)) ]] 
\end{align*}

We now use \reflem{characterise_equal_PONFs} to show the two claimed pointed flags are equal, verifying (iii) there, which has three conditions.

The first condition is $A\cdot \f(\kappa) = \f(A \kappa)$; call this point $p$. This follows from equivariance of $\f$ (\reflem{restricted_actions_on_H}).  

The second condition is that $A\cdot D_\kappa \f(\ZZ(\kappa))$ and $D_{A \kappa} \f(\ZZ(A \kappa))$ both lie in the tangent space to $\HH_0^+$ at $p$, and are not real multiples of $p$.

Since $\f$ has image in $\HH_0^+$, the image of the derivative $D_\kappa \f$ lies in $T_{\f(\kappa)} \HH_0^+$, and hence $D_\kappa \f (\ZZ(\kappa)) \in T_{\f(\kappa)} \HH_0^+$. Moreover, by \reflem{flag_well_defined}, $D_\kappa \f(\ZZ(\kappa))$ is not a real multiple of $\f(\kappa)$. As $A$ acts linearly on $\HH$ preserving $\HH_0^+$, then $A\cdot D_\kappa \f(\ZZ(\kappa)) \in T_{p} \HH_0^+$. Similarly, the image of the derivative of $\f$ at $A \kappa$ lies in $T_{\f(A\kappa)} \HH_0^+$, so $D_{A \kappa} \f(\ZZ(A \kappa)) \in T_p \HH_0^+$.

Applying $A$, which acts linearly on $\HH$, sends $\f(\kappa)$ to $A\cdot \f(\kappa) = p$ and $D_\kappa \f(\ZZ(\kappa))$ to $A\cdot D_\kappa \f(\ZZ(\kappa))$. If these two did not span a plane, then the action of $A$ would send a 2-plane to a smaller dimensional subspace, contradicting \reflem{SL2C_action_preserves_dimension}. Thus $A\cdot D_\kappa \f(\ZZ(\kappa))$ is not a real multiple of $p$. Applying \reflem{flag_well_defined} to $A \kappa$ gives that $D_{A \kappa} \f(\ZZ(A \kappa))$ is not a real multiple of $\f(A \kappa) = p$ either.

The third condition is that there exist real numbers $a,b,c$ such that 
\begin{equation}
\label{Eqn:want_these_abc}
a \left( p \right) + b \left( A\cdot D_\kappa \f(\ZZ(\kappa)) \right)
+ c \left( D_{A \kappa} \f(\ZZ(A \kappa)) \right) = 0,
\end{equation}
where $b$ and $c$ have opposite signs. We calculate $p = A\cdot \f(\kappa) = A \kappa \kappa^* A^*$, and from \refeqn{F_explicitly} we have $D_\kappa \f(\ZZ(\kappa)) = \kappa \kappa^T J + J \overline{\kappa} \kappa^*$ so 
\[
A\cdot D_\kappa \f(\ZZ(\kappa)) 
= A\cdot \left( \kappa \kappa^T J + J \overline{\kappa} \kappa^* \right) 
= A \left( \kappa \kappa^T J + J \overline{\kappa} \kappa^* \right) A^*.
\]
and 
\[
D_{A\kappa} \f(\ZZ(A \kappa))
= (A\kappa) (A\kappa)^T J + J \overline{(A \kappa)} (A\kappa)^* 
= A \kappa \kappa^T A^T J + J \overline{A} \, \overline{\kappa} \kappa^* A^*.
\]
We can then rewrite \refprop{crazy_identity} as
\[
\left[ A\cdot D_\kappa \f(\ZZ(\kappa))  \right] 
\left( \kappa^* A^* A \kappa \right)
-
\left[ D_{A\kappa} \f(\ZZ(A \kappa)) \right] \left( \kappa^* \kappa \right) 
-
\left[ p \right]
\left( \kappa^T J A^* A \kappa + \kappa^* A^* A J \overline{\kappa} \right)
=
0,
\]
where the expressions in parentheses are real numbers. For any $\tau \in \C^2_\times$ written as a column vector, $\tau^* \tau$ is positive real; taking $\tau$ to be $A \kappa$ and $\kappa$ respectively, we see that 
$\kappa^* A^* A \kappa > 0$ and $-\kappa^* \kappa < 0$. Thus we have the required $a,b,c$ for \refeqn{want_these_abc}.
\end{proof}

\subsubsection{$SL(2,\C)$ action on flags in Minkowski space}
\label{Sec:flags_Minkowski_space}

We now translate all the above results on flags in $\HH$ into Minkowski space, using the maps $\g \colon \HH \To \R^{1,3}$ (\refdef{g_H_to_R31}) and $\G \colon \mathcal{F_P^O}(\HH) \To \mathcal{F_P^O}(\R^{1,3})$ (\refdef{G}). Essentially, $\g$ and $\G$ preserve all the structure required, so statements about flags in $\HH$ translate immediately to Minkowski space.

We have already defined a null flag (\refdef{null_flag_in_Minkowski}), pointed null flag (\refdef{pointed_null_flag}), pointed oriented null flag (\refdef{pointed_oriented_null_flag}), and $[[p,v]]$ notation for flags (\refdef{pv_notation_PONF}) in both $\HH$ and $\R^{1,3}$, and observed that $\g$ sends each object in $\HH$ to the corresponding object in $\R^{1,3}$, giving rise to the bijection $\G$.

We now define the $SL(2,\C)$ action on $\mathcal{F_P^O}(\R^{1,3})$ and show $\G$ is equivariant. We extend the action of $SL(2,\C)$ on $\R^{1,3}$ (\refdef{SL2C_on_R31}) to subspaces of $\R^{1,3}$, quotient spaces, and orientations. As in \refdef{SL2C_on_R31}, these actions are  imported directly from the corresponding actions in $\HH$. 

Throughout this section, $V \subset W$ are subspaces of $\R^{1,3}$, and $A \in SL(2,\C)$.
\begin{defn}
\label{Def:SL2C_on_R31_subspace}
\label{Def:SL2C_on_R31_orientations}
\label{Def:SL2C_on_PONF_R31}
The action of $A$ on:
\begin{enumerate}
\item a vector subspace $V$ of $\R^{1,3}$ is given by
\[
A\cdot V = \{A\cdot v \mid v \in V \} 
= \left\{ \g \left( A\cdot  \left( \g^{-1} v \right) \right) \mid v \in V \right\} 
= \g \left( A\cdot  \left( \g^{-1} (V) \right) \right)
= \g \left( A \left( \g^{-1} V \right) A^* \right);
\]
\item a quotient space $W/V$ is given by $A \cdot (W/V) = A \cdot W/A \cdot V$;
\item
an orientation $o$ on $W/V$ is given by 
$A \cdot o = \g \left( A\cdot \g^{-1} (o) \right)$;
\item
a flag $(p,V,o)\in\mathcal{F_P^O}(\R^{1,3})$, is given by
$A\cdot (p,V,o) = (A\cdot p, A\cdot V, A\cdot o)$.
\end{enumerate}
\end{defn}

Note that as $V \subset W$, then $A \cdot V \subset A \cdot W$, so (ii) above makes sense.

All these actions essentially derive from the action of $SL(2,\C)$ on $\R^{1,3}$. If $A \in SL(2,\C)$ acts on $\R^{1,3}$ via a linear map $M \in SO(1,3)^+$, then all of the actions above essentially just apply $M$. In particular, for a flag $(p,V,o)$, we have $A\cdot (p,V,o)=(Mp,MV,Mo)$.

It follows immediately from the fact that $\g$ is a linear isomorphism, and the results of  \refsec{SL2c_action_on_flags_HH}, that these definitions give actions of $SL(2,\C)$ on the following sets.
\begin{enumerate}
\item 
The set of subspaces of $\R^{1,3}$, acting by linear isomorphisms, using \reflem{SL2C_action_preserves_dimension}; also on each Grassmannian $\Gr(k,\R^{1,3})$.
\item
The set of quotients of subspaces of $\R^{1,3}$, acting by linear isomorphisms, using \reflem{SL2C_action_subspaces_facts} and subsequent comment.
\item
The set of orientations of quotients of subspaces of $\R^{1,3}$, using \reflem{action_on_coorientation} and subsequent comment.
\item the set of flags $\mathcal{F_P}(\R^{1,3})$, using \reflem{SL2C_act_on_PONF_H} and subsequent comment.
\end{enumerate}

Similarly we obtain the following immediate translation of \reflem{action_on_pv_notation}
\begin{lem}
\label{Lem:SL2c_action_on_PONF_R31_works}
For $[[p,v]] \in \mathcal{F_P^O}(\R^{1,3})$, we have
\[
A\cdot [[p,v]] = [[A\cdot p,A\cdot v]] 
\]
\qed
\end{lem}

All the actions of $SL(2,\C)$ on objects in $\R^{1,3}$ are defined by applying $\g^{-1}$, then apply the action in $\HH$, then applying $\g$. Hence they are all equivariant. In particular, We obtain the following  statement.
\begin{prop}
\label{Prop:FG_equivariant}
The actions of $SL(2,\C)$ on $\mathcal{F_P^O}(\HH)$ and $\mathcal{F_P^O}(\R^{1,3})$ are equivariant with respect to  $\G$. In other words, for any $A \in SL(2,\C)$ and any $(S,V,o) \in \mathcal{F_P^O}(\HH)$,
\[
\G( A \cdot (S,V,o)) = A \cdot \G(S,V,o),
\quad \text{i.e.} \quad
\begin{array}{ccc}
\mathcal{F_P^O}(\HH) & \stackrel{\G}{\To} & \mathcal{F_P^O}(\R^{1,3}) \\
\downarrow A && \downarrow A \\
\mathcal{F_P^O}(\HH) & \stackrel{\G}{\To} & \mathcal{F_P^O}(\R^{1,3})
\end{array}
\quad \text{commutes}.
\]
\qed
\end{prop}

\subsubsection{Flag intersection with the celestial sphere}
\label{Sec:calculating_flags_Minkowski}

Let us calculate some details of the flag of a spin vector. In particular, it will be useful to  describe its intersections with the celestial sphere $\S^+ = L^+ \cap \{T=1\}$ (\refdef{celestial_sphere}(ii))

Given a flag $(p,V,o) \in \mathcal{F_P^O}(\R^{1,3})$, the line $\R p$ intersects $\S^+$ in a point $q$. The 2-plane $V$ contains $\R p$, so is transverse to the 3-plane $T = 1$, and intersects this 3-plane in a 1-dimensional line. Because $V$ is tangent to the light cone, the line $V \cap \{T=1\}$ is tangent to $\S^+$ at $q$. The orientation $o$ on $V/\R p$ yields an orientation on this line $V \cap \{T=1\}$. 

Now, given a spin vector $\kappa = (\xi, \eta)$, by \reflem{GoF_in_pv_form} the associated flag $\G \circ \F(\kappa)$ in $\R^{1,3}$ is $[[p,v]]$, where $p = \g \circ \f (\kappa)$, and $v = \g (D_\kappa \f(\ZZ(\kappa)))$. The 2-plane $V$ is the span of $p$ and $v$, with orientation on $V/\R p$ given by $v$. In \refsec{f_compose_g} we gave explicit descriptions of $p$ (\reflem{spin_vector_to_TXYZ}), and the intersection point $q$ of the line $\R p$ with $\S^+$ (\reflem{gof_celestial_sphere}):
\begin{align*}
p &= \g \circ \f (\kappa)
= \left( a^2 + b^2 + c^2 + d^2, 2(ac+bd), 2(bc-ad), a^2 + b^2 - c^2 - d^2 \right) \\
q &= \left( 1, \frac{2(ac+bd)}{a^2+b^2+c^2+d^2}, \frac{2(bc-ad)}{a^2+b^2+c^2+d^2}, \frac{a^2+b^2-c^2-d^2}{a^2+b^2+c^2+d^2} \right).
\end{align*}

As we now see, $v$ has no $T$-component, and so gives a tangent vector to $\S^+$ at $q$, which is the oriented direction of the line $V \cap \{T=1\}$. See \reffig{flag_intersect_celestial_sphere}.

\begin{center}
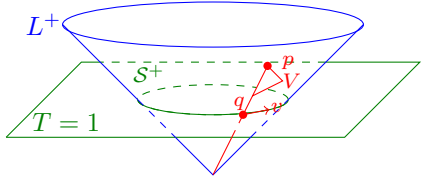

  \begin{tikzpicture}
  \draw[blue] (3.75,1.5) ellipse (2cm and 0.3cm);
  \draw[green!50!black] (3.75,0.5) ellipse (1cm and 0.2cm);
  \fill[white] (2.75,0.5)--(4.75,0.5)--(4.75,0.72)--(2.75,0.72);
  \draw[dashed, green!50!black] (3.75,0.5) ellipse (1cm and 0.2cm);
  \draw[green!50!black] (1,0)--(5.5,0)--(6.5,1)--(5.25,1);
  \draw[green!50!black] (2.25,1)--(2,1)--(1,0);
  \draw[dashed,green!50!black] (5.25,1)--(2.25,1);
  \draw[dashed,blue] (2.75,0.5)--(3.25,0);
  \draw[blue] (2.75,0.5)--(1.75,1.5);
  \draw[dashed, blue] (4.25,0)--(4.75,0.5);
  \draw[blue] (4.75,0.5)--(5.75,1.5);
  \draw[blue] (3.25,0)--(3.75,-0.5)--(4.25,0.0);
  \draw[red] (3.75,-0.5)--(4,0);
  \draw[dashed,red] (4,0)--(4.1875,0.375);
  \fill[white] (4.475,0.95)--(4.675,0.75)--(4.275,0.55);
  \draw[red] (4.1375,0.275)--(4.475,0.95)--(4.675,0.75)--(4.275,0.55);
  \node[blue] at (1.5,1.5){$L^+$};
  \fill[red] (4.475,0.95) circle (0.055cm);
  \fill[red] (4.15,0.3) circle (0.055cm);
  \node[red] at (4.75,1){\footnotesize$p$};
  \node[red] at (4.8,0.75){\footnotesize$V$};
  \node[red] at (4.1,0.45){\footnotesize$q$};
  \node[red] at (4.6,0.4){\footnotesize$v$};
  \draw[->,red](4.15,0.3)--(4.5,0.37);
  \node[green!50!black] at (1.8,0.2){$T=1$};
  \node[green!50!black] at (2.9,0.85){\footnotesize$\mathcal{S}^+$};
\end{tikzpicture} 
    \captionof{figure}{The intersection of a flag with the celestial sphere.}
    \label{Fig:flag_intersect_celestial_sphere}
\end{center}

For the rest of this section, we let $\kappa = (\xi, \eta) = (a+bi, c+di) \in \C^2_\times$ where $a,b,c,d \in \R$.

\begin{lem}
\label{Lem:null_flag_tricky_vector}
\label{Lem:null_flag_tricky_vector_PONF}
The 2-plane of the flag $\G \circ \F (\kappa)$ intersects any 3-plane of constant $T$ in a 1-dimensional line, and the orientation on the flag yields an orientation on this line. The oriented line's direction is
\[
v = \g (D_\kappa \f(\ZZ(\kappa))) = 2 \left( 0, 2(cd-ab), a^2 - b^2 + c^2 - d^2, 2(ad+bc) \right).
\]
\end{lem}

To see why $v$ has $T$-component zero, observe that $\kappa$ lies in a $3$-sphere $S^3_r$ of radius $r = |\xi|^2 + |\eta|^2 > 0$, and by \reflem{C2_to_R31_Hopf_fibrations}, each such 3-sphere maps under $\g \circ \f$ to a constant-$T$ slice of $L^+$, namely $L^+ \cap \{T=r^2\}$. Now the tangent vector $\ZZ(\kappa)$ at $\kappa$ in $\C^2$ is in fact tangent to $S^3_r$. Indeed, as discussed in \refsec{Z}, regarding $\kappa$ as a quaternion, $\ZZ(\kappa) = - \pmb{k} \kappa$, so that $\ZZ(\kappa)$ is orthogonal to the position vector of $\kappa$. Thus, under $D_\kappa (\g \circ \f) = \g \circ D_\kappa \f$, the vector $\ZZ(\kappa)$ tangent to $S^3_r$ is mapped to a tangent vector to $L^+ \cap \{ T = r^2 \}$, hence has $T$-component zero.

The expressions for $p$ and $v$ look quite similar. Indeed, their $X,Y,Z$ coordinates can be obtained from each other by permuting variables,  coordinates, and signs. As we see in the next section, this is not a coincidence.

In any case, we now calculate this vector.
\begin{proof}
Using 
\refdef{Z_C2_to_C2_and_J} and \refeqn{derivative_flag_dirn}, we calculate
\begin{align*}
D_\kappa \f (\ZZ(\kappa)) &= \kappa \kappa^T J + J \overline{\kappa} \kappa^* 
= \begin{pmatrix} \xi \\ \eta \end{pmatrix}
\begin{pmatrix} \xi & \eta \end{pmatrix}
\begin{pmatrix} 0 & i \\ -i & 0 \end{pmatrix}
+
\begin{pmatrix} 0 & i \\ -i & 0 \end{pmatrix}
\begin{pmatrix} \overline{\xi} \\ \overline{\eta} \end{pmatrix}
\begin{pmatrix} \overline{\xi} & \overline{\eta} \end{pmatrix} \\
&=
\begin{pmatrix} -i \xi \eta & i \xi^2 \\ -i \eta^2 & i \xi \eta \end{pmatrix}
+
\begin{pmatrix} i \overline{\xi \eta} & i \overline{\eta}^2 \\ -i \overline{\xi^2} & -i \overline{\xi \eta} \end{pmatrix} 
=
\begin{pmatrix}
i \left( \overline{\xi \eta} - \xi \eta \right) 
& i \left( \xi^2 + \overline{\eta}^2 \right) \\
-i \left( \overline{\xi}^2 + \eta^2 \right)
& i \left( \xi \eta - \overline{\xi \eta} \right)
\end{pmatrix}
\end{align*}
Thus, applying \refdef{g_H_to_R31},
\begin{align}
v = \g \left( D_\kappa \f(\ZZ(\kappa)) \right)
&=
\left( 0, 
2 \Re \left( i \left( \xi^2 + \overline{\eta}^2 \right) \right), 
2 \Im \left( i \left( \xi^2 + \overline{\eta}^2 \right) \right),
2i \left( \overline{\xi \eta} - \xi \eta \right)
 \right) \nonumber \\ 
\label{Eqn:flag_direction_in_terms_of_alpha_beta}
&=
 \left( 0, 
-2 \Im \left( \xi^2 + \overline{\eta}^2 \right), 
2 \Re \left( \xi^2 + \overline{\eta}^2 \right), 
4 \Im \left( \xi \eta \right)
\right),
\end{align}
using the identities $i(\overline{z}-z) = 2 \Im z$,  $\Re(iz) = -\Im(z)$ and $\Im(iz) = \Re(z)$.
We then directly calculate 
\begin{align*}
\xi^2 + \overline{\eta}^2  &= (a+bi)^2 + (c-di)^2 = a^2 - b^2 +c^2 - d^2 + 2(ab-cd)i, \\
\xi \eta &= (a+bi)(c+di) = ac-bd + (ad+bc)i
\end{align*}
and substituting real and imaginary parts give the desired expression for $v$.

Since $v$ has $T$-coordinate $0$, when we intersect $V$ with a 3-plane $T = $ constant, $V$ yields a line in the direction of $v$. The orientation on $V/\R p$ given by $v$ yields the orientation on this line given by $v$.
\end{proof}

\begin{eg}
\label{Eg:flag_of_simple_spinors}
Let us compute the flag of the spinor $\kappa_0 = (1,0)$.
By direct calculation, or using \reflem{spin_vector_to_TXYZ}, we have $\g \circ \f (\kappa_0) = (1, 0, 0, 1)$; let this point be $p_0$.
From \reflem{null_flag_tricky_vector} we have
\[
\G \circ \F (\kappa_0) = [[p_0, (0,0,1,0)]]
\]
i.e. the flag points in the  $Y$-direction. The quotient $V/\R p_0$ is spanned and oriented by $(0,0,1,0)$.

More generally, if we take $\kappa = (e^{i\theta}, 0)$, we obtain $\g \circ \f (\kappa_0) = (1,0,0,1) = p_0$ again, but now (again using \reflem{null_flag_tricky_vector} with $a=\cos \theta$, $b = \sin \theta$), we have 
\[
\G \circ \F(\kappa) = [[p_0, (0, -\sin 2\theta, \cos 2\theta, 0)]].
\]
Now $V/\R p_0$ is spanned and oriented by the vector $(0,-\sin2\theta, \cos 2\theta, 0)$. 

Thus as $\kappa$ rotates from $(1,0)$ by an angle of $\theta$, multiplying $\kappa$ by $e^{i\theta}$, $p$ remains constant, but the flag rotates by an angle of $2\theta$. Indeed, as the direction is $(0,\sin(-2\theta),\cos(-2\theta),0)$, it may be better to say that the flag rotates by an angle of $-2\theta$. 
\end{eg}

We will next see that this principle applies to spinors generally: multiplying a spinor by $e^{i\theta}$ rotates a flag by $-2\theta$, in an appropriate sense.

\subsubsection{Rotating flags}
\label{Sec:rotating_flags}

Given $p\in L^+$, we now consider the set of flags $(p,V,o)$ based at $p$. We first consider which 2-planes $V$ may arise, and for this we need a description of the tangent space to the light cone.
\begin{lem}
\label{Lem:light_cone_orthogonal_complement}
At any $p \in L^+$, the tangent space to $L^+$ is the orthogonal complement $p^\perp$ with respect to the Minkowski inner product:
\[
T_p L^+ = \{ v \in \R^{1,3} \mid \langle p,v \rangle = 0 \} = p^\perp.
\]
\end{lem}

\begin{proof}
A smooth curve $p(s)$ on $L^+$ passing through $p(0) = p$ satisfies $\langle p(s),p(s) \rangle = 0$ for all $s$. Differentiating and setting $s=0$ yields $\langle p, p'(0) \rangle = 0$ Thus $T_p L^+ \subseteq p^\perp$. As both are 3-dimensional linear subspaces they are equal.
\end{proof}

Thus, the 2-planes $V$ which may arise in a flag based at $p \in L^+$ are precisely those satisfying   $\R p \subset V \subset p^\perp = T_p L^+$. Since $p \in L^+$, $p$ has positive $T$-coordinate, so the ray $\R p$ is transverse to any 3-plane $T =$ constant; moreover, $V$ and $p^\perp$ are also transverse to $T=$ constant. Thus such a $V$ intersects a 3-plane $T=$ constant in a line, which also lies in $p^\perp$. Conversely, a line in a 3-plane $T=$ constant, which also lies in $p^\perp$ spans, together with $p$, a 2-plane $V$ such that $\R p\subset V \subset p^\perp$. So the 2-planes $V$ arising in pointed null flags starting from $p$ can be characterised via their 1-dimensional intersections with 3-planes of constant $T$. The intersections of such 2-planes $V$ with the 3-plane $T=0$ are precisely the 1-dimensional subspaces of the 2-plane $\{T=0\} \cap p^\perp$.

A flag also includes an orientation $o$ on $V/\R p$. As $p$ has positive $T$-coordinate, each vector in $V/\R p$ has a unique representative with $T$-coordinate zero, giving an isomorphism $V/\R p \cong V \cap \{T=0\}$. The orientation $o$ on $V/\R p$ is thus equivalent to an orientation on the 1-dimensional subspace $V \cap \{T=0\}$.

Thus, the flags based at $p$ can be characterised by their oriented intersections with $\{T=0\}$, and correspond precisely to the oriented 1-dimensional subspaces of the 2-plane $\{T=0\} \cap p^\perp$. There is an $S^1$ family of oriented lines through the origin in a 2-plane, and so there is an $S^1$ family of flags based at $p$. 

To investigate how flags rotate, we set up a useful basis.

Let $\kappa = (\xi, \eta) = (a+bi, c+di) \in \C^2_\times$ where $a,b,c,d \in \R$, and let $|\xi|^2+|\eta|^2=r^2$, where $r>0$. Also let $S^3_r = \{ \kappa \in \C^2 \, \mid \, |\xi|^2 + |\eta|^2 = r^2 \}$ be the 3-sphere of radius $r>0$ in $\C^2$. The corresponding flag $\G \circ \F(\kappa)$ is $[[p,v]]$ where $p = \g \circ \f (\kappa) \in L^+$ and $v = \g \circ D_\kappa \f (\ZZ(\kappa)) \in T_p L^+$ (\reflem{GoF_in_pv_form}).  We calculated $p$ and $v$ explicitly in \reflem{spin_vector_to_TXYZ} and \reflem{null_flag_tricky_vector}. In \refsec{calculating_flags_Minkowski} we observed the algebraic similarity between the expressions for $p$ and $v$. We now extend them to provide a useful basis of the $XYZ$ 3-plane.

The $T$-coordinate of $p$ is $r^2$, so $p \in L^+ \cap \{T=r^2\}$, which is a 2-sphere of Euclidean radius $r$ in the 3-plane $T=r^2$ in Minkowski space. Indeed $L^+ \cap \{T=r^2\} = r^2 \S^+$, where the celestial sphere $\S^+ = L^+ \cap \{T=1\}$ is the unit sphere in the plane $T=1$ (\refdef{celestial_sphere}(ii)). Indeed, as observed in in \reflem{C2_to_R31_Hopf_fibrations}, $\g \circ \f$ restricts to a Hopf fibration $S^3_r \To r^2 \S^+$.

Thus the projection of $p$ to the $XYZ$ 3-plane has Euclidean length $r$. Similarly, (because of the algebraic similarity of $p$ and $v$), one can check that the $XYZ$-projection of $v$ also has length $r$. Since $v \in T_p L^+ = p^\perp$ we have $\langle p, v \rangle = 0$, and since the $T$-coordinate of $v$ is $0$ (\reflem{null_flag_tricky_vector} and discussed in \refsec{calculating_flags_Minkowski}), we deduce that the $XYZ$-projections of $p$ and $v$ are orthogonal in $\R^3$. Thus, they extend naturally to an orthogonal basis where all vectors have length $r$. 

When $r=1$, i.e. $\kappa \in S^3$, we saw in \reflem{gof_Hopf} that the $XYZ$-projection of $\g \circ \f$ is the Hopf fibration composed with stereographic projection. And in this case we obtain an orthonormal basis.

\begin{lem}
\label{Lem:orthonormal_basis_from_spinor}
For any $\kappa \in \C^2_\times$, the vectors $e_1(\kappa), e_2(\kappa), e_3(\kappa)$ below all have length $r$ and form a right-handed orthogonal basis of $\R^3$. Moreover, identifying $\R^3$ with the $T=0$ plane in $\R^{1,3}$, $e_1(\kappa)$ and $e_2 (\kappa)$ form an orthogonal basis for the 2-plane $\{T=0\} \cap p^\perp$.
\[ \begin{array}{rll}
    e_1 (\kappa) &= \left( a^2 - b^2 - c^2 + d^2, \; 2(ab+cd), 2(bd-ac) \right) &= \frac{1}{2} \pi_{XYZ} \circ \g \circ D_\kappa \f \left( i \ZZ(\kappa) \right) \\
    e_2 (\kappa) &= \left( 2(cd-ab), \; a^2 - b^2 + c^2 - d^2, \; 2(ad+bc) \right) &= \frac{1}{2} \pi_{XYZ} (v) = \frac{1}{2} \pi_{XYZ} \circ \g \circ D_\kappa \f \left( \ZZ(\kappa) \right)\\
    e_3(\kappa) &= \left( 2(ac+bd), \; 2(bc-ad), \; a^2 + b^2 - c^2 - d^2 \right) &= \pi_{XYZ} (p) 
    = \frac{1}{2} \pi_{XYZ} \circ \g \circ D_\kappa \f (\kappa) \\
\end{array} \]
\end{lem}

In \reflem{structure_of_derivative_of_f} we identified 3 vectors $\kappa, \ZZ(\kappa), i \ZZ(\kappa) \in \C^2$, which are orthogonal and have equal length $r$; at $\kappa$ they consist of  a radial vector and two tangent vectors to $S^3_r$. We showed that their images under the the derivative of $\f$ spanned the image of $D_\kappa \f$. Here we calculate that their images under the derivative of $\g \circ \f$ are also orthogonal and have equal length $r$.

\begin{proof}
These are direct calculations. In addition to the preceding lemmas mentioned above giving $e_2(\kappa)$ and $e_3 (\kappa)$, we can also use \reflem{derivatives_of_f_in_easy_directions} that $D_\kappa \f (\kappa) = 2 \f(\kappa)$.
A similar method as in the proof of  \reflem{null_flag_tricky_vector}, using \refeqn{derivative_formula}, gives $e_1 (\kappa)$.
One can check that the cross product of the first and second vectors yields $a^2 + b^2 + c^2 + d^2 = r^2$ times the third, so we have the correct orientation.

Now $p = (r^2, e_3(\kappa))$, using \reflem{spin_vector_to_TXYZ}. When regarded in $\R^{1,3}$, the $e_i$ have $T$-coordinate zero, so $\langle p, e_i \rangle = - e_3 \cdot e_i$, which is zero for $i=1,2$. Thus $e_1, e_2 \in \{T=0\} \cap p^\perp$. Since $e_1, e_2$ are orthogonal, and since as argued above $\{T=0\} \cap p^\perp$ is 2-dimensional, we have an orthogonal basis.
\end{proof}
We now have an explicit picture of the intersection of the flag of $\kappa$ in the 3-plane $T=r^2$ of Minkowski space. In this 3-plane, the light cone appears as a 2-sphere of radius $r^2$, $p$ appears at $e_3 (\kappa)$, and the tangent space to the light cone $T_p L^+ = p^\perp$ appears as the tangent 2-plane to the 2-sphere at $p$. The flag 2-plane appears as an oriented line through $p$ in the direction of $e_2 \sim v$; the possible flag 2-planes based at $p$ appear as oriented lines through $p$ tangent to the 2-sphere. See \reffig{flag_intersect_T_r_squared}.

\begin{center}
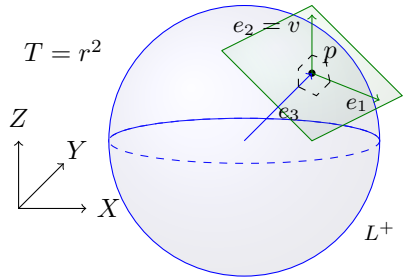

    \begin{tikzpicture}[scale=1.2]
    \draw[blue] (0,0) ellipse (1.5cm and 0.25cm);
       \fill[white] (-1.5,-0.25)--(1.5,-0.25)--(1.5,0.05)--(-1.5,0.05);
       \draw[dashed,blue] (0,0) ellipse (1.5cm and 0.25cm);
        \shade[ball color = blue!40, opacity = 0.1] (0,0) circle (1.5cm);
        \draw[blue] (0,0) circle (1.5cm);
        \shade[ball color=green!40,opacity=0.1] (-0.25,1)--(0.75,0)--(1.75,0.5)--(0.75,1.5)--(-0.25,1);
        \draw[green!50!black] (-0.25,1)--(0.75,0)--(1.75,0.5)--(0.75,1.5)--(-0.25,1);
        \fill (0.75,0.75) circle (0.04cm);
        \draw[blue, ->] (0,0)--(0.75,0.75);
        \draw[green!50!black,->](0.75,0.75)--(1.5,0.45);
        \draw[green!50!black,->] (0.75,0.75)--(0.75,1.4);
        \node at (-2,1){$T=r^2$};
        \node at (-2.5,0.25){$Z$};
        \node at (-1.5,-0.75){$X$};
        \node at (-1.85,-0.1){$Y$};
        \draw[<->](-2.5,0)--(-2.5,-0.75)--(-1.75,-0.75);
        \draw[->](-2.5,-0.75)--(-2,-0.25);
        \node at (0.95,0.95){$p$};
        \node at (0.5,0.3){\small$e_3$};
        \node at (0.25,1.25){\small$e_2=v$};
        \node at (1.25,0.4){\small$e_1$};
        \node at (1.5,-1){\footnotesize$L^+$};
        \draw[dashed] (0.6,0.6)--(0.8,0.5)--(0.95,0.65);
        \draw[dashed] (0.6,0.6)--(0.6,0.8)--(0.75,0.95);
        \draw[dashed] (0.95,0.65)--(0.9,0.9)--(0.75,0.95);
    \end{tikzpicture}
    \captionof{figure}{The intersection of the light cone, tangent space, and flag with the plane $T = r^2$.}
    \label{Fig:flag_intersect_T_r_squared}
\end{center}

As an aside, we note that
\[
\kappa = (\xi, \eta) \in S^3
\quad \text{corresponds to a matrix} \quad
\begin{pmatrix} \xi & - \overline{\eta} \\ \eta & \overline{\xi} \end{pmatrix}
\in SU(2),
\]
which in turn corresponds to a rotation of $\R^3$,  under the standard double covering map $SU(2) \To SO(3)$ (a subset of the double cover $SL(2,\C) \To SO(1,3)^+$ considered at length here). The images of the standard basis vectors in $\R^3$ under this rotation are precisely the $e_i (\kappa)$ here.

When $\kappa = (1,0)$, from \refeg{flag_of_simple_spinors}, $e_1, e_2, e_3$ are just unit vectors in the $X,Y,Z$ directions respectively, and we calculated that multiplying $\kappa$ by $e^{i\theta}$ preserved $e_3$ ($= \g \circ \f(\kappa)$) but rotated the flag direction $e_2$ by $-2\theta$ about $e_3$.

We now show this holds in general. In general, a rotation of $\R^3$ about $e_3$ by angle $\theta$ fixes $e_3$, sends $e_1 \mapsto e_1 \cos \theta  + e_2 \sin \theta$, and $e_2 \mapsto -e_1 \sin \theta + e_2 \cos \theta$. 
\begin{lem}
\label{Lem:flag_basis_rotation}
Each $e_i (e^{i\theta} \kappa)$ is obtained from $e_i (\kappa)$ by a rotation of angle $-2\theta$ about $e_3 (\kappa)$.
\end{lem}

\begin{proof}
We first observe that $\f(\kappa) = \f(e^{i\theta} \kappa)$ (\reflem{when_f_equal}) implies $e_3 (\kappa) = e_3 (e^{i \theta} \kappa)$.
We now calculate $e_2 (e^{i\theta} \kappa)$ directly. In \refeqn{flag_direction_in_terms_of_alpha_beta} we calculated an expression for $\g \circ D_\kappa \f (\ZZ(\kappa))$ in terms of $(\xi, \eta)$; replacing them with $e^{i\theta} (\xi, \eta)$ we obtain
\[
\g \circ D_\kappa \f (\ZZ (e^{i \theta} \kappa))
=
\left( 0, 
-2 \Im \left( e^{2 i \theta} \xi^2 + e^{-2i\theta} \overline{\eta}^2 \right), 
2 \Re \left( e^{2 i \theta} \xi^2 + e^{-2i\theta} \overline{\eta}^2 \right), 
4 \Im \left( e^{2 i \theta} \xi \eta \right)
\right).
\]
Now direct computations yield
\begin{align*}
e^{2 i \theta} \xi^2 + e^{-2i\theta} \overline{\eta}^2
&= \left( (a^2-b^2+c^2-d^2) \cos 2\theta - 2(ab+cd) \sin 2\theta \right) \\ & \quad \quad  + i \left( 2(ab-cd) \cos 2\theta + (a^2 - b^2 - c^2 + d^2) \sin 2\theta \right) \\
e^{2i\theta} \xi \eta &= \left( (ac-bd) \cos 2\theta - (ad+bc) \sin 2\theta \right) + i \left( (ad+bc) \cos 2\theta + (ac-bd) \sin 2\theta \right)
\end{align*}
so that $\pi_{XYZ} \circ \g \circ D_\kappa \f (\ZZ (e^{i \theta} \kappa))$ is given by
\begin{align*}
2 \Big( 2(cd-ab) \cos 2\theta &+ (-a^2 + b^2 + c^2 - d^2) \sin 2\theta, 
\;  (a^2 - b^2 + c^2 - d^2) \cos 2\theta - 2(ab+cd) \sin 2\theta, \\
& \quad \quad \quad 2(ad+bc) \cos 2\theta + 2(ac-bd) \sin 2\theta \Big)
\end{align*}
hence $e_2 (e^{i \theta} \kappa) = \frac{1}{2} \pi_{XYZ} \circ \g \circ D_\kappa \f (\ZZ (e^{i \theta} \kappa))$ is given by
\begin{align*}
\cos 2\theta & \left( 2(cd-ab), a^2 - b^2 + c^2 - d^2, 2(ad+bc) \right)
+ \sin 2\theta \left( -a^2 + b^2 + c^2 - d^2, -2(ab+cd), 2(ac-bd) \right) \\
&= e_2 (\kappa) \cos (-2\theta) + e_1 (\kappa) \sin (-2\theta)
\end{align*}
Thus both $e_2$ and $e_3$ behave as claimed. Since $e_1 (e^{i\theta} \kappa)$ forms a right-handed orthonormal basis with $e_2 (e^{i\theta} \kappa)$ and $e_3 (e^{i\theta} \kappa)$, the same must be true of $e_1$.
\end{proof}

\subsubsection{Surjectivity of maps to flags}
\label{Sec:F_surjectivity}

We now show that all flags arise via the maps $\F$ and $\G$. 
\begin{prop}
\label{Prop:F_G_surjective}
The maps $\F$ and $\G \circ \F$ are surjective.
\end{prop}

\begin{proof}
Since $\G$ is a bijection, it suffices to prove $\G \circ \F$ is a surjection $\C_\times^2 \To \mathcal{F_P^O}(\R^{1,3})$. 
As explained in \refsec{rotating_flags} above, there is an $S^1$ family of flags at a given basepoint $p \in L^+$, which can be characterised by their oriented 1-dimensional intersections with $\{T=0\}$, and these intersections are precisely the oriented 1-dimensional subspaces of the 2-plane $\{T=0\} \cap p^\perp$.

\refsec{rotating_flags} essentially shows that multiplying a spinor by $e^{i\theta}$ fixes the basepoint of a flag, but rotates through this $S^1$ family of flags based at $p$ by an angle of $-2\theta$. 

To see this explicitly, take $\kappa \in \C^2_\times$, which  yields the flag $\G \circ \F (\kappa) = [[p , \g \circ D_\kappa \f (\ZZ(\kappa))]]$ based at $p$, where $p = \g \circ \f (\kappa)$ (\reflem{GoF_in_pv_form}). Since $\g \circ D_\kappa \f (\ZZ(\kappa))$ has $T$-coordinate zero (\reflem{null_flag_tricky_vector}), the 2-plane of the flag intersects $\{T=0\}$ along $\g \circ D_\kappa \f (\ZZ(\kappa))$. So the flag $\G \circ \F (\kappa)$ corresponds to the oriented 1-dimensional subspace of $\{T=0\} \cap p^\perp$ given by $\g \circ D_\kappa \f (\ZZ(\kappa))$ or, if we regard $\R^3$ as the $T=0$ subset of Minkowski space, by $e_2 (\kappa)$. By \reflem{orthonormal_basis_from_spinor}, $e_1 (\kappa)$ and $e_2(\kappa) $ span the 2-plane $\{T=0\} \cap p^\perp$. By \reflem{flag_basis_rotation}, multiplying $\kappa$ by $e^{i\theta}$ rotates this plane in $\R^3$ by an angle of $-2\theta$, about the orthogonal vector $e_3 (\kappa)$. Thus as $\theta$ ranges through $[0,2\pi]$ (or even just $[0,\pi)$), all flags based at $p$ are obtained.

Thus, if $\G \circ \F$ contains in its image a flag based at a point $p \in L^+$, then it contains all flags based at $p$. It thus remains to show that all points of $L^+$ arise in the image of $\g \circ \f$. But we showed this in \reflem{gof_properties}.
\end{proof}

\begin{lem}
\label{Lem:F_G_2-1}
The maps $\F$ and $\G \circ \F$ are 2--1. More precisely, 
$\F(\kappa) = \F(\kappa')$ iff  $\G \circ \F (\kappa) = \G \circ \F (\kappa')$ iff $\kappa = \pm \kappa'$.
\end{lem}

\begin{proof}
Again as $\G$ is a bijection it suffices to show that $\G \circ \F$ is 2--1. Suppose two spinors $\kappa, \kappa'$ yield the same flag. Then in particular these flags have the same basepoint $p$, i.e. $\g \circ \f (\kappa) = \g \circ \f (\kappa') = p$. Hence $\kappa' = e^{i \theta} \kappa$ (\reflem{gof_properties}). We have seen (\reflem{flag_basis_rotation}) that the flag of $e^{i \theta} \kappa$ is is obtained from that of $\kappa$ by rotation by an angle of $-2\theta$ through the $S^1$ family of flags based at $p$. This $S^1$ family is characterised by the family of oriented lines in a 2-dimensional Euclidean plane, namely $\{T=0\} \cap p^\perp$. Thus, rotating a flag, we obtain the same flag when the rotation angle is an integer multiple of $2\pi$. Thus $\kappa = \pm \kappa'$. The converse follows equally from these observations: $-\kappa = e^{i\pi} \kappa$ has flag obtained from that of $\kappa$ by a rotation of $-2\pi$, hence yields the same flag.
\end{proof}
(If we ignore orientations, and consider only pointed null flags as per \refdef{pointed_null_flag}, then flags coincide when they are rotated by $\pi$ rather than $2\pi$, yielding 4--1 rather than 2--1 maps.)

We point out that there should be an extension of \refprop{complex_Minkowski_inner_products} using rotations between flags. There we found that for two spinors $\kappa, \kappa'$, the magnitude of $\{\kappa, \kappa'\}$ gave the Minkowski inner product of $p = \g \circ \f (\kappa)$ and $p' = \g \circ \f (\kappa')$. The argument of $\{\kappa, \kappa'\}$ should be related to the angles between the geodesic connecting $p$ to $p'$, and the flag directions of $\G \circ \F(\kappa), \G \circ \F (\kappa')$ at $p,p'$ respectively (or indeed, the directions $e_2(\kappa), e_2 (\kappa')$.

\subsection{From Minkowski space to the hyperboloid model}
\label{Sec:Minkowski_to_hyperboloid}

The third step in our journey is from Minkowski space to the hyperboloid model; we now finally enter hyperbolic space. We define the map $\h$ from the light cone to horospheres, and the map $\H$ from flags to decorated horospheres.

We proceed as follows. We first introduce and discuss the hyperboloid model (\refsec{hyperboloid_model}) and horospheres (\refsec{horospheres}). In \refsec{light_cone_to_horosphere} we define and discuss the map $\h$; in \refsec{SL2C_on_hyperboloid} we prove it is $SL(2,\C)$-equivariant. We briefly digress in \refsec{distances_between_horospheres} to discuss distances between horospheres, and how they can be found from spinors. In \refsec{flags_and_horospheres} we introduce the map $\H$, which produces an oriented line field on a horosphere; however at this stage we do not know that the line field is parallel. In \refsec{examples_from_10} we compute in detail flags and horospheres and decorations from the single spinor $(1,0)$; this work then pays off in \refsec{parallel_line_fields} when we show that oriented line fields obtained from $\H$ are parallel. In \refsec{decorated_horospheres} we define decorated horospheres and show $\H$ is a bijection. Finally, in \refsec{SL2c_on_decorated_horospheres} we show $\H$ is $SL(2,\C)$-equivariant.

\subsubsection{The hyperboloid model}
\label{Sec:hyperboloid_model}

\begin{defn}
The \emph{hyperboloid model} $\hyp$ is the Riemannian submanifold of $\R^{1,3}$ consisting of $x = (T,X,Y,Z) \in \R^{1,3}$ such that 
\[
T>0 \quad \text{and} \quad
\langle x,x \rangle = T^2 - X^2 - Y^2 - Z^2 = 1,
\]
with metric $ds^2 = dX^2 + dY^2 + dZ^2 - dT^2$.
\end{defn}

To see that $\hyp$ is a Riemannian (not Lorentzian or semi-Riemannian) manifold, observe that, by essentially the same proof as \reflem{light_cone_orthogonal_complement} for the light cone (which, like the hyperboloid, is part of a level set of the Minkowski norm function), we have, for any $q \in \hyp$,
\begin{equation}
\label{Eqn:hyperboloid_tangent_space}
T_q \hyp = q^\perp.
\end{equation}
As $q$ by definition has timelike position vector, all nonzero vectors in $q^\perp$ are spacelike. Thus all nonzero tangent vectors to $\hyp$ are spacelike. Reversing the sign of the metric on $\R^{1,3}$, we have a positive definite Riemannian metric on $\hyp$.

The cross section of $\hyp$ with a 3-plane of constant $T \geq 1$ is a Euclidean 2-sphere (of radius $\sqrt{T^2-1}$). The cross section of $L^+$ with such a 3-plane is also a Euclidean 2-sphere (of radius $T$). When $T$ becomes large, these 2-spheres become arbitrarily close and represent the possible directions of geodesics from a point in $\hyp$. Thus we may regard the \emph{sphere at infinity} of $\hyp$, which we write as $\partial \hyp$, as the celestial sphere $\S^+$ (the projectivisation of $L^+$, \refdef{celestial_sphere}(i)).

We denote the isometry group of $\hyp$ by $\Isom \hyp$, and its subgroup of orientation-preserving isometries by $\Isom^+ \hyp$. It is well known that $\Isom \hyp \cong O(1,3)^+$ and $\Isom^+ \hyp \cong SO(1,3)^+$, acting by linear transformations on $\R^{1,3}$. We saw a few examples in \refsec{Minkowski_space_and_g} of how the action of $SL(2,\C)$ gives rise to linear transformations of $\R^{1,3}$ in $SO(1,3)^+$. It is well known that this  map $SL(2,\C) \To SO(1,3)^+$ is a surjective homomorphism which is 2--1, with kernel $\pm I$.

\subsubsection{Horospheres}
\label{Sec:horospheres}

Horospheres in $\hyp$ are given by intersection with certain 3-planes $\Pi$ in $\R^{1,3}$; we now say precisely which.
As mentioned in \refsec{intro_horospheres_decorations}, they are analogous to 2-planes which cut out parabolic conic sections.

\begin{lem}
Let $\Pi$ be an affine 3-plane in $\R^{1,3}$. The following are equivalent.
\begin{enumerate}
\item
$\Pi$ has a lightlike tangent vector, and no timelike tangent vector.
\item
There exist a lightlike vector $n$ and $c \in \R$ so that 
$\Pi=\{x \in \R^{1,3}|\langle x, n \rangle = c \}$.
\item
$\Pi$ is parallel to $n^\perp$ where $n$ is lightlike.
\end{enumerate}
We call such a plane a \emph{lightlike 3-plane}.
\end{lem}

\begin{proof}
Let $n$ be a Minkowski normal vector to $\Pi$, so that $\Pi=\{x\in\R^{1,3}|\langle x, n \rangle = c\}$ for some $c\in\R$. Such $n$ is unique up to a nonzero real scalar; we take it to be future pointing, i.e. have non-negative $T$-coordinate. The tangent space to $\Pi$ is then the orthogonal complement $n^\perp$, and $\Pi$ is parallel to $n^\perp$.

If $n$ is timelike, after changing basis by a rotation in the $XYZ$ 3-plane (which is an isometry in $SO(1,3)^+$), we may arrange that $n = (T,X,0,0)$ where $T,X>0$. Similarly, if $n$ is spacelike (resp. timelike) then by a change of basis by boost in the $XT$ 2-plane, we may assume $n = (0,X,0,0)$ and $X>0$  (resp. $(T,0,0,0)$ and $T>0$).

If $n$ is spacelike, $n=(0,X,0,0)$ then $n^\perp$ contains $(1,0,0,0)$, which is timelike. Thus none of (i)--(iii) hold. Similarly, if $n$ is timelike, $n=(T,0,0,0)$, then $n^\perp=\{p=(T,X,Y,Z)|\ T=0\}$, so every nonzero vector in $n^\perp$ is spacelike, and again none of (i)--(iii) hold.

If $n$ is lightlike, $n=(T,X,0,0)$ with $T,X>0$, then $n^\perp=\{x = (T,X,Y,Z)|\ T=X\}$. Any such $x$ satisfies $\langle x,x \rangle = -Y^2-Z^2 \leq 0$ so is lightlike or spacelike. Thus all of (i)--(iii) hold.
\end{proof}
Not all lightlike 3-planes intersect $\hyp$; some pass below (in the past of) the positive light cone. 
\begin{lem}
\label{Lem:plane_intersect_hyperboloid}
A lightlike 3-plane $\Pi$ satisfies $\Pi\cap\hyp\neq\emptyset$ iff $\Pi=\{x\in\R^{1,3}|\langle x, n \rangle = c,\ n \in L^+,\ c>0\}$ for some $n$ and $c$.
\end{lem}
Any lightlike 3-plane has an equation $\langle x,n \rangle = c$ where $n \in L^+$; the point here is that only those with $c>0$ intersect $\hyp$.
\begin{proof}
Let $\Pi$ have equation $\langle x,n \rangle = c$ with $n \in L^+$. By a change of basis in $SO(1,3)^+$, we may assume $n = (1,1,0,0)$. Such a change of basis preserves $\langle \cdot, \cdot \rangle$ and $L^+$, hence $\Pi$ is given by an equation of the desired form iff its equation satisfies the desired form after this change of basis. The 3-plane $\Pi$ then has equation $T-X=c$. The plane intersects $\hyp$ iff there exist $(T,X,Y,Z)$ such that $T-X=c$, $T>0$ and $T^2 - X^2 - Y^2 - Z^2 = 1$.

Substituting the former into the latter yields $T^2 - (T-c)^2 -Y^2-Z^2=1 = 2cT-c^2-Y^2-Z^2=1$. If $c \leq 0$ then, as $T>0$, every term on the left is non-positive and we have a contradiction. If $c>0$ then there certainly are solutions, for instance $(T,X,Y,Z) = ((1+c^2)/2c, (1-c^2)/2c,0,0)$.
\end{proof}

\begin{defn}
\label{Def:set_of_horospheres}
A \emph{horosphere} in $\hyp$ is a non-empty intersection of $\hyp$ with a lightlike 3-plane. The set of all horospheres in $\hyp$ is denoted $\mathfrak{H}(\hyp)$.
\end{defn}
It is perhaps not obvious that this definition agrees with \refdef{intro_horosphere}; it is better seen via other models. In any case, a lightlike 3-plane $\Pi$ intersecting $\hyp$ determines a horosphere $\mathpzc{h}$; and conversely, $\mathpzc{h}$ determines the plane $\Pi$ as the unique affine 3-plane containing $\mathpzc{h}$. So there is a bijection
\[
\{ \text{Lightlike 3-planes $\Pi$ such that $\Pi \cap \hyp \neq \emptyset$} \}
\To \mathfrak{H}(\hyp), 
\]
given by intersection with $\hyp$.

A horosphere determines a distinguished point at infinity, i.e. ray on the light cone, as follows.
\begin{lem}
\label{Lem:horosphere_centre_exists}
Let $\mathpzc{h} \in \mathfrak{H}(\hyp)$ be the intersection of $\hyp$ with the lightlike 3-plane $\Pi$ with equation $\langle x,n \rangle = c$, where $n \in L^+$ and $c>0$. Then $\Pi$ intersects every ray of $L^+$ except the ray containing $n$.
\end{lem}

\begin{proof}
The 3-plane $\Pi$ is parallel to, and disjoint from, the 3-plane $n^\perp$, which contains the ray of $L^+$ through $n$. Thus $\Pi$ does not intersect the ray containing $n$.

To see that $\Pi$ intersects every other ray, let $p \in L^+$ be a point not on the ray through $n$. By a change of basis as in \reflem{plane_intersect_hyperboloid}, we may assume $n=(1,1,0,0)$, so $\Pi$ has equation $T-X=c$. Let $p = (T_0, X_0, Y_0, Z_0)$. Note that $T_0 > X_0$, for if $T_0 \leq X_0$ then $T_0^2 \leq X_0^2$ so $0 = \langle p,p \rangle = T_0^2 - X_0^2 - Y_0^2 - Z_0^2 \leq -Y_0^2 - Z_0^2$, so $Y_0 = Z_0 = 0$, so $p$ is on the ray through $n$. We then observe that the point $cp/(T_0 - X_0)$ lies on both the ray through $p$ (since it is a positive multiple of $p$), and $\Pi$ (since the $T$-coordinate $cT_0/(T_0 - X_0)$ and $X$-coordinate $cX_0/(T_0-X_0)$ differ by $c$).
\end{proof}

\begin{defn}
Let $\mathpzc{h} \in \mathfrak{H}(\hyp)$, corresponding to the lightlike 3-plane $\Pi$. The \emph{centre} of $\mathpzc{h}$ is the unique point of $\partial \hyp \cong \S^+$ such that $\Pi$ does not intersect the corresponding ray of $L^+$.
\end{defn}
Here we regard $\S^+$ as the projectivisation of $L^+$, \refdef{celestial_sphere}(i). By \reflem{horosphere_centre_exists}, if $\Pi$ has equation $\langle x, n \rangle = c$ where $n \in L^+$ and $c>0$, then the centre of $\mathpzc{h}$ is the point of $\S^+$ corresponding to the ray through the normal vector $n$.

\begin{defn}
Let $\mathpzc{h}$ be a horosphere, corresponding to the 3-plane $\Pi$. The \emph{horoball} bounded by $\mathpzc{h}$ is the subset of $\hyp$ bounded by $\h$, on the same side of $\Pi$ as its centre. The \emph{centre} of a horoball is the centre of its bounding horosphere.
\end{defn}
We may regard a horoball as a neighbourhood in $\hyp$ of its centre, a point at infinity in $\partial \hyp$.

{\flushleft \textbf{Remark.} } A horosphere appears in the hyperboloid model as a 2-dimensional paraboloid. To see this, again as in \reflem{plane_intersect_hyperboloid} we may change basis in $SO(1,3)^+$ and assume the lightlike 3-plane has equation $T-X=c$ where $c>0$ (we could in fact obtain equation $T-X=1$). Eliminating $T$ from $T-X=c$ and $T^2-X^2-Y^2-Z^2=1$ yields $(X+c)^2-X^2-Y^2-Z^2=1$, so $2cX-Y^2-Z^2=1-c^2$, hence $X=\frac{1}{2c} \left( Y^2 +Z^2 + 1-c^2 \right)$, which is the equation of a 2-dimensional paraboloid in $\R^3$. Thus the horosphere is the image of the paraboloid $X=\frac{1}{2c} \left( Y^2 +Z^2 + 1-c^2 \right)$ in $\R^3$ under the injective linear map $\R^3 \To \R^{1,3}$ given by $(X,Y,Z) \mapsto (X+c,X,Y,Z)$.

This remark makes clear that a horosphere has the topology of a 2-plane. In fact, a horosphere is isometric to the Euclidean plane; this is easier to see in other models of hyperbolic space.

\subsubsection{The map from the light cone to horospheres}
\label{Sec:light_cone_to_horosphere}

The following idea, assigning horospheres to points of $L^+$, goes back at least to Penner \cite{Penner87}, at least in 2-dimensional hyperbolic space.
\begin{defn}
\label{Def:h}
There is a bijection
\[
\h \colon L^+ \To \horos(\hyp)
\]
which sends $p \in L^+$ to the horosphere $\mathpzc{h}$ given by the intersection of $\hyp$ with the lightlike 3-plane with equation $\langle x, p \rangle = 1$.
\end{defn}

\begin{proof}
If $p \in L^+$ then by \reflem{plane_intersect_hyperboloid} the 3-plane $\langle x, p \rangle = 1$ is lightlike and intersects $\hyp$ nontrivially, yielding a horosphere, so the map is well defined.

To show $\h$ is bijective, we construct its inverse. So let $\mathpzc{h}$ be a horosphere, with corresponding lightlike 3-plane $\Pi$. By \reflem{plane_intersect_hyperboloid}, $\Pi$ has an equation of the form $\langle x, n \rangle = c$ where $n \in L^+$ and $c>0$. Dividing through by $c$, $\Pi$ has equivalent equation $\langle x, n/c \rangle = 1$. Now $n/c \in L^+$, and with the constant normalised to $1$, $\Pi$ has a unique equation of this form. Thus $n/c$ is the unique point in $L^+$ such that $\h(n/c) = \horo$.
\end{proof}

By \reflem{horosphere_centre_exists}, the horosphere $\h(p)$ has centre given by the ray through $p$.

Let us consider the geometry of the map $\h$. As $p$ is scaled up or down by multiples of $c>0$, the 3-plane $\langle x, p \rangle = 1$ is translated through a family of lightlike 3-planes with common normal, namely the ray through $p$. This is because $\langle x, cp \rangle = 1$ is equivalent to $\langle x, p \rangle = \frac{1}{c}$. The family of lightlike 3-planes are disjoint, and their intersections with $\hyp$ yield a family of horospheres with common centre foliating $\hyp$.
As $p$ goes to infinity, the 3-planes approach tangency with the light cone, and the corresponding horospheres also ``go to infinity", bounding decreasing horoballs, and eventually becoming arbitrarily far from any given point in $\hyp$.

The set $\horos(\hyp)$ naturally has the topology of $S^2 \times \R$. For instance, a horosphere is uniquely specified by its centre, a point of $\partial \hyp \cong \S^+ \cong S^2$, and a real parameter  specifying the position of $\horo$ in the foliation of $\hyp$ by horospheres about $p$. With this topology, $\h$ is a diffeomorphism.

Forgetting everything about the horosphere except its centre, we obtain the following, which is useful in the sequel.
\begin{defn}
\label{Def:h_partial_light_cone_to_hyp}
The map from the positive light cone to the boundary at infinity of $\hyp$
\[
\h_\partial \colon L^+ \To \partial \hyp = \S^+
\]
sends $p$ to the centre of $\h(p)$.
\end{defn}
Since the centre of $\h(p)$ is the ray through $p$, $\h_\partial$ is just the projectivisation map collapsing each ray of $L^+ \cong S^2 \times \R$ to a point, producing $\S^+ = \partial \hyp$. 

The map $\h$ also provides a nice description of the tangent spaces of a horosphere. We demonstrate this after giving a straightforward lemma that will be useful in the sequel.
\begin{lem}
\label{Lem:lightlike_intersection}
Let $q \in \hyp$ and $1 \leq k \leq 4$ be an integer. The intersection of the 3-plane $T_q \hyp = q^\perp$ with a $k$-plane $V \subset \R^{1,3}$ containing a lightlike or timelike vector is transverse, and hence $T_q \hyp \cap V$ has dimension $k-1$.
\end{lem}

\begin{proof}
As $T_q \hyp$ is spacelike, but $V$ contains a lightlike or timelike vector, $T_q \hyp + V$ has dimension more than $3$, hence $4$. Thus the intersection is transverse, and the intersection is as claimed.
\end{proof}

\begin{lem}
\label{Lem:tangent_space_of_horosphere}
Let $p \in L^+$ and let $q$ be a point on the horosphere $\h(p)$. Then the tangent space $T_q \h(p)$ is the 2-plane given by the following transverse intersection of 3-planes:
\[
T_q \h(p) = p^\perp \cap q^\perp.
\]
\end{lem}

\begin{proof}
Observe that $p^\perp$ is the tangent space to the 3-plane $\langle x,p \rangle = 1$ cutting out $\h(p)$, and $q^\perp$ is the tangent 3-plane to $\hyp$ at $q$, by \refeqn{hyperboloid_tangent_space}. So $T_q \h(p)$ is given as claimed. We explicitly calculated that horospheres are paraboloids, hence 2-dimensional manifolds, so the intersection must be transverse to obtain a 2-dimensional result. This can also be seen directly from \reflem{lightlike_intersection}, since $p^\perp$ contains the lightlike vector $p$.
\end{proof}

\subsubsection{$SL(2,\C)$ action on hyperboloid model}
\label{Sec:SL2C_on_hyperboloid}

We have seen that $SL(2,\C)$ acts on $\R^{1,3}$ in \refdef{SL2C_on_R31}, by linear maps in $SO(1,3)^+$. Linear maps in $SO(1,3)^+$ preserve the Minkowski metric, the positive light cone $L^+$, the hyperboloid $\hyp$, and lightlike 3-planes. They also send rays of $L^+$ to rays of $L^+$, send horospheres to horospheres, and act as orientation-preserving isometries on $\hyp$. Thus we can make the following definitions.
\begin{defn} \
\label{Def:SL2C_action_on_hyperboloid_model}
\begin{enumerate}
\item
$SL(2,\C)$ acts on $\hyp$ by restriction of its action on $\R^{1,3}$.
\item
$SL(2,\C)$ acts on $\partial \hyp$ by restriction of its action to $L^+$ and projectivisation to $\S^+ = \partial \hyp$.
\item
$SL(2,\C)$ acts on $\horos(\hyp)$ via its action on $\hyp$.
\end{enumerate}
\end{defn}

\begin{lem} \
\label{Lem:h_equivariance}
\begin{enumerate}
\item
The actions of $SL(2,\C)$ on $L^+$ and $\horos(\hyp)$ are equivariant with respect to $\h$. 
\item
The actions of $SL(2,\C)$ on $L^+$ and $\partial \hyp$ are equivariant with respect to $\h_\partial$. 
\end{enumerate}
That is, for $A \in SL(2,\C)$ and $p \in L^+$,
\[
\h(A\cdot p) = A\cdot (\h(p)) 
\quad \text{and} \quad
\h_\partial (A\cdot p) = A\cdot \h_\partial(p).
\]
\end{lem}

\begin{proof} 
The horosphere $\h(p)$ is cut out of $\hyp$ by the 3-plane $\langle x,p \rangle = 1$. Upon applying $A$, we see that $A\cdot \h(p)$ is cut out of $\hyp$ by the equation $\langle A^{-1}\cdot x, p \rangle = 1$, which is equivalent to $\langle x, A\cdot p \rangle = 1$, and this equation cuts out $\h(A\cdot p)$. Thus $A\cdot \h(p) = \h(A\cdot p)$ as desired for (i). Forgetting everything but points at infinity, we obtain (ii).
\end{proof}

We will need the following in the sequel. To those familiar with hyperbolic geometry it will be known or a simple exercise, but we can give an argument using spinors, which may be of interest.
\begin{lem}
The action of $SL(2,\C)$ on $\mathfrak{H}(\hyp)$ is transitive.
\end{lem}
In other words, if $\mathpzc{h}, \mathpzc{h}'$ are horospheres then there exists $A \in SL(2,\C)$ such that $A \cdot \mathpzc{h} = \mathpzc{h}'$. This $A$ is not unique.
\begin{proof}
As $\h$ is bijective (\refdef{h}) and $\g \circ \f\colon \C^2_\times \To L^+$ is surjective (\reflem{gof_properties}), there exist $\kappa, \kappa' \in \C^2_\times$ such that $\h \circ \g \circ f (\kappa) = \mathpzc{h}$ and $\h \circ \g \circ f (\kappa') = \mathpzc{h'}$. Now by \reflem{SL2C_on_C2_transitive} the action of $SL(2,\C)$ on $\C^2_\times$ is transitive, so there exists $A \in SL(2,\C)$ such that $A \cdot \kappa = \kappa'$. Then by equivariance of $\h$ (\reflem{h_equivariance}) and $\g \circ \f$ (\reflem{gof_properties}) we have
\[
A \cdot \mathpzc{h} 
= A \cdot \left( \h \circ \g \circ \f (\kappa) \right) 
= \h \circ \g \circ \f \left( A \cdot \kappa \right)
= \h \circ \g \circ \f (\kappa')
= \mathpzc{h'}
\]
as desired.
\end{proof}

\subsubsection{Distances between horospheres}
\label{Sec:distances_between_horospheres}

We now consider distances between horospheres and points in $\hyp^3$. Later, in \refsec{complex_lambda_lengths}, we will define \emph{complex} and \emph{directed} distances between horospheres with decorations, but for now we only need a simpler, undirected notion of distance. The arguments of this subsection are based on  \cite{Penner87}.

Let $\mathpzc{h}, \mathpzc{h}'$ be two horospheres, with centres $p \neq p'$ respectively. Let $\gamma$ be the geodesic with endpoints $p,p'$, and let $q = \gamma \cap \mathpzc{h}$ and $q' = \gamma \cap \mathpzc{h}'$. If $\mathpzc{h}$ and $\mathpzc{h}'$ are disjoint, then the shortest arc from $\mathpzc{h}$ to $\mathpzc{h'}$ is the segment $\gamma_{q,q'}$ of the geodesic $\gamma$ between $q$ and $q'$. When $\mathpzc{h}, \mathpzc{h'}$ overlap, one might think their distance should be zero, but instead we it turns out to be useful to use the same segment $\gamma_{q,q'}$, but count the distance negatively. When $\horo, \horo'$ have the same centre, there is no distinguished geodesic $\gamma$, we define a distance of $-\infty$ (see \refsec{complex_lambda_lengths} for justification). 
\begin{defn}
\label{Def:signed_undirected_distance}
The \emph{signed (undirected) distance} $\rho$ between $\mathpzc{h}$ and $\mathpzc{h'}$ is defined as follows.
\begin{enumerate}
\item 
If $p = p'$ then $\rho = - \infty$.
\item
If $p \neq p'$ and
\begin{enumerate}
\item
$\mathpzc{h}, \mathpzc{h}'$ are disjoint, then $\rho$ is the length of $\gamma_{q,q'}$;
\item 
$\mathpzc{h}, \mathpzc{h}'$ are tangent, then $\rho=0$; 
\item 
$\mathpzc{h}, \mathpzc{h}'$ overlap, then $\rho$ is the negative length of $\gamma_{q,q'}$.
\end{enumerate}
\end{enumerate}
\end{defn}

We can apply a similar idea for the distance between a horosphere $\horo$ and a point $q$. Let  $p$ be the centre of $\horo$, let $\gamma$ the geodesic with an endpoint at $p$ passing through $q$, and let $q' = \horo \cap \gamma$. let $\gamma_{q,q'}$ be the segment of $\gamma$ between $q$ and $q'$. This segment provides the shortest path between $\horo$ and $q$.
\begin{defn}
The \emph{signed distance} $\rho$ between $\horo$ and $q$ is defined as follow.
\begin{enumerate}
\item 
If $q$ lies outside the horoball bounded by $\horo$, then $\rho$ is the length of $\gamma_{q,q'}$.
\item
If $q$ lies on $\horo$, then $\rho = 0$.
\item
If $q$ lies inside the horoball bounded by $\horo$, then $\rho$ is the negative length of $\gamma_{q,q'}$.
\end{enumerate}
\end{defn}

\begin{lem}
\label{Lem:geodesic}
Let $q_0 = (1,0,0,0) \in \hyp$ and $p = (T,X,Y,Z) \in L^+$. Then the signed distance $\rho$ between $\h(p) \in\mathfrak{H}(\hyp)$ and $q_0$ is $\log T$.
\end{lem}
Here $q_0$ can be regarded as ``the centre of $\hyp$", the unique point with $X,Y,Z$-coordinates all zero. 

\begin{proof}
The strategy is as follows: consider the affine line in $\R^{1,3}$  from $p$ to $q_0$; calculate where this line intersects the cone on the horosphere $\h(p)$; this intersection point will be on the ray through the the point of $\h(p)$ closest to $q_0$; then we find the desired distance.

As the horosphere $\h(p)$ consists of the points $x \in \hyp$ (which satisfy $\langle x,x \rangle = 1$) with $\langle x,p \rangle = 1$, the \emph{cone} on $\h(p)$ consists of constant multiples $cx$ ($c \in \R$) of such points, which satisfy $\langle cx, p \rangle = c$ and $\langle cx,cx \rangle = c^2$, hence $\langle cx, p \rangle = \langle cx, cx \rangle^2$.

Recall that the centre of $\h(p)$ is the point of $\partial \hyp$ represented by $p$, i.e. the ray through $p$. Note $\langle p,p \rangle = 0$. For points $x$ on this ray we have $\langle x,x \rangle^2 = 0 = \langle x, p \rangle^2$.

From the previous two paragraphs, we observe that points $x$ in the cone on $\h(p)$ and on the ray through $p$ satisfy $\langle x, p \rangle^2 = \langle x,x \rangle$. Conversely, if a point $x$ satisfies $\langle x,p \rangle^2 = \langle x,x \rangle$ then we claim it is either on this cone or this ray. To see this, note the equation implies $\langle x,x \rangle \geq 0$. If $\langle x,x \rangle = 0$, we have $\langle x, p \rangle = 0$, so that $x$ lies on the ray through $p$;. If $\langle x,x \rangle > 0$ then there is a real multiple $x'$ of $x$ on $\hyp$, and then we have $\langle x', x' \rangle = 1$ and $\langle p, x' \rangle^2 = 1$. But as $p \in L^+$ and $x' \in \hyp$ we cannot have $\langle p, x' \rangle < 0$; thus $\langle p, x' \rangle = 1$, so $x' \in \h(p)$ and $x$ lies on the cone on $\h(p)$. 

Therefore, the equation
\begin{equation}
\label{Eqn:cone_on_horosphere}
\langle x,p \rangle^2 = \langle x,x \rangle
\end{equation}
characterises points in the cone on $\h(p)$ and the ray through $p$. We now parametrise the affine line from $p$ to $q_0$ by $x(s) = sp+(1-s)q_0$ and find where $x(s)$ satisfies \refeqn{cone_on_horosphere}.
We calculate
\begin{align*}
\langle x,p \rangle = \langle sp+(1-s)q_0 ,p \rangle
= s \langle p,p \rangle + (1-s) \langle q_0 , p \rangle
= (1-s)T,
\end{align*}
using $p= (T,X,Y,Z)$, $q_0 = (1,0,0,0)$, and since $p \in L^+$ so that $\langle p,p \rangle = 0$. Similarly, 
\begin{align*}
\langle x,x \rangle 
&= s^2 \langle p,p \rangle + 2s(1-s) \langle p, q_0 \rangle + (1-s)^2 \langle q_0, q_0 \rangle \\
&= 2s(1-s)T + (1-s)^2 = (1-s) \left( 2sT + 1-s \right).
\end{align*}
The equation $\langle x,p \rangle^2 = \langle x,x \rangle$ then yields
\[
(1-s)^2 T^2 = (1-s) \left( 2sT + 1-s \right) 
\]
The solution $s=1$ corresponds to $x=p$, the other solution is $s = \frac{T^2-1}{T^2+2T-1}$. For this $s$, $x(s)$ lies on the cone above $\h(p)$ at the point closest to $q_0$, and normalising its length gives the closest point in $\h(p)$ to $q_0$ as
\[
q' = \left( \frac{T^2 + 1}{2T^2}T, \frac{T^2-1}{2T^2} X, \frac{T^2-1}{2T^2} Y, \frac{T^2-1}{2T^2} Z \right),
\]
When $T>1$, the $X,Y,Z$ coordinates of $q'$ are positive multiples of $X,Y,Z$, so $q'$ lies on the geodesic from $q_0$ to the point at infinity represented by $p$, on the same side of $q_0$ as $p$. The horoball bounded by $\h(p)$ is thus disjoint from $q_0$, so  $\rho>0$. Conversely, when $T<1$, $\rho<0$.

The distance $d$ from $q'$ to $q_0$ can now be found from the formula $\cosh d =  \langle x,y \rangle$, where $d$ is the hyperbolic distance between points $x,y \in \hyp$. (Note $d = \pm \rho$.) Thus
\[
\cosh d = \langle q', q_0  \rangle
= \frac{T^2+1}{2T} = \frac{1}{2} \left( T + \frac{1}{T} \right).
\]
Since $\cosh d = \frac{1}{2} \left( e^d + e^{-d} \right)$, we have $e^d = T$ or $e^d = \frac{1}{T}$, i.e. $d = \pm \log T$. We just saw that when $T>1$, $\rho>0$ and when $T<1$, $\rho<0$. Thus $\rho = \log T$.
\end{proof}

\begin{prop}
\label{Prop:point_horosphere_distance_hyp}
Let $q \in \hyp$ and $p \in L^+$. Then the signed distance between $q$ and the horosphere $\h(p)$ is $\log \langle q,p \rangle$.
\end{prop}

\begin{proof}
We reduce to the previous lemma. Let $M \in SO(1,3)^+$ be an isometry which sends $q$ to $q_0$, and let $M(p) = (T,X,Y,Z) \in L^+$. By \reflem{geodesic}, the signed distance $\rho$ between $q_0$ and $\h(M(p))$ is given by $\rho = \log T = \log \langle q_0, (T,X,Y,Z) \rangle$. Now as $M$ is an isometry, we have $\langle q_0, (T,X,Y,Z) \rangle = \langle M(q), M(p) \rangle = \langle q,p \rangle$. Thus $\rho = \log \langle q,p \rangle$.
\end{proof}

\begin{lem}
\label{Lem:geodesic2}
Let $p_0 = (1,0,0,1)$ and  $p = (T,X,Y,Z)$ be points on $L^+$. Then the signed distance between the two horospheres $\h(p)$ and $\mathpzc{h}_0 = \h(p_0)$ is $\log \frac{T-Z}{2}$.
\end{lem}
Note that for any point $(T,X,Y,Z) \in L^+$, $T \geq Z$, with equality iff the point is a multiple of $p_0$. The case $T=Z$ arises when $p_0$ and $p$ lie on the same ray of $L^+$, and we regard $\log 0 $ as $-\infty$.

\begin{proof}
We follow a similar strategy to the previous lemma. The two horospheres have centres on $\partial \hyp$ given by rays through $p_0$ and $p$. We consider the affine line between $p$ and $p_0$, parametrised as $x(s) = sp+(1-s)p_0$, and find which points on this line lie on the cones of $\h(p)$ and $\mathpzc{h}_0$. The cone on $\h(p)$ is defined again by $\langle x,p \rangle^2 =  \langle x,x \rangle$, and the cone on $\mathpzc{h}_0$ is defined by $\langle x, p_0 \rangle^2 = \langle x,x \rangle$. We find that the closest points on $\h(p)$ and $\mathpzc{h}_0$ to each other are
\[
q = \left( \frac{T}{2} + \frac{1}{T-Z}, \frac{X}{2}, \frac{Y}{2}, \frac{Z}{2} + \frac{1}{T-Z} \right)
\quad \text{and} \quad
q_0 = \frac{1}{2(T-Z)} \left( 3T-Z, 2X, 2Y, T+Z \right).
\]
respectively.

Now $\mathpzc{h}_0$ is cut out of $\hyp$ by the equation $T-Z=1$, and $T-Z=0$ contains its centre $p_0$. So the horoball bounded by $\mathpzc{h}_0$ consists of points in $\hyp$ satisfying $T-Z<1$. Thus the two horoballs are disjoint iff $q$ lies outside the horoball of $\mathpzc{h}_0$, which occurs iff $q$ satisfies $T-Z>1$. This happens precisely when
\[
\left( \frac{T}{2} + \frac{1}{T-Z} \right) - \left( \frac{Z}{2} + \frac{1}{T-Z} \right) = \frac{T-Z}{2} > 1.
\]
Thus the horoballs are disjoint precisely when $T-Z>2$. We then find the distance $d$ between the closest points using $\cosh d =  \langle q, q_0 \rangle$, which reduces to
\[
\frac{1}{2} \left( e^d + e^{-d} \right) = \frac{1}{2} \left( \frac{T-Z}{2} + \frac{2}{T-Z} \right).
\]
Thus $e^d = \frac{T-Z}{2}$ or $\frac{2}{T-Z}$, i.e. $d = \pm \log \frac{T-Z}{2}$. As we have seen, when $T-Z>2$ the horoballs are disjoint, so that $d>0$. Hence $\rho = \log \frac{T-Z}{2}$ as desired.
\end{proof}

\begin{prop}[Cf. \cite{Penner87} lemma 2.1]
\label{Prop:horosphere_distance_hyp}
Let $p, p' \in L^+$. Then the signed distance $\rho$ between the horospheres $\h(p), \h(p')$ satisfies
\begin{equation}
\label{Eqn:horosphere_distance_from_Minkowski_inner_product}
\langle p, p' \rangle = 2 e^{\rho}.
\end{equation}
Further, suppose $\kappa, \kappa' \in \C^2_\times$ satisfy $\g \circ \f(\kappa) = p$ and $\g \circ \f(\kappa') = p'$. Then
\begin{equation}
\label{Eqn:horosphere_distance_from_spinor_inner_product}
\left| \{ \kappa, \kappa' \} \right|^2 = e^\rho
\end{equation}
\end{prop}
Equation \refeqn{horosphere_distance_from_spinor_inner_product} is equivalent to the modulus of the equation in \refthm{main_thm}. It is perhaps interesting that we can obtain this result without yet having considered spin at all. This proposition is closely related to \refprop{complex_Minkowski_inner_products}.

\begin{proof}
We begin with equation \refeqn{horosphere_distance_from_spinor_inner_product}, reducing it to the previous lemma. By \reflem{SL2C_on_C2_transitive}, there exists $A \in SL(2,\C)$ such that $A(\kappa) = (1,0)$. Let $A(\kappa') = \kappa''$. Then by \reflem{SL2C_by_symplectomorphisms},
\begin{equation}
\label{Eqn:reduction_to_10}
\{\kappa, \kappa'\} = \{A \kappa, A \kappa'\} = \{ (1,0), \kappa''\}.
\end{equation}
As $A$ acts by an isometry of hyperbolic space, the signed distance between the horospheres $A \cdot \h \circ \g \circ \f (\kappa)$ and $A \cdot \h \circ \g \circ \f (\kappa')$ is also $\rho$. By equivariance of $\f,\g,\h$ these horospheres can also be written as $\h \circ \g \circ \f (1,0)$ and $\h \circ \g \circ \f (\kappa'')$. Now  $\g \circ \f (1,0) = p_0 = (1,0,0,1)$. Let $\g \circ \f (\kappa'') = (T,X,Y,Z)$. By \reflem{geodesic2}, $\rho = \log \frac{T-Z}{2}$. Rearranging this and noting that $\langle p_0, (T,X,Y,Z) \rangle = T-Z$, we have
\[
e^\rho = \frac{1}{2} \left\langle p_0, (T,X,Y,Z) \right\rangle
= \frac{1}{2} \langle \g \circ \f (1,0), \g \circ \f (\kappa'') \rangle.
\]
Applying \refprop{complex_Minkowski_inner_products} we then obtain
\[
e^\rho = \left| \{ (1,0), \kappa'' \} \right|^2,
\]
which by \refeqn{reduction_to_10} is equal to $| \{ \kappa, \kappa' \} |^2$ as desired.

To obtain equation \refeqn{horosphere_distance_from_Minkowski_inner_product}, note that as $\g \circ \f$ is surjective, there exist $\kappa, \kappa'$ such that $\g \circ \f (\kappa) = p$ and $\g \circ \f (\kappa') = p'$. Then the first equation follows directly from the second, using \refprop{complex_Minkowski_inner_products}.
\end{proof}

\subsubsection{The map from flags to horospheres}
\label{Sec:flags_and_horospheres}

We consider how flags behave under $\h$ and how to obtain corresponding tangent data on a horosphere. So, let $(p,V, o)\in\mathcal{F_P^O}(\R^{1,3})$ and consider the effect of $\h$. The situation is schematically depicted in \reffig{flag_horosphere}.

First, consider the point $p$. Under  $\h$, $p$ corresponds to a horosphere $\h(p)\in\mathfrak{H}$. At a point $q$ of $\h(p)$, by \reflem{tangent_space_of_horosphere} we have $T_q \h(p) = p^\perp \cap q^\perp$

Second, consider the 2-plane $V$; recall $\R p \subset V \subset p^\perp$ (\reflem{light_cone_orthogonal_complement}). Consider how $V$ intersects the tangent space to $\h(p)$ at $q$.  We have 
\[
T_q \h(p) \cap V = ( q^\perp \cap p^\perp) \cap V = q^\perp \cap V,
\]
where the latter equality used $V \subset p^\perp$. Now as $\R p \subset V$, $V$ contains the the lightlike vector $p$, so by \reflem{lightlike_intersection} the latter intersection is transverse and the result is 1-dimensional.

Third, consider the  orientation $o$; recall $o$ is an orientation on the 1-dimensional space $V / \R p$. We will try to use $o$ to provide an orientation on the 1-dimensional space $T_q \h(p) \cap V$. We can regard $o$ as singling out as positive one the two sides of the origin in the line $V/\R p$ (the other side being negative). Then, any vector $w \in V$ which does not lie in $\R p$ obtains a sign, depending on the side of $\R p$ to which it lies; these two sides of $\R p$ project to the two sides of the origin in $V/\R p$.
\begin{lem}
If $p \in L^+$, $q \in \h(p)$ and $\R p \subset V \subset p^\perp$ (as above), then $T_q \h(p) \cap V \neq \R p$.
\end{lem}

\begin{proof}
As $T_q \h(p) \cap V \subset T_q \hyp$, it  is spacelike, so  cannot contain the lightlike vector $p$.
\end{proof}
Thus the 1-dimensional subspace $T_q \h(p) \cap V$ is a line in the 2-plane $V$ transverse to $\R p$. So $o$ singles out one side of the origin in this line; or equivalently, induces an orientation on this line.

To summarise: given a flag $(p,V,o)$, the point $p \in L^+$ singles out a horosphere $\h(p)$; at a point $q$ on this horosphere, $V$ singles out a distinguished 1-dimensional subspace $T_q \h(p) \cap V$ of the tangent space $T_q \h(p)$ to the horosphere; and $o$ induces an orientation on the 1-dimensional space $V \cap T_q \h(p)$.

Considering the above construction  over all $q \in h(p)$, the 1-dimensional spaces $T_q \h(p) \cap V$ form a \emph{tangent line field} on the horosphere $\h(p)$, and with the orientation from $o$ we in fact have an \emph{oriented tangent line field}  on the horosphere $\h(p)$, i.e. a smoothly varying choice of oriented 1-dimensional subspace of each tangent space $T_q \h(p)$. We denote this oriented tangent line field by $V \cap T\h(p)$, as it is given by intersections with the various fibres in the tangent bundle to $\h(p)$.

We can then make the following definitions.
\begin{defn}
\label{Def:overly_decorated_horosphere}
An \emph{overly decorated horosphere} is a pair $(\mathpzc{h},L^O)$ consisting of $\mathpzc{h}\in\horos(\hyp)$ together with an oriented tangent line field $L^O$ on $\mathpzc{h}$. The set of overly decorated horospheres is denoted $\mathfrak{H_D^O}(\hyp)$.
\end{defn}

\begin{defn}
\label{Def:H_PONF_to_decorated_horospheres}
The map $\H$ sends (pointed oriented null) flags in $\R^{1,3}$ to overly decorated horospheres 
\[
\H \colon \mathcal{F_P^O}(\R^{1,3}) \To \mathfrak{H_D^O}(\hyp), \quad
\H(p,V,o) = \left( \h(p), V \cap T \h(p) \right),
\]
where $V \cap T \h(p)$ is endowed with the orientation induced from $o$.
\end{defn}

We say the horospheres are ``overly" decorated, because it turns out that the oriented line fields $V \cap T\h(p)$ are of a very specific type: they are \emph{parallel}. A parallel oriented line field is determined by the single oriented line at one point; keeping track of an entire oriented line field is overkill.

\subsubsection{Illustrative examples from the spinor $(1,0)$}
\label{Sec:examples_from_10}

Let us return to the spinor $\kappa_0 = (1,0)$. In \refeg{flag_of_simple_spinors} we calculated that, in Minkowski space, the flag $\G \circ \F (\kappa_0)$ is based at $\g \circ \f (\kappa_0) = (1,0,0,1)$; let this point by $p_0$. We also calculated that the flag has 2-plane $V$ spanned by $p_0$ and the vector $(0,0,1,0)$ in the $Y$-direction, which we denote $\partial_Y$. This flag has $V/\R p_0$ is oriented in the direction of $\partial_Y$. In other words, the flag is $[[p_0, \partial_Y]]$

\begin{eg}[The horosphere of $(1,0)$ and oriented line field at a point]
\label{Eg:horosphere_of_10_at_point}
Let us now find the corresponding horosphere, which we denote $\horo_0$, i.e. $\horo_0 = \h(p_0) = \h \circ \g \circ \f (\kappa_0)$. It is cut out of $\hyp$ by the 3-plane $\Pi$ with equation $\langle x, p_0 \rangle = 1$, i.e. $T-Z=1$. Thus, $\mathpzc{h}_0$ is the paraboloid defined by equations $T^2-X^2-Y^2-Z^2=1$ and $T-Z=1$. By the comment after \refdef{h}, the centre of $\mathpzc{h}_0$ is the ray of $L^+$ through $p_0$.

A useful perspective on this horosphere $\mathpzc{h}_0$ may be obtained by noting that $\Pi$, with equation  $T-Z=1$, is foliated by lines in the direction $(1,0,0,1)$ (i.e. the direction of the position vector of $p_0$). Each such line contains exactly one point with $T=0$, i.e. in the $XYZ$ 3-plane. Since $T-Z=1$, when $T=0$ we have $Z=-1$. This $\Pi$ intersects the $XYZ$ 3-plane in the 2-plane consisting of points of the form $(0,X,Y,-1)$. Denote this 2-plane $\Pi_{XY}$. It is a Euclidean 2-plane.

Each of the lines parallel to $p_0$ foliating $\Pi$ intersects the horosphere $\mathpzc{h}_0$ exactly once. To see this, note that such a line has parametrisation $(0,X,Y,-1) + s(1,0,0,1) = (s,X,Y,s-1)$, and intersects $\horo_0$ when it intersects $\hyp$, i.e. when $s^2 - X^2 - Y^2 - (s-1)^2 = 1$. This equation is linear in the parameter $s$ and has a unique solution, giving the unique intersection point with $\mathpzc{h}_0$.

Thus the projection $\Pi \To \Pi_{XY}$, projecting along the lines in the direction of $p_0$, restricts to a bijection $\mathpzc{h}_0 \To \Pi_{XY}$. In fact, as $p_0$ is a lightlike direction and the tangent planes to $\Pi$ are precisely the orthogonal complement $p_0^\perp$, this bijection is an isometry. This shows the horosphere $\mathpzc{h}_0$ is isometric to a Euclidean 2-plane. It also shows that a point of $\mathpzc{h}_0$ is determined by its $X$ and $Y$ coordinates, and that all $(X,Y) \in \R^2$ arise as $X,Y$ coordinates of points on $\mathpzc{h}_0$. See \reffig{plane_Pi_projection}.

\begin{center}
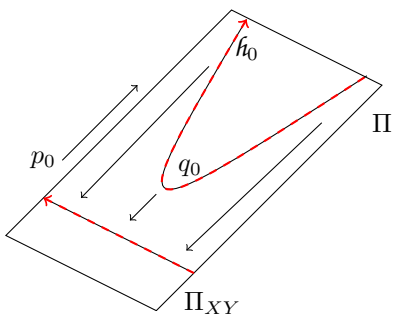

    \begin{tikzpicture}
        \draw(0,0)--(3,3)--(1,4)--(-2,1)--(0,0);
        \draw(0.5,0.5)--(-1.5,1.5);
        \draw (1.2,3.875) .. controls (-0.5,1) .. (2.8,3.125);
        \draw[red, dashed, thick, ->](0.5,0.5)--(-1.5,1.5);
        \draw[red, dashed, thick, <-](1.2,3.875) .. controls (-0.5,1) .. (2.8,3.125);
        \draw[->](0.7,3.25)--(-1,1.5);
        \draw[->](2.2,2.5)--(0.4,0.8);
        \draw[->](0,1.55)--(-0.35,1.2);
        \node at (0.75,0.1){$\Pi_{XY}$};
        \node at (3,2.5){$\Pi$};
        \node at (0.45,1.9){$q_0$};
        \node at (1.2,3.5){$\mathpzc{h}_0$};
        \node at (-1.5,2){$p_0$};    
        \draw[->](-1.25,2)--(-0.25,3);
    \end{tikzpicture}
    \captionof{figure}{Projection of the plane $\Pi$ to $\Pi_{XY}$ (schematically drawn a dimension down).}
    \label{Fig:plane_Pi_projection}
\end{center}

Let us examine the horosphere $\horo_0$ at a particular point. One can verify that $(1,0,0,0) \in \mathpzc{h}_0$; let this point be $q_0$. 

The tangent space of $\hyp$ at $q_0$ is  $q_0^\perp$ by \refeqn{hyperboloid_tangent_space}, which has equation $T=0$. So $T_{q_0} \hyp$ is the $XYZ$ 3-plane. 

The tangent space of $\mathpzc{h}_0$ at $q_0$ is $p_0^\perp \cap q_0^\perp$ by \reflem{tangent_space_of_horosphere}, thus is defined by equations $T-Z=0$ and $T=0$. So $T_{q_0} \mathpzc{h}_0$ is the $XY$ 2-plane.

The decoration, or oriented line, obtained on the horosphere in $\G \circ \F (\kappa_0)$, at $q_0$, by \refdef{H_PONF_to_decorated_horospheres} is given by $V \cap T_{q_0} \mathpzc{h}_0$. We have calculated that $V$ is spanned by $p_0$ and $\partial_Y$, while $T_{q_0} \mathpzc{h}_0$ is the $XY$-plane, so the intersection is the line in the $Y$ direction. Since the flag $V / \R p_0$ is oriented in the direction of $\partial_Y$, this line is oriented in the $\partial_Y$ direction.

Note that a quotient by $\R p_0$, when restricted to the 3-plane $\Pi$, is essentially the same as the projection along the lines in the $p_0$ direction discussed above. At each point of $\Pi$ (given by $T-Z=1$), the tangent space is given by $p_0^\perp = \{T-Z=0\}$, and $V$ is a 2-dimensional subspace of this tangent space. When we project $\Pi \To \Pi_{XY}$, the 2-plane $V$ of the flag projects to a 1-dimensional subspace of $\Pi_{XY}$, which we may regard as $V/\R p_0$. Since $V$ is spanned by $p_0$ and $\partial_Y$, the projection along $p_0$ is spanned by $\partial_Y$.
\end{eg}

\begin{eg}[Action of parabolic matrices on flag and horosphere of $(1,0)$]
\label{Eg:parabolic_action_on_h0}
Consider the following matrices in $SL(2,\C)$:
\begin{equation}
\label{Eqn:P}
P_\alpha = \begin{pmatrix} 1 & \alpha \\ 0 & 1 \end{pmatrix}
\text{  for $\alpha \in \C$}, \quad
P = \left\{ P_\alpha \; \mid \; \alpha \in \C \right\} .
\end{equation}
It is not difficult to see that $P$ is a subgroup $P$ of $SL(2,\C)$. Indeed, for $\alpha,\alpha' \in \C$ we have $P_\alpha P_{\alpha'} = P_{\alpha'} P_\alpha = P_{\alpha+\alpha'}$, and the correspondence $\alpha \mapsto P_\alpha$ gives an isomorphism from $\C$, as an additive group, to $P$. Thus $P \cong \C \cong \R^2$.

The matrices $P_\alpha$ are all \emph{parabolic} in the sense that  they have trace $2$. They are also \emph{parabolic} in the sense that, at least when $\alpha \neq 0$, as complex linear maps on $\C^2$, they have only one 2-dimensional eigenspace (i.e. their Jordan block decomposition consists of a single 2-dimensional block). The word parabolic can have other meanings too, which do not concern us here.

As a subgroup of $SL(2,\C)$, $P$ acts on all the spaces that $SL(2,\C)$ does. It will be useful to consider its action on various objects deriving from the spinor $\kappa_0 = (1,0)$ of the previous example.

Each $P_\alpha$ acts on $\C^2$ by complex linear maps preserving $\kappa_0$. In fact, for the action of $SL(2,\C)$ on $\C^2$ of \refdef{SL2C_action_on_C2}, $P$ is precisely the stabiliser of $\kappa_0$.

Under the map $\g \circ \f$ from $\C^2$ to $\R^{1,3}$, $\kappa_0$ maps to $p_0$. As $P$ preserves $\kappa_0$, by equivariance of $\g \circ \f$ (\reflem{gof_properties}), the action of $P$ on $\R^{1,3}$ preserves $p_0$. Precisely, for any $P_\alpha \in P$ we have
\begin{equation}
\label{Eqn:parabolics_fix_p0}
P_\alpha \cdot p_0 
= P_\alpha \cdot \left( (\g \circ \f) (\kappa_0) \right)
= (\g \circ \f ) \left( P_\alpha \cdot (\kappa_0) \right)
= (\g \circ \f) (\kappa_0)
= p_0
\end{equation}
Thus, each $P_\alpha$ acts on $\R^{1,3}$ by a real linear map in $SO(1,3)^+$ (\reflem{SL2C_action_on_light_cones} and subsequent comments) which preserves $p_0$, and hence also $p_0^\perp$. So, it can't be ``too bad"; we compute it explicitly.

On the Hermitian matrix $S$ corresponding to the point $2(T,X,Y,Z) \in \R^{1,3}$ (see \refdef{g_H_to_R31}), $P_\alpha$ acts by
\begin{align*}
P_\alpha \cdot S
&= P_\alpha S P_\alpha^*
= \begin{pmatrix} 1 & \alpha \\ 0 & 1 \end{pmatrix}
\begin{pmatrix} T+Z & X+iY \\ X-iY & T-Z \end{pmatrix}
\begin{pmatrix} 1 & 0 \\ \overline{\alpha} & 1 \end{pmatrix} \\
&=
\begin{pmatrix} 
T+Z + \alpha(X-iY) + \overline{\alpha}(X+iY) + |\alpha|^2 (T-Z)
& 
X+iY+\alpha(T-Z) \\
X-iY+\overline{\alpha}(T-Z) 
&
T-Z
\end{pmatrix}.
\end{align*}
This is equal to the Hermitian matrix corresponding to a point $2(T',X',Y',Z') \in \R^{1,3}$
\[
\begin{pmatrix} T'+Z' & X'+iY' \\ X'-iY' & T'-Z' \end{pmatrix}
\]
where, letting $\alpha = a+bi$ with $a,b \in \R$,
\begin{equation}
\begin{array}{cc}
\label{Eqn:transform_TXYZ_under_simple_parabolic_first}
T' 
= T + a X + b Y + \frac{|\alpha|^2}{2} (T-Z), &
X' = X + a  (T-Z), \\
Y' = Y + b  (T-Z), &
Z' 
= Z + a  X + b Y + \frac{|\alpha|^2}{2} (T-Z)
\end{array}
\end{equation}
Indeed, one can verify that $(T,X,Y,Z) = p_0$ implies $(T',X',Y',Z') = p_0$. This describes the action of $P$ on $\R^{1,3}$.

Now consider the action of $P$ on the flag $\G \circ \F(\kappa_0) = [[p_0, \partial_Y]] \in \mathcal{F_P^O}(\R^{1,3})$ from \refeg{flag_of_simple_spinors} and the previous \refeg{horosphere_of_10_at_point}. Using equivariance again (of $\G \circ \F$ this time, \refprop{SL2C_spinors_PNF_H_equivariant} and \refprop{FG_equivariant}), as $P$ stabilises $\kappa_0$, it also stabilises $[[p_0, \partial_Y]]$. Precisely,
for $P_\alpha \in P$ we have
\[
P_\alpha \cdot [[p_0, \partial_Y]]
= P_\alpha \cdot \left( \G \circ \F \right) (\kappa_0)
= \left( \G \circ \F \right) \left( P_\alpha \cdot (\kappa_0) \right)
= \left( \G \circ \F \right) (\kappa_0)
= [[p_0, \partial_Y]]
\]
Thus each $P_\alpha$ must fix the flag 2-plane $V$ spanned by $p_0$ and $\partial_Y$; we saw in \refeqn{parabolics_fix_p0} that $P_\alpha$ fixes $p_0$; we compute  $P_\alpha \cdot \partial_Y$  explicitly to see how $P$ acts on $V$. Using \refeqn{transform_TXYZ_under_simple_parabolic_first} gives
\[
P_\alpha \cdot \partial_Y = P_\alpha \cdot (0,0,1,0)
= (b, 0, 1, b)
= \partial_Y + b p_0.
\]
Thus indeed each $P_\alpha$ preserves the plane $V$ spanned by $p_0$ and $\partial_Y$. In fact, it acts as the identity on $V/\R p_0$, so definitely preserves the orientation in the flag.

Each $P_\alpha$ fixes $p_0^\perp$, the 3-dimensional orthogonal complement of $p_0$, which has a basis given by $p_0, \partial_Y$ and $\partial_X = (0,1,0,0)$. We have already computed $P_\alpha$ on the first two of these; the third is no more difficult, and we find that $P_\alpha$ acts on $p_0^\perp$ by
\begin{equation}
\label{Eqn:parabolic_on_p0_perp}
P_\alpha \cdot p_0 = p_0, \quad
P_\alpha \cdot \partial_X = \partial_X + a p_0, \quad
P_\alpha \cdot \partial_Y = \partial_Y + b p_0,
\end{equation}
adding multiples of $p_0$ to $\partial_X$ and $\partial_Y$ according to the real and imaginary parts of $\alpha$.

Having considered both $p_0$ and $p_0^\perp$, we observe that $\R p_0 \subset p_0^\perp$ and so we can consider their quotient $p_0^\perp / \R p_0$. This is a 2-dimensional vector space, and has a basis represented by $\partial_X$ and $\partial_Y$. From \refeqn{parabolic_on_p0_perp} we observe that each $P_\alpha$ acts on $p_0^\perp / \R p_0$ as the identity.

Next we turn to horospheres. \refeg{horosphere_of_10_at_point} above calculated $\h(p_0) = \h \circ \g \circ \f (\kappa_0)$ to be the horosphere $\mathpzc{h}_0$ cut out of $\hyp$ by the plane $\Pi$ with equation $T-Z=1$. We found that the point $q_0 = (1,0,0,0)$ was on this horosphere. At this point we have $T_{q_0} \hyp$ equal to the $XYZ$ 3-plane,  $T_{q_0} \h(p_0)$ equal to the the $XY$ 2-plane, and the oriented decoration $V \cap T_{q_0} \h(p_0)$ given by $\partial_Y$. Again by equivariance (\reflem{gof_properties}, \reflem{h_equivariance}), $P$ must fix $\mathpzc{h}_0$: for any $P_\alpha \in P$ we have
\[
P_\alpha \cdot \mathpzc{h}_0
= P_\alpha \cdot \left( \h \circ \g \circ \f \right) (\kappa_0)
= \left( \h \circ \g \circ \f \right) \left( P_\alpha \cdot (\kappa_0) \right)
= \h \circ \g \circ \f (\kappa_0)
= \mathpzc{h}_0.
\]

Let us see explicitly how $P_\alpha$ acts on the horosphere $\mathpzc{h}_0$, starting from the point $q_0$. Using \refeqn{transform_TXYZ_under_simple_parabolic_first}, and recalling that every point of $\mathpzc{h}_0$ satisfies $T-Z=1$, we obtain
\begin{equation}
\label{Eqn:general_point_on_h0}
P_\alpha \cdot q_0 = \left( 1 + \frac{|\alpha|^2}{2}, a, b, \frac{|\alpha|^2}{2} \right)
= \left( 1 + \frac{a^2 + b^2}{2}, a, b, \frac{a^2+b^2}{2} \right).
\end{equation}
The $X$ and $Y$ coordinates of $P_\alpha \cdot q_0$ are the real and imaginary parts of $\alpha$, and as mentioned in \refeg{horosphere_of_10_at_point}, $X$ and $Y$ coordinates determine points of $\horo_0$. Thus for any point $q \in \mathpzc{h}_0$ there is precisely one $\alpha \in \C$ such that $P_\alpha \cdot q_0 = q$, namely $\alpha=X+Yi$. In other words, the action of $P$ on $\mathpzc{h}_0$ is simply  transitive. The expression in \refeqn{general_point_on_h0} is a parametrisation of $\mathpzc{h}_0$ by $(a,b) \in \R^2$ or $\alpha\in \C$. If we project $\mathpzc{h}_0$ to $\Pi_{XY}$ as in \refeg{horosphere_of_10_at_point}, then $P_\alpha$ acts by addition by $(0,a,b,0)$.
\end{eg}

\begin{eg}[Oriented line field on the horosphere of $(1,0)$]
\label{Eg:horosphere_of_10_generally}
We again consider the horosphere $\mathpzc{h}_0 = \h(p_0) = \h \circ \g \circ \f (\kappa_0)$. In \refeg{horosphere_of_10_at_point} we found the tangent space to $\mathpzc{h}_0$ at a specific point $q_0$, and its intersection with the flag $\G \circ \F(\kappa_0)$. In \refeg{parabolic_action_on_h0} we found that the group $P$ acts simply transitively on $\mathpzc{h}_0$, so each point $q \in \mathpzc{h}_0$ can be written as $P_\alpha \cdot q_0$ for a unique $\alpha = a+bi$. We now find the tangent space to $\mathpzc{h}_0$ at $q$ explicitly, and its decoration, given by intersection with the flag $\G \circ \F (\kappa_0)$. 

Having calculated $q$ explicitly in \refeqn{general_point_on_h0}, using \refeqn{hyperboloid_tangent_space} we  have
\begin{equation}
\label{Eqn:tangent_space_general_point_on_h0}
T_q \hyp = q^\perp = \left\{ (T,X,Y,Z) \mid 
\left( 1 + \frac{|\alpha|^2}{2} \right) T - a X - b Y -  \frac{|\alpha|^2}{2}  Z = 0 \right\}
\end{equation}
The tangent space to the horosphere $\mathpzc{h}_0$ at $q$ is given by the intersection of $T_q \hyp$ with $p_0^\perp$ (\reflem{tangent_space_of_horosphere}). As in \refeg{horosphere_of_10_at_point}, the 3-plane $p_0^\perp$ has equation $T-Z=0$. 
Substituting  $T=Z$ into \refeqn{tangent_space_general_point_on_h0}  simplifies the equation to
\[
Z = a X  + b  Y
\]
and so we can obtain various descriptions of the tangent space to $\mathpzc{h}_0$ at $q$,
\begin{align*}
T_q \mathpzc{h}_0 &= q^\perp \cap p_0^\perp = \left\{ (T,X,Y,Z) \; \mid \; T=Z, \; Z = a  X  + b  Y \right\} \\
&= \left\{ \left( aX+bY, X, Y, aX+bY \right) \; \mid \; X,Y \in \R \right\} \\
&= \Span \left\{  (a,1,0,a), (b,0,1,b) \right\}
= \Span \left\{ \partial_X + a p_0, \partial_Y + b p_0 \right\}
\end{align*}
As in \refeg{flag_of_simple_spinors} and \refeg{horosphere_of_10_at_point}, the flag 2-plane $V$ of $\G \circ \F (\kappa_0)$ is spanned by $p_0$ and $\partial_Y$, with $V/\R p_0$ oriented by $\partial_Y$. One of the generators of $T_q \mathpzc{h}_0$ identified above already lies in this subspace, so  the line field on $\mathpzc{h}_0$ at $q$ is given by
\[
V \cap T_{q} \mathpzc{h}_0 = 
\Span \left\{  (b,0,1,b) \right\}
= \Span \left\{  \partial_Y + b p_0 \right\}
\]
The orientation on $V/\R p_0$ given by $\partial_Y + \R p_0$  induces the orientation on the 1-dimensional space $V \cap T_q \mathpzc{h}_0$ given by $\partial_Y + b p_0$. In other words, the oriented line field of $\H \circ \G \circ \F (\kappa_0)$ at $q = P_\alpha \cdot p_0$ is spanned and oriented by $\partial_Y + b p_0$. Denote this oriented line field by $L^O$, so that its value at $q$ is given by
\[
L^O_q = \Span \left\{ \partial_Y + b p_0 \right\}.
\]
In the parametrisation of \refeqn{general_point_on_h0} by $(a,b) \in \R^2$, $L_q^O$ points in the direction of constant $a$ and increasing $b$, i.e. the partial derivative with respect to $b$.

Since the action of $P$ on $\R^{1,3}$ is linear and preserves $\hyp$, $V$, and $\mathpzc{h}_0$, it also preserves tangent spaces of $\horo_0$: for any $\alpha \in \C$, we have $P_\alpha \cdot T_q \mathpzc{h}_0 = T_{P_\alpha \cdot q} \mathpzc{h}_0$. Hence the action of $P$ must preserve the intersections $V \cap T_q \mathpzc{h}_0$ which form the decoration on $\mathpzc{h}_0$:
\[
P_\alpha \cdot \left( V \cap T_q \mathpzc{h}_0 \right)
= V \cap T_{P_\alpha \cdot q} \mathpzc{h}_0
\]
Indeed, we can check this explicitly at any $q \in \mathpzc{h}_0$. Letting $q = P_\alpha \cdot q_0$, we just saw that the oriented line field at $q$ is  spanned and oriented by $\partial_Y + b p_0$. Applying $P_{\alpha'}$, where $\alpha' = a'+b' i$ with $a',b' \in \R$, from \refeqn{transform_TXYZ_under_simple_parabolic_first} we obtain
\[
P_{\alpha'} \cdot \left( \partial_Y + b p_0 \right)
= P_{\alpha'} \cdot (b,0,1,b)
= (b+b', 0, 1, b+b')
= \partial_Y + (b+b') p_0,
\]
the same vector spanning and orienting $L^O_{q'}$ where $q' = P_{\alpha'} \cdot q = P_{\alpha+\alpha'} q_0$.  So, for any $q \in \mathpzc{h}_0$ and any $A \in P$,
\[
A \cdot L^O_q = L^O_{A \cdot q}
\]
Thus, the oriented line field $L^O$ on $\mathpzc{h}_0$ given by $\H \circ \G \circ \F (\kappa_0)$ is a quite special type of oriented line field: it is parallel. Its value at any one point determines all the others, by applying the isometries given by $P$. The group $P$ of isometries of $\hyp$ is precisely the set of translations of $\mathpzc{h}_0$, which acts simply transitively on $\mathpzc{h}_0$ and carries with it the oriented line field $L^O$.

It is worth noting what happens if we project $\mathpzc{h}_0$ to the plane $\Pi_{XY}$ from \refeg{horosphere_of_10_at_point}. As discussed there, this projection is an isometry, and is effectively a quotient by $\R p_0$, expressing $\mathpzc{h}_0$ as a Euclidean 2-plane. Under this projection, $V$ becomes an oriented  line field in the direction $\partial_Y$. We saw in \refeg{parabolic_action_on_h0} that after applying this projection, $P_\alpha$ acts by translation by $(0,a,b,0)$. Thus in particular it preserves the oriented line field in the direction $\partial_Y$, which is the oriented line field of $\H \circ \G \circ \F(\kappa_0)$.
\end{eg}

\subsubsection{Parallel line fields}
\label{Sec:parallel_line_fields}

The type of oriented line field found as $\H \circ \G \circ \F(1,0)$ is known as \emph{parallel}, which we now define.

\begin{defn}
An element $A \in SL(2,\C)$, or the corresponding element $M \in SO(1,3)^+$, is called
\begin{enumerate}
\item
\emph{parabolic} if $\Trace A = \pm 2$;
\item 
\emph{elliptic} if $\Trace A \in (-2,2)$.
\item 
\emph{loxodromic} if $\Trace A \in \C \setminus [-2,2] = \pm 2$.
\end{enumerate}
\end{defn}
(There are other characterisations of these types of elements, but this is all we need.) It follows that the type of $A$ and any conjugate $MAM^{-1}$ are the same.

All the matrices $P_\alpha$ of the previous section are parabolic. (Their negatives $-P_\alpha$ are also parabolic, but a matrix $A \in SL(2,\C)$ and its negative $-A$ produce the same element of $SO(1,3)^+$, so these do not produce any new isometries of $\hyp$).

The oriented line field calculated on $\mathpzc{h}_0$ in the previous section thus satisfies the following definition.
\begin{defn}
Let $\mathpzc{h}\in\mathfrak{H}(\hyp)$. An oriented line field  on $\mathpzc{h}$ is \emph{parallel} if it is invariant under the  parabolic isometries of $\hyp$ fixing $\mathpzc{h}$.
\end{defn}
Thus, to describe a parallel oriented line field on a horosphere $\horo$, it suffices to describe it at one point: the oriented lines at other points can be found by applying parabolic isometries. Indeed, a horosphere is isometric to the Euclidean plane, and the parabolic isometries preserving $\mathpzc{h}$ act by Euclidean translations. A parallel oriented line field is therefore parallel in the sense of ``invariant under parallel translation". By the Gauss--Bonnet theorem no such line field exists on a surface of nonzero curvature.

As we now see, all oriented line fields produced by $\H$ (\refdef{H_PONF_to_decorated_horospheres}) are parallel.
\begin{lem}
\label{Lem:image_of_H_parallel}
Let  $(p,V,o) \in \mathcal{F_P^O}(\R^{1,3})$ be a flag, and let $\H(p,V,o) = (\h(p), L^O) \in \mathfrak{H_D^O}(\hyp)$ the corresponding overly decorated horosphere. Then the oriented line field $L^O$ on $\h(p)$ is parallel.
\end{lem}

\begin{proof}
The proof proceeds by reducing to the examples of the previous \refsec{examples_from_10}.

As $\G \circ \F$ is surjective (\refprop{F_G_surjective}), there exists $\kappa \in \C_\times^2$ such that $(p,V,o) = \G \circ \F(\kappa)$. As the action of $SL(2,\C)$ on $\C^2_\times$ is transitive (\reflem{SL2C_on_C2_transitive}), there exists  $A \in SL(2,\C)$ be a matrix such that $A \cdot \kappa = (1,0)$. Then by equivariance of $\f,\g,\h$ (\reflem{gof_properties}, \reflem{h_equivariance})  $A$ sends the given horosphere $\h(p)$ to $\horo_0 = \h(p_0) = \h \circ \g \circ \f (1,0)$ from \refsec{examples_from_10}:
\[
A \cdot \h(p) 
= A \cdot \left( \h \circ \g \circ \f  (\kappa) \right)
= \h \circ \g \circ \f \left( A \cdot \kappa \right)
= \h \circ \g \circ \f (1,0)
= \mathpzc{h}_0.
\]
Similarly, by equivariance of $\F$ and $\G$, $A$ sends the flag $(p,V,o)$ to the standard one $\G \circ \F(1,0)$ from \refsec{examples_from_10}, which we denote $(p_0, V_0, o_0)$:
\[
A (p,V,o) 
= A \cdot \left( \G \circ \F (\kappa) \right)
=  \G \circ \F \left(A \cdot \kappa \right)
= \G \circ \F (1,0)
= (p_0, V_0, o_0).
\]

Consider now the action of $A$ on oriented line fields. Recall that $SL(2,\C)$ acts on $\R^{1,3}$ via linear maps in $SO(1,3)^+$. If there is an oriented line field $L^O$ on $\h(p)$, then $A$ (via its derivative; but $A$ acts on $\R^{1,3}$ by a linear map) takes $L^O$ to an oriented line field on $\h(p_0)$, and $A^{-1}$ does the opposite. Thus $A$ and $A^{-1}$ provide a bijection
\begin{equation}
\label{Eqn:oriented_line_field_bijection}
\left\{ \text{Oriented line fields on $\h(p)$} \right\}
\cong
\left\{ \text{Oriented line fields on $\mathpzc{h}_0$} \right\}.
\end{equation}
Now, if $P$ is a parabolic isometry fixing $\h(p)$ then $A P A^{-1}$ is a parabolic isometry fixing $\mathpzc{h}_0 = A \cdot \h(p)$. This conjugation operation $P \mapsto A P A^{-1}$ has inverse $P \mapsto A^{-1} P A$, and provides a bijection between parabolic isometries fixing $\h(p)$ and parabolic isometries fixing $\mathpzc{h}_0 = A \cdot \h(p)$.

Thus, if we have a parallel oriented line field $L^O$ on $\h(p)$, then it is preserved under all parabolics $P$ fixing $\h(p)$, $P \cdot L^O = L^O$. Then the corresponding line field $A L^O$ on $\mathpzc{h}_0 = A \cdot \h(p)$ is preserved by all parabolics $A P A^{-1}$ fixing $\mathpzc{h}_0$, so $A \cdot L^O$ is parallel. In other words, the bijection \refeqn{oriented_line_field_bijection} above restricts to a bijection
\begin{equation}
\label{Eqn:parallel_oriented_line_field_bijection}
\left\{ \text{Parallel oriented line fields on $\h(p)$} \right\}
\cong
\left\{ \text{Parallel oriented line fields on $\mathpzc{h}_0$} \right\}.
\end{equation}

Now taking the given oriented line field $L^O$ from $\H(p,V,o)$ and applying $A$ gives an oriented lie field on $\mathpzc{h}_0$. We compute
\[
A L^O 
= A \left( V \cap T \h(p)) \right)
=  A \cdot V \cap T \left( A \cdot \h(p) \right)
=  V_0 \cap T \mathpzc{h}_0 
\]
which is precisely the oriented line field from $\H \circ \G \circ \F (1,0)$ in \refsec{examples_from_10}, which we calculated to be parallel. As $A$ sends $L^O$ to a parallel oriented line field, by \refeqn{parallel_oriented_line_field_bijection} $L^O$ is also parallel.
\end{proof}

The proof above essentially shows that any horosphere $\mathpzc{h}$, and the group of parabolics preserving it, behave like any other. The group of parabolics preserving a horosphere is isomorphic to the additive group $\C$ and acts by Euclidean translations on the horosphere. By a similar argument as above, one can show that if $A$ is parabolic and fixes $p \in L^+$, then $A$ fixes the horosphere $\h(p)$, the line $\R p$, the orthogonal complement $p^\perp$, and the quotient $p^\perp / \R p$, where it acts by translations.

\subsubsection{Decorated horospheres}
\label{Sec:decorated_horospheres}

Parallel oriented line fields are precisely the type of decoration we want on horospheres (at least, until we introduce spin in \refsec{spin}). As we see now, they make $\H$ into a bijection.
\begin{defn}
\label{Def:decorated_horosphere}
An \emph{decorated horosphere} is a pair $(\mathpzc{h}, L^O_P)$ consisting of $\mathpzc{h}\in\mathfrak{H}$ together with an oriented parallel line field $L^O_P$ on $\mathpzc{h}$. The set of all decorated horospheres is denoted $\mathfrak{H_D}$.
\end{defn}
We often refer to the oriented parallel line field on a horosphere as its \emph{decoration}. By definition, $\mathfrak{H_D} \subset \mathfrak{H_D^O}$.

Note that \refdef{decorated_horosphere} does not refer to any particular model of hyperbolic space. When we refer to decorated horospheres in a particular model we add it in brackets, e.g. $\mathfrak{H_D}(\hyp)$. 

Although $\H$ was originally defined (\refdef{H_PONF_to_decorated_horospheres}) as a map $\mathcal{F_P^O}(\R^{1,3}) \To \mathfrak{H_D^O}(\hyp)$, by  \reflem{image_of_H_parallel} $\H$ in fact has image $\mathfrak{H_D}(\hyp)$. Thus, we henceforth regard $\H$ as a map to the set of decorated horospheres, i.e.
\[
\H \colon \mathcal{F_P^O} (\R^{1,3}) \To \mathfrak{H_D}(\hyp).
\]
We will no longer need to refer to arbitrary line fields or overly decorated horospheres.

\begin{lem}
\label{Lem:H_bijection}
$\H \colon \mathcal{F_P^O}(\R^{1,3}) \To \mathfrak{H_D}(\hyp)$ is a bijection.
\end{lem}

\begin{proof}
From \refdef{h}, $\h \colon L^+ \To \mathfrak{H}(\hyp)$ is a bijection. Since the horosphere of $\H(p,V,o)$ is just $\h(p)$, every horosphere is obtained in the image of $\H$.

As explained in \refsec{rotating_flags}, there is an $S^1$ family of flags at any given basepoint $p \in L^+$. The 2-planes $V$ in this family all contain the line $\R p$, and rotate in the $3$-dimensional subspace $T_p L^+$ of $\R^{1,3}$. In defining the map $\H$, the horosphere $\h(p)$ is cut out of $\hyp$ by the 3-plane $\Pi$  with equation $\langle x, p \rangle = 1$. This 3-plane is parallel to the 3-plane $\langle x,p \rangle = 0$, which is $p^\perp = T_p L^+$. So in fact the tangent space to $\Pi$ at any point is just $T_p L^+$. We saw in \refsec{flags_and_horospheres} that $V$ always intersects the tangent space to $\h(p)$ in a 1-dimensional set, i.e. transversely in $\Pi$, and we saw in \reflem{image_of_H_parallel} that the resulting oriented line field is always parallel, hence determined by its value at one point. Moreover, the horosphere (being a spacelike surface) is transverse to the lightlike direction $\R p$. So as the flags based at $p$ rotate about $\R p$, they can also be considered to rotate in $T_p L^+ \cong T \Pi$, and transversely and bijectively cut out the $S^1$ family of oriented parallel directions on the 2-dimensional horosphere $\h(p)$ at each point.
\end{proof}

\subsubsection{$SL(2,\C)$ action on decorated horospheres}
\label{Sec:SL2c_on_decorated_horospheres}

\begin{defn} \
\label{Def:SL2C_action_UODHOR_hyp}
$SL(2,\C)$ acts on $\mathfrak{H_D}(\hyp)$ via its action on $\mathfrak{H}(\hyp)$ and its derivative.
\end{defn}
This action of $A \in SL(2,\C)$ derives from its action on $\R^{1,3}$ (\refdef{SL2C_on_R31}) via linear maps in $SO(1,3)^+$, the orientation-preserving isometries of $\hyp$. A horosphere $\mathpzc{h}$ is sent to $A \cdot \mathpzc{h}$ as in \refdef{SL2C_action_on_hyperboloid_model}. The derivative of this linear map (which is the same linear map, on the tangent space to the horosphere) applies to the decoration. Thus if $(\mathpzc{h}, L_P^O)$ is a decorated horosphere then $A \cdot (\mathpzc{h}, L_P^O) = (A \cdot \mathpzc{h}, A \cdot L_P^O)$ where both $A \cdot \mathpzc{h}$ and $A \cdot L_P^O$ mean to apply $A$ as a linear map in $SO(1,3)^+$.

\begin{lem}
\label{Lem:H_equivariant}
The actions of $SL(2,\C)$ on $\mathcal{F_P^O}(\R^{1,3})$ (\refdef{SL2C_on_PONF_R31}), and $\mathfrak{H_D}(\hyp)$ are equivariant with respect to $\H$.
\end{lem}

\begin{proof}
The equivariance basically follows from the fact that $A$ acts via a linear map in $SO(1,3)^+$ on both spaces. Explicitly, let $A \in SL(2,\C)$, and let $M \in SO(1,3)^+$ be the induced map on $\R^{1,3}$. For a flag $(p,V,o) \in \mathcal{F_P^O}(\R^{1,3})$, the action of $A$ on $p, V$ and $o$ is via the linear map $M$ on $\R^{1,3}$, and we have $A\cdot (p,V,o)=(Mp,MV,Mo)$ where $M$ acts linearly in the usual way.

Now $\H(p,V,o) = (\h(p), V \cap T\h(p))$ where the horosphere $\h(p)\in\mathfrak{H}(\hyp)$ is cut out of $\hyp$ by the plane with equation $\langle x,p \rangle = 1$, and $V \cap T \h(p)$ is a line which obtains an orientation from $o$.

Thus, $A\cdot \H(p,V,o) = (M\h(p), M(V \cap T\h(p)))$ is simply obtained by applying the linear map $M$ to the situation. 

On the other hand, $\H(Mp,MV,Mo)) = (\h(Mp), MV \cap M(T\h(p)))$. By equivariance of $\h$ (\reflem{h_equivariance}),  $\h(Mp)=M \h(p)$. And $M(V \cap T\h(p)) = MV \cap M(T\h(p)) = MV \cap TM\h(p)$: the image under $M$ of the intersection of 2-plane $V$ with the tangent space of $\h(p)$ is the intersection of $MV$ with the tangent space of $M\h(p) = \h(Mp)$.
\end{proof}

\subsection{From the hyperboloid model to the disc model}
\label{Sec:hyperboloid_to_disc}

The fourth step of our journey is from the hyperboloid model $\hyp$ to the disc model $\Disc$, via the maps $\i$ (and $\I$) from horospheres (with decorations) in $\hyp$ to horospheres (with decorations) in $\Disc$. The map from $\hyp$ to $\Disc$ is a standard isometry and we discuss it briefly. All constructions in $\hyp$ translate directly to $\Disc$, but we only consider the model briefly here. In \refsec{disc_model} we introduce the model and the maps $\i$ and $\I$; in \refsec{SL2C_disc_model} we discuss $SL(2,\C)$ actions and equivariance; in \refsec{examples_computations_disc_model} we discuss some examples and computations.

\subsubsection{The disc model}
\label{Sec:disc_model}

For a point $(X,Y,Z) \in \R^3$ let $r$ be its Euclidean length, i.e. $r > 0$ is such that $r^2 = X^2 + Y^2 + Z^2$.
\begin{defn}
The \emph{disc model} $\Disc$ of $\hyp^3$ is the set 
\[
\{(X,Y,Z) \in \R^3 \, \mid \, r < 1 \}
\quad \text{with Riemannian metric} \quad
ds^2 = \frac{4 \left( dX^2 + dY^2 + dZ^2 \right)}{\left( 1-r^2 \right)^2}.
\]
The boundary at infinity $\partial \Disc$ of $\Disc$ is $\{(X,Y,Z) \in \R^3 \, \mid r = 1 \}$.
\end{defn}

\begin{center}
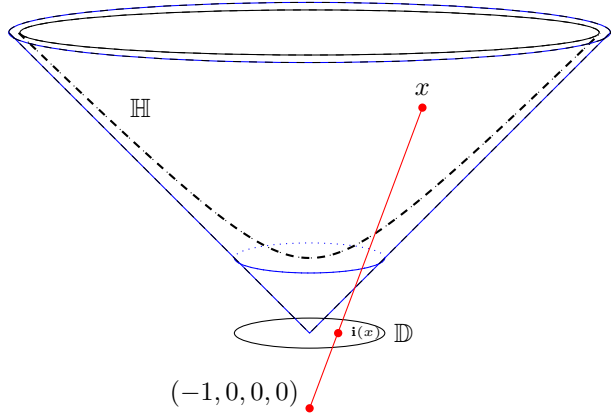

    \begin{tikzpicture}
  \draw[blue] (0,1) ellipse (1cm and 0.2cm);
  \fill[white] (-1,1)--(1,1)--(1,1.5)--(-1,1.5);
  \draw[blue,dotted] (0,1) ellipse (1cm and 0.2cm);
  \draw (0,0) ellipse (1cm and 0.2cm);
  \draw[blue] (-4,4)--(0,0)--(4,4);
  \draw[dashed, thick] plot[variable=\t,samples=1000,domain=-75.5:75.5] ({tan(\t)},{sec(\t)});
  \draw[blue] (0,4) ellipse (4cm and 0.4cm);
  \draw (0,4) ellipse (3.85cm and 0.3cm);
  \fill[red] (1.5,3) circle (0.055cm);
  \node at (1.5,3.25){$x$};
  \fill[red] (0.38,0) circle (0.055cm);
  \node at (0.75,0){\tiny$\i(x)$};
  \fill[red] (0,-1) circle (0.055cm);
  \node at (-1,-0.8){$(-1,0,0,0)$};
  \draw[dotted, thin] plot[variable=\t,samples=1000,domain=-75.5:75.5] ({tan(\t)},{sec(\t)});
  \draw[dashed] (0,4) ellipse (4cm and 0.4cm);
  \draw[dashed] (0,4) ellipse (3.85cm and 0.3cm);
  \draw[dashed] (-4,4)--(0,0)--(4,4);
  \node at (-2.25,3){$\hyp$};
  \draw[red] (1.5,3)--(0,-1);
  \node at (1.25,0){$\Disc$};
\end{tikzpicture} 
\label{Fig:hyperboloid_to_disc}
\captionof{figure}{From the hyperboloid $\hyp$ to the disc $\Disc$ (drawn a dimension down).}
\end{center}

The standard isometry from the hyperboloid model $\hyp$ to the disc model $\Disc$ regards  $\Disc$ as the unit 3-disc in the 3-plane $T=0$, i.e.
\[
\Disc = \{ (0,X,Y,Z) \mid X^2 + Y^2 + Z^2 < 1 \},
\]
and is given by straight-line projection from $(-1,0,0,0)$. See  \reffig{hyperboloid_to_disc}.
This gives the following map.
\begin{defn}
\label{Def:isometry_hyp_disc}
The isometry $\i$ from the hyperboloid model $\hyp$ to the disc model $\Disc$ is given by
\[
\i \colon \hyp \To \Disc, \quad
\i (T,X,Y,Z) = \frac{1}{1+T} (X,Y,Z).
\]
The map $\i$ extends to a map on spheres at infinity, which is essentially the identity on $\S^+$, but the domain can be taken to be $L^+$,
\[
\i \colon \partial \hyp = \S^+ \To \partial \Disc \text{ or } L^+ \To \partial \Disc, \quad
\i (T,X,Y,Z) = \left( \frac{X}{T}, \frac{Y}{T}, \frac{Z}{T} \right).
\]
The map $\i$ yields a map on horospheres, which we also denote $\i$,
\[
\i \colon \mathfrak{H}(\hyp) \To \mathfrak{H}(\Disc).
\]
\end{defn}

Horospheres in $\Disc$ appear as Euclidean spheres tangent to the boundary sphere $\partial \Disc$. The point of tangency with $\partial \Disc$ is the centre of the horosphere. The horoball bounded by the horosphere is the interior of the Euclidean sphere.

If a horosphere in $\hyp$ has an oriented tangent line field, we can transport it to $\Disc$ using the derivative of $\i$. One of these oriented tangent line fields is parallel if and only if the other is. So we obtain the following. 
\begin{defn}
\label{Def:I}
The map
\[
\I \colon \mathfrak{H_D}(\hyp) \To \mathfrak{H_D}(\Disc).
\]
is given by $\i$ and its derivative.
\end{defn}
It is clear that $\i$ and $\I$ are both bijections.

\subsubsection{$SL(2,\C)$ action on disc model}
\label{Sec:SL2C_disc_model}

The action of $SL(2,\C)$ extends to $\Disc$ and $\partial \Disc$, $\mathfrak{H}(\Disc)$, as follows:
\begin{defn}
The action of $A \in SL(2,\C)$ on 
\label{Def:SL2C_action_disc_model}
\label{Def:SL2C_action_UODHOR_Disc}
\begin{enumerate}
\item
$\Disc$ sends each $x \in \Disc$ to $A\cdot x =  \i \left( A\cdot  \left( \i^{-1} x \right) \right)$.
\item
$\partial \Disc$ sends each $x \in \partial \Disc$ to $
A\cdot x = \i \left( A\cdot  \left( \i^{-1} x \right) \right)$.
\item
$\mathfrak{H}(\Disc)$ is induced by the action on $\Disc$, which sends $\mathfrak{H}(\Disc)$ to $\mathfrak{H}(\Disc)$.
\item
$\mathfrak{H_D}(\Disc)$ is induced by its action on $\mathfrak{H}(\Disc)$ and its derivative.
\end{enumerate}
\end{defn}
Note that in (i), $\i^{-1} x \in \hyp$, so $A \cdot \i^{-1}(x)$ uses the action on $\hyp$, and in (ii), $\i^{-1} (x) \in \partial \hyp$, so $A \cdot \i^{-1}(x)$ uses the  action on $\partial \hyp$ (\refdef{SL2C_action_on_hyperboloid_model}).

The actions on $\Disc$ and $\partial \Disc$ are equivariant by definition: if we take a point $p \in \hyp$ or $\partial \hyp$, then $\i(p) \in \Disc$ or $\partial \Disc$, and by definition
\[
A \cdot \i (p) = \i \left( A \cdot p \right).
\]
The action on $\horos(\Disc)$ is induced by the pointwise action on $\Disc$, immediately giving the following.
\begin{lem} 
The actions of $SL(2,\C)$ on 
\label{Lem:SL2C_actions_on_Hyp_Disc_equivariant}
\[
\text{(i) } \hyp \text{ and } \Disc, \quad
\text{(ii) } \partial \hyp \text{ and } \partial \Disc, \quad
\text{(iii) } \mathfrak{H}(\hyp) \text{ and } \mathfrak{H}(\Disc)
\]
are equivariant with respect to $\i$.
\qed
\end{lem}

\begin{lem} 
\label{Lem:I_equivariant}
The actions of $SL(2,\C)$ on $\mathfrak{H_D}(\hyp)$ and $\mathfrak{H_D}(\Disc)$ are equivariant with respect to $\I$.
\end{lem}

\begin{proof}
We just saw the action of $A \in SL(2,\C)$ on $\mathfrak{H}(\hyp)$ and $\mathfrak{H}(\Disc)$ are equivariant with respect to $\i$. Both $A$ and $\I$ transport tangent line fields using the derivative, so they commute.
\end{proof}

\subsubsection{Examples and computations}
\label{Sec:examples_computations_disc_model}

We give some facts about the isometry $\i$.
\begin{lem}
\label{Lem:i_facts}
Under the map $\i \colon \hyp \To \Disc$,
\begin{enumerate}
\item
$q_0 = (1,0,0,0) \in \hyp$ maps to the origin $(0,0,0) \in \Disc$.
\item 
The point in $\partial \hyp$ represented by the ray in $L^+$ through $(1,X,Y,Z)$, maps to $(X,Y,Z) \in \partial \Disc$.
\item
In particular, the point of $\partial \hyp$ represented by the ray of $L^+$ through $p_0 = (1,0,0,1)$, maps to the north pole $(0,0,1) \in \partial \Disc$.
\end{enumerate}
\end{lem}

\begin{proof}
These are immediate from \refdef{isometry_hyp_disc}.
\end{proof}

\begin{eg}[Decorated horosphere in $\Disc$ of spinor $(1,0)$]
\label{Eg:decorated_horosphere_of_10_Disc}
Let $\kappa_0 = (1,0)$. The horosphere $\mathpzc{h}_0 =\h(p_0) = \h \circ \g \circ \f (\kappa_0)$ in $\hyp$, considered at length in the examples of \refsec{examples_from_10}, corresponds to a horosphere $\mathpzc{h}'_0 = \i(\mathpzc{h}_0)$ in $\Disc$. Since $\mathpzc{h}_0$ has centre the ray through $p_0 = (1,0,0,1)$ and passes through $q_0 = (1,0,0,0)$, using \reflem{i_facts}, $\mathpzc{h}'_0$ has centre $(0,0,1)$ and passes through the origin. Thus it is a Euclidean sphere of diameter $1$.

In \refeqn{general_point_on_h0} we found a parametrisation of $\mathpzc{h}_0$ by $\alpha = a+bi \in \C$ or $(a,b) \in \R^2$. Applying $\i$ yields a parametrisation of $\mathpzc{h}'_0$,
\begin{equation}
\label{Eqn:parametrisation_of_10_horosphere_in_disc}
\i \left( 1+ \frac{|\alpha|^2}{2},a, b, \frac{|\alpha|^2}{2} \right)
=
\frac{2}{4+a^2 + b^2} \left( a, b, \frac{a^2 + b^2}{2} \right).
\end{equation}
One can verify explicitly that this parametrises a Euclidean sphere in $\Disc$, tangent to $\partial \Disc$ at $(0,0,1)$ and passing through the origin (except for the point of tangency).

In \refeg{horosphere_of_10_generally} we found the oriented tangent line field $L^O$ on $\mathpzc{h}_0$ given by $\H \circ \G \circ \F(\kappa_0)$ explicitly: at the point $q$ parametrised by $(a,b)$, $L^O_q$ is spanned and oriented by $(b, 0, 1, b)$, which is the direction of constant $a$ and increasing $b$. Applying $\I$ we obtain a decoration on $\mathpzc{h}'_0$. This amounts to applying the derivative of $\i$ in the appropriate direction, which is just the partial derivative of $\i$ with respect to $b$. We find that the corresponding oriented line field on $\mathpzc{h}'_0$ is  spanned and oriented by 
\begin{equation}
\label{Eqn:decoration_on_10_horosphere_disc}
\frac{2}{(4+a^2+b^2)^2} \left( -2ab, 4+a^2-b^2,4b \right).
\end{equation}
This gives an explicit description of $\I \circ \H \circ \G \circ \F(\kappa_0)$. In particular, at the origin $(a,b)=(0,0)$, the decoration points in the direction $(0,1,0)$.
\end{eg}

For a general spin vector $\kappa$, we can  explicitly compute the centre of the corresponding horosphere in $\Disc$.
\begin{lem}
For $\kappa = (a+bi, c+di) \in \C^2_\times$ with $a,b,c,d \in \R$, we have
\[
\i \circ \h_\partial \circ \g \circ \f (\kappa) = 
\frac{1}{a^2+b^2+c^2+d^2} \left( 2(ac+bd), 2(bc-ad), a^2 + b^2 - c^2 - d^2 \right).
\]
\end{lem}

\begin{proof}
In \refsec{light_cone_to_horosphere} we observed that $\h_\partial$ is just the projectivisation map $L^+ \To \S^+$. So $\h_\partial \circ \g \circ \f (\kappa)$ is the point on $\partial \hyp$ given by the ray through $\g \circ \f (\kappa)$, calculated in \reflem{spin_vector_to_TXYZ}. Applying $\i$ to a point on that ray, such as the point calculated in \reflem{gof_celestial_sphere}, we obtain the result.
\end{proof}

A few further remarks:
\begin{itemize}
\item
In \refsec{calculating_flags_Minkowski} we considered $\g \circ D_\kappa \f (\ZZ(\kappa))$, which is involved in defining the flag $\G \circ \F (\kappa)$. Explicit calculation (\reflem{null_flag_tricky_vector}) showed $\g \circ D_\kappa \f (\ZZ(\kappa))$ has no $T$-component. It thus defines a tangent vector to the $S^2$ given by intersecting $L^+$ with any slice of constant positive $T$. The map from this $S^2$ to $\partial \Disc$ is just a dilation from the origin, and so we immediately obtain these flag directions on $\partial \Disc$. From \reflem{null_flag_tricky_vector} we find that when $\kappa = (a+bi, c+di)$ with $a,b,c,d \in \R$, the direction is
\begin{equation}
\label{Eqn:flag_direction_disc}
\left( 2(cd-ab), a^2-b^2+c^2-d^2,2(ad+bc) \right).
\end{equation}
\item
More generally, in \refsec{rotating_flags} we found an orthogonal basis $e_1 (\kappa), e_2(\kappa), e_3 (\kappa)$ for $\R^3$, obtained by projecting to the $XYZ$ 3-plane the point $p = \g \circ \f (\kappa)$, and derivatives of $\g \circ \f$ in the directions $\ZZ(\kappa)$ and $i \ZZ(\kappa)$. As discussed there, this basis yields an explicit picture of the flag of $\kappa$ in the 3-plane $T=r^2$, on which the light cone appears as a 2-sphere of radius $r^2$. Projection to the $XYZ$ 3-plane, and rescaling to the unit sphere, then gives a description of the flag on $\partial \Disc$. So \reffig{flag_intersect_T_r_squared} can be regarded also as a picture of a flag in $\Disc$.
\item
With this in mind, return to the decorated horosphere $\horo'_0$ of \refeg{decorated_horosphere_of_10_Disc}: described by $\kappa_0 = (1,0)$, it has centre $(0,0,1)$, Euclidean diameter 1, parametrisation \refeqn{parametrisation_of_10_horosphere_in_disc}, and decoration \refeqn{decoration_on_10_horosphere_disc}. From \refeqn{flag_direction_disc}, the flag direction at $(0,0,1)$ is (setting $\kappa = \kappa_0$) is $(0,1,0)$.

Now consider what happens as a point $q$ in the horosphere approaches $(0,0,1) \in \partial \Disc$ along the line field. This corresponds to holding $a$ constant and letting $b \rightarrow \pm \infty$. One can check that the oriented line field on $\mathpzc{h}'_0$ approaches $(0,-1,0)$. This is the negative of the flag direction at $(0,0,1)$ calculated above, and we appear to have a ``mismatch" of decorations at infinity. See \reffig{5}.

This is worth noting, to avoid future confusion, but not particularly surprising: in Minkowski space, the flag direction along $L^+$ and the oriented line field on a horosphere come from intersections with different, parallel 3-planes. Also note that, approaching the centre of the horosphere from other directions on the horosphere, the oriented line field can approach any arbitrary direction.
\end{itemize}

\begin{center}
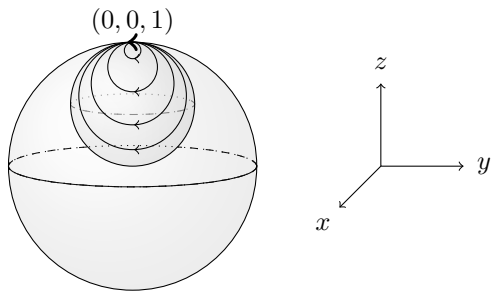

    \begin{tikzpicture}[scale=1.1]
    \draw (0,0) ellipse (1.5cm and 0.25cm);
    \fill[white] (-1.45,-0)--(1.45,-0)--(1.45,0.3)--(-1.45,0.3);
    \draw[dashed] (0,0) ellipse (1.5cm and 0.25cm);
    \fill[white] (0,0.75) circle (0.75cm);
    \draw[gray, dashed] (0,0.75) ellipse (0.75cm and 0.125cm);
    \fill[white] (-0.7,0.75)--(0.7,0.75)--(0.7,0.9)--(-0.7,0.9);
    \draw[gray, dotted] (0,0.75) ellipse (0.75cm and 0.125cm);
    \shade[ball color = gray!40, opacity = 0.1] (0,0) circle (1.5cm);
    \draw (0,0) circle (1.5cm);
    \shade[ball color = gray!40, opacity = 0.1] (0,0.75) circle (0.75cm);
    \draw (0,0.75) circle (0.75cm);
    \draw[dotted] (0,0) ellipse (1.5cm and 0.25cm);
    \draw[<->] (3,1)--(3,0)--(4,0);
    \draw[->] (3,0)--(2.5,-0.5);
    \node at (3,1.25){$z$};
    \node at (2.3,-0.7){$x$}; 
    \node at (4.25,0){$y$};
    \node at (0,1.75){$(0,0,1)$};
    \draw (0,0.85) circle (0.65cm);
    \draw (0,1) circle (0.5cm);
    \draw (0,1.2) circle (0.3cm);
    \draw (0,1.4) circle (0.1cm);
    \draw[<-] (0.02,1.3)--(0.04,1.3);
    \draw[<-] (0.02,0.9)--(0.04,0.9);
    \draw[<-] (0.02,0.5)--(0.04,0.5);
    \draw[<-] (0.02,0.2)--(0.04,0.2);
    \draw[line width=0.5mm, ->] (-0.04,1.5)--(-0.06,1.5);
    \end{tikzpicture}
    \captionof{figure}{Decoration ``mismatch" at $\infty$.}
\label{Fig:5}
\end{center}

\subsection{From the disc model to the upper half space model}
\label{Sec:Disc_to_U}

Finally, in our fifth step, we pass to the upper half space model $\U$, via the maps $\j$ (and $\J$) sending horospheres (with decorations) from $\Disc$ to $\U$. We have already discussed $\U$ to some extent in the introduction. The map $\Disc \To \U$ is another standard isometry and we discuss it briefly. We introduce $\U$, $\j$ and $\J$ in \refsec{U_horospheres_decorations} and prove their $SL(2,\C)$ equivariance in \refsec{SL2C_on_U}.

\subsubsection{The upper half space model, horospheres, and decorations}
\label{Sec:U_horospheres_decorations}

As discussed in introductory \refsec{intro_horospheres_decorations}, we may denote points in $\U$ by Cartesian coordinates $(x,y,z)$ with $z>0$, or combine $x$ and $y$ into a complex number $x+yi$, writing points of $\U$ as $(x+yi,h) \in \C \times \R^+$. Regarding $\C$ as $\C \times \{0\}$, the boundary at infinity is $\partial \U = \C \cup \{\infty\} = \CP^1$.

Stereographic projection $S^2 \To \CP^1$ (the inverse of the map in \refdef{stereographic_projection}) yields the map $\partial \Disc \To \partial \U$. 
\begin{defn}
\label{Def:isometry_D_U}
The isometry $\j$ from the disc model $\Disc$ to the upper half space model $\U$ is induced by its map on spheres at infinity,
\[
\j = \Stereo^{-1} \colon \partial \Disc = S^2 \To \partial \U = \C \cup \{\infty\}, \quad
\j(x,y,z) = \frac{x+iy}{1-z}.
\]
This map extends uniquely to an isometry $\j \colon \Disc \To \U$ and then restricts to a map on horospheres, which we also denote $\j$,
\[
\j \colon \mathfrak{H}(\Disc) \To \mathfrak{H}(\U).
\]
\end{defn}
As with $\i$ and $\I$, the derivative of the isometry $\j$ can be used to transport a decoration on a horosphere from $\Disc$ to $\U$. 
\begin{defn}
\label{Def:J}
The map
\[
\J \colon \mathfrak{H_D}(\Disc) \To \mathfrak{H_D}(\U)
\]
is given by $\j \colon \Disc \To \U$ and its derivative.
\end{defn}
Clearly $\j$ (in all its forms) and $\J$ are bijections.

We have discussed horospheres and decorations in $\U$ in introductory \refsec{intro_horospheres_decorations}; we now elaborate. A horosphere $\horo \in \horos(\U)$ centred at $\infty$ appears in $\U$ as a horizontal Euclidean plane. The group of parabolic isometries fixing $\mathpzc{h}$ appear in $\U$ as horizontal translations. An oriented tangent line field on $\horo$ is then parallel if and only if it appears \emph{constant}. So to describe a decoration on $\mathpzc{h}$, we only need to specify a direction at one point; the decoration points in the same direction at all other points. Since $\horo$ appears in $\U$ as a plane parallel to the complex plane, we can describe a decoration by a complex number. Since it is an oriented line field, that complex number is only well defined up to multiplication by positive reals. See \reffig{decorated_horospheres}(b).

On the other hand, if a horosphere $\mathpzc{h} \in \horos(\U)$ is not entered at $\infty$, then it appears in $\U$ as a Euclidean sphere tangent to $\C$. As discussed in \refsec{parallel_line_fields}, to specify a decoration, it suffices to specify an oriented tangent line at any point of $\horo$; the oriented line field then propagates over the rest of $\horo$ by parallel translation. The point at which it is most convenient to specify a decoration is at the point which appears highest in $\U$, which we call the \emph{north pole} of $\horo$. The tangent space to $\horo$ at its north pole is parallel to $\C$, and so a decoration there can be specified by a complex number (again, up to multiplication by positive reals). Precisely, at the north pole, a tangent vector $(a,b,0)$ in Cartesian coordinates corresponds to the complex number $a+bi$. See \reffig{upper_half_space_decorated_horosphere}.

\begin{defn}
\label{Def:decoration_specification}
Let $(\horo, L_P^O) \in \mathfrak{H_D}(\U)$, where $\horo$ is a horosphere and $L_P^O$ a parallel oriented line field.
\begin{enumerate}
\item
If the centre of $\horo$ is $\infty$, then a \emph{specification} of $L_P^O$ is a complex number directing $L_P^O$ at any point of $\horo$, identifying each tangent space of $\horo$ with $\C$.
\item
If the centre of $\horo$ is not $\infty$, then a \emph{north-pole specification}, or just \emph{specification}, of $L_P^O$ is a complex number directing $L_P^O$ at the north pole $n$ of $\horo$, identifying $T_n \horo$ with $\C$.
\end{enumerate}
\end{defn}
Thus any decorated horosphere in $\U$ has a specification, but it is not unique: if $\alpha \in \C$ is a specification for $\horo$, then so is $c \alpha$ for any $c > 0$.

\subsubsection{$SL(2,\C)$ action on the upper half space model}
\label{Sec:SL2C_on_U}

The $SL(2,\C)$ actions on various aspects of $\U$ are similar to previous models of $\hyp^3$, using actions defined previously.
\begin{defn} 
\label{Def:SL2C_action_upper_half_space_model}
\label{Def:SL2C_action_UODHOR_U}
The action of $A \in SL(2,\C)$ on
\begin{enumerate}
\item
$\U$ sends each $x \in \U$ to  $A\cdot x = \j \left( A\cdot  \left( \j^{-1} x \right) \right)$.
\item
$\partial \U$ sends each $x \in \partial \U$ to $A\cdot x = \j \left( A\cdot  \left( \j^{-1} x \right) \right)$.
\item
$\mathfrak{H}(\U)$ in induced by the action on $\U$, which sends $\horos(\U)$ to $\horos(\U)$.
\item 
$\mathfrak{H_D}(\U)$ is induced by its action on $\horos(\U)$ and its derivative.
\end{enumerate}
\end{defn}
As with the disc model, the actions on $\U$ and $\partial \U$ are  defined to be equivariant, and as the action on $\horos(\U)$ is induced pointwise by the action on $\U$, we immediately have the following.
\begin{lem}
\label{Lem:D_U_actions_equivariant}
The actions of $SL(2,\C)$ on 
\[
\text{(i) } \Disc \text{ and } \U, \quad
\text{(ii) } \partial \Disc \text{ and } \partial \U, \quad
\text{(iii) } \mathfrak{H}(\Disc) \text{ and } \mathfrak{H}(\U)
\]
are equivariant with respect to $\j$.
\qed
\end{lem}
Similarly, both $\J$ and $A \in SL(2,\C)$ transport line fields using the derivative, giving the following.
\begin{lem} \
\label{Lem:J_equivariant}
The actions of $SL(2,\C)$ on $\mathfrak{H_D}(\Disc)$ and $\mathfrak{H_D}(\U)$ are equivariant with respect to $\J$.
\qed
\end{lem}

\subsection{Putting the maps together}
\label{Sec:putting_maps_together}

We now have two sequences of maps, $\f,\g,\h,\i,\j$ and $\F,\G,\H,\I,\J$, as discussed in the introduction. We now consider their compositions.

In \refsec{boundary_points_isometries} we consider the effect of these maps on points at infinity, and show that the action of $SL(2,\C)$ on $\partial \U$ yields the standard description of isometries via M\"{o}bius transformation. In \refsec{fghij_2}, we calculate the compositions of $\f, \g, \h, \i, \j$ and $\F,\G,\H,\I,\J$. 

\subsubsection{Boundary points and isometries}
\label{Sec:boundary_points_isometries}

Before considering the composition of $\f,\g,\h,\i,\j$, we consider the composition
\[
\C_\times^2 \stackrel{\f}{\To} \HH_0^+
\stackrel{\g}{\To} L^+ 
\stackrel{\h_\partial}{\To} \partial \hyp
\stackrel{\i}{\To} \partial \Disc
\stackrel{\j}{\To} \partial \U.
\]
These map to the points of $\partial\hyp, \partial\Disc, \partial\U$ which are the centres of the horospheres produced by $\h, \i, \j$. For convenience, we abbreviate the composition to
\[
\k_\partial = \j \circ \i \circ \h_\partial \circ \g \circ \f
\]

There are $SL(2,\C)$ actions on all these spaces. A matrix $A \in SL(2,\C)$ acts on $\C_\times^2$ via matrix-vector multiplication (\refdef{SL2C_action_on_C2}); on $S \in \HH_0^+$, $A$ acts as $A\cdot S = ASA^*$  (\reflem{restricted_actions_on_H}); on $L^+ \subset \R^{1,3}$, $A$ essentially has the same action, which via $\g$ becomes a linear map in $SO(1,3)^+$ (\refdef{SL2C_on_R31}); for $x \in \partial \hyp$, $A \in SL(2,\C)$ acts similarly (\refdef{SL2C_action_on_hyperboloid_model}); the action is then transferred to the other models using the isometries $\i$ and $\j$ (\refdef{SL2C_action_disc_model}, \refdef{SL2C_action_upper_half_space_model}).

We have seen that these actions are all equivariant with respect to these maps: $\f$ \reflem{restricted_actions_on_H}, $\g$ (remark after \refdef{SL2C_on_R31}), $\h_\partial$ (\reflem{h_equivariance}), $\i$ (\reflem{SL2C_actions_on_Hyp_Disc_equivariant}), and $\j$ (\reflem{D_U_actions_equivariant}). Thus, $\k_\partial$ is also $SL(2,\C)$-equivariant.

Let us now compute the composition $\k_\partial$!
\begin{prop}
\label{Prop:explicit_fghij}
The composition $\k_\partial = \j \circ \i \circ \h_\partial \circ \g \circ \f \colon \C_\times^2 \To \partial \U = \C \cup \{\infty\}$ is given by
\[
\k_\partial  (\xi, \eta) = \frac{\xi}{\eta}.
\]
\end{prop}

We give two proofs of this result. This first is more conceptual, using our previous observations about the Hopf fibration and stereographic projection. The second is explicitly computational.
\begin{lem}
\label{Lem:Stereo_Hopf_p}
Let $\p \colon \C^2_\times \To S^3$ be the map that collapses each real ray from the origin to its intersection with the unit 3-sphere. Then
\[
\Stereo \circ \Hopf \circ \, \p = \i \circ \h_\partial \circ \g \circ \f
\]
In other words, the following diagram commutes. 

\begin{center}
    \begin{tikzpicture}
        \node (a) at (0,0){$\C^2_\times$};
        \node (b) at (2,1){$S^3$};
        \node (c) at (4,1){$\CP^1$};
        \node (d) at (6,0){$S^2=\partial\Disc$};
        \node (e) at (1,-1){$\HH_0^+$};
        \node (f) at (3,-1){$L^+$};
        \node (g) at (5,-1){$\partial\hyp$};
        \draw[->] (a) -- (b) node [pos=0.5,above] {$\p$};
        \draw[->] (b) -- (c) node [pos=0.5,above] {$\Hopf$};
        \draw[->] (c) -- (d);
        \node at (5.5,0.8) {$\Stereo$};
        \draw[->] (a) -- (e) node [pos=0.75,above] {$\f$};
        \draw[->] (e) -- (f) node [pos=0.5,above] {$\g$};
        \draw[->] (f) -- (g) node [pos=0.5,above] {$\h_\partial$};
        \draw[->] (g) -- (d) node [pos=0.25,above] {$\i$};
    \end{tikzpicture}
\end{center}
\end{lem}

\begin{proof}
We already saw in \reflem{gof_Hopf} that, for $\kappa = (\xi, \eta) \in S^3$, the $XYZ$ coordinates of $\g \circ \f (\kappa)$ are precisely $\Stereo \circ \Hopf (\kappa)$. In this case  (\reflem{spin_vector_to_TXYZ}), the $T$ coordinate of $\g \circ \f (\kappa)$ is $1$. Now the map $\h_\partial$ (\refdef{h_partial_light_cone_to_hyp}) projectivises the light cone, and then $\i$  (\refdef{isometry_D_U}) maps it to the unit Euclidean sphere in such a way that the ray through  $(1,X,Y,Z)$ maps to $(X,Y,Z)$.  Hence we have
\begin{equation}
\label{Eqn:hgf=stereohopf_in_S3}
\i \circ \h_\partial \circ \g \circ \f (\kappa) = \Stereo \circ \Hopf (\kappa) \quad \text{for $\kappa \in S^3$}
\end{equation}

Now for general $\kappa \in \C^2_\times$, let $\kappa = r\kappa'$ where $r>0$ and $\kappa'  \in S^3$. Then $\p(\kappa) = \kappa'$ and $\i \circ \h_\partial \circ \g \circ \f (\kappa') = \Stereo \circ \Hopf (\kappa')$. Applying $\f$ we have $\f(\kappa) = \f(r \kappa') = (r \kappa')(r \kappa')^* = r^2 \kappa' \kappa'^*= r^2 \f(\kappa')$. Applying the linear map $\g$ we then have $\g \circ \f (\kappa) = r^2 \g \circ \f (\kappa')$; then $\h_\partial$ then collapses rays to a point, so $\h_\partial \circ \g \circ \f (\kappa) = \h_\partial \circ \g \circ \f (\kappa')$. Putting this together we obtain the result:
\[
\i \circ \h_\partial \circ \g \circ \f (\kappa)
= \i \circ \h_\partial \circ \g \circ \f (\kappa')
= \Stereo \circ \Hopf (\kappa')
= \Stereo \circ \Hopf \circ \, \p (\kappa).
\]
\end{proof}

\begin{proof}[Proof 1 of \refprop{explicit_fghij}]
From the preceding lemma, we may replace $\i \circ \h_\partial \circ \g \circ \f$ with $\Stereo \circ \Hopf \circ \p$. The final map $\j$ (\refdef{isometry_D_U}) is the inverse of $\Stereo$ (\refdef{stereographic_projection}). Thus
\[
\k(\xi, \eta)
= \j \circ \i \circ \h_\partial \circ \g \circ \f (\xi,\eta)
= \Stereo^{-1} \circ \Stereo \circ \Hopf \circ \, \p (\xi, \eta)
= \Hopf \circ \, \p (\xi, \eta).
\]
Writing $(\xi, \eta) = r(\xi',\eta')$ where $r>0$ and $(\xi', \eta') \in S^3$, we have $\p (\xi, \eta) = (\xi', \eta')$ and 
\[
\Hopf \circ \, \p (\xi, \eta) 
= \Hopf (\xi', \eta')
= \frac{\xi'}{\eta'}
= \frac{\xi}{\eta}.
\]
\end{proof}

\begin{proof}[Proof 2 of \refprop{explicit_fghij}]
Let $\xi = a+bi$ and $\eta = c+di$ where $a,b,c,d \in \R$. In \reflem{spin_vector_to_TXYZ} we computed
\[
\g \circ \f (\xi, \eta) = \left( a^2+b^2+c^2+d^2, 2(ac+bd), 2(bc-ad), a^2+b^2-c^2-d^2 \right) \in L^+.
\]
The map $\h_\partial$ then projectivises, and $\i$ (\refdef{isometry_hyp_disc}) then maps  $(T,X,Y,Z) \mapsto (X/T,Y/T,Z/T)$, so we have
\[
\i \circ \h_\partial \circ \g \circ \f (\xi, \eta) = \left( \frac{2(ac+bd)}{a^2+b^2+c^2+d^2}, \frac{2(bc-ad)}{a^2+b^2+c^2+d^2}, \frac{a^2+b^2-c^2-d^2}{a^2+b^2+c^2+d^2} \right).
\]
(This may also be obtained from \reflem{gof_celestial_sphere}).
Finally, applying $\j$ (\refdef{isometry_D_U}) we have
\begin{align*}
\k_\partial (\xi, \eta) = \j \circ \i \circ \h_\partial \circ \g \circ \f (\xi, \eta) 
&=
\frac{ \frac{2(ac+bd)}{a^2+b^2+c^2+d^2} + i \frac{2(bc-ad)}{a^2+b^2+c^2+d^2} }{1 - \frac{a^2+b^2-c^2-d^2}{a^2+b^2+c^2+d^2} }
= \frac{ (ac+bd) + i(bc-ad) }{ c^2+d^2 } \\
&= \frac{(a+bi)(c-di)}{(c+di)(c-di)} = \frac{a+bi}{c+di} = \frac{\xi}{\eta}.
\end{align*}
\end{proof}

\begin{lem}
An $A \in SL(2,\C)$ acts on $\partial \U = \C \cup \{\infty\} = \CP^1$ by M\"{o}bius transformations:
\[
\text{if} \quad
A = \begin{pmatrix} \alpha & \beta \\ \gamma & \delta \end{pmatrix}
\quad \text{and} \quad
z \in \C \cup \{\infty\}
\quad \text{then} \quad
A\cdot z = \frac{\alpha z + \beta}{\gamma z + \delta}.
\]
\end{lem}
Note that when $A$ is the negative identity matrix, the corresponding M\"{o}bius transformation is just the identity. Thus the above action of $SL(2,\C)$ descends to an action of $PSL(2,\C)$. It is a standard fact that a M\"{o}bius transformation on $\partial \U$ extends to an orientation-preserving isometry of $\U$. In fact, the orientation preserving isometry group of $\U$ is $PSL(2,\C)$, acting in this way.

\begin{proof}
We use the equivariance of $\k_\partial \colon \C_\times^2 \To \partial \U = \C \cup \{\infty\}$. Starting from $\kappa = (\xi, \eta) \in \C_\times^2$ we have
\[
A\cdot\kappa = \begin{pmatrix} \alpha & \beta \\ \gamma & \delta \end{pmatrix} \begin{pmatrix} \xi \\ \eta \end{pmatrix} = \begin{pmatrix} \alpha \xi + \beta \eta \\ \gamma \xi + \delta \eta \end{pmatrix}.
\]
On the other hand we just computed $\k_\partial (\kappa) = \xi/\eta$. Thus the action of $A$ on this point of $\C \cup \{\infty\}$ is given by
\[
A\cdot \k_\partial (\kappa) = \k_\partial (A\cdot\kappa) = \k_\partial  \begin{pmatrix} \alpha \xi + \beta \eta \\ \gamma \xi + \delta \eta \end{pmatrix} = \frac{\alpha \xi + \beta \eta}{\gamma \xi + \delta \eta}
\]
which is precisely the action of the claimed M\"{o}bius transformation on $\xi/\eta$. Every point of $\C \cup \{\infty\}$ can be written as $\xi/\eta$ for some such $(\xi, \eta)$, and hence the action on $\C \cup \{\infty\}$ is as claimed. Even better, we can regard $\CP^1$ and its points as $[\xi:\eta]$, and then $A$ simply acts linearly.
\end{proof}

\subsubsection{Maps to horospheres and decorations}
\label{Sec:fghij_2}
\label{Sec:FGHIJ}

Consider now the following compositions, which map to horospheres and decorated horospheres. 
\begin{gather*}
\C_\times^2 \stackrel{\f}{\To} \HH_0^+
\stackrel{\g}{\To} L^+ 
\stackrel{\h}{\To} \mathfrak{H}(\hyp)
\stackrel{\i}{\To} \mathfrak{H}(\Disc)
\stackrel{\j}{\To} \mathfrak{H}(\U), \\
\C_\times^2 \stackrel{\F}{\To} \mathcal{F_P^O}(\HH) \stackrel{\G}{\To} \mathcal{F_P^O} (\R^{1,3}) \stackrel{\H}{\To} \mathfrak{H_D}(\hyp) \stackrel{\I}{\To} \mathfrak{H_D}(\Disc) \stackrel{\J}{\To} \mathfrak{H_D}(\U).
\end{gather*}
We abbreviate the compositions to
\[
\k = \j \circ \i \circ \h \circ \g \circ \f.
\quad \text{and} \quad
\K = \J \circ \I \circ \H \circ \G \circ \F.
\]
Again, $SL(2,\C)$ acts on all these spaces; additionally to those seen in \refsec{boundary_points_isometries}, $A \in SL(2,\C)$ acts on horospheres $\horos(\hyp)$ via  its action on $\R^{1,3}$ (\refdef{SL2C_action_on_hyperboloid_model}), and on horospheres in other models by using the isometries between the models (\refdef{SL2C_action_disc_model}, \refdef{SL2C_action_upper_half_space_model}). We have seen these actions are all equivariant with respect to $\h$ (\reflem{h_equivariance}), $\i$ (\reflem{SL2C_actions_on_Hyp_Disc_equivariant}), and $\j$ (\reflem{D_U_actions_equivariant}). 
Further, $A \in SL(2,\C)$ acts on a flag $(p,V,o) \in \mathcal{F_P^O}(\HH)$  via its action on $\HH$ (\refdef{matrix_on_PONF}); on a flag in $\R^{1,3}$ via the isomorphism $\g$ (\refdef{SL2C_on_PONF_R31}); on a decorated horosphere in $\hyp$ via its action on $\hyp$ (and its derivative)  (\refdef{SL2C_action_UODHOR_hyp}); and on decorated horospheres in other models by the using isometries between the models (\refdef{SL2C_action_UODHOR_Disc}, \refdef{SL2C_action_UODHOR_U}).
Moreover, all the maps are equivariant: 
$\F$ (\refprop{SL2C_spinors_PNF_H_equivariant}),
$\G$ (\refprop{FG_equivariant}),
$\H$ (\reflem{H_equivariant}),
$\I$ (\reflem{I_equivariant}), and 
$\J$ (\reflem{J_equivariant}). 

Thus, the compositions $\k$ and $\K$ are $SL(2,\C)$-equivariant.

It is worth pointing out that this composition $\K$ is \emph{almost} a bijection. Only $\F$ is not a bijection, but we have seen that it is surjective and 2--1, with $\F(\kappa) =\F(\kappa')$ iff $\kappa = \pm \kappa'$ (\reflem{F_G_2-1}). We have seen that  $\G,\H,\I,\J$ are bijections (\reflem{G_bijection}, \reflem{H_bijection}, remark after \refdef{I}, remark after \refdef{J}). Indeed, it is not hard to see that $\G,\H,\I,\J$ are all smooth and have smooth inverses, so we in fact have diffeomorphisms between these spaces. We will see how to produce a complete bijection in \refsec{lifts_of_maps_spaces}.

We now compute the compositions. The following proposition includes a precise statement of \refthm{explicit_spinor_horosphere_decoration}, for (non-spin-)decorated horospheres.
\begin{prop}
\label{Prop:JIHGF_general_spin_vector}
\label{Prop:U_horosphere_general}
For $(\xi, \eta) \in \C_\times^2$ the decorated horosphere $\K(\xi, \eta) \in \mathfrak{H_D}(\U)$ is centred at $\xi/\eta$ and
\begin{enumerate}
\item
is a sphere with Euclidean diameter $|\eta|^{-2}$ and decoration north-pole specified by $i \eta^{-2}$, if $\eta \neq 0$;
\item
is a horizontal plane at Euclidean height $|\xi|^2$ and decoration specified by $i \xi^2$, if $\eta = 0$.
\end{enumerate}
The horosphere $\k(\xi, \eta) \in \horos(\U)$ is the horosphere of $\K(\xi, \eta)$, without the decoration.
\end{prop}

Specifications here are in the sense of \refdef{decoration_specification}. As in \refsec{fghij_2}, the strategy is to prove the proposition for $(1,0)$ and build to the general case by equivariance.

The strategy is to first prove the proposition for $\kappa = (1,0)$, then use equivariance to prove it for $(0,1)$, then general $\kappa$. We have studied the horosphere of $(1,0)$ extensively; we now just need to map it to $\U$ via $\j$.

\begin{lem}
\label{Lem:j_facts}
The map $\j$ has the following properties, illustrated in \reffig{D_to_U}.
\begin{enumerate}
\item
It maps the following points $\partial \Disc \To \partial \U \cong \C \cup \{\infty\}$:
\[
\begin{array}{ccc}
\j(-1,0,0) = -1, & 
\j(0,-1,0) = -i, & 
\j(0,0,-1) = 0, \\
\j(1,0,0) = 1, & 
\j(0,1,0) = i, &
\j(0,0,1)= \infty. 
\end{array}
\]
\item
Denoting by $[p \rightarrow q]$ the oriented geodesic from a  point at infinity $p \in \partial \Disc$ or $\partial \U$ to $q$, we have
\[
\j\left[ (-1,0,0) \rightarrow (1,0,0) \right] = \left[ -1 \rightarrow 1 \right] 
\quad \text{and} \quad
\j\left[ (0,-1,0) \rightarrow (0,1,0) \right] = \left[ -i \rightarrow i \right].
\]
\item
$\j$ maps $(0,0,0) \in \Disc$ to $(0,0,1) \in \U$, and at this point the derivative maps $(0,1,0)$ to $(0,1,0)$.
\end{enumerate}
\end{lem}

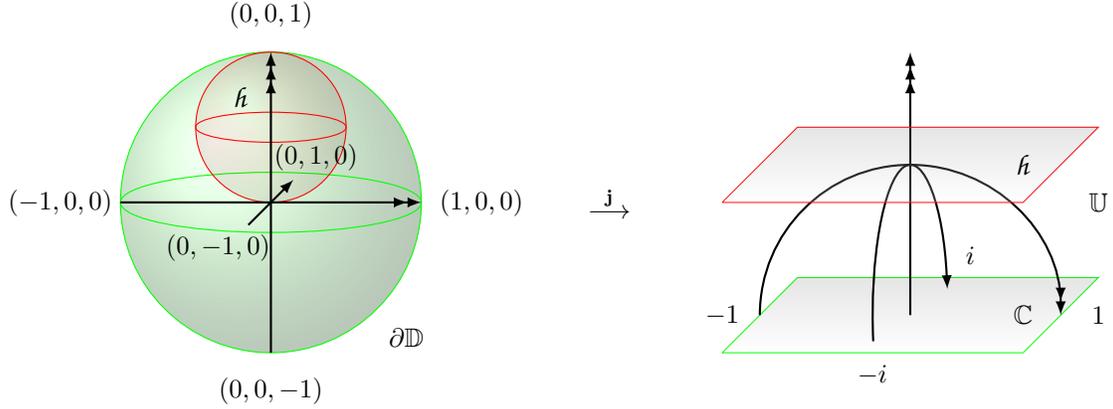
\begin{figure}
\begin{center}
\begin{tikzpicture}
\tikzset{
    partial ellipse/.style args={#1:#2:#3}{
        insert path={+ (#1:#3) arc (#1:#2:#3)}
    }
}
    \shade[ball color = green!40, opacity = 0.2] (0,0) circle (2cm);
    \shade[ball color = green!40, opacity = 0.2] (0,0) circle (2cm);
    \draw[green] (0,0) circle (2cm);
    \draw[green] (0,0) ellipse (2cm and 0.4cm);
    \draw[red] (0,1) circle (1cm);
    \shade[ball color = red!80, opacity = 0.1] (0,1) circle (1cm);    
    \draw[red] (0,1) ellipse (1cm and 0.2cm);
    \draw[>=latex, thick, ->>>] (0,-2) -- (0,2);
    \draw[>=latex, thick, ->>] (-2,0) -- (2,0);
    \draw[>=latex, thick, ->] (-0.3,-0.3)--(0.3,0.3);
    \node[black] at (-2.8,0) {$(-1,0,0)$};
    \node[black] at (2.8,0) {$(1,0,0)$};
    \node[black] at (0,-2.5) {$(0,0,-1)$};
    \node[black] at (0,2.5) {$(0,0,1)$};
    \node[black] at (-0.7,-0.6) {$(0,-1,0)$};
    \node[black] at (0.6,0.6) {$(0,1,0)$};
	\node[black] at (1.8,-1.8) {$\partial \Disc$};
    \node[black] at (-0.4,1.4) {$\horo$};
    \node at (4.5,0){$\stackrel{\j}{\To}$};
\begin{scope}[xshift = 1cm]
    \draw[green] (5,-2)--(9,-2)--(10,-1)--(6,-1)--(5,-2);
    \shade[color = green, opacity=0.2] (5,-2)--(9,-2)--(10,-1)--(6,-1)--(5,-2);
    \draw[>=latex, thick, ->>>] (7.5,-1.5) -- (7.5,2);
    \draw[>=latex, thick, ->>] (5.5,-1.5) arc[start angle=180, end angle=0,radius=2cm];
    \draw[>=latex, thick, ->] (7.5,-1.5) [partial ellipse=190:10:0.5cm and 2cm];
    \draw[red] (5,0)--(9,0)--(10,1)--(6,1)--(5,0);
    \shade[color = red, opacity=0.2] (5,0)--(9,0)--(10,1)--(6,1)--(5,0);
    \node[black] at (5,-1.5) {$-1$};
    \node[black] at (10,-1.5) {$1$};
    \node[black] at (7,-2.3) {$-i$};
    \node[black] at (8.3,-0.7) {$i$};
    \node[black] at (9,0.5) {$\horo$};
    \node[black] at (9,-1.5) {$\C$};
    \node[black] at (10,0) {$\U$};
\end{scope}
\end{tikzpicture}
\caption{The map $\j$, showing various boundary points, geodesics, and horospheres.}
\label{Fig:D_to_U}
\end{center}
\end{figure}

\begin{proof}
Applying \refdef{isometry_D_U} immediately gives (i). Since $\j$ is an isometry $\Disc \To \U$, it must preserve geodesics and their endpoints at infinity, so (ii) follows. Finally, the origin in $\Disc$ is the intersection point of the two geodesics in $\Disc$ specified in (ii), so maps to the intersection of the two corresponding geodesics in $\U$. The intersection point in $\U$ of the geodesics $\left[ -1 \rightarrow 1 \right]$ and $\left[ -i \rightarrow i \right]$ is $(0,0,1)$. The specified tangent direction at the origin in $\Disc$ is the direction of the latter geodesic, thus it maps to the claimed tangent direction at $(0,0,1) \in \U$.
\end{proof}

\begin{lem}
\label{Lem:U_horosphere_10}
\label{Lem:JIHGF10}
$\k (1,0)\in\mathfrak{H}(\U)$ is centred at $\infty$ at (Euclidean) height $1$. $\K (1,0) \in \mathfrak{H_D}(\U)$ is the same horosphere, with decoration specified by $i$.
\end{lem}

\begin{proof}
In \refeg{decorated_horosphere_of_10_Disc} we described explicitly the decorated horosphere in $\Disc$ given by $(1,0)$, i.e. $\I\circ \H \circ \G \circ \F (1,0)$. It is the horosphere in $\Disc$ centred at $(0,0,1)$, passing through the origin $(0,0,0)$. At the origin, the decoration points in the direction of $(0,1,0)$. Forgetting the decoration yields $\i \circ \h \circ \g \circ \f (1,0)$.

Applying $\j$, \reflem{j_facts} shows that the horosphere centre $(0,0,1)$ maps to $\infty$, the origin of $\Disc$ maps to  $(0,0,1) \in \U$, and the direction $(0,1,0)$ at the origin maps to to the direction $(0,1,0)$ at $(0,0,1) \in \U$. 

Thus $\k(1,0)$ is centred at $\infty$ and passes through $(0,0,1)$, hence lies at Euclidean height 1. The decoration $(0,1,0)$ there is the $i$ direction, so the decoration on $\K(1,0)$ is specified by $i$. See \reffig{D_to_U}
\end{proof}

\begin{lem}
\label{Lem:U_horosphere_01}
\label{Lem:JIHG010}
$\k(0,1)\in\mathfrak{H}(\U)$ is centred at $0$ and has Euclidean diameter $1$. $\K (0,1)\in\mathfrak{H_D}(\U)$ is the same horosphere, with decoration north-pole specified by $i$.
\end{lem}

\begin{proof}
We use the previous lemma and equivariance. Note
\[
\begin{pmatrix} 0 \\ 1 \end{pmatrix} = A \begin{pmatrix} 1 \\ 0 \end{pmatrix}
\quad \text{where} \quad
A = \begin{pmatrix} 0 & -1 \\ 1 & 0 \end{pmatrix} \in SL(2,\C),
\]
so
\[
\K \begin{pmatrix} 0 \\ 1 \end{pmatrix}
= \K  \left( A \begin{pmatrix} 1 \\ 0 \end{pmatrix} \right)
= A \cdot \left( \K  \begin{pmatrix} 1 \\ 0 \end{pmatrix} \right),
\]
and similarly for $\k$. Thus $\K  (0,1)$ is obtained from $\K(1,0)$ of \reflem{U_horosphere_10} by applying $A$, and similarly for $\k$.

On $\U$, $A$ acts by the M\"{o}bius transformation $z \mapsto -1/z$, which is an involution sending $\infty \leftrightarrow 0$. It yields an isometry of $\U$ which is a half turn about the geodesic between $-i$ and $i$. As the point $(0,0,1)$ lies on this geodesic, it is fixed by the action of $A$. The vector $(0,1,0)$ at $(0,0,1)$ is tangent to the geodesic, so is also preserved by the half turn.

Since $\k(1,0)$ has centre $\infty$ and passes through $(0,0,1)$, then $A \cdot \k(1,0)$ has centre $0$ and also passes through $(0,0,1)$.  Hence $\k(0,1)$  has centre $0$ and Euclidean diameter $1$. The decoration of $\K(1,0)$ is directed by $(0,1,0)$ at $(0,0,1)$, and this vector is preserved by $A$. Hence this vector also directs the oriented parallel line field of $\K (0,1)$, which is thus north pole specified by $(0,1,0)$, corresponding to the complex number $i$. See \reffig{K10_to_K01}.
\end{proof}

\begin{figure}
\begin{center}
\begin{tikzpicture}[scale=1.2]
\tikzset{
    partial ellipse/.style args={#1:#2:#3}{
        insert path={+ (#1:#3) arc (#1:#2:#3)}
    }
}
    \draw[green!50!black] (4,-2)--(10,-2)--(11,-1)--(5,-1)--(4,-2);
    \shade[ball color = red, opacity = 0.2] (7.5,-0.5) circle (1cm);
    \draw[thick] (7.5,-1.5) [partial ellipse=190:170:0.5cm and 2cm];
    \draw[>=latex, thick, ->] (7.5,-1.5) [partial ellipse=167:10:0.5cm and 2cm];    
    \draw[red] (4,0)--(10,0)--(11,1)--(5,1)--(4,0);
    \shade[color = red, opacity=0.2] (4,0)--(10,0)--(11,1)--(5,1)--(4,0);
    \draw[red, fill=red] (7.5,0.5) circle (0.05cm);
    \draw[red, thick, -latex] (7.5,0.5)--(8,1);
    \node[red] at (7.9,1.3) {$i$};
    \draw[black, fill=black] (7,-1.8) circle (0.05cm);
    \draw[black, fill=black] (8,-1.2) circle (0.05cm);
    \node[black] at (7,-2.3) {$-i$};
    \node[black] at (8.3,-0.7) {$i$};
    \node[black] at (10,0.7) {$\K(1,0)$};
    \node[black] at (5.9,-0.3) {$\K(0,1)$};
    \node[black] at (9,-1.5) {$\C$};
    \node[black] at (10,-0.5) {$\U$};
    \draw[thick, ->] (6.875,-1.5) arc (225:-45: 0.25cm);
    \draw[black, fill=black] (7.5,-1.5) circle (0.05cm);
    \node[black] at (7.7,-1.7) {$0$};
    \node[black] at (5.9,-1.4) {$z \mapsto -1/z$};
\end{tikzpicture}
\caption{The decorated horospheres $\K(1,0)$ and $\K(0,1)$ are related by the M\"{o}bius transformation $z \mapsto -1/z$.}
\label{Fig:K10_to_K01}
\end{center}
\end{figure}
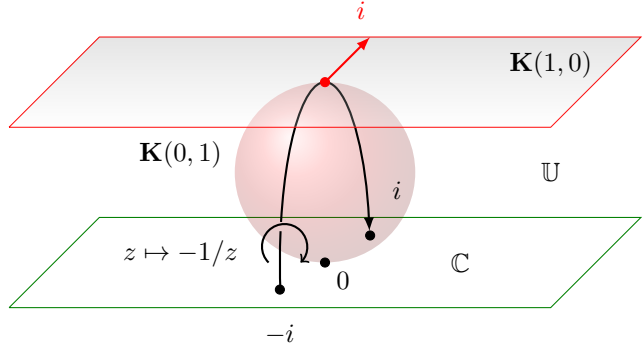

\begin{proof}[Proof of \refprop{U_horosphere_general}]
We use the previous two lemmas and $SL(2,\C)$-equivariance. Observe that
\[
\begin{pmatrix} \xi \\ 0 \end{pmatrix}
= \begin{pmatrix} \xi & 0 \\ 0 & \xi^{-1} \end{pmatrix}
\begin{pmatrix} 1 \\ 0 \end{pmatrix}
\quad \text{and} \quad
\begin{pmatrix} \xi \\ \eta \end{pmatrix}
= \begin{pmatrix} \eta^{-1} & \xi \\ 0 & \eta \end{pmatrix}
\begin{pmatrix} 0 \\ 1 \end{pmatrix}.
\]
If $\eta = 0$, then
we have
\[
\K \begin{pmatrix} \xi \\ 0 \end{pmatrix}
= \K \left( \begin{pmatrix} \xi & 0 \\ 0 & \xi^{-1} \end{pmatrix} \cdot \begin{pmatrix} 1 \\ 0 \end{pmatrix} \right)
= \begin{pmatrix} \xi & 0 \\ 0 & \xi^{-1} \end{pmatrix} \cdot \left( \K \begin{pmatrix} 1 \\ 0 \end{pmatrix} \right),
\]
and similarly for $\k$. The matrix $A \in SL(2,\C)$ involved corresponds to the isometry of $\U$ described by the M\"{o}bius transformation $z \mapsto \xi^2 z$. Thus $\K(\xi,0)$ is the image of $\K(1,0)$ under this isometry. By \reflem{JIHGF10}, $\K(1,0)$ is the horosphere centred at $\infty$ at Euclidean height $1$ with decoration specified by $i$. In $\U$, the isometry appears as a Euclidean dilation from the origin by factor $|\xi|^2$, and a rotation about the $z$-axis by $2 \arg \xi$. The resulting horosphere is again centred at $\infty$, i.e. a plane, but now has height $|\xi|^2$, and parallel oriented line field directed by $i \xi^2$. Thus $\K(\xi,0)$ is as claimed, and forgetting the decoration, $\k(\xi,0)$ is as claimed.

If $\eta \neq 0$ then 
\[
\K \begin{pmatrix} \xi \\ \eta \end{pmatrix}
= \K \left( \begin{pmatrix} \eta^{-1} & \xi \\ 0 & \eta \end{pmatrix} \begin{pmatrix} 0 \\ 1 \end{pmatrix} \right)
= \begin{pmatrix} \eta^{-1} & \xi \\ 0 & \eta \end{pmatrix} \cdot \left( \K \begin{pmatrix} 0 \\ 1 \end{pmatrix} \right).
\]
The matrix $A \in SL(2,\C)$ involved corresponds to the M\"{o}bius transformation $z \mapsto z \eta^{-2} + \xi \eta^{-1}$. The desired decorated horosphere $\K(\xi, \eta)$ is the image under $A$ of $\K(0,1)$, i.e. (by \reflem{U_horosphere_01}) the decorated horosphere centred at $0$ of Euclidean diameter $1$ and north-pole specification $i$. In $\U$, the corresponding isometry appears as a dilation from the origin by factor $|\eta|^{-2}$, a rotation about the $z$-axis by $-2 \arg \eta$, and then a translation in the horizontal ($\C$) plane by $\xi/\eta$. The resulting decorated horosphere $\K(\xi, \eta)$ has Euclidean diameter $|\eta|^{-2}$, center $\xi/\eta$, and north-pole specification $i \eta^{-2}$, as claimed. Forgetting the decoration, $\k(\xi, \eta)$ is as claimed.
\end{proof}

{\flushleft \textbf{Remark.} } It is perhaps not so surprising that a pair of complex numbers $(\xi, \eta)$ should correspond to an object centred at $\xi/\eta \in \partial \U$, with a tangent decoration in the direction of $i/\eta^2$. These are precisely the type of things preserved by M\"{o}bius transformations. Indeed, a M\"{o}bius transformation
\[
m \colon \CP^1 \To \CP^1, \quad m(z) = \frac{\alpha z+ \beta}{\gamma z+\delta},
\quad \text{corresponding to }
\begin{pmatrix} \alpha & \beta \\ \gamma & \delta \end{pmatrix} \in SL(2,\C),
\]
sends
\[
\frac{\xi}{\eta} \mapsto 
\frac{ \alpha \frac{\xi}{\eta} + \beta }{ \gamma \frac{\xi}{\eta} + \delta}
= \frac{\alpha \xi + \beta \eta}{\gamma \xi + \delta \eta} = \frac{\xi'}{\eta'}
\]
where
\[
\xi' = \alpha \xi + \beta \eta 
\quad \text{and} \quad 
\eta' = \gamma \xi + \delta \eta,
\quad \text{i.e.} \begin{pmatrix} \xi' \\ \eta' \end{pmatrix}
= \begin{pmatrix} \alpha &  \beta \\ \gamma & \delta \end{pmatrix} \begin{pmatrix} \xi \\ \eta \end{pmatrix}.
\]
Its derivative is then
\[
m'(z) = \frac{1}{(\gamma z+\delta)^2},
\quad \text{so that} \quad
m' \left( \frac{\xi}{\eta} \right) = \frac{1}{ \left( \gamma \frac{\xi}{\eta} + \delta \right)^2 }
= \frac{\eta^2}{ \left( \gamma \xi + \delta \eta \right)^2 } = \frac{\eta^2}{\eta'^2}.
\]
When applied to a tangent vector $i/\eta^2$ at $\xi/\eta$, one obtains
\[
m' \left( \frac{\xi}{\eta} \right) \frac{i}{\eta^2}
= \frac{\eta^2}{\eta'^2} \frac{i}{\eta^2} = \frac{i}{\eta'^2}
\quad \text{at} \quad
m \left( \frac{\xi}{\eta} \right) = \frac{\xi'}{\eta'}.
\]
In other words, a tangent decoration $i/\eta^2$ at $\xi/\eta$ maps to a tangent decoration $i/\eta'^2$ and $\xi'/\eta'$. In this way, the $SL(2,\C)$ equivariance arises naturally and geometrically.

\section{Spin decorations and complex lambda lengths}
\label{Sec:spin}

Finally, we incorporate spin into our considerations.

\subsection{Spin-decorated horospheres}
\label{Sec:spin-decorated_horospheres}

We now define the requisite notions for spin decorations on horospheres. In section \refsec{frame_fields} we discuss how decorations on horospheres give rise to certain frame fields; then we can define spin frame and spin isometries (\refsec{spin_frames_isometries}), and then spin decorations (\refsec{spin_decorations}).

Throughout this section we consider hyperbolic 3-space $\hyp^3$ independent of model. We will use the cross product $\times$ of vectors in the elementary sense that if $v,w$ are tangent vectors to $\hyp^3$ at the same point $p \in \hyp^3$ making an angle of $\theta$, then $v \times w$ has length $|v| \, |w| \sin \theta$ and points in the direction perpendicular to $v$ and $w$ as determined by the right hand rule. 

We will make much use of frames. By \emph{frame} we mean right-handed orthonormal frame in $\hyp^3$. In other words, a frame is a triple $(f_1, f_2, f_3)$ where all $f_i$ are unit tangent vectors to $\hyp^3$ at the same point and $f_1 \times f_2 = f_3$.

\subsubsection{Frame fields of decorated horospheres}
\label{Sec:frame_fields}

Throughout this section, let $\horo$ be a horosphere in $\hyp^3$. As with any smooth surface in a 3-manifold, at any point of $\mathpzc{h}$ there are two normal directions. 
\begin{defn}  \
\label{Def:horosphere_normals}
\begin{enumerate}
\item
The \emph{outward} normal direction to $\mathpzc{h}$ is the normal direction towards its centre. The outward unit normal vector field to $\mathpzc{h}$ is denoted $N^{out}$.
\item
The \emph{inward} normal direction to $\mathpzc{h}$ is the normal direction away from its centre. The inward unit normal vector field to $\mathpzc{h}$ is denoted $N^{in}$.
\end{enumerate}
\end{defn}
Intuitively, ``inwards" means in towards the bulk of $\hyp^3$, and ``outwards" means out towards the boundary at infinity. (This means that the ``outwards" direction from a horosphere points into the horoball it bounds.)

We now associate \emph{frames} to horospheres equipped with certain vector fields. . 

\begin{defn}
\label{Def:inward_outward_frame_fields}
Let $\V$ be a unit parallel vector field on $\mathpzc{h}$.
\begin{enumerate}
\item
The \emph{outward frame field of $\V$} is the frame field on $\mathpzc{h}$ given by
\[
f^{out}(\V) = \left( N^{out}, \V, N^{out} \times \V \right).
\]
\item
The \emph{inward frame field of $\V$} is the frame field on $\mathpzc{h}$ given by
\[
f^{in}(\V) = \left( N^{in}, \V, N^{in} \times \V \right).
\]
\end{enumerate}
A frame field on $\horo$ is an \emph{outward} (resp. \emph{inward}) frame field if it is the outward (resp. inward) frame field of some unit parallel vector field on $\horo$.
\end{defn}

\begin{defn}
If $(\mathpzc{h}, L^O_P) \in\mathfrak{H_D}$ with oriented parallel line field $L^O_P$, the \emph{associated outward (resp.inward) frame field} on $\mathpzc{h}$ is the outward (resp. inward) frame field of $\V$, where $\V$ is the unit tangent vector field on $\mathpzc{h}$ directing $L^O$.
\end{defn}
A decoration on $\horo$ thus  determines an outward and an  inward frame field on $\mathpzc{h}$. See \reffig{frames_from_decoration}.

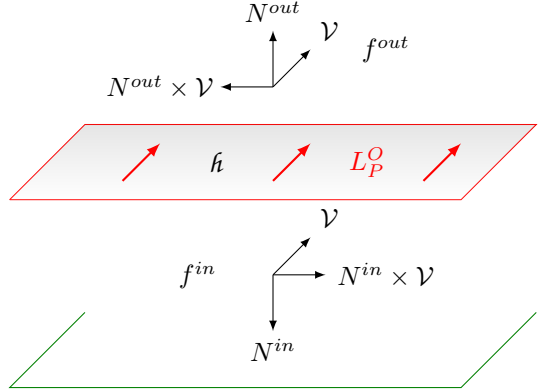
\begin{figure}
\begin{center}
\begin{tikzpicture}
    \draw[green!50!black] (5,-1.5)--(4,-2.5)--(10,-2.5)--(11,-1.5);
    \draw[red] (4,0)--(10,0)--(11,1)--(5,1)--(4,0);
    \shade[color = red, opacity=0.2] (4,0)--(10,0)--(11,1)--(5,1)--(4,0);
    \draw[red, thick, -latex] (5.5,0.25)--(6,0.75);
    \draw[red, thick, -latex] (7.5,0.25)--(8,0.75);
    \draw[red, thick, -latex] (9.5,0.25)--(10,0.75);
    \node[red] at (8.75,0.5) {$L_P^O$};
    \node[black] at (6.75,0.5) {$\horo$};
    \draw[black, -latex] (7.5,1.5)--(7.5,2.25);
    \node[black] at (7.5,2.5) {$N^{out}$};
    \draw[black, -latex] (7.5,1.5)--(8,2);
    \node[black] at (8.25,2.25) {$\V$};
    \draw[black, -latex] (7.5,1.5)--(6.8,1.5);
    \node[black] at (6,1.5) {$N^{out} \times \V$};
    \node[black] at (9,2) {$f^{out}$};
    \draw[black, -latex] (7.5,-1)--(7.5,-1.75);
    \node[black] at (7.5,-2) {$N^{in}$};
    \draw[black, -latex] (7.5,-1)--(8,-0.5);
    \node[black] at (8.25,-0.25) {$\V$};
    \draw[black, -latex] (7.5,-1)--(8.2,-1);
    \node[black] at (9,-1) {$N^{in} \times \V$};
    \node[black] at (6.5,-1) {$f^{in}$};
\end{tikzpicture}
\caption{A decoration $L^P_O$ on a horosphere $\horo$ determines inward and outward frame fields.}
\label{Fig:frames_from_decoration}
\end{center}
\end{figure}

\subsubsection{Spin frames and spin isometries}
\label{Sec:spin_frames_isometries}

The bundle of (right-handed orthonormal) frames over $\hyp^3$ is a principal $SO(3)$ bundle. As $\pi_1(SO(3)) \cong \Z/2\Z$, the double cover of $SO(3)$ is also its universal cover, and this is the spin group $\Spin(3)$.
\begin{defn}
\label{Def:Fr}
Denote by $\Fr \To \hyp^3$ the principal $SO(3)$ bundle of (right-handed orthonormal) frames over $\hyp^3$, and $\Spin \To \hyp^3$ its double cover, a principal $\Spin(3)$ bundle.
\end{defn}
A point of (the total space of) $\Fr$ consists of a point of $\hyp^3$ together with a frame there; similarly, a point of $\Spin$ consists of a point of $\hyp^3$ together with one of the two lifts of a frame there.
\begin{defn}
A point of the total space of $\Spin$ is called a \emph{spin frame}.
\end{defn}

The orientation preserving isometry group $\Isom^+ \hyp^3$ of $\hyp^3$ acts simply transitively on $\Fr$: there is a unique orientation-preserving isometry sending any frame at any point of $\hyp^3$ to any other frame at any other point. Using the isomorphism $\Isom^+(\hyp^3) \cong PSL(2,\C)$ yields a diffeomorphism
\begin{equation}
\label{Eqn:PSL2C_Fr}
PSL(2,\C) \cong \Fr.
\end{equation}

We can make this homeomorphism explicit by choosing a specific frame, a ``base frame" $f_0$. The identity $1 \in PSL(2,\C)$ corresponds to the frame $f_0$, and
then a general element $A \in PSL(2,\C) \cong \Isom^+ \hyp^3$ corresponds to the frame obtained by applying the  isometry $A$ (and its derivative) to $f_0$. In other words, he correspondence is given by $A \leftrightarrow A\cdot f_0$. The actions of $PSL(2,\C)$ on itself by multiplication, and on $\Fr$ by orientation-preserving isometries, are equivariant with respect to this correspondence; so we have an identification of $PSL(2,\C)$-spaces.

This identification then lifts to universal covers: a path in $PSL(2,\C)$ from $1$ to an element $A$ corresponds to a path in $\Fr$ from $f_0$ to $A \cdot f_0$. Recalling the definition of a universal cover, this gives an identification between points of the universal cover of $PSL(2,\C)$, and the universal cover of $\Fr$. These universal covers are $SL(2,\C)$, and the space of spin frames $\Spin$, respectively. So we  obtain a homeomorphism which identifies $SL(2,\C)$ with spin frames.
\begin{equation}
\label{Eqn:SL2C_Spin}
SL(2,\C) \cong \Spin
\end{equation}
Under this identification,  the two matrices $A,-A \in SL(2,\C)$ lifting $\pm A \in PSL(2,\C)$ correspond to the two spin frames above the frame $(\pm A).f_0$. The two spin frames lifting a common frame are related  by a $2\pi$ rotation about any axis at their common point.

Indeed, $SL(2,\C)$ acts freely and transitively on $\Spin$, whose elements are spin frames in $\hyp^3$.
\begin{defn}
A \emph{spin isometry} is an element of the universal cover of $\Isom^+ \hyp^3$.
\end{defn}
Thus, a spin isometry is just an element of $SL(2,\C)$, regarded as the double/universal cover of $PSL(2,\C) \cong \Isom^+ \hyp^3$. Each orientation-preserving isometry of $\hyp^3$ lifts to two spin isometries, which differ by a $2\pi$ rotation. Just as an orientation-preserving isometry sends frames to frames, a spin isometry sends spin frames to spin frames.

\subsubsection{Spin decorations}
\label{Sec:spin_decorations}

Let $\horo$ be a horosphere in $\hyp^3$. A frame field on $\mathpzc{h}$ is a continuous section of $\Fr$ along $\mathpzc{h}$, and such a frame field has two continuous lifts to $\Spin$.
\begin{defn} 
An \emph{outward (resp. inward) spin decoration} on $\mathpzc{h}$ is a continuous lift of an outward (resp. inward) frame field on $\mathpzc{h}$ from $\Fr$ to $\Spin$.
\end{defn}
In other words, an outward (resp. inward) spin decoration on $\mathpzc{h}$ is a choice of lift to $\Spin$ of a frame field of the form $f^{out}(\V)$ (resp. $f^{in}(\V)$), for some unit parallel vector field $\V$ on $\mathpzc{h}$.

Given an inward frame field $f^{in}(\V) = (N^{in}, \V, N^{in} \times \V)$ on $\mathpzc{h}$ corresponding to a unit parallel vector field $\V$, we can obtain $f^{out}(\V) = (N^{out}, \V, N^{out} \times \V)$ by rotating the frame at each point by an angle of $\pi$ about $\V$. This rotation preserves $\V$ and sends $N^{in}$ to $N^{out}$, hence sends one frame to the other, and a similar rotation sends $f^{out}(\V)$ back to $f^{in}(\V)$. Each rotation of angle $\pi$ can be done in either direction around $\V$. However, once we take spin lifts, rotations of angle $\pi$ clockwise or anticlockwise about $\V$ yield distinct results, since the results are related by a $2\pi$ rotation. Thus we make the following definition, where rotations about vectors are made in the usual right-handed way.
\begin{defn} \
\label{Def:associated_inward_outward_spindec}
\begin{enumerate}
\item
If $W^{out}$ is an outward spin decoration on $\mathpzc{h}$ lifting an outward frame field $(N^{out}, \V, N^{out} \times \V)$ for some unit parallel vector field $\V$, the \emph{associated inward spin decoration} is the inward spin decoration obtained by rotating $W^{out}$ by angle $\pi$ about $\V$ at each point of $\mathpzc{h}$.
\item
If $W^{in}$ is an inward spin decoration on $\mathpzc{h}$ lifting an inward frame field $(N^{in}, \V, N^{in} \times \V)$ for some unit parallel vector field $\V$, the \emph{associated outward spin decoration} is the outward spin decoration obtained by rotating $W^{in}$ by angle $-\pi$ about $\V$ at each point of $\mathpzc{h}$.
\end{enumerate}
\end{defn}
The choice of $\pi$ and $-\pi$ is somewhat arbitrary but is required for our main theorem to hold.

By construction, if $W^{out}$ (resp. $W^{in}$) is a lift of $f^{out}(\V)$ (resp. $f^{in}(\V)$), then the associated inward (resp. outward) spin decoration is a spin decoration lifting $f^{in}(\V)$ (resp. $f^{out}(\V)$). Moreover, these associations are inverses so we obtain pairs $(W^{in}, W^{out})$ where each is associated to the other. Given $\V$, the frame fields $f^{in}(\V)$ and $f^{out}(\V)$ are determined, and then there are two choices of lift for $W^{in}$ and two choices of lift for $W^{out}$. Each choice of $W^{in}$ has an associated $W^{out}$. Thus, the choice of $W^{in}$ determines the associated $W^{out}$ and vice versa. 

Later, in \refsec{complex_lambda_lengths}, inward and outward fields feature equally in the definition of a complex lambda length. So we prefer to use both of them, as a pair, in the following definition.
\begin{defn}
\label{Def:spin_decoration}
A \emph{spin decoration} on $\mathpzc{h}$ is a pair $W = (W^{in}, W^{out})$ where $W^{in}$ is an inward spin decoration on $\mathpzc{h}$, $W^{out}$ is an outward spin decoration on $\mathpzc{h}$, and each is associated to the other. The pair $(\horo, W)$ is called a \emph{spin-decorated horosphere}.
\end{defn}

{\flushleft \textbf{Remark.} } Under the identification $PSL(2,\C) \cong \Fr$, decorated horospheres correspond to certain cosets of $PSL(2,\C)$. Let us make the homeomorphism \refeqn{PSL2C_Fr} explicit by choosing the base frame $f_0$ to be the frame $(e_z, e_y, -e_x) \in \Fr$ at the point $p_0 = (0,0,1)$ in the upper half space model, where $e_x, e_y, e_z$ denote unit vectors in the $x,y,z$ directions. Then $1\in PSL(2,\C)$ corresponds to the base frame $f_0$ at $p_0$. This $f_0$ forms part of an outward frame field $f^{out}_0$ on the horosphere $\mathpzc{h}_0$ centred at $\infty$ passing through $p_0$. This outward frame field $f^{out}_0$ arises from the decoration on $\horo_0$ in the $y$-direction.

The frames of $f^{out}_0$ are obtained from $f_0$ by parabolic isometries which appear as horizontal translations in $\U$. These isometries form the subgroup of $PSL(2,\C)$ given by
\[
\underline{P} = \left\{ \pm \begin{pmatrix} 1 & \alpha \\ 0 & 1 \end{pmatrix} \mid \alpha \in \C \right\}.
\]
The cosets $g \underline{P}$, over $g \in PSL(2,\C)$, then yield the outward frame fields associated to oriented parallel line fields on horospheres, and we obtain a bijection
\begin{equation}
\label{Eqn:decorated_horospheres_cosets}
PSL(2,\C)/ \underline{P} \cong \mathfrak{H_D}.
\end{equation}

\begin{defn}
\label{Def:spin-decorated_horospheres}
The set of all spin-decorated horospheres is denoted $\mathfrak{H_D^S}$.
\end{defn}
There is a 2-1 projection map $\mathfrak{H_D^S} \To \mathfrak{H_D}$ given as follows. A spin decorated horosphere $(\horo, W)$ contains a pair $W = (W^{in}, W^{out})$ of associated inward and outward spin decorations on a horosphere $\mathpzc{h}$, which project down to  inward and outward frame fields on $\mathpzc{h}$. The inward frame is of the form $f^{in}(\V)$ for some unit parallel vector field $\V$ on $\mathpzc{h}$, and the outward frame is of the form $f^{out}(\V)$, for the same $\V$. This $\V$ directs an oriented parallel line field $L_P^O$ on $\horo$, i.e. a decoration on $\horo$. The spin decoration $W$ projects to the decoration $L_P^O$. There are two spin decorations on $\horo$ which project to this $L_P^O$, namely $W$, and the spin decoration $W' = (W'^{in}, W'^{out})$ obtained from rotating $W^{in}$ and $W^{out}$ through $2\pi$ at each point.

{\flushleft \textbf{Remark.} }Just as decorated horospheres correspond to certain cosets of $PSL(2,\C)$ \refeqn{decorated_horospheres_cosets}, spin-decorated horospheres correspond to certain cosets of $SL(2,\C)$. Starting from the identification $SL(2,\C) \cong \Spin$ \refeqn{SL2C_Spin}, we can make it explicit by choosing a base spin frame $\widetilde{f_0}$, a lift of the base frame $f_0$. An $A\in SL(2,\C)$, being a point of the universal cover of $PSL(2,\C) \cong \Isom^+(\hyp^3)$, can be regarded as a (homotopy class of a) path in $PSL(2,\C)$ from the identity to the element $\pm A$ of $PSL(2,\C)$. This can be regarded as a path of isometries starting at the identity, and its action on frames yields a path from $\widetilde{f_0}$ to the spin frame corresponding to $A$. On $\mathpzc{h}_0\in\mathfrak{H}$ centred at $\infty$ passing through $p_0$, the frame $f_0$ forms part of a unique outward frame field $f_0^{out}$. This outward frame field lifts to two distinct outward spin decorations on $\mathpzc{h}_0$. One of these contains $\widetilde{f_0}$, corresponding to the identity in $SL(2,\C)$, and the spin frames of this outward spin decoration correspond to the elements of $SL(2,\C)$ forming the parabolic subgroup
\[
P = \left\{ \begin{pmatrix} 1 & \alpha \\ 0 & 1 \end{pmatrix} \mid \alpha \in \C \right\}.
\]
The other lift of $f_0^{out}$ is the outward spin decoration on $\mathpzc{h}_0$ whose spin frames are obtained from those of the previous spin decoration by a $2\pi$ rotation; these correspond to the negative matrices in $SL(2,\C)$, and correspond to the coset
\[
-P = \begin{pmatrix} -1 & 0 \\ 0 & -1 \end{pmatrix} P.
\]
In general, cosets $gP$, over $g \in SL(2,\C)$, yield the outward spin decorations corresponding to spin decorations on horospheres, and we obtain a bijection
\begin{equation}
\label{Eqn:SL2C_mod_P}
SL(2,\C)/P \cong \mathfrak{H_D^S}.
\end{equation}

\subsection{Topology of spaces and maps}
\label{Sec:topology_of_spaces_and_maps}

We now consider the various spaces and maps in the composition $\K$:
\[
\C_\times^2 \stackrel{\F}{\To} \mathcal{F_P^O}(\HH) \stackrel{\G}{\To} \mathcal{F_P^O} (\R^{1,3}) \stackrel{\H}{\To} \mathfrak{H_D}(\hyp) \stackrel{\I}{\To} \mathfrak{H_D}(\Disc) \stackrel{\J}{\To} \mathfrak{H_D}(\U).
\]
In turn, we consider the topology of spaces (\refsec{topology_of_spaces}), the topology of the maps (\refsec{topology_of_maps}), then lift them to incorporate spin (\refsec{lifts_of_maps_spaces}).

\subsubsection{Topology of spaces}
\label{Sec:topology_of_spaces}

Topologically, $\C_\times^2 \cong \R^4 \setminus \{0\} \cong S^3 \times \R$, which is simply connected: $\pi_1 (\C^2_\times) \cong \pi_1 (S^3) \times \pi_1 (\R)$ is trivial.

The space of flags $\mathcal{F_P^O}(\R^{1,3})$ naturally has the topology of $UTS^2 \times \R$, where $UTS^2$ is the unit tangent bundle of $S^2$. A point of $UTS^2$ describes a point on the celestial sphere $\S^+ \cong S^2$, or equivalently a lightlike ray, together with a tangent direction to $\S^+$ at that point, which precisely provides a flag 2-plane containing that ray. There is also an $\R$ family of points on each lightlike ray. This provides an identification $\mathcal{F_P^O}(\R^{1,3}) \cong UTS^2 \times \R$ and we use it to provide a topology and smooth structure on $\mathcal{F_P^O}(\R^{1,3})$. Since $\g$ is a linear isomorphism between $\HH$ and $\R^{1,3}$, we can similarly identify $\mathcal{F_P^O}(\HH) \cong UTS^2 \times \R$ so that $\G$ is a diffeomorphism.

The space $UTS^2$ is not simply connected; it is diffeomorphic to $SO(3)$. One way to see this standard fact is to note that a point of $S^2$ yields a unit vector $v_1$ in $\R^3$; a unit tangent vector to $S^2$ at $v_1$ yields an orthonormal unit vector $v_2$; and then $v_1, v_2$ uniquely determines a right-handed orthonormal frame for $\R^3$. This gives a diffeomorphism between $UTS^2$ and the space of frames in $\R^3$, i.e. $UTS^2 \cong SO(3)$. Thus $\pi_1 (UTS^2) \cong \pi_1 (SO(3)) \cong \Z/2\Z$, and each space of flags has fundamental group $\pi_1 (UTS^2 \times \R) \cong \pi_1 (UTS^2) \times \pi_1 (\R) \cong \Z/2\Z$.

The spaces of decorated horospheres $\mathfrak{H_D}$ naturally have the topology of $UTS^2 \times \R$, with fundamental group $\Z/2\Z$. This is true for any model of $\hyp^3$. A point of $UTS^2$ describes the point at infinity in $\partial \hyp^3 \cong S^2$ of a horosphere, together with a parallel tangent field direction, and at each point at infinity there is an $\R$ family of horospheres. This provides an identification $\mathfrak{H_D} \cong UTS^2 \times \R$ and we use it to provide a topology and smooth structure on $\mathfrak{H_D}$. Since $\i,\j$ are isometries between different models of $\hyp^3$, $\I$ and $\J$ provide diffeomorphisms between $\mathfrak{H_D}(\hyp)$, $\mathfrak{H_D}(\Disc)$ and $\mathfrak{H_D}(\U)$.

\subsubsection{Topology of maps}
\label{Sec:topology_of_maps}

We saw above that $\G, \I, \J$ are diffeomorphisms, so it remains to consider the maps $\F$ and $\H$, which topologically are maps $S^3 \times \R \To UTS^2 \times \R$ and $UTS^2 \times \R \To UTS^2 \times \R$ respectively.

First, consider the map $\F$. Since $\G$ is a diffeomorphism, we may equivalently consider the map $\G \circ \F \colon S^3 \times \R \To UTS^2 \times \R$. Both $S^3 \times \R$ and $UTS^2 \times \R$ are both naturally $S^1$ bundles over $S^2 \times \R$, the former via the Hopf fibration, the latter as a unit tangent bundle. We saw in \reflem{C2_to_R31_Hopf_fibrations} that $\g \circ \f \colon S^3 \times \R \To L^+$, sends each 3-sphere $S^3_r$ of constant radius $r$, to the 2-sphere $L^+ \cap \{ T = r^2\}$, via a Hopf fibration. Since $L^+ \cong S^2 \times \R$, topologically $\g \circ \f \colon S^3 \times \R \To S^2 \times \R$ is the product of the Hopf fibration with the identity.

The map $\G \circ \F$ is then a map $S^3 \times \R \To UTS^2 \times \R$ which adds the data of a flag to the point on $L^+$ described by $\g \circ \f$. It thus projects to $\g \circ \f$ under the projection map $UTS^2 \times \R \To S^2 \times \R$. That is, the following diagram commutes. 
\begin{center}
    \begin{tikzpicture}
        \node (a) at (0,0){$S^3\times\R$};
        \node (b) at (3,0){$UTS^2\times\R$};
        \node (c) at (3,-1){$S^2\times\R$};
        \draw[->] (a) -- (b) node [pos=0.5,above] {$\G\circ\F$};
        \draw[->] (a) -- (c) node [pos=0.35,below] {$\g\circ\f$};
        \draw[->] (b) -- (c);
    \end{tikzpicture}
\end{center}
Another way of viewing this diagram is that $\G \circ \F$ is a map of $S^1$ bundles over $S^2 \times \R$. Let us consider the fibres over a point $p \in S^2 \times \R \cong L^+$, which can equivalently be described by a pair $\underline{p} \in \S^+ \cong \CP^1$, and a length $r>0$ (or $T$-coordinate $T=r^2$). In $S^3 \times \R$, the fibre over $p \in \S^2 \times \R$ is the set of $(\xi, \eta)$ such that $|\xi|^2 + |\eta|^2 = r^2$ and $\xi/\eta = \underline{p}$. Given one point in the fibre $(\xi_0, \eta_0)$ over $p$, the other points in the fibre are of the form $e^{i\theta}(\xi_0, \eta_0)$, by \reflem{gof_properties}, and form an $S^1$. Under $\G \circ \F$, this fibre maps to the fibre of unit tangent directions to $S^2$ at $\underline{p}$, or equivalently, the fibre of flag directions over $\R p$.

Proceeding around an $S^1$ fibre in $\C_\times^2 \cong S^3 \times \R$ corresponds to a path $e^{i\theta}(\xi_0, \eta_0)$ for $\theta$ from $0$ to $2\pi$. Proceeding around the $S^1$ factor in a fibre in $\mathcal{F_P^O}(\R^{1,3})$ corresponds to rotating the 2-plane of a null flag through $2\pi$ about a fixed ray. As we saw in \refsec{rotating_flags}, and explicitly in \reflem{flag_basis_rotation}, as we move through the $S^1$ fibre above $p$ in $S^3 \times \R$, the point $e^{i\theta}(\xi_0, \eta_0)$ under $\G \circ \F$ produces a flag rotation of angle $-2\theta$. So $\G \circ \F$ is a smooth 2--1 map on each fibre. We discussed this explicitly in the proof of \refprop{F_G_surjective}.

The map $\G$ is also a bundle isomorphism: $\g$ is a linear isomorphism between $\HH$ and $\R^{1,3}$, and the diffeomorphism provided by $\G$ between $\mathcal{F_P^O}(\HH)$ and $\mathcal{F_P^O}(\R^{1,3})$, both diffeomorphic to $UTS^2 \times \R$, respects their structure as $S^1$ bundles over $S^2 \times \R$.

Thus, both $\F$ and $\G \circ \F$ are bundle maps $S^3 \times \R \To UTS^2 \times \R$ of $S^1$-bundles over $S^2 \times \R$, which are 2--1 on each fibre. They are also covering maps, since $UTS^2 \cong \RP^3$, so topologically both $\F$ and $\G \circ \F$ they are maps $S^3 \times \R \To \RP^3 \times \R$ which are topologically the product of the 2-fold covering map with the identity. 

We now turn to the map $\H \colon \mathcal{F_P^O}(\R^{1,3}) \To \mathfrak{H_D}(\hyp)$, which is topologically a map $UTS^2 \times \R \To UTS^2 \times \R$. Again, both spaces are $S^1$-bundles over $S^2 \times \R$. As discussed in \refsec{light_cone_to_horosphere}, the map $\h \colon L^+ \To \horos(\hyp)$ is a diffeomorphism, both spaces being diffeomorphic to $S^2 \times \R$. We have seen that $\mathcal{F_P^O}(\R^{1,3})$ is an $S^1$-bundle over $L^+ \cong \R^2 \times S^1$, with an $S^1$ worth of flag directions at each point of $L^+$. And $\mathfrak{H_D}(\hyp)$ is an $S^1$-bundle over $\horos(\hyp)$, with an $S^1$ of decorations over each horosphere. Thus we have a commutative diagram
\[
\begin{array}{ccc}
UTS^2 \times \R \cong \mathcal{F_P^O}(\R^{1,3}) & \stackrel{\H}{\To}&  \mathfrak{H_D}(\hyp) \cong UTS^2 \times \R \\
\downarrow  & & \downarrow \\
S^2 \times \R \cong L^+ & \stackrel{\h}{\To} & \horos(\hyp) \cong S^2 \times \R
\end{array}
\]
As argued in \reflem{H_bijection}, $\H$ maps the $S^1$ fibre of flags above a point $p \in L^+$, to the $S^1$ fibre of decorations on the horosphere $\h(p) \in \horos(\hyp)$, in bijective fashion. This map is in  fact smooth: as the 2-plane of the flag rotates, the same 2-plane rotates to provide different decorations on a horosphere, always intersecting the horosphere transversely. So $\H$ is a diffeomorphism and a bundle isomorphism.

Combining the above with \reflem{F_G_2-1}, we have now proved the following. This is the non-spin version of the main \refthm{spinors_to_horospheres}, using spinors up to sign.
\begin{prop}
\label{Prop:main_thm_up_to_sign}
The map $\K \colon \C^2_\times \To \mathfrak{H_D}(\U)$ is smooth, surjective, 2--1, and $SL(2,\C)$-equivariant. It yields a smooth, bijective, $SL(2,\C)$-equivariant map
\[
\frac{\C^2_\times}{ \{ \pm 1 \} } \To \mathfrak{H_D}(\U)
\]
between nonzero spin vectors up to sign, and decorated horospheres. The action of $SL(2,\C)$ on both $\C^2_\times/\{\pm 1\}$ and $\mathfrak{H_D}(\U)$ factors through $PSL(2,\C)$.
\qed
\end{prop}

\subsubsection{Spin lifts of maps and spaces}
\label{Sec:lifts_of_maps_spaces}
Let us now consider spin lifts, or universal covers, of the above spaces. 

We observe that the 2--1 projection $\mathfrak{H_D^S} \To \mathfrak{H_D}$ is a double cover. This can be seen directly, or via the identifications with $SL(2,\C)/P$ and $PSL(2,\C)/\underline{P}$ of \refeqn{SL2C_mod_P} and \refeqn{decorated_horospheres_cosets}.
Since $\mathfrak{H_D^S}$ is a double cover of $\mathfrak{H_D} \cong UTS^2 \times \R \cong SO(3) \times \R \cong \RP^3 \times \R$, we have $\mathfrak{H_D^S} \cong S^3 \times \R$, and $\mathfrak{H_D^S}$ is in fact the universal cover of $\mathfrak{H_D}$. We also have a commutative diagram
\[ \begin{array}{ccccc}
SL(2,\C) & \To & SL(2,\C)/P & \cong & \mathfrak{H_D^S} \\
\downarrow && \downarrow && \downarrow \\
PSL(2,\C) & \To & PSL(2,\C)/(\underline{P}) & \cong & \mathfrak{H_D}
\end{array} \]
where the vertical maps are double covers and universal covers.

Similarly, the spaces $\mathcal{F_P^O}$ are diffeomorphic to $\RP^3 \times \R$, so have double and universal covers diffeomorphic to $S^3 \times \R$, and these arise from bundle maps which are 2--1 on each fibre. In $\mathcal{F_P^O}$, a fibre is the $S^1$ family of flags with a given base point and flagpole. In the double cover, rotating a flag about its flagpole through $2\pi$ (and keeping the base point fixed) does not return to the same null flag, but a rotation of $4\pi$ does return to the same fixed point. 
\begin{defn}
\label{Def:covers_of_flags}
We denote by $\mathcal{SF_P^O}(\HH)$ and $\mathcal{SF_P^O}(\R^{1,3})$ the double (universal) covers of $\mathcal{F_P^O}(\HH)$ and $\mathcal{F_P^O}(\R^{1,3})$ respectively. We call an element of $\mathcal{SF_P^O}(\HH)$ or $\mathcal{SF_P^O}(\R^{1,3})$ a \emph{spin flag}.
\end{defn}
A spin flag in \cite{Penrose_Rindler84} is called a \emph{null flag}.

The maps $\G,\H,\I,\J$ are all diffeomorphisms, and these lift to diffeomorphisms of double covers of spaces $\mathfrak{H_D^S}$ and $\mathcal{SF_P^O}$. We denote these diffeomorphisms $\widetilde{\G}, \widetilde{\H}, \widetilde{\I}, \widetilde{\J}$.

Since $\C_\times^2$ is simply connected, we also obtain a lift $\widetilde{\F}$ of $\F$ from $\C^2_\times$ to $\mathcal{SF_P^O}(\HH)$. The result is a  sequence of diffeomorphisms lifting $\F, \G, \H, \I, \J$, between spaces all  diffeomorphic to $S^3 \times \R$; they are also isomorphisms of $S^1$ bundles over $S^2 \times \R$.
\begin{equation}
\label{Eqn:fghij_lifts}
\C_\times^2 \stackrel{\widetilde{\F}}{\To} \mathcal{SF_P^O}(\HH) \stackrel{\widetilde{\G}}{\To} \mathcal{SF_P^O} (\R^{1,3}) \stackrel{\widetilde{\H}}{\To} \mathfrak{H_D^S}(\hyp) \stackrel{\widetilde{\I}}{\To} \mathfrak{H_D^S}(\Disc) \stackrel{\widetilde{\J}}{\To} \mathfrak{H_D^S}(\U).
\end{equation}

We have already seen that $\F,\G,\H,\I,\J$ are all $SL(2,\C)$ equivariant; we now argue that their lifts are too. First, note that the actions of $SL(2,\C)$ on $\mathcal{F_P^O}(\HH)$, $\mathcal{F_P^O}(\R^{1,3})$ and  $\mathfrak{H_D}$ all factor through $PSL(2,\C)$. The action on $\mathcal{F_P^O}(\HH)$ derives from the action of $A \in SL(2,\C)$ on $S \in \HH$ as $S \mapsto ASA^*$, which when $A=-1$ is trivial. The same is true for the action on $\mathcal{F_P^O}(\R^{1,3})$, which is equivalent via the diffeomorphism $\G$. Similarly for the action on $\horos_D$, the action of $SL(2,\C)$ factors through $PSL(2,\C)$ since $PSL(2,\C) \cong \Isom^+ \hyp^3$.

As $SL(2,\C)$ is the universal cover of $PSL(2,\C)$, we may regard elements of $SL(2,\C)$ as homotopy classes of paths in $PSL(2,\C)$ starting from the identity, and the action of elements in such a path on $\C^2_\times$, $\mathcal{F_P^O}(\HH)$, $\mathcal{F_P^O}(\R^{1,3})$, or $\mathfrak{H_D}$ in any model of hyperbolic space, is equivariant. The resulting paths in $\mathcal{F_P^O}$ or $\mathfrak{H_D}$ lifts to paths in the universal covers $\mathcal{SF_P^O}$ or $\mathfrak{H_D^S}$, and so we obtain equivariant actions of $SL(2,\C)$ on the universal covers, proving the following proposition.
\begin{prop}
\label{Prop:spin_decoration_equivariance}
The maps $\widetilde{\F},\widetilde{\G},\widetilde{\H},\widetilde{\I},\widetilde{\J}$ are all diffeomorphisms, equivariant with respect to the actions of $SL(2,\C)$ on $\C_\times^2$, $\mathcal{SF_P^O}(\HH)$, $\mathcal{SF_P^O}(\R^{1,3})$, $\mathfrak{H_D^S}(\hyp)$, $\mathfrak{H_D^S}(\Disc)$ and $\mathfrak{H_D^S}(\U)$.
\qed
\end{prop}
Abbreviating the composition to
\[
\widetilde{\K} = \widetilde{\J} \circ \widetilde{\I} \circ \widetilde{\H} \circ \widetilde{\G} \circ \widetilde{\F},
\]
and observing that $\widetilde{\K}$ projects to $\K$ upon forgetting spin, mapping spin-decorated horospheres to decorated horospheres,
we now have the following precise version of the main \refthm{spinors_to_horospheres} and \refthm{explicit_spinor_horosphere_decoration}.
\begin{theorem}
\label{Thm:main_thm_precise}
The map $\widetilde{\K} \colon \C^2_\times \To \mathfrak{H_D^S}(\U)$ is an $SL(2,\C)$-equivariant diffeomorphism. Under $\widetilde{\K}$, a nonzero spinor corresponds to a spin-decorated horosphere which projects to the decorated horosphere described in \refprop{JIHGF_general_spin_vector}.
\end{theorem}

\subsection{Complex lambda lengths}
\label{Sec:complex_lambda_lengths}

We define requisite notions for lambda lengths. In this section we consider $\hyp^3$ independent of model.

\begin{defn}
Let $q$ be a point on an oriented geodesic $\gamma$ in $\hyp^3$.
\begin{enumerate}
\item
Let $f = (f_1, f_2, f_3)$ be a (right-handed orthonormal) frame at $q$. We say $f$ is \emph{adapted to $\gamma$} if $f_1$ is positively tangent to $\gamma$.
\item
Let $\widetilde{f}$ be a spin frame at $q$. We say $\widetilde{f}$ is \emph{adapted to $\gamma$} if it is the lift of a frame  adapted to $\gamma$.
\end{enumerate}
\end{defn}

Suppose now that $\gamma$ is an oriented geodesic in $\hyp^3$, and $q_1, q_2$ are two points on this line (not necessarily distinct). Suppose we have a frame $f^i$ at $q_i$ adapted to $\gamma$, for $i=1,2$; let $f^i = (f^i_1, f^i_2, f^i_3)$. We can then consider parallel translation along $\gamma$ from $q_1$ to $q_2$; this translation is by some distance $\rho$, which we regard as positive or negative by reference to the orientation on $\gamma$. This parallel translation takes $f^1$ to a frame ${f^1}'$ at $q_2$. Since $f^1$ is adapted to $\gamma$, its first vector points positively along $\gamma$, and since ${f^1}'$ is related to $f^1$ by parallel translation along $\gamma$, ${f^1}'$ is also adapted to $\gamma$.  Thus ${f^1}'$ and $f^2$ lie at the same point $q_2$ and have the same first vector. A further rotation of same angle $\theta$ about $\gamma$ (signed using the orientation of $\gamma$, using the standard right-handed convention) then takes ${f^1}'$ to $f^2$. We regard $\rho + i\theta$ as a complex length from $f^1$ to $f^2$, which we also denote by $d$. Note that $\theta$ is only well defined modulo $2\pi$.

If the frames $f^1, f^2$ are lifted to spin frames, the same applies, except that $\theta$ is then well defined modulo $4\pi$. We summarise in the following definition.
\begin{defn}
\label{Def:complex_distance}
Let $f^1, f^2$ be frames, or spin frames, at points $q_1, q_2$ on an oriented geodesic $\gamma$, adapted to $\gamma$. The \emph{complex translation distance}, or just \emph{complex distance} from $f^1$ to $f^2$ is $d = \rho+i\theta$, where a translation along $\gamma$ of signed distance $\rho$, followed by a rotation about $\gamma$ of angle $\theta$, takes $f^1$ to $f^2$.
\end{defn}

Two arbitrarily chosen frames, or spin frames, will usually not be adapted to any single oriented geodesic. If they are both adapted to a single oriented geodesic, then that geodesic is unique. So we may simply speak of the complex distance from $f^1$ to $f^2$, when it exists, without reference to any geodesic.

The complex distance  between two frames adapted to a common geodesic is well defined modulo $2\pi i$. The complex distance  between two spin frames adapted to a common geodesic is well defined modulo $4\pi i$.

Suppose now that we have two horospheres. We first consider decorations on them, then lift to spin decorations. 

So, let $(\mathpzc{h}_i, L^O_i)\in\mathfrak{H_D}$, for $i=1,2$, with $\mathpzc{h}_i\in\mathfrak{H}$ and $L^O_i$ an oriented parallel line field on $\horo_i$. Let $p_i \in \partial \hyp^3$ be the centre of $\mathpzc{h}_i$, and assume $p_1 \neq p_2$. Let $\gamma_{12}$ be the oriented geodesic from $p_1$ to $p_2$. Let $q_i = \gamma_{12} \cap \mathpzc{h}_i$. So if $\horo_1, \horo_2$ are disjoint then $q_1$ is the closest point on $\mathpzc{h}_1$ to $\mathpzc{h}_2$, $q_2$ is the closest point on $\mathpzc{h}_2$ to $\mathpzc{h}_1$, and $\gamma_{12}$ is the unique common perpendicular geodesic to $\mathpzc{h}_1$ and $\mathpzc{h}_2$, oriented from $p_1$ to $p_2$. However, these constructions apply even if $\horo_1, \horo_2$ are tangent or overlap. The oriented parallel line field $L^O_i$ on $\mathpzc{h}_i$ determines an associated outward frame field $f_i^{out}$, and inward frame field $f_i^{in}$, on $\mathpzc{h}_i$. Note that $f_1^{in}(q_1)$ and $f_2^{out}(q_2)$ are both adapted to $\gamma_{12}$, while $f_1^{out}(q_1)$ and $f_2^{in}(q_2)$ are not; rather $f_1^{out}(q_1)$ and $f_2^{in}(q_2)$ are both adapted to the oriented geodesic $\gamma_{21}$ from $p_2$ to $p_1$.

If we instead have spin decorations $(\mathpzc{h}_i, W_i)\in\mathfrak{H_D^S}$, then each $\mathpzc{h}_i\in\mathfrak{H}$ has a spin decoration $W_i$, from which we obtain an outward spin decoration $W_i^{out}$ and an inward spin decoration $W_i^{in}$ on each $\mathpzc{h}_i$. Note that $W_i^{out}$ and $W_i^{in}$ here project to $f_i^{out}$ and $f_i^{in}$ as in  the previous paragraph. So $W_1^{in}(q_1)$ and $W_2^{out}(q_2)$ are adapted to $\gamma_{12}$, and $W_1^{out}(q_1)$ and $W_2^{in}(q_2)$ are adapted to $\gamma_{21}$.
\begin{center}
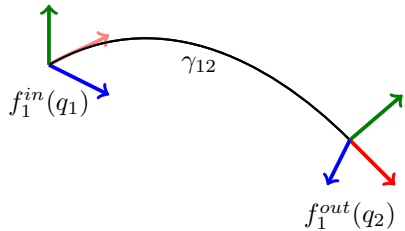

\begin{tikzpicture}
    \draw[thick] (0,2) to[in=135,out=30](4,1);
    \draw[red!50, ->, line width=0.5mm](0,2) to [out=30,in=210] (0.8,2.4);
    \draw[green!50!black, ->, line width=0.5mm](0,2)--(0,2.8);
    \draw[blue, ->, line width=0.5mm](0,2)--(0.8,1.6);
    \draw[thick] (0,2) to[in=135,out=30](4,1);
    \draw[red, ->, line width=0.5mm](4,1) to [out=315,in=135] (4.6,0.4);
    \draw[green!50!black, ->, line width=0.5mm](4,1)--(4.7,1.6);
    \draw[blue, ->, line width=0.5mm](4,1)--(3.7,0.4);
    \node at (0,1.5){$f_1^{in}(q_1)$};
    \node at (4,0){$f_1^{out}(q_2)$};
    \node at (2,2){$\gamma_{12}$};
\end{tikzpicture}
\captionof{figure}{Complex Translation Distance between $f^{in}$ and $f^{out}$}.
\label{Fig:6}
\end{center}
\begin{defn} \
\label{Def:complex_lambda_length}
\begin{enumerate}
\item
If $(\mathpzc{h}_1, L^O_1),(\mathpzc{h}_2, L^O_2)\in\mathfrak{H_D}$ have distinct centres, the \emph{complex lambda length} from $(\mathpzc{h}_1, L^O_1)$ to $(\mathpzc{h}_2, L^O_2)$ is
\[
\lambda_{12} = \exp \left( \frac{d}{2} \right),
\]
where $d$ is the complex distance from $f_1^{in}(q_1)$ to $f_2^{out}(q_2)$. 
\item
If $(\mathpzc{h}_1, W_1),(\mathpzc{h}_2, W_2)\in\mathfrak{H_D^S}$ have distinct centres, the \emph{complex lambda length} from $(\mathpzc{h}_1, W_1)$ to $(\mathpzc{h}_2, W_2)$ is
\[
\lambda_{12} = \exp \left( \frac{d}{2} \right),
\]
where $d$ is the complex distance from $W_1^{in}(q_1)$ to $W_2^{out}(q_2)$.
\end{enumerate}
If $\horo_1, \horo_2$ have common centre then in both cases $\lambda_{12} = 0$.
\end{defn}
See \reffig{6}. We abbreviate complex lambda length to \emph{lambda length}.

In the decorated case, $d$ is well defined modulo $2\pi i$, so $\lambda_{12}$ is a well defined complex number up to sign. In the spin-decorated case, $\lambda_{12}$ is a well defined complex number. In either case $|\lambda_{12}|$ is well defined.

Assume $\horo_1, \horo_2$ have distinct centres, so the geodesic $\gamma$ and the points $q_1, q_2$ exist. Writing the complex distance $d$ from $f_1^{in}(q_1)$ to $f_2^{out}(q_2)$ or $W_1^{in}(q_1)$ to $W_2^{out}(q_2)$ as $d = \rho + i \theta$ with $\rho, \theta \in \R$, then $\rho$ is the signed distance from  $q_1$ to $q_2$ along the oriented geodesic $\gamma_{12}$. When $\horo_1, \horo_2$ are disjoint, then $\rho$ is positive, and  gives the shortest distance between $\horo_1$ and $\horo_2$. When $\horo_1, \horo_2$ are tangent, $\rho=0$. When $\horo_1, \horo_2$ overlap, $\rho$ is negative. 

Setting $\lambda_{12} = 0$ when $\horo_1$ and $\horo_2$ have the same centre extends $\lambda$ to a continuous function $\mathfrak{H_D^S} \times \mathfrak{H_D^S} \To \C$, since when two horospheres (of fixed size, say, as they appear in the disc model) approach each other, their common perpendicular geodesic moves out to infinity and the length of the interval lying in the intersection of the horoballs becomes arbitrarily large, so that $\rho \rightarrow -\infty$ and hence $\lambda \rightarrow 0$.

These observations show that $\rho$ agrees with the signed undirected distance of \refdef{signed_undirected_distance}. Although $d$ is defined in a ``directed" way from $\horo_1$ to $\horo_2$, its real part $\rho$ does not depend on the direction. Its imaginary part, the angle $\theta$, is also undirected in the decorated case, but in the spin-decorated case $\theta$ does depend on the direction, as we see below in \reflem{lambda_antisymmetric}.

Taking moduli of both sides of the equations in \refdef{complex_lambda_length}, we obtain
\[
\left| \lambda_{12} \right| = \exp \left( \frac{\rho}{2} \right).
\]
which by \refeqn{horosphere_distance_from_Minkowski_inner_product} and \refeqn{horosphere_distance_from_spinor_inner_product} implies
\[
\left| \lambda_{12} \right|^2
= \frac{1}{2} \left\langle \h^{-1}(\horo_1), \h^{-1}(\horo_2) \right\rangle
= \left| \left\{ \kappa_1, \kappa_2 \right\} \right|^2
\]
where $\h^{-1}(\horo_i) \in L^+$ is the point on the light cone corresponding to the horosphere $\horo_i$ under $\h$, and $\kappa_i$ is a spinor corresponding to the horosphere $\horo_i$, i.e. such that $\h \circ \g \circ \f (\kappa_i) = \horo_i$. These equations include the modulus of the equation in \refthm{main_thm}.

We now show that lambda length is antisymmetric, in the sense that if we measure it between spin-decorated horospheres in reverse order, it changes by a sign. This is necessary for \refthm{main_thm}, since the spinor inner product $\{ \cdot, \cdot \}$ of \refdef{bilinear_form_defn} is also antisymmetric.
\begin{lem}
\label{Lem:lambda_antisymmetric}
Let $(\mathpzc{h}_i, W_i)\in\mathfrak{H_D^S}$, for $i=1,2$. Let $d_{ij}$ be the complex distance from $W_i^{in}(q_i)$ to $W_j^{out}(q_j)$, so that  $\lambda_{ij} = \exp \left( d_{ij}/2 \right)$ is the lambda length from $(\mathpzc{h}_i, W_i)$ to $(\mathpzc{h}_j, W_j)$. Then 
\[
d_{ij} = d_{ji} + 2 \pi i \quad \text{mod} \quad 4\pi i
\quad \text{and} \quad
\lambda_{ij} = -\lambda_{ji}.
\]
\end{lem}

\begin{proof}
First, if the horospheres have common centre then $\lambda_{ij} = \lambda_{ji} = 0$, by definition. So we may assume they have distinct centres. Then $\lambda_{ij} = \exp(d_{ij}/2)$, where $d_{ij}$ is the complex distance from $W_i^{in}$ to $W_j^{out}$ along $\gamma_{ij}$, the oriented geodesic from the centre of $\horo_i$ to the centre of $\horo_j$. Let $W_i^{in}, W_j^{out}$ project to the frames $f_i^{in}(\V_i), f_j^{out}(\V_j)$ of unit parallel vector fields $\V_i, \V_j$ on $\mathpzc{h}_i, \horo_j$.

Recall that $W_2^{in}$ is obtained from $W_2^{out}$ by a rotation of $\pi$ about $\V_2$, and $W_1^{out}$ is obtained from $W_1^{in}$ by a rotation of $-\pi$ about $\V_1$ (\refdef{associated_inward_outward_spindec}). Let $Y_1^{out}$ be obtained from $W_1^{in}$ by a rotation of $\pi$ about $\V_1$, so $Y_1^{out}$ and $W_1^{out}$ both project to $f_1^{out}$, but differ by a $2\pi$ rotation. 

Now the spin isometry which takes $W_1^{in}(p_1)$ to $W_2^{out}(p_2)$ also takes $Y_1^{out}(p_1)$ to $W_2^{in}(p_2)$, since the latter pair are obtained from the former pair by rotations of $\pi$ about $\V_1, \V_2$ respectively. So the complex distance from $W_1^{in}(p_1)$ to $W_2^{out}(p_2)$ along $\gamma_{12}$ is equal to the complex  distance from $W_2^{in}(p_2)$ to $Y_1^{out}(p_1)$ along $\gamma_{21}$. But this latter complex distance is equal to $d_{21} + 2\pi i$ (mod $4\pi i$), since $Y_1^{out}(p_1)$ and $W_1^{out}(p_1)$ differ by a $2\pi$ rotation. Thus we obtain $d_{12} = d_{21} + 2 \pi i$ mod $4\pi i$, hence $\lambda_{12} = - \lambda_{21}$ as desired.
\end{proof}

\subsection{Proof of \refthm{main_thm_2}}
\label{Sec:proof_main_thm}

The strategy of the proof of \refthm{main_thm_2} is to first prove it in simple cases, and then extend to the general case by equivariance. Before doing so, however, we first  establish how lambda lengths are invariant under $SL(2,\C)$.

\begin{lem}
\label{Lem:lambda_length_invariant_under_isometry}
Let $(\mathpzc{h}_i, W_i)\in\mathfrak{H_D^S}$ for $i=1,2$ and let $A \in SL(2,\C)$. Let $\lambda_{12}$ be the complex lambda length from $(\mathpzc{h}_1, W_1)$ to $(\mathpzc{h}_2, W_2)$, and let $\lambda_{A1,A2}$ be the complex lambda length from $A\cdot (\mathpzc{h}_1, W_1)$ to $A\cdot (\mathpzc{h}_2, W_2)$. Then $\lambda_{12} = \lambda_{A1,A2}$.
\end{lem}

\begin{proof}
As $A \in SL(2,\C)$, the universal cover of $\Isom^+ \hyp^3 \cong PSL(2,\C)$, $A$ is represented by a path of isometries $M_t \in PSL(2,\C)$, where $M_0$ is the identity and $M_1 = \pm A$. As in the definition of complex lambda length, let $\gamma_{12}$ be the oriented geodesic from the centre of $\horo_1$ to the centre of $\horo_2$, and let $q_i = \gamma_{12} \cap \horo_i$. Then the spin frames $W_1^{in} (q_1)$ and $W_2^{out} (q_2)$ are adapted to $\gamma_{12}$ and their complex distance $d$ satisfies $\lambda_{12} = \exp(d/2)$.

As each $M_t$ is an isometry, applying $M_t$ to the horospheres and spin frames involved yields a 1-parameter family of horospheres $M_t \cdot \horo_1, M_t \cdot \horo_2$ for $t \in [0,1]$, with mutually perpendicular geodesic $M_t \cdot \gamma_{12}$, intersecting the horospheres at points $q_1^t = M_t \cdot q_1$ and $q_2^t = M_t \cdot q_2$, at which there are spin frames $M_t \cdot W_1^{in} (q_1^t), M_t \cdot W_2^{out} (q_2^t)$ adapted to $M_t \cdot \gamma_{12}$. As $M_t$ is an isometry, the complex distance $d$ between the spin frames $M_t \cdot W_1^{in} (q_1^t)$ and $M_t \cdot W_2^{out} (q_2^t)$ remains constant.  Hence the lambda length $\lambda_{12} = \exp(d/2)$ also remains constant.

At time $t=1$, we arrive at the decorated horospheres $A \cdot (\horo_1, W_1)$ and $A \cdot (\horo_2, W_2)$. Their complex distance remains $d$, and their lambda length $\lambda_{A1,A2}$ remains equal to $\lambda = e^{d/2}$.
\end{proof}

\begin{lem}
\label{Lem:main_thm_for_10_and_01}
Let $\kappa_1 = (1,0)$ and $\kappa_2 = (0,1)$, and let $(\horo_1, W_1), (\horo_2, W_2) \in \mathfrak{H_D^S}(\U)$ be the corresponding spin-decorated horospheres under $\widetilde{\K}$. Then the lambda length from $(\horo_1, W_1)$ to $(\horo_2, W_2)$ is $1$.
\end{lem}

\begin{proof}
By \refprop{JIHGF_general_spin_vector}, $\mathpzc{h}_1$ is centred at $\infty$, at Euclidean height $1$, with spin decoration $W_1$ projecting to the decoration specified by $i$. Similarly, $\mathpzc{h}_2$ is centred at $0$, with Euclidean diameter $1$, and spin decoration $W_2$ projecting to the decoration north-pole specified by $i$. These two horospheres are tangent at $q = (0,0,1)$, and both spin decorations  $W_1^{in}$ and $W_2^{out}$ both project to the same frame at $q$, namely  $(-e_z,e_y,e_x)$. So the complex distance from $W_1^{in}(q)$ to $W_2^{out}(q)$ is $d = i\theta$, where the rotation angle $\theta$ is $0$ or $2\pi$ mod $4\pi$; we claim it is in fact $0$ mod $4\pi$.

To see this, consider the following path in $PSL(2,\C) \cong \Isom^+ \U$:
\[
M_t = \pm \begin{pmatrix} \cos t & -\sin t \\ \sin t & \cos t \end{pmatrix} \in PSL(2,\C), \quad \text{from} \quad t=0 \quad \text{to} \quad t=\frac{\pi}{2}.
\]
As an isometry of $\U$, each $M_t$ is a rotation by angle $2t$ about the oriented geodesic $\delta$ from $-i$ to $i$. Hence $M_t$ preserves each point on $\delta$, including $q$. Thus $M_t$ rotates $\horo_1$ about $\delta$ through to the horosphere $M_{\pi/2} \horo_1$, which is centred at $M_{\pi/2} (0) = \infty$ and passes through $q$, hence is $\horo_2$. Throughout this family of rotations, the point $q$ is preserved, as is the tangent vector at $q$ in the $y$-direction, which is positively tangent to $\delta$. In particular, over $t \in [0, \pi/2]$, the family of rotations $M_t$ rotates the frame of $W_1^{in}$ to the frame of $W_2^{in}$.

In fact, the path $M_t$ rotates the \emph{spin} frame of $W_1^{in}$ to the spin frame $W_2^{in}$. The path $M_t$ is a path in $PSL(2,\C)$ starting at the identity, and 
lifts to a unique path in $SL(2,\C)$ starting at the identity
\[
\widetilde{M_t} = \begin{pmatrix} \cos t & - \sin t \\ \sin t & \cos t \end{pmatrix}
\quad \text{from} \quad
\widetilde{M_0} = \begin{pmatrix} 1 & 0 \\ 0 & 1 \end{pmatrix}
\quad \text{to} \quad
A = \widetilde{M_{\frac{\pi}{2}}} = \begin{pmatrix} 0 & -1 \\ 1 & 0 \end{pmatrix}.
\]
Regarding $SL(2,\C)$ as a universal cover of $PSL(2,\C)$, $M_t$ is a path representing the spin isometry $A$. Note that $A \cdot (0,1) = (1,0)$, i.e. $A \cdot \kappa_1 = \kappa_2$. So by $SL(2,\C)$-equivariance (\refthm{main_thm_precise}), we have $A \cdot (\mathpzc{h}_1, W_1) = (\mathpzc{h}_2, W_2)$, and hence $A \cdot W_1^{in} = W_2^{in}$. 

Thus on the one hand $A \cdot W_1^{in} = W_2^{in}$. But on the other hand, $A$ is represented by the path $M_t$,  which rotates about the geodesic $\delta$ by an angle of $2t$,  for $t \in [0, \pi/2]$. Therefore $W_2^{in}(q)$ is obtained from $W_1^{in}(q)$ by a rotation of angle $\pi$ about $e_y$, the vector pointing along $\delta$.

Then, by \refdef{associated_inward_outward_spindec}, $W_2^{out}(q)$ is obtained from $W_2^{in}(q)$ by a rotation of angle $-\pi$ about $e_y$, i.e. by $-\pi$ about the oriented geodesic $\delta$.

Thus, from $W_1^{in}(q)$, we obtain $W_2^{in}(q)$ by a rotation of $\pi$ about $\delta$; and then obtain $W_2^{out}(q)$ by a rotation of $-\pi$ about $\delta$. So $W_1^{in}(q) = W_2^{out}(q)$, and the rotation angle $\theta$ is $0$ mod $4\pi$ as claimed. Then $d=0$ and $\lambda = \exp(d/2) = 1$.
\end{proof}

\begin{lem}
\label{Lem:main_thm_for_10_and_0D}
Let $0 \neq D \in \C$, and let $\kappa_1 = (1,0)$ and $\kappa_2 = (0,D)$. Let $(\horo_1, W_1), (\horo_2, W_2) \in \mathfrak{H_D^S}(\U)$ be the corresponding spin-decorated horospheres under $\widetilde{\K}$. Then the lambda length from $(\horo_1, W_1)$ to $(\horo_2, W_2)$ is $D$.
\end{lem}

\begin{proof}
The previous \reflem{main_thm_for_10_and_01} verified this statement when $D=1$. As there, $\horo_1$ is centred at $\infty$, of height $1$, with spin decoration $W_1$ projecting to the decoration specified by $i$. By \refprop{JIHGF_general_spin_vector}, $\horo_2$ is centred at $0$, with Euclidean height $|D|^{-2}$, and spin decoration $W_2$ projecting to the decoration north-pole specified by $i D^{-2}$. The common perpendicular geodesic $\gamma_{12}$ is the vertical line in $\U$ from $\infty$ to $0$, which intersects $\mathpzc{h}_1$ at $q_1 = (0,0,1)$ and $\mathpzc{h}_2$ at $q_2 = (0,0,|D|^{-2})$. Thus the signed distance from $q_1$ to $q_2$ along $\gamma$ is $\rho = 2 \log |D|$. The  rotation angle $\theta$ between decorations, measured with respect to $\gamma_{12}$ is $2 \arg D$, modulo $2\pi$. We will show that $\theta$  is in fact $2 \arg D$ modulo $4\pi$.

From \reflem{main_thm_for_10_and_01}, we know that when $D=1$, the points $q_1, q_2$ coincide, and the frames $W_1^{in}$ and $W_2^{out}$ coincide at this point. Denote the spin-decorated horosphere $\widetilde{\K} (0,1)$ by $(\horo_{2,{D=1}}, W_{2,{D=1}})$. We consider a spin isometry taking the $D=1$ case to the general $D$ case. 

Consider the following path $M_t$ in $PSL(2,\C)$ for $t \in [0,1]$, representing the spin isometry $A$:
\[
A = \begin{pmatrix} D^{-1} & 0 \\ 0 & D \end{pmatrix}
, \quad
M_t = \pm \begin{pmatrix} e^{-t \left( \log |D| + i \arg D \right)} & 0 \\ 0 & e^{t \left( \log |D| + i \arg D \right)} \end{pmatrix}
\]
Note $M_t$ effectively has diagonal entries $D^{-t}$ and $D^t$, we just make them precise using logarithm and argument. We can take, for instance, $\arg D \in [0, 2\pi)$. The path $M_t$ lifts to a path in $SL(2,\C)$ beginning at the identity and ending at $A$, so indeed $M_t$ represents $A$.

On the one hand, $A \cdot (0,1) = (0,D)$, so by equivariance (\refthm{main_thm_precise}), when applied to the corresponding horospheres, $A \cdot (\horo_{2,{D=1}}, W_{2,{D=1}}) = (\horo_2, W_2)$. 

On the other hand, each $M_t$ is a loxodromic isometry of $\U$, which translates along  $\gamma_{12}$ by signed distance $2t \log |D|$, and rotates around the oriented geodesic $\gamma_{12}$ by angle $2t \arg D$, for $t \in [0,1]$. So $A \cdot (\horo_{2,{D=1}}, W_{2,{D=1}}) = (\horo_2, W_2)$ is obtained from $(\horo_{2,{D=1}}, W_{2,{D=1}})$ by a translation along $\gamma_{12}$ of distance $2 \log |D|$, and rotation around $\gamma_{12}$ of angle $2 \arg D$.

Now from \reflem{main_thm_for_10_and_01}, the spin frames $W_1^{in} (q_1)$ and $W_{2,{D=1}}^{out} (q_1)$ coincide. From above,  $W_2^{out} (q_2)$ is obtained from $W_{2,{D=1}}^{out} (q_1)$ by a complex translation of $d = 2 \log |D| + 2 i \arg D$. Thus the  lambda length from $(\horo_1, W_1)$ to $(\horo_2, W_2)$ is  
\[
\lambda_{12} = e^{d/2} = \exp \left( \log |D| + i \arg(D) \right) = D.
\]
\end{proof}

We now state and prove a precise version of \refthm{main_thm_2}.
\begin{theorem}
\label{Thm:main_thm_2_precise}
Let $\kappa_1, \kappa_2 \in \C_\times^2$, and let $\widetilde{\K}(\kappa_1)= (\mathpzc{h}_1, W_1)$ and $\widetilde{\K}(\kappa_2)=(\mathpzc{h}_2, W_2)$ be the corresponding spin-decorated horospheres. Then the lambda length $\lambda_{12}$  from $(\mathpzc{h}_1, W_1)$ to $(\mathpzc{h}_2, W_2)$ is given by
\[
\lambda_{12} = \{\kappa_1, \kappa_2 \}.
\]
\end{theorem}

\begin{proof}
If $\kappa_1, \kappa_2$ are linearly dependent then one is a complex multiple of the other, and the two horospheres $\mathpzc{h}_1, \mathpzc{h}_2$ have the same centre. Then $\{\kappa_1, \kappa_2\} = \lambda_{12} = 0$. We can thus assume $\kappa_1, \kappa_2$ are linearly independent.

By \refthm{main_thm_precise}, $\widetilde{\K}$ is $SL(2,\C)$-equivariant. By \reflem{SL2C_by_symplectomorphisms}, the bilinear form $\{\cdot, \cdot \}$ is invariant under applying $A \in SL(2,\C)$ to spin vectors. By \reflem{lambda_length_invariant_under_isometry}, complex lambda length is invariant under applying $A \in SL(2,\C)$ to spin-decorated horospheres. So it suffices to show the desired equality after applying an element $A$ of $SL(2,\C)$ to both $\kappa_1, \kappa_2$ and $(\mathpzc{h}_1, W_1), (\mathpzc{h}_2, W_2)$. 

Since $\kappa_1, \kappa_2$ are linearly independent, we take $A$ to be the unique matrix in $SL(2,\C)$ such that $A\cdot\kappa_1 = (1,0)$ and $A\cdot\kappa_2 = (0,D)$ for some $D$. In fact then $D = \{ \kappa_1, \kappa_2\}$. To see this, note that $A$ is the inverse of the matrix with columns $\kappa_1$ and $\kappa_2/D$, with $D$ chosen so that $\det A = 1$. By definition of the bilinear form $\{ \cdot, \cdot \}$, we have $1 = \det A = \{ \kappa_1, \kappa_2/D \} = \frac{1}{D} \{\kappa_1, \kappa_2 \}$. Thus $D = \{ \kappa_1, \kappa_2\}$.

Thus, it suffices to prove the result when $\kappa_1 = (1,0)$ and $\kappa_2 = (0,D)$, i.e. that in this case the lambda length is $\{\kappa_1, \kappa_2\} = D$. This is precisely the result of \reflem{main_thm_for_10_and_0D}.
\end{proof}

\section{Applications}

\label{Sec:applications}

\subsection{Three-dimensional hyperbolic geometry}
\label{Sec:3d_hyp_geom}
\subsubsection{Ptolemy equation for spin-decorated ideal tetrahedra}

We now prove \refthm{main_thm_Ptolemy}. In fact, we prove the following slightly stronger theorem.
\begin{theorem}
Let $(\mathpzc{h}_i, W_i)\in\mathfrak{H_D^S}$ for $i=0,1,2,3$ be four spin-decorated horospheres in $\hyp^3$, and let $\lambda_{ij}$ be the lambda length from $(\mathpzc{h}_i, W_i)$ to $(\mathpzc{h}_j, W_j)$. Then
\[
\lambda_{01} \lambda_{23} + \lambda_{03} \lambda_{12} = \lambda_{02} \lambda_{13}.
\]
\end{theorem}

\begin{proof}
By \refthm{main_thm_precise}, each $(\mathpzc{h}_i, W_i)$ corresponds via $\widetilde{\K}$ to a unique $\kappa_i = (\xi_i, \eta_i) \in \C_\times^2$. Let $M\in\mathcal{M}_{2 \times 4}(\C)$ be the matrix whose $j^{\text{th}}$ column is $\kappa_j$. For  $i,j \in \{0,1,2,3\}$, let $M_{ij}\in\mathcal{M}_{2 \times 2}(\C)$ be the submatrix whose columns are $\kappa_i$ and $\kappa_j$ in order. By definition $\det M_{ij} = \{ \kappa_i, \kappa_j \}$ and by \refthm{main_thm_2_precise} this is also equal to $\lambda_{ij}$. Thus the claimed equation can be rewritten as
\[
\det M_{01} \det M_{23} + \det M_{03} \det M_{12} = \det M_{02} \det M_{12}
\]
which is a well known Pl\"{u}cker relation, as seen previously in \refeqn{Plucker_24}.
\end{proof}

As mentioned in introductory \refsec{Ptolemy_matrices}, this theorem  generalises Penner's Ptolemy equation for decorated ideal triangles in $\hyp^2$. When the four horosphere centres lie in a plane, Penner defines each $\lambda_{ij}$ to be our $\exp \left( \rho_{ij}/2 \right)$, by reference to distances only, without angles or decorations. As we will see in \refsec{spin_coherent_positivity} below, decorations can be chosen to obtain the same $\lambda_{ij}$.

Note that multiplying any $\kappa_i$, corresponding to $(\mathpzc{h}_i,W_i)$, with a scalar $\alpha\in\C$ also scales each term of \refeqn{ptolemy} containing index $i$. For instance, multiplying $\kappa_1$ by $\alpha$ scales $\lambda_{01},\lambda_{12}$ and $\lambda_{13}$. In this case, the choice of decorated horosphere at each vertex is like a choice of ``gauge".

\subsubsection{Shape parameters and Ptolemy equations}
\label{Sec:shape_parameters}

The following definition is standard in hyperbolic geometry; shape parameters for instance form the variables in Thurston's gluing equations \cite{Thurston_notes}.
\begin{defn}
\label{Def:shape_parameter}
Let $e$ be an edge of an ideal hyperbolic tetrahedron $\Delta$. Suppose the endpoints of $e$ are placed at $0$ and $\infty$ in $\U$, and the remaining ideal vertices are placed at $1$ and $z_e \in \C$ such that $\Im z_e > 0$. Then $z_e$ is the \emph{shape parameter} of $\Delta$ along $e$.
\end{defn}

As is well known, the shape parameters at opposite edges of an ideal tetrahedron are equal. And if one shape parameter $z$ is known, then the others can be denoted $z, z', z''$ such that
\[
z' = \frac{1}{1-z} \quad \text{and} \quad z'' =\frac{z-1}{z}.
\]
In particular, we have the equation
\begin{equation}
\label{Eqn:shapeparamter}
    z + z'^{-1} = 1,
\end{equation}
which also holds under cyclic permutations $(z, z', z'')\mapsto (z', z'', z)$.

\begin{theorem}
    Numbering the ideal vertices of $\Delta$ by $0, 1, 2, 3$ as in \reffig{7}, let the shape parameter of edge $ij$ by $z_{ij}$. Choose $(\mathpzc{h}_i,W_i)\in\mathfrak{H_D^S}$ at ideal vertex $i$ and let $\lambda_{ij}$ be the complex lambda length from $(\mathpzc{h}_i,W_i)$ to $(\mathpzc{h}_j,W_j)$. Then
\begin{equation}z_{01}=z_{23}=\frac{\lambda_{02}\lambda_{13}}{\lambda_{03}\lambda_{12}},\ \ z_{02}=z_{13}=-\frac{\lambda_{03}\lambda_{12}}{\lambda_{01}\lambda_{23}},\ \ z_{03}=z_{12}=\frac{\lambda_{01}\lambda_{23}}{\lambda_{02}\lambda_{13}}.\label{Eqn:sp}\end{equation}
\end{theorem}
\begin{center}
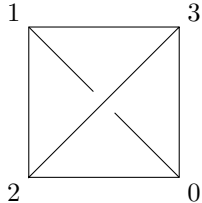

\begin{tikzpicture}[scale=2]
\draw (1,0)--(0,1);
\fill[white] (0.5,0.5) circle (0.1 cm);
    \draw (0,0)--(1,0)--(1,1)--(0,1)--(0,0)--(1,1);
    \node at (-0.1,-0.1){2};
    \node at (-0.1,1.1){1};
    \node at (1.1,-0.1){0};
    \node at (1.1,1.1){3};
\end{tikzpicture}
\captionof{figure}{Tetrahedron with vertices labeled $0,1,2,3$}.
\label{Fig:7}
\end{center}
\begin{proof}
Transforming a spin-decorated ideal tetrahedron by a spin isometry leaves all shape parameters and complex lambda lengths invariant. Noting the orientation of \reffig{7}, we may thus place the ideal vertices $0, 1, 2, 3$ at $0, \infty, z, 1\in\partial\hyp^3$ respectively, so that $z = z_{01} = z_{23}$, $z' = z_{02} = z_{13}$ and $z'' = z_{03} = z_{12}$. Then multiplying a spinor $\kappa_i$ corresponding to $(\mathpzc{h}_i,W_i)$ by a complex scalar $\alpha$, the homogeneous expressions in lambda lengths in \refeqn{sp} are invariant. Thus it suffices to prove the claim for any choice of spin decoration, or spinor, at each vertex. By \refthm{main_thm_precise}, each of the following spinors $\kappa_j$ corresponds under $\widetilde{\K}$ to a spin-decorated horosphere at ideal vertex $j$:
\[
\kappa_0 = (0, 1), \quad 
\kappa_1 = (1, 0), \quad
\kappa_2 = (z, 1), \quad
\kappa_3 = (1, 1).
\]
By \refthm{main_thm} we calculate the complex lambda lengths to be 
\[
\lambda_{01} =-1, \quad 
\lambda_{02} =-z, \quad
\lambda_{03} =-1, \quad
\lambda_{12} =1, \quad
\lambda_{13} =1, \quad
\lambda_{23} =z-1.
\]
This immediately gives the first equation. Permuting labels $(0,1,2,3)\mapsto (0,2,3,1)$ on $\Delta$ and similarly the indices on each $\lambda$, and using the antisymmetry of $\lambda$, gives the subsequent equations.
\end{proof}

Substituting $z = z_{01} = z_{23}$ and $z' = z_{02} = z_{13}$ with the corresponding expression in the $\lambda_{ij}$ from \refeqn{sp}, the equation \refeqn{shapeparamter} relating shape parameters becomes precisely the Ptolemy equation \refeqn{ptolemy}. From this perspective, arguably the shape parameter equation is a Ptolemy equation in disguise!

\subsection{Real spinors and 2-dimensional hyperbolic geometry}
\label{Sec:real_spinors_H2}

When a spinor $(\xi, \eta)$ has real coordinates, the corresponding horosphere has centre $\xi/\eta$ in $\R \cup \{\infty\}$, which forms the boundary of a copy of the upper half-plane model of 2-dimensional hyperbolic space, inside the upper half-space model $\U$ of 3-dimensional hyperbolic space. With this observation, we can use real spinors to describe 2-dimensional hyperbolic geometry. Accordingly, we regard $\hyp^2$ as the upper half plane embedded in $\U = \{(x,y,z) \mid z>0 \}$ at  $y=0$.

In this section we explore some of this lower-dimensional geometry. In \refsec{horocycles_decorations} we adapt spin decorations to 2 dimensions, and show they correspond to real spinors. In \refsec{spin_coherent_positivity} we discuss the relationship between ordering points around the boundary circle of $\hyp^2$, and positive lambda lengths. In \refsec{triangulations} we briefly mention triangulations of polygons, and in \refsec{Ford} we discuss how Ford circles arise from integer spinors.

\subsubsection{Horocycles and their decorations}
\label{Sec:horocycles_decorations}

Any horocycle $\mathpzc{h}$ in $\hyp^2$ extends to a unique horosphere $\widetilde{\mathpzc{h}}$ in $\hyp^3$, with centre in $\R \cup \{\infty\}$, and $\widetilde{\horo}$ can be given a spin decoration as usual.
\begin{defn}
\label{Def:planar_spin_decoration}
Let $\mathpzc{h}$ be a horocycle in $\hyp^2$. A \emph{planar spin decoration} on $\mathpzc{h}$ is a spin decoration $(W^{in}, W^{out})$ on the corresponding $\widetilde{\mathpzc{h}} \in \mathfrak{H}(\hyp^3)$ whose inward and outward spin frames $W^{in}, W^{out}$ project to frames specified by $i$.
\end{defn}
The requirement that frames be specified by $i$ determines a unique decoration on $\widetilde{\horo}$, but not a unique spin decoration. There are precisely two planar spin decorations on $\horo$, corresponding to two spinors $\pm \kappa$.

\begin{lem}
\label{Lem:real_spinor_planar_decoration}
Let $\kappa = (\xi,\eta) \in \C_\times^2$ be a spin vector corresponding to $(\mathpzc{h}, W)\in\mathfrak{H_D^S}$. Then $(\mathpzc{h},W)$ arises from a planar spin decoration if and only if $(\xi,\eta) \in \R_\times^2$.
\end{lem}

\begin{proof}
A spin decoration on a horosphere is planar if and only if its centre is real and its decoration is specified by $i$. Since the centre is $\xi/\eta$ and its frames are (north-pole) specified by $i/\eta^2$ (if $\eta \neq 0$) or $i \xi^2$ (if $\eta = 0$), the decoration is planar iff $\xi/\eta$  is real and $\eta^2$ is positive, or $\xi/\eta = \infty$ (i.e. $\eta = 0)$ and $\xi^2$ is positive; i.e. iff $\xi/\eta$  is real and $\eta$ is real, or $\eta = 0$ and $\xi$ is nonzero real. This is precisely the case when $\xi$ and $\eta$ are both real and not both zero.
\end{proof}

Suppose we have two horocycles $\mathpzc{h}_1, \mathpzc{h}_2$ in $\hyp^2$ with planar spin decorations. Then the complex distance $d_{12}$ from $W_1^{in}$ to $W_2^{out}$ will be of the form $\rho_{12} + i \theta_{12}$ where $\theta_{12}$ is $0$ or $2\pi$ mod $4\pi$. The lambda length $\lambda_{12}$ from $(\mathpzc{h}_1, W_1)$ to $(\mathpzc{h}_2, W_2)$ will be positive or negative accordingly.

Since lambda lengths are antisymmetric (\reflem{lambda_antisymmetric}), $\lambda_{21} = - \lambda_{12}$, the two complex distances $d_{12}, d_{21}$ satisfy $d_{21} = d_{12} + 2\pi i$ mod $4\pi i$. So in particular, $\theta_{12}$ is $0$ mod $4\pi$ precisely when $\theta_{21}$ is $2\pi$ mod $4\pi$, and vice versa.

\subsubsection{Ordering and positivity in the hyperbolic plane}
\label{Sec:spin_coherent_positivity}

We will use a specific orientation of $\hyp^2$. Since we regard  $\hyp^2$ is a surface embedded in $\U$, each point of $\hyp^2$ has two normal directions $\pm e_y$. By \refdef{planar_spin_decoration}, the positive $y$-direction which corresponds to real spinors, and so we make the following definition.
\begin{defn}
We orient $\hyp^2$ by the normal direction $e_y$. We orient $\partial \hyp^2$ as the boundary of $\hyp^2$.
\end{defn}
Note that this is the opposite of the usual orientation on the upper half plane, and in particular $\partial \hyp^2 = \R \cup \{\infty\}$ is oriented along $\R$ in the negative direction.

\begin{defn}
\label{Def:in_order}
A collection of distinct points $z_1, \ldots, z_d \in \partial \hyp^2$ are \emph{in order} around $\partial \hyp^2$ if they lie in cyclic order on the oriented circle $\partial \hyp^2$. 
\end{defn}
Since $\partial \hyp^2$ is oriented in the negative direction on $\R$, ``in order" here roughly means ``in decreasing order", but as $\partial \hyp^2$ is a circle, we allow for cyclic permutation. Thus $z_1, \ldots, z_d \in \R$ are in order if $z_1 > z_2 > \cdots > z_d$. But they are also in order if, for instance, 
$\infty = z_k > k_{k+1} > \cdots > z_d > z_1 > z_2 > \cdots > z_{k-1}$ for some $k$.

We will define, as usual, an ideal $d$-gon in the hyperbolic plane to have vertices at infinity, but we prefer the vertices to be in order around the circle at infinity. we allow various relaxations of usual assumptions, or degenerations, of those vertices, as follows.

\begin{defn} \
\label{Def:d-gons}
\begin{enumerate}
\item
An \emph{ideal $d$-gon} is a collection of distinct points $z_1, \ldots, z_d$ in $\partial \hyp^2$, labelled in order around $\partial \hyp^2$.
\item
A \emph{generalised ideal $d$-gon} is a collection of points $z_1, \ldots, z_d$ in $\partial \hyp^2$. 
\item
A generalised ideal $d$-gon is \emph{degenerate} if all its vertices coincide. Otherwise it is \emph{non-degenerate}.
\end{enumerate}
\end{defn}
Thus, a generalised ideal $d$-gon may not have its vertices labelled in order around $\partial \hyp^2$, and some or even all vertices may coincide. When all of them coincide, it is degenerate.

For the rest of this section, let $\horo_1, \ldots, \horo_d$ be horocycles in $\hyp^2$, with centres $p_1, \ldots, p_d \in \partial \hyp^2 = \R \cup \{\infty\}$, and let $\gamma_{ij}$ be the oriented geodesic from $p_i$ to $p_j$. Let $q_{i,j}$ be the intersection of $\gamma_{ij}$ with $\horo_i$.

It turns out that when the $p_i$ are distinct and in order as per \refdef{in_order}, one can choose planar spin decorations on the $\horo_i$ in a ``coherent" way.
\begin{prop}
\label{Prop:distinct_vertices_coherent_decorations}
Suppose $p_1, \ldots, p_d \in \partial \hyp^2$ are distinct and in order around $\partial \hyp^2$. Then there exist planar spin decorations $W_1, \ldots, W_d$ on $\horo_1, \ldots, \horo_d$ such that the lambda lengths $\lambda_{ij}$ from $(\horo_i, W_i)$ to $(\horo_j, W_j)$ satisfy
\[
\lambda_{ij} > 0 \quad \text{when} \quad i<j
\quad \text{and} \quad
\lambda_{ij} < 0 \quad \text{when} \quad i>j.
\]
\end{prop}

\begin{proof}
Recall there are two choices for the spin decoration on each horocycle. Choose the spin decoration on $\horo_1$ arbitrarily, and then choose the spin decoration on each other $\horo_j$ so that $\lambda_{1j} >0$. This means that each complex distance $d_{1j} = \rho_{1j} + i \theta_{1j}$ is chosen to be real, with each $\theta_{1j} = 0$ mod $4\pi$, and the spin frame $W_j$ on each $\widetilde{\horo_j}$ is chosen so that $W_1^{in} (p_{1,j})$ and $W_j^{out} (p_{j,1})$ are related by parallel translation  along $\gamma_{1j}$; there is no rotation, as $\theta_{1j} = 0$ mod $4\pi$. 

It suffices then to prove that for any $i<j$ we have $\lambda_{ij} > 0$, or equivalently, $\theta_{ij} = 0$ mod $4\pi$. We give two proofs of this claim, geometric then algebraic.

For the first proof, consider the oriented geodesics $\gamma_{1i}, \gamma_{1j}, \gamma_{ij}$. By \refdef{associated_inward_outward_spindec}, $W_i^{out}$ is obtained from $W_i^{in}$ at any point on $\horo_i$ by a rotation of $-\pi$ about the decoration on $W_i$, which points in $y$-direction. Moreover, as decorations arise from parallel line fields, $W_i^{in}$ is obtained at any point of $\horo_i$ from any other by parallel translation on $\horo_i$; similarly $W_j^{in}$ is obtained at any point of $\horo_j$ from any other by parallel translation on $\horo_j$.

Thus, we may proceed from $W_i^{in}(p_{i,j})$ to $W_j^{out}(p_{j,i})$ in several steps, as follows:
\begin{enumerate}
\item Parallel translation along $\horo_i$ from $p_{i,j}$ to $p_{i,1}$ takes $W_i^{in}(p_{i,j})$ to $W_i^{in}(p_{i,1})$;
\item Rotation by $-\pi$ about the $y$-direction takes $W_i^{in}(p_{i,1})$ to $W_i^{out}(p_{i,1})$;
\item Parallel translation along $\gamma_{1i}$ from $p_{i,1}$ to $p_{1,i}$ takes $W_i^{out}(p_{i,1})$ to $W_1^{in}(p_{1,i})$;
\item Parallel translation along $\horo_1$ from $p_{1,i}$ to $p_{1,j}$ takes $W_1^{in}(p_{1,i})$ to $W_1^{in}(p_{1,j})$;
\item Parallel translation along $\gamma_{1j}$ from $p_{1,j}$ to $p_{j,1}$ takes $W_1^{in}(p_{1,j})$ to $W_j^{out}(p_{j,1})$;
\item Parallel translation along $\horo_j$ from $p_{j,1}$ to $p_{j,i}$ takes $W_j^{out}(p_{j,1})$ to $W_j^{out}(p_{j,i})$.
\end{enumerate}
Now, we note that the result of these translations must be the same, or differ by $2\pi$, from the parallel translation along $\gamma_{ij}$, which takes $W_i^{in}(p_{i,j})$ to $W_j^{out}(p_{j,i})$. By tracking the rotation of the frame (which may be done by hand, or by measuring angles), we see that the results are the same. See \reffig{spin_coherence}. Thus $\theta_{ij} = 0$ mod $4\pi$, and $\lambda_{ij} > 0$.

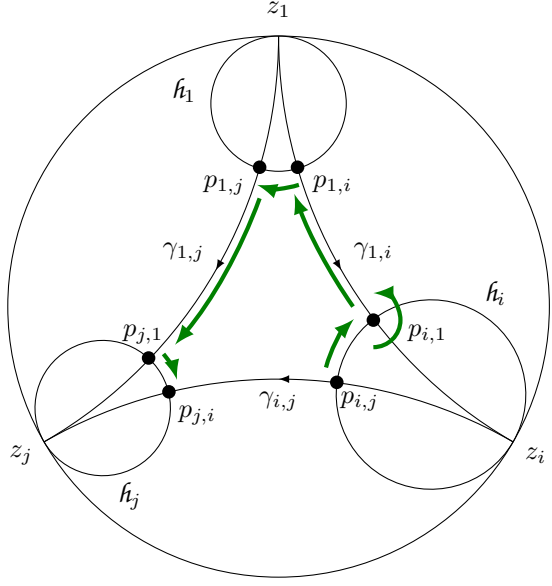
\begin{figure}
\begin{center}
 \begin{tikzpicture}[scale=1.8]
    \draw[black] (0,0) circle (2cm);
    \draw[black] (0,1.5) circle (0.5cm);
    \draw[black] (-1.732, -1) arc (210:-150:0.5cm);
    \draw[black] (1.732,-1) arc (-30:330:0.7 cm);
\begin{scope}[decoration={markings,mark=at position 0.5 with {\arrow{latex}}}] 
    \draw[black, postaction={decorate}] (0,2) arc (0:-60:3.464 cm);
    \draw[black, postaction={decorate}] (0,2) arc (180:240:3.464 cm);
    \draw[black, postaction={decorate}] (1.732,-1) arc (60:120:3.464cm);
\end{scope}
    \node[black] at (0,2.2) {$z_1$};
    \node[black] at (-1.9,-1.1) {$z_j$};
    \node[black] at (1.9,-1.1) {$z_i$};
    \fill[black] (-0.14,1.03) circle (0.05 cm);
    \node[black] at (-0.4,0.9) {$p_{1,j}$};
    \fill[black] (0.14,1.03) circle (0.05 cm);
    \node[black] at (0.4,0.9) {$p_{1,i}$};
    \fill[black] (-0.96,-0.38) circle (0.05 cm);
    \node[black] at (-1,-0.2) {$p_{j,1}$};
    \fill[black] (-0.81,-0.63) circle (0.05 cm);
    \node[black] at (-0.6,-0.8) {$p_{j,i}$};
    \fill[black] (0.7, -0.1) circle (0.05 cm);
    \node[black] at (1.1,-0.2) {$p_{i,1}$};
    \fill[black] (0.43,-0.56) circle (0.05 cm);
    \node[black] at (0.6,-0.7) {$p_{i,j}$};
    \draw[-latex, ultra thick, green!50!black] (0.35,-0.45) arc (165:135:0.8 cm);
    \draw[-latex, ultra thick, green!50!black] (0.7, -0.3) arc (-90:90:0.2 cm);
    \draw[-latex, ultra thick, green!50!black] (0.55, 0) arc (215:200:3.6 cm);
    \draw[-latex, ultra thick, green!50!black] (0.15, 0.9) arc (-72:-108:0.5 cm);
    \draw[-latex, ultra thick, green!50!black] (-0.14,0.8) arc (-20:-40:3.6 cm);
    \draw[-latex, ultra thick, green!50!black] (-0.85, -0.35) arc (40:15:0.5 cm);
    \node[black] at (0,-0.7) {$\gamma_{i,j}$};
    \node[black] at (0.7,0.4) {$\gamma_{1,i}$};
    \node[black] at (-0.7,0.4) {$\gamma_{1,j}$};
    \node[black] at (-0.7,1.6) {$\horo_1$};
    \node[black] at (1.6, 0.1) {$\horo_i$};
    \node[black] at (-1.1, -1.4) {$\horo_j$};
\end{tikzpicture} 
\caption{Constructing spin-coherent planar spin decorations.}
\label{Fig:spin_coherence}
\end{center}
\end{figure}

We give the second proof in the case where none of the $p_i$ is $\infty$, leaving the rest to the reader. Let $\kappa_j = (\xi_j, \eta_j)$ be spinors corresponding to the $(\horo_j, W_j)$. Then we have $p_j = \xi_j/\eta_j$ so
\begin{equation}
\label{Eqn:pi-pj}
p_i - p_j 
= \frac{\xi_i}{\eta_i} - \frac{\xi_j}{\eta_j}
= \frac{\xi_i \eta_j - \xi_j \eta_i}{\eta_i \eta_j}
= \frac{ \{\kappa_i, \kappa_j \} }{ \eta_i \eta_j}
= \frac{\lambda_{ij}}{\eta_i \eta_j}
\end{equation}
and similarly
\begin{equation}
\label{Eqn:p1-pi}
p_1 - p_i = \frac{\lambda_{1i}}{\eta_1 \eta_i}, \quad
p_1 - p_j = \frac{\lambda_{1i}}{\eta_1 \eta_j}.
\end{equation}
Recall $i<j$ and the $p_i$ are in order. We consider two cases, $p_i < p_j$ and $p_j < p_i$. 

If $p_i < p_j$ then since $p_1, p_i, p_j$ are in order we have $p_i < p_1 < p_j$. So $p_1 - p_i$ and $p_1 - p_j$ have opposite signs. Hence by equation \refeqn{p1-pi}, $\eta_i$ and $\eta_j$ have opposite signs. Hence in equation \refeqn{pi-pj}, the left hand side is negative and the right hand side denominator is negative. Thus the numerator is positive, $\lambda_{ij} > 0$.

If $p_i > p_j$ then the ordering of $p_1, p_i, p_j$ tells us that $p_1 < p_j < p_i$ or $p_j < p_i < p_1$. Either way, $p_1 - p_i$ and $p_1 - p_j$ have the same sign, so by \refeqn{p1-pi}, $\eta_i$ and $\eta_j$ have the same sign. Hence in \refeqn{pi-pj}, the left hand side is positive and the right hand side denominator is positive. Thus $\lambda_{ij} > 0$ again.
\end{proof}

We formalise this ``coherence" of spin decorations, or positivity of lambda lengths, in the following definitions.
\begin{defn}
\label{Def:spin-coherent}
Let $\horo_1, \ldots, \horo_d$ be horocycles with distinct centres $p_1, \ldots, p_d$. A set of planar spin decorations on the $\horo_i$ is \emph{spin-coherent} if
\[
\left\{
\begin{array}{rcl}
\theta_{ij} &=& 0\mod4\pi \text{ when } i<j \\
\theta_{ij} &=& 2\pi\mod4\pi \text{ when } i>j,
\end{array}
\right.
\quad \text{or equivalently,} \quad
\left\{
\begin{array}{rcl}
\lambda_{ij} &>& 0 \text{ when } i<j \\
\lambda_{ij} &<& 0 \text{ when } i>j.
\end{array}
\right.
\]
\end{defn}

\begin{defn}
A collection of spinors $\kappa_1, \ldots, \kappa_d$ is \emph{totally positive} if they are all real, and satisfy
\begin{equation}
\label{Eqn:total_positivity}
\{ \kappa_i, \kappa_j \} > 0 \quad \text{when} \quad i<j
\quad \text{and} \quad
\{ \kappa_i, \kappa_j \} < 0 \quad \text{when} \quad i>j.
\end{equation}
\end{defn}
This notion of \emph{total positivity} arises in other contexts in mathematics and physics. See, for example, \cite{ABCGPT16, Lusztig94, Postnikov06, Williams21}).

\begin{lem}
\label{Lem:positive_spin-coherent}
A collection of spinors $\kappa_1, \ldots, \kappa_d$ is totally positive and if and only if the corresponding spin-decorated horospheres have planar spin decorations and are spin-coherent.
\end{lem}

\begin{proof}
By \reflem{real_spinor_planar_decoration}, the $\kappa_j$ are real if and only if the $\horo_j$ have planar spin decorations. By \refthm{main_thm}, the $\{\kappa_i, \kappa_j\}$ satisfy the total positivity condition \refeqn{total_positivity} if and only if the $\lambda_{ij}$ satisfy the spin-coherence condition of \refdef{spin-coherent}.
\end{proof}

Note that in \refdef{spin-coherent} of spin coherence, the centres $p_i$ are distinct, but there is no requirement that the vertices lie in order around $\partial \hyp^2$. However, it turns out that the spin coherence condition \emph{implies} that the $p_i$ are in order around $\partial \hyp^2$. The following converse to \refprop{distinct_vertices_coherent_decorations} is proved in \cite{Mathews_Spinors_horospheres}.

\begin{prop}
\label{Prop:spin-coherent_in_order}
If $\horo_1, \ldots, \horo_d$ have spin-coherent planar spin decorations, then their centres $p_1, \ldots, p_d$ are in order around $\partial \hyp^2$. In other words, $p_1, \ldots, p_d$ form an ideal $d$-gon as in \refdef{d-gons}.
\end{prop}

\subsubsection{Triangulations of polygons}
\label{Sec:triangulations}

Let $d \geq 3$ be an integer and consider an ideal $d$-gon with a triangulation $T$. This $T$ consists of $d-3$ diagonals, so the polygon and the triangulation together contain $2d-3$ edges.

If each ideal vertex of the polygon has a planar spin decoration, we then have a lambda length associated to each edge of the polygon and the triangulation. Labelling the vertices of the $d$-gon from $1$ to $d$, we write $\lambda_{ij}$ for the lambda length from vertex $i$ to vertex $j$.

A natural operation on a triangulation is a diagonal flip, which replaces two adjacent triangles with two different triangles. These two adjacent triangles share an edge and together they form an ideal $4$-gon.  If the four vertices involved in a diagonal flip have labels $i,j,k,l$ in order, in the diagonal flip, lambda length $\lambda_{ik}$ is replaced with $\lambda_{jl}$ or vice versa. The variables are related by the Ptolemy equation
\[
\lambda_{ij} \lambda_{kl} + \lambda_{il} \lambda_{jk} = \lambda_{ik} \lambda_{jl}.
\]
Considering such variables and equations leads to the notion of a cluster algebra: see e.g. \cite{Fomin_Shapiro_Thurston08, Fomin_Thurston18, Williams14} for details.

If we allow the ideal vertices to vary along $\partial \hyp^2$ and the horospheres to also vary, the space of such choices gives a description of what Penner calls the \emph{decorated Teichm\"{u}ller space} $\widetilde{T}$. Then each $\lambda_{ij}$ can be regarded as a function $\widetilde{T} \To \R^{+}$ and indeed the collection of $\lambda_{ij}$ from the edges of the polygon and the diagonals of a triangulation yields a system of coordinates on $\widetilde{T}$ and a homeomorphism $\widetilde{T} \To \R^{2d-3}$. See e.g. \cite{Penner87, Penner12} for further details.

\subsubsection{Ford circles}
\label{Sec:Ford}

Suppose we have a spinor $\kappa = (\xi, \eta)$ whose coordinates $\xi, \eta$ are \emph{integers}. As a real spinor, by \reflem{real_spinor_planar_decoration}, $\kappa$ corresponds to a horocycle with a planar spin decoration. Its centre is at the \emph{rational} number $\xi/\eta$, and its Euclidean diameter is $1/\eta^2$. If we require that $(\xi, \eta) \in \Z^2_\times$ be \emph{relatively prime}, then we obtain a family of horocycles with precisely one horocycle centred at each point of $\Q \cup \{\infty\}$. They are known as \emph{Ford circles} \cite{Ford38}. See \reffig{Ford_farey}.

A standard fact about Ford circles is that the circles at $a/b$ and $c/d$, written as fractions in simplest form, are tangent if and only if $|ad-bc| = 1$. We see this immediately from the spinor-horosphere correspondence, since the corresponding real spinors $\kappa_1 = (a,b)$ and $\kappa_2 = (c,d)$ satisfy $\{\kappa_1, \kappa_2\} = ad-bc = \pm 1$ so the corresponding spin-decorated horospheres have lambda length $\lambda = \pm 1$. Thus the horocycles are tangent.

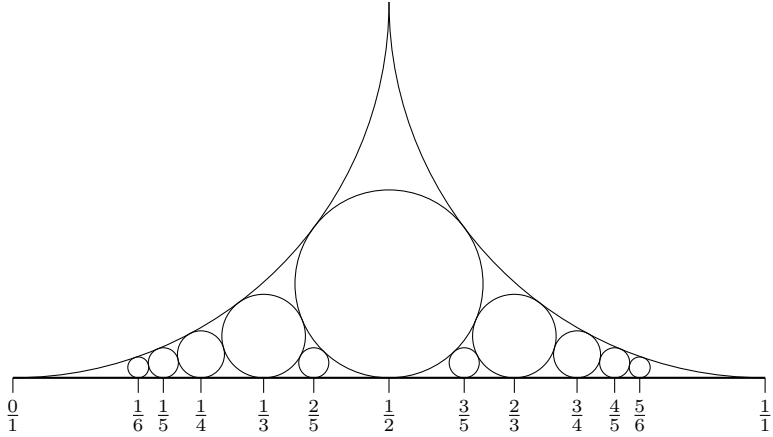
\begin{figure}
\begin{center}
\begin{tikzpicture}[scale=10]
    \draw[black, thick] (0,0) -- (1,0);
    \draw[black] (0,0) arc (-90:0:0.5 cm);
    \draw[black] (0,-0.02) -- (0,0);
    \node at (0,-0.05) {$\frac{0}{1}$};
    \draw[black] (1,0) arc (-90:-180:0.5 cm);
    \draw[black] (1,-0.02) -- (1,0);
    \node at (1,-0.05) {$\frac{1}{1}$};
    \draw[black] (0.5,0) arc (-90:270:0.125 cm);
    \draw[black] (0.5,-0.02) -- (0.5,0);
    \node at (0.5,-0.05) {$\frac{1}{2}$};
    \draw[black] ( {1/3} ,0) arc (-90:270: { 1/18 });
    \draw[black] ( {1/3} ,-0.02) -- ( {1/3} ,0);
    \node at ( {1/3} ,-0.05) {$\frac{1}{3}$};
    \draw[black] ( {2/3}, 0) arc (-90:270: {1/18} );
    \draw[black] ( {2/3} ,-0.02) -- ( {2/3} ,0);
    \node at ( {2/3} ,-0.05) {$\frac{2}{3}$};
    \draw[black] ( {1/4}, 0) arc (-90:270: {1/32} );
    \draw[black] ( {1/4} ,-0.02) -- ( {1/4} ,0);
    \node at ( {1/4} ,-0.05) {$\frac{1}{4}$};
    \draw[black] ( {3/4}, 0) arc (-90:270: {1/32} );
    \draw[black] ( {3/4} ,-0.02) -- ( {3/4} ,0);
    \node at ( {3/4} ,-0.05) {$\frac{3}{4}$};
    \draw[black] ( {1/5}, 0) arc (-90:270: {1/50} );
    \draw[black] ( {1/5} ,-0.02) -- ( {1/5} ,0);
    \node at ( {1/5} ,-0.05) {$\frac{1}{5}$};
    \draw[black] ( {2/5}, 0) arc (-90:270: {1/50} );
    \draw[black] ( {2/5} ,-0.02) -- ( {2/5} ,0);
    \node at ( {2/5} ,-0.05) {$\frac{2}{5}$};
    \draw[black] ( {3/5}, 0) arc (-90:270: {1/50} );
    \draw[black] ( {3/5} ,-0.02) -- ( {3/5} ,0);
    \node at ( {3/5} ,-0.05) {$\frac{3}{5}$};
    \draw[black] ( {4/5}, 0) arc (-90:270: {1/50} );
    \draw[black] ( {4/5} ,-0.02) -- ( {4/5} ,0);
    \node at ( {4/5} ,-0.05) {$\frac{4}{5}$};
    \draw[black] ( {1/6}, 0) arc (-90:270: {1/72} );
    \draw[black] ( {1/6} ,-0.02) -- ( {1/6} ,0);
    \node at ( {1/6} ,-0.05) {$\frac{1}{6}$};
    \draw[black] ( {5/6}, 0) arc (-90:270: {1/72} );
    \draw[black] ( {5/6} ,-0.02) -- ( {5/6} ,0);
    \node at ( {5/6} ,-0.05) {$\frac{5}{6}$};
\end{tikzpicture}
\caption{Ford circles and Farey fractions.}
\label{Fig:Ford_farey}
\end{center}
\end{figure}

More generally, the hyperbolic distance $\rho$ between Ford circles based at $a/b$ and $c/d$, again given as fractions in simplest form, is given from $e^\rho = |\lambda|^2$ as
\[
\rho = 2 \log |\lambda| = 2 \log \left| \{ \kappa_1, \kappa_2 \} \right|
= 2 \log \left| ad - bc \right|.
\]
This was observed by McShane and Sergiescu in \cite{McShane_Eisenstein, McShane_Sergiescu}.

If we have two tangent Ford circles, based at fractions in simplest form $a/b$ and $c/d$, then there is a unique Ford circle between them, tangent to both. This Ford circle is based at the \emph{mediant} or \emph{Farey sum} of $a/b$ and $c/d$, namely
\[
\frac{a+c}{b+d},
\]
which has corresponding spinor $\kappa_1 + \kappa_2$. Its tangency with the circles corresponding to $\kappa_1, \kappa_2$ then follows from its lambda length, or spinor inner product, with $\kappa_1$ and $\kappa_2$ being $\pm 1$:
\[
\{\kappa_1 + \kappa_2, \kappa_1 \} = \{ \kappa_2, \kappa_1 \} = \pm 1, \quad
\{\kappa_1 + \kappa_2, \kappa_2 \} = \{ \kappa_1, \kappa_2 \} = \pm 1.
\]
Here we  used the antisymmetry of the spinor inner product, and the tangency of horospheres corresponding to $\kappa_1, \kappa_2$.

\subsection{Polygons, polyhedra and matrices}
\label{Sec:polygons_polyhedra_matrices}

\subsubsection{Matrices and ideal hyperbolic polygons}

If we have an ideal $d$-gon in $\hyp^2$ (\refdef{d-gons}), with horospheres and planar spin decorations on each ideal vertex (\refdef{planar_spin_decoration}), then we have a real spinor associated to each vertex (\reflem{real_spinor_planar_decoration}). We can arrange these spinors as the columns of a matrix. Various properties of this matrix correspond to properties of corresponding ideal polygon.

\begin{lem}
\label{Lem:ideal_polygons_matrices}
The map which takes planar spin-decorated horocycles $(\horo_1, W_1), \ldots, (\horo_d, W_d)$ in $\hyp^2$ to matrices whose columns are the corresponding collection of spin vectors $\kappa_1, \ldots, \kappa_d \in \R^2_\times$, induces the following correspondences:
\begin{gather*}
\left\{ \begin{array}{c}
\text{Generalised ideal $d$-gons in $\hyp^2$} \\
\text{with planar spin decorations}
\end{array} \right\}
\cong 
\left\{ \begin{array}{c}
\text{$2 \times d$ real matrices} \\
\text{with no zero columns}
\end{array} \right\} \\
\left\{ \begin{array}{c}
\text{Nondegenerate generalised ideal $d$-gons in $\hyp^2$} \\
\text{with planar spin decorations}
\end{array} \right\}
\cong 
\left\{ \begin{array}{c}
\text{$2 \times d$ real matrices of rank $2$} \\
\text{with no zero columns}
\end{array} \right\} \\
\left\{ \begin{array}{c}
\text{Spin-coherent ideal $d$-gons in $\hyp^2$} 
\end{array} \right\}
\cong 
\left\{ \begin{array}{c}
\text{$2 \times d$ real matrices} \\
\text{where all $2 \times 2$ submatrices} \\
\text{have positive determinant}
\end{array} \right\} 
\end{gather*}
\end{lem}

\begin{proof}
In a generalised ideal $d$-gon (\refdef{d-gons}(ii)) there is no restriction on where the vertices lie, so any collection of $\kappa_i \in \R^2_\times$ provide a matrix. Generalised ideal $d$-gons with planar spin decorations are thus bijective with $2 \times d$ real matrices, none of whose columns consist entirely of zeroes.

Adding the requirement of nondegeneracy (\refdef{d-gons}(iii)) means that two of the $\xi_i/\eta_i$ are distinct, so that there are two linearly independent columns in the matrix, so it has rank 2. 

Adding the requirement that all vertices are  distinct means that all $\xi_i/\eta_i$ are distinct, so all $2 \times 2$ submatrices of $M$ have rank 2.

The requirement of spin-coherence corresponds to all $2 \times 2$ matrices have positive determinant (\reflem{positive_spin-coherent}), and in this case the ideal vertices are in order around $\partial \hyp^2$, so we have an ideal $d$-gon (\refprop{spin-coherent_in_order}).
\end{proof}

\subsubsection{Matrices and ideal hyperbolic polyhedra}

We can find similar correspondences between complex matrices and polyhedra in $3$ dimensions. This notion of ideal polyhedron is somewhat unsatisfactory, forgetting all its structure except for listing out its vertices. Nonetheless it is analogous to \refdef{d-gons} and allows us to produce corresponding matrices. However, since $\partial \hyp^3 \cong S^2$, there is no natural way to order vertices.

\begin{defn} \
\begin{enumerate}
\item
An \emph{ideal $d$-hedron} is a collection of distinct points $p_1, \ldots, p_d$ in $\partial \hyp^3$.
\item
A \emph{generalised ideal $d$-hedron} is a collection of points $p_1, \ldots, p_d$ in $\partial \hyp^3$.
\item
A generalised ideal $d$-hedron is \emph{degenerate} if all of its vertices coincide. Otherwise it is \emph{non-degenerate}.
\end{enumerate}
\end{defn}

\begin{lem}
The map which takes spin-decorated horospheres $(\horo_1, W_1), \ldots, (\horo_d, W_d)$ in $\hyp^3$ to matrices whose columns are the corresponding collection of spin vectors $\kappa_1, \ldots, \kappa_d \in \C^2_\times$, induces the following correspondences:
\begin{gather*}
\left\{ \begin{array}{c}
\text{Generalised ideal $d$-hedra in $\hyp^3$} \\
\text{with spin decorations}
\end{array} \right\}
\cong
\left\{ \begin{array}{c}
\text{$2 \times d$ complex matrices} \\
\text{with no zero columns}
\end{array} \right\} \\
\left\{ \begin{array}{c}
\text{Nondegenerate generalised ideal $d$-hedra in $\hyp^3$} \\
\text{with spin decorations}
\end{array} \right\}
\cong 
\left\{ \begin{array}{c}
\text{$2 \times d$ complex matrices} \\
\text{of rank $2$} \\
\text{with no zero columns}
\end{array} \right\} \\
\left\{ \begin{array}{c}
\text{Ideal $d$-hedra in $\hyp^3$} \\
\text{with spin decorations}
\end{array} \right\}
\cong
\left\{ \begin{array}{c}
\text{$2 \times d$ complex matrices} \\
\text{where all $2 \times 2$ submatrices} \\
\text{have nonzero determinant}
\end{array} \right\}
\end{gather*}
\end{lem}

\begin{proof}
A generalised ideal $d$-hedron yields a matrix with all columns in $\C_\times^2$, without further restriction. Just as in two dimensions, adding the requirement of nondegeneracy corresponds to the matrix having rank 2. All vertices being distinct corresponds to all $2 \times 2$ submatrices having nonzero determinant.
\end{proof}

\subsubsection{Ideal hyperbolic polygons up to isometry}

After characterising various spaces of hyperbolic ideal polygons as spaces of matrices in \reflem{ideal_polygons_matrices}, we now consider such objects \emph{up to isometry}.

Just as $SL(2,\C)$ is the double and universal cover of $PSL(2,\C)$, $SL(2,\R)$ is the double and universal cover of $PSL(2,\R)$. The latter is the orientation-preserving isometry group of $\hyp^2$. Since the polygons we have considered all have various types of spin decorations, we consider them under the action of spin isometries preserving $\hyp^2$, i.e. of $SL(2,\R)$. 

\begin{defn} 
A \emph{spin isometry class} of a set $X$ of generalised ideal $d$-gons with planar spin decorations is an orbit of the $SL(2,\R)$ action on $X$.
\end{defn}
In this definition, $X$ could for example be any of the sets of $d$-gons considered in \reflem{ideal_polygons_matrices}.

The correspondence between spin vectors and spin-decorated horospheres is $SL(2,\C)$-equivariant, and $SL(2,\R)$ is the subgroup of $SL(2,\C)$ which preserves real spin vectors, and also preserves $\hyp^2$. Given a collection of real spin vectors $\kappa_1, \ldots, \kappa_d$ forming a $2 \times d$ matrix $M$, an $A \in SL(2,\C)$ acts simultaneously on all the $\kappa_j$ in the matrix multiplication $AM$. The correspondences of \reflem{ideal_polygons_matrices} then yield correspondences between  spin isometry classes of various types of decorated $d$-gons, and the orbit spaces of the $SL(2,\R)$ action on various sets of matrices. For instance, we have the following.
\begin{gather}
\label{Eqn:corresondences_orbit_spaces_first}
\left\{ \begin{array}{c}
\text{Spin isometry classes of} \\
\text{generalised ideal $d$-gons in $\hyp^2$} \\
\text{with planar spin decorations}
\end{array} \right\}
\cong 
SL(2,\R) \backslash
\left\{ \begin{array}{c}
\text{$2 \times d$ real matrices} \\
\text{with no zero columns}
\end{array} \right\} \\
\left\{ \begin{array}{c}
\text{Spin isometry classes of} \\
\text{nondegenerate generalised ideal $d$-gons in $\hyp^2$} \\
\text{with planar spin decorations}
\end{array} \right\}
\cong 
SL(2,\R) \backslash
\left\{ \begin{array}{c}
\text{$2 \times d$ real matrices} \\
\text{of rank $2$} \\
\text{with no zero columns}
\end{array} \right\} \\
\left\{ \begin{array}{c}
\text{Spin isometry classes of} \\
\text{spin-coherent ideal $d$-gons in $\hyp^2$} 
\end{array} \right\}
\cong 
SL(2,\R) \backslash
\left\{ \begin{array}{c}
\text{$2 \times d$ real matrices} \\
\text{where all $2 \times 2$ submatrices} \\
\text{have positive determinant}
\end{array} \right\} 
\label{Eqn:corresondences_orbit_spaces_last}
\end{gather}

In the case of spin-coherent ideal $d$-gons, the spin isometry classes reduce to isometry classes in the standard sense.
\begin{prop}
\label{Prop:spin_isometry_spin-coherent_ideal_d-gons}
Let $d \geq 3$. 
There is a bijection between 
spin isometry classes of spin-coherent ideal $d$-gons in $\hyp^2$
and
isometry classes of ideal $d$-gons in $\hyp^2$, decorated with a horocycle at each vertex.
\end{prop}
That is,
\begin{gather}
\label{Eqn:spin_isometry_isometry}
\left\{ \begin{array}{c}
\text{Spin isometry classes of} \\
\text{$d$-gons in $\hyp^2$} \\
\text{with spin-coherent decorations}
\end{array} \right\}
\cong
\left\{ \begin{array}{c}
\text{Isometry classes of} \\
\text{ideal $d$-gons in $\hyp^2$} \\
\text{decorated with horocycles} 
\end{array} \right\}
\end{gather}
The latter is Penner's decorated Teichm\"{u}ller space of a disc with $d$ boundary punctures.

\begin{proof}
As discussed in \refsec{horocycles_decorations}, an ideal $d$-gon in $\hyp^2$, decorated with horocycles, has two possible planar spin decorations at each ideal vertex, for a total of $2^d$ choices. However, as discussed in the proof of \refprop{distinct_vertices_coherent_decorations}), if the decorations are spin-coherent, then once we choose one planar spin decoration at one ideal vertex, the rest are determined. Thus there are precisely $2$ spin-coherent decorations. Under the action of $SL(2,\R)$ the negative identity will identify planar spin decorations in pairs. So there will be $2^{d-1}$ spin isometry classes of planar spin decorations, but only one spin isometry class of spin-coherent decorations.

If two spin-coherent ideal $d$-gons are related by a spin isometry, then the underlying isometry of $\hyp^2$ maps the underlying ideal $d$-gons and horocycles to each other. So we have a map from the left to right set of \refeqn{spin_isometry_isometry}. By \refprop{distinct_vertices_coherent_decorations}, we can always find spin-coherent decorations on an ideal $d$-gon, so this map is surjective. As argued above, there is only one spin isometry class of spin-coherent decorations, so the map is injective. 
\end{proof}

Orbit spaces of $SL(2,\R)$ action on sets of matrices are quite similar to \emph{Grassmannians}. The real Grassmannian $\Gr_\R (2,d)$ is the space of all real $2$-planes in $\R^d$, and it can be described as
\begin{equation}
\label{Eqn:Grassmannian_description}
\Gr_\R (2,d) \cong GL(2,\R) \backslash \left\{ \begin{array}{c}
\text{$2 \times d$ real matrices} \\
\text{of rank $2$}
\end{array} \right\}.
\end{equation}
To see this, we associate to a $2 \times d$ matrix its row space, a subspace of $\R^d$. If $M$ has rank 2 then this row space is a 2-plane in $\R^d$, hence an element of $\Gr_\R (2,d)$. This gives a map from matrices to $\Gr_\R (2,d)$. But multiplying a $2 \times d$ matrix on the left by a matrix in $GL(2,\R)$ preserves its row space, and indeed the $GL(2,\R)$-orbits of matrices correspond to points in $\Gr_\R (2,d)$. 

The correspondences in \refeqn{corresondences_orbit_spaces_first}--\refeqn{corresondences_orbit_spaces_last} describe various classes of spin isometry classes of ideal $d$-gons in similar terms to the Grassmannian in \refeqn{Grassmannian_description}. Various statements along these lines are made in \cite{Mathews_Spinors_horospheres}, including with complex matrices and polyhedra. It is also worth pointing out that determinants of $2 \times 2$ submatrices of the $2 \times d$ matrix, which in our context are lambda lengths, correspond to \emph{Pl\"{u}cker coordinates} on the Grassmannian.

Combining \refeqn{corresondences_orbit_spaces_last} and \refeqn{spin_isometry_isometry} gives an algebraic description of the decorated Teichm\"{u}ller space of a punctured disc
\[
SL(2,\R) \backslash
\left\{ \begin{array}{c}
\text{$2 \times d$ real matrices} \\
\text{where all $2 \times 2$ submatrices} \\
\text{have positive determinant}
\end{array} \right\} 
\cong
\left\{ \begin{array}{c}
\text{Isometry classes of} \\
\text{ideal $d$-gons in $\hyp^2$} \\
\text{decorated with horocycles} 
\end{array} \right\},
\]
and in fact, as shown in \cite{Mathews_Spinors_horospheres}, this is the \emph{affine cone} on the \emph{positive Grassmannian} $\Gr^+(2,d)$.

\newpage

\appendix

\section{Notation}
\label{Sec:Notation}

\begin{tabular}{ll}
\toprule
\multicolumn{2}{l}{\textbf{General}} \\
$D_p f (v)$ & Derivative of function $f$ at point $p$ in direction $v$ \\
$T_p M$ & Tangent space to manifold $M$ at point $p$ \\
$\f,\g,\h,\i,\j$ & Maps from spinors, to Hermitian matrices, to light cone, to horospheres \\
$\h_\partial$ & Simplification of $\h$ mapping to $\partial \hyp$ \\
$\F,\G,\H,\I,\J$ & Maps from spinors, to flags, to decorated horospheres \\
$\widetilde{\mathbf{F}}, \widetilde{\mathbf{G}}, \widetilde{\mathbf{H}}, \widetilde{\mathbf{I}}, \widetilde{\mathbf{J}}$ & Spin lifts of $\mathbf{F},\mathbf{G},\mathbf{H},\mathbf{I},\mathbf{J}$, maps from spinors, to spin flags, to spin-decorated horospheres\\
$\k_\partial, \k, \K, \widetilde{\K}$ & Compositions of $\f,\g,\h_\partial,\i,\j$, and $\f,\g,\h,\i,\j$, and $\F,\G,\H,\I,\J$, and $\widetilde{\mathbf{F}}, \widetilde{\mathbf{G}}, \widetilde{\mathbf{H}}, \widetilde{\mathbf{I}}, \widetilde{\mathbf{J}}$. \\
\midrule
\multicolumn{2}{l}{\textbf{Algebra, matrices}} \\
$\R, \R^{0+}, \R^+$ & Reals, non-negative reals, positive reals \\
$i$ & Square root of $-1$ \\
$\alpha, \beta$ & General complex numbers \\
$a,b,c,d$ & General real numbers \\
$A, A'$ & Matrices in $SL(2,\C)$ \\
$M,N$ & General matrices \\
$\mathcal{M}_{m\times n}(\mathbb{F})$ & $m \times n$ matrices with entries in field $\mathbb{F}$  \\
$\mathbb{F}_\times$ & Nonzero elements of field / ring / vector space $\mathbb{F}$  \\
$\HH$ & $2 \times 2$ Hermitian matrices \\
$\HH_0$ & $2 \times 2$ Hermitian matrices with determinant $0$ \\
$\HH_0^{0+}$ & $2 \times 2$ Hermitian matrices with determinant $0$ and trace $\geq 0$ \\
$\HH_0^+$ & $2 \times 2$ Hermitian matrices with determinant $0$ and trace $>0$ \\
$S$ & Hermitian matrix \\
$\underline{S}$ & Image of Hermitian matrix in quotient space \\
$\underline{B}$ & Basis of quotient space of Hermitian matrices \\
\midrule
\multicolumn{2}{l}{\textbf{Spin vectors}} \\
$\kappa$ & Spin vector  \\
$\xi, \eta$ & Coordinates of spin vector (elements of $\C$)  \\
$\nu$ & Tangent vector to $\C^2$ \\
$\{ \cdot, \cdot \}$ & Inner product of spin vectors \\
$\ZZ$ & Map $\C^2 \To \C^2$ in definition of $\F$ giving direction of flag \\
$J$ & Complex linear map in definition of $\ZZ$ \\
$\p$ & Projection $\C^2_\times \To S^3$ (\reflem{Stereo_Hopf_p}) \\
\midrule
\multicolumn{2}{l}{\textbf{Minkowski geometry}} \\
$L$ & Light cone (\refdef{light_cones}) \\
$L^{0+}$ & Non-negative light cone (\refdef{light_cones}) \\
$L^+$ & Positive light cone  (\refdef{light_cones}) \\
$p = (T,X,Y,Z)$ & Point in Minkowski space $\R^{1,3}$ \\
$\langle \cdot, \cdot \rangle$ & Inner product (signature $+---$) on Minkowski space $\R^{1,3}$ \\
$\pi_{XYZ}$ & Projection to $XYZ$ 3-plane \\
$V,W$ & Subspaces of $\R^{1,3}$ \\
$\S^+$ & Future celestial sphere (\refdef{celestial_sphere}) \\
$\Pi$ & Lightlike $3$-plane in $\R^{1,3}$ \\
$e_1, e_2, e_3$ & Orthonormal basis determined by spinor (\reflem{orthonormal_basis_from_spinor}) \\
\bottomrule
\end{tabular}

\begin{tabular}{ll}
\toprule
\multicolumn{2}{l}{\textbf{Hyperbolic geometry}} \\
$\hyp^n$ & Hyperbolic $n$-space (model-independent) \\
$\partial \hyp^n$ & Boundary at infinity of hyperbolic $n$-space  \\
$\Disc$, $\partial \Disc$ & Conformal ball model of $\hyp^3$, boundary at infinity \\
$\U, \partial \U$ & Upper half space model of $\hyp^3$, boundary at infinity \\
$x,y,z$ & Coordinates in upper half space model \\
$\hyp$ & Hyperboloid model of $\hyp^3$ \\
$\mathpzc{h}$ & Horosphere \\
$d = \rho + i \theta$ & Complex distance between horospheres \\
$\rho$ & Signed distance between horospheres \\
$\theta$ & Angular distance between horosphere decorations \\
$\lambda$, $\lambda_{ij}$  & Lambda-length, between horospheres indexed by $i,j$ \\
$\mathfrak{H}, \mathfrak{H}(\hyp), \mathfrak{H}(\Disc)$ & Set of all horospheres in $\hyp^3, \hyp, \Disc$ (\refdef{set_of_horospheres}) \\
$\Delta$ & Ideal tetrahedron in $\hyp^3$ \\
$\zeta_i$ & Ideal points, ideal vertices \\
$E$ & Edge of tetrahedron \\
$z_e$ & Shape parameter of tetrahedron along edge $e$ (\refdef{shape_parameter}) \\
$P_\alpha$ & Parabolic matrix in $SL(2,\C)$ \refeqn{P} \\
$P$ & parabolic subgroup of $SL(2,\C)$ or $PSL(2,\C)$ \\
\midrule
\multicolumn{2}{l}{\textbf{Flags}} \\
$\mathcal{F_P}$; $\mathcal{F_P}(\HH), \mathcal{F_P}(\R^{1,3})$ & Set of pointed null flags;  in $\HH$, in $\R^{1,3}$ \\
$(p,V), [[p,v]]$ & Pointed null flag in $\HH$ or $\R^{1,3}$ (\refdef{pointed_null_flag}, \refdef{null_flag_in_Minkowski}) \\
$\mathcal{F_P^O}$; $\mathcal{F_P^O}(\HH), \mathcal{F_P^O}(\R^{1,3})$ & Set of pointed oriented null flags; in $\HH$, in $\R^{1,3}$ \\
$(p,V,o), [[p,v]]$ & Pointed oriented null flag (\refdef{pointed_oriented_null_flag}, \refdef{pv_notation_PONF}) \\
$\mathcal{SF_P^O}(\HH), \mathcal{SF_P^O}(\R^{1,3})$
& Spin flags (\refdef{covers_of_flags}) \\
\midrule
\multicolumn{2}{l}{\textbf{Frames, decorations}} \\
$\Fr$ & Frame bundle over $\hyp^3$ (\refdef{Fr}) \\
$\Spin$ & Spin frame bundle over $\hyp^3$ (\refdef{Fr}) \\
$f = (f_1, f_2, f_3)$ & Frame  \\
$\widetilde{f}$ & Spin frame  \\
$L^O$ & Oriented line field on horosphere \\
$\mathfrak{H_D^O}$ & Set of overly decorated horospheres (\refdef{overly_decorated_horosphere}) \\
$L^O_P$ & Oriented parallel line field on horosphere \\
$\mathfrak{H_D}$ & Set of decorated horospheres (\refdef{decorated_horosphere}) \\
$N^{in}, N^{out}$ & Inward, outward normal vector fields to horosphere (\refdef{horosphere_normals}) \\
$\V$ & Unit parallel vector field on horosphere \\
$f^{in}(\V), f^{out}(V)$ & Inward, outward frame fields of vector field $V$ (\refdef{inward_outward_frame_fields}) \\
$W^{in}, W^{out}$ & Inward, outward spin decoration (\refdef{associated_inward_outward_spindec}) \\
$\mathfrak{H_D^S}$ & Set of spin-decorated horospheres (\refdef{spin-decorated_horospheres}) \\
\bottomrule
\end{tabular}

\small

\bibliography{spinref}
\bibliographystyle{amsplain}

\end{document}

%% file: complex_lambda_lengths_v5.pdf_tex
\begingroup%
  \makeatletter%
  \providecommand\color[2][]{%
    \errmessage{(Inkscape) Color is used for the text in Inkscape, but the package 'color.sty' is not loaded}%
    \renewcommand\color[2][]{}%
  }%
  \providecommand\transparent[1]{%
    \errmessage{(Inkscape) Transparency is used (non-zero) for the text in Inkscape, but the package 'transparent.sty' is not loaded}%
    \renewcommand\transparent[1]{}%
  }%
  \providecommand\rotatebox[2]{#2}%
  \newcommand*\fsize{\dimexpr\f@size pt\relax}%
  \newcommand*\lineheight[1]{\fontsize{\fsize}{#1\fsize}\selectfont}%
  \ifx\svgwidth\undefined%
    \setlength{\unitlength}{260.73690339bp}%
    \ifx\svgscale\undefined%
      \relax%
    \else%
      \setlength{\unitlength}{\unitlength * \real{\svgscale}}%
    \fi%
  \else%
    \setlength{\unitlength}{\svgwidth}%
  \fi%
  \global\let\svgwidth\undefined%
  \global\let\svgscale\undefined%
  \makeatother%
  \begin{picture}(1,0.71300705)%
    \lineheight{1}%
    \setlength\tabcolsep{0pt}%
    \put(0,0){\includegraphics[width=\unitlength,page=1]{complex_lambda_lengths_v5.pdf}}%
    \put(0.49475175,0.62370917){\color[rgb]{0,0,0}\makebox(0,0)[lt]{\lineheight{1.25}\smash{\begin{tabular}[t]{l}$\gamma$\end{tabular}}}}%
    \put(0,0){\includegraphics[width=\unitlength,page=2]{complex_lambda_lengths_v5.pdf}}%
    \put(0.72095237,0.54438302){\color[rgb]{0,0,0}\makebox(0,0)[lt]{\lineheight{1.25}\smash{\begin{tabular}[t]{l}$\mathpzc{h}_2$\end{tabular}}}}%
    \put(0,0){\includegraphics[width=\unitlength,page=3]{complex_lambda_lengths_v5.pdf}}%
    \put(0.27929939,0.11836984){\color[rgb]{0,0,0}\makebox(0,0)[lt]{\lineheight{1.25}\smash{\begin{tabular}[t]{l}$\mathpzc{h}_1$\end{tabular}}}}%
    \put(0,0){\includegraphics[width=\unitlength,page=4]{complex_lambda_lengths_v5.pdf}}%
    \put(0.36207469,0.38161453){\color[rgb]{0,0,0}\makebox(0,0)[lt]{\lineheight{1.25}\smash{\begin{tabular}[t]{l}$\rho$\end{tabular}}}}%
    \put(0.54701809,0.48778086){\color[rgb]{0,0,0}\makebox(0,0)[lt]{\lineheight{1.25}\smash{\begin{tabular}[t]{l}$\theta$\end{tabular}}}}%
    \put(0,0){\includegraphics[width=\unitlength,page=5]{complex_lambda_lengths_v5.pdf}}%
    \put(0.45621154,0.02205662){\color[rgb]{0,0,0}\makebox(0,0)[lt]{\lineheight{1.25}\smash{\begin{tabular}[t]{l}$p_1$\end{tabular}}}}%
    \put(0.41594102,0.68939613){\color[rgb]{0,0,0}\makebox(0,0)[lt]{\lineheight{1.25}\smash{\begin{tabular}[t]{l}$p_2 = \infty$\end{tabular}}}}%
    \put(0.42019581,0.53848394){\color[rgb]{0,0,0}\makebox(0,0)[lt]{\lineheight{1.25}\smash{\begin{tabular}[t]{l}$q_2$\end{tabular}}}}%
    \put(0.50105884,0.33559074){\color[rgb]{0,0,0}\makebox(0,0)[lt]{\lineheight{1.25}\smash{\begin{tabular}[t]{l}$q_1$\end{tabular}}}}%
    \put(0,0){\includegraphics[width=\unitlength,page=6]{complex_lambda_lengths_v5.pdf}}%
  \end{picture}%
\endgroup%